\include{macpap2}
\documentclass[12pt,a4paper]{ociamthesis}
\usepackage{amsfonts,amscd,amsthm,amsgen,amsmath,amssymb,epsfig}
\usepackage[all]{xy}
\graphicspath{{./}{Figs/}{Plots/}}
\newtheorem{prop}{Proposition}[section]
\newtheorem{thm}{Theorem}[section]

\newtheorem{lemp}{Lemma}[prop]
\newtheorem{conj}{Conjecture}[section]

\newtheorem{corp}{Corollary}[prop]
\theoremstyle{remark}
\newtheorem{rem}{Remark}[section]

\theoremstyle{definition}
\newtheorem{defn}{Definition}[section]
\newcommand{\leftexp}[2]{{\vphantom{#2}}^{#1}{#2}}
\title{$G_2$ Geometry And \\ Integrable Systems}
\author{David Baraglia}
\college{Hertford College}
\degree{Doctor of Philosophy}
\degreedate{Trinity 2009}
\renewcommand{\crest}{\beltcrest}
\usepackage[pdftex]{hyperref}
\begin{document}


\baselineskip=18pt plus1pt

\setcounter{secnumdepth}{3}
\setcounter{tocdepth}{3}

\maketitle

\begin{abstract}
We study the Hitchin component in the space of representations of the fundamental group of a Riemann surface into a split real Lie group in the rank $2$ case. We prove that such representations are described by a conformal structure and class of Higgs bundle we call cyclic and we show cyclic Higgs bundles correspond to a form of the affine Toda equations. We also relate various real forms of the Toda equations to minimal surfaces in quadrics of arbitrary signature. In the case of the Hitchin component for ${\rm PSL}(3,\mathbb{R})$ we provide a new proof of the relation to convex $\mathbb{RP}^2$-structures and hyperbolic affine spheres. For ${\rm PSp}(4,\mathbb{R})$ we prove such representations are the monodromy for a special class of projective structure on the unit tangent bundle of the surface. We prove these are isomorphic to the convex-foliated projective structures of Guichard and Wienhard.

We elucidate the geometry of generic $2$-plane distributions in $5$ dimensions, work which traces back to Cartan. Nurowski showed that there is an associated signature $(2,3)$ conformal structure. We clarify this as a relationship between a parabolic geometry associated to the split real form of $G_2$ and a conformal geometry with holonomy in $G_2$. Moreover in terms of the conformal geometry we prove this distribution is the bundle of maximal isotropics corresponding to the annihilator of a spinor satisfying the twistor-spinor equation.

The moduli space of deformations of a compact coassociative submanifold $L$ in a $G_2$ manifold is shown to have a natural local embedding as a submanifold of $H^2(L,\mathbb{R})$. We consider $G_2$-manifolds with a $T^4$-action of isomorphisms such that the orbits are coassociative tori and prove a local equivalence to minimal $3$-manifolds in $\mathbb{R}^{3,3}\cong H^2(T^4,\mathbb{R})$ with positive induced metric. By studying minimal surfaces in quadrics we show how to construct minimal $3$-manifold cones in $\mathbb{R}^{3,3}$ and hence $G_2$-metrics from a set of affine Toda equations. The relation to semi-flat special Lagrangian fibrations and the Monge-Amp\`ere equation is explained.

\end{abstract}

\begin{acknowledgements}
I would like to thank my supervisor Professor Nigel Hitchin for his supervision and guidance over the years. He has been tremendously supportive and encouraging throughout my studies.

I am indebted to the University of Adelaide for awarding me the George Murray scholarship which financed my study in Oxford and without which I could not have come to Oxford. I would also like to thank my supervisors in Adelaide Nicholas Buchdahl and Mathai Varghese for their supervision and their continuing support over the years.

\end{acknowledgements}

\begin{romanpages}    
\tableofcontents
\end{romanpages}


\chapter{Introduction}
The space of representations of the fundamental group $\pi_1(\Sigma)$ of a compact Riemann surface into a simple Lie group is a rich source of geometry. If we restrict attention to the case that $\Sigma$ has genus $g >1$ and consider representations into ${\rm PSL}(2,\mathbb{R})$, there is a component in the space of such representations consisting of the Fuchsian representations. These correspond to constant negative curvature metrics by means of an identification $\Sigma \simeq \mathbb{H} / \pi_1(\Sigma)$ where $\mathbb{H}$ is the upper half-plane. The space of Fuchsian representations is then identified with Teichm\"uller space.\\

For groups other than ${\rm PSL}(2,\mathbb{R})$ one might hope to similarly find geometric structures on $\Sigma$ from which the representations arise. To be more specific let $H$ be the adjoint form of a simple, split real Lie group. Hitchin \cite{hit1} has identified a component of the space of representations of $\pi_1(\Sigma)$ into $H$ that is contractible and naturally contains a copy of Teichm\"uller space. It is for representations in this space, which following Labourie \cite{lab} we call the {\em Hitchin component}, that we seek a geometric interpretation.
Our approach succeeds in the case that $H$ has rank $2$ while for higher rank it gives only a hint as to what such representations may describe.\\

Associated to a representation in the Hitchin component is a Higgs bundle which is determined by a series of holomorphic differentials. We prove that when all but the highest differential vanishes the Higgs bundle takes on a special form. The corresponding Higgs bundle equations relate to a real form of the affine Toda equations. We then consider in generality the affine Toda equations and prove their relation to harmonic maps.\\

In Chapter \ref{chap2} we study Higgs bundles and their relation to harmonic maps and the affine Toda equations. We begin in Section \ref{higgsandharm} with a review of Higgs bundles and the relationship with harmonic maps. In Section \ref{hitcomp} we introduce the Hitchin component and the construction of Higgs bundles for representations in this component. The link with the Toda equations comes about through studying a special class of Higgs bundles that we call cyclic. We introduce cyclic Higgs bundles in Section \ref{speccase} and prove their equivalence to the Toda equations with particular reality conditions. In Section \ref{sectoda} we propose a general form of the affine Toda equations with a reality condition that covers a wide range of instances of the equations. In Section \ref{msaate} we define superconformal and superminimal surfaces in quadrics of arbitrary signature and show how they relate to the Toda equations and to Higgs bundles.\\

In Chapter \ref{chap3} we then use our results on cyclic Higgs bundles to understand the geometry of representations in the Hitchin component for rank $2$ groups. That is for ${\rm PSL}(3,\mathbb{R})$, ${\rm PSp}(4,\mathbb{R})$ and the split real form of $G_2$. In such cases we find that all such representations can be related to a set of Toda equations which allows us to understand in some detail the geometry underlying the representations. In the rank $2$ case all such representations are described by two holomorphic differentials a quadratic differential and a higher differential. In Section \ref{qdae} we examine how the quadratic differential is closely related to the space of complex structure on the surface and based on work by Labourie \cite{lab} we show it is sufficient to consider representations where the quadratic differential vanishes.

In the case of ${\rm PSL}(3,\mathbb{R})$ the Hitchin component is known to correspond to convex projective structures and hyperbolic affine spheres. In Section \ref{crp2s} we directly construct the projective and affine structures using Higgs bundles.

For ${\rm PSp}(4,\mathbb{R})$ it is known that such representations correspond to the convex-foliated projective structures of Guichard and Wienhard \cite{gui} on the unit tangent bundle of the surface. In Section \ref{hcfpsp} we construct projective structures on the unit tangent bundle and prove they differ from the convex-foliated structures only by homeomorphism of the surface, but have their own distinct features. In particular the fibres of the unit tangent bundle define lines in our projective structure.

Section \ref{g2reps} we are able to relate the Hitchin component for $G_2$ to almost complex curves in a quadric. This is a split real version of the usual notion of almost complex curves for the compact real form of $G_2$, although this is a less intrinsic structure than in the other rank $2$ cases.\\


The next two chapters move away from Higgs bundles but continue the theme of $G_2$ geometry. We study a geometry associated with the split real form of the exceptional Lie algebra $G_2$. This geometry is most conveniently described by the language of Cartan geometries, more specifically parabolic geometries.\\

Cartan geometries are a curved generalisation of homogeneous spaces. As the definition is very broad more important to us will be the notion of a parabolic geometry. Here the homogeneous model space is one associated to a parabolic subgroup. In addition to projective and conformal geometries the parabolic geometry of interest to us is one associated to the split real form of $G_2$. This geometry is the structure that naturally arises from a generic $2$-distribution in $5$ dimensions. Here generic means that taking successive commutators generates the entire tangent bundle. Such distributions were studied by Cartan in a section of his famous 5 variables paper \cite{cartan1}. This geometry has been briefly touched upon in a few places for example \cite{gardner} and \cite{stern} and later Nurowski \cite{nurowski1} in the context of studying a class of differential equations discovered there is an associated conformal structure on the $5$-manifold with split $G_2$ holonomy. Our contribution is to clarify the geometry from the parabolic Cartan geometry point of view and also to give an interpretation in terms of spinors.\\

The homogeneous space for this $G_2$ geometry is a $5$-dimensional quadric which can also be identified as the homogeneous model space for conformal geometry of signature $(2,3)$. As a result we are able to relate the structure to the more familiar conformal geometry. To each generic $2$-distribution we can therefore associate a unique conformal structure. We show the distribution consists of maximal isotropic subspaces in the conformal structure. Moreover since the distribution $V$ is maximal isotropic there is locally a spinor $\psi$ defined up to scale for which $V$ is the annihilator of $\psi$. We prove that we can scale $\psi$ so that it satisfies the twistor spinor equation.\\

Chapter \ref{chap4} begins with background material on parabolic geometries and Cartan connections in Section \ref{cartgeom} and continues with tractor bundles in Section \ref{tract}. We also consider conformal geometry specifically in Section \ref{confgeom} as this will relate to the $G_2$ geometry of the following chapter. Chapter \ref{chap5} begins with Section \ref{octonions} where we introduce a parabolic subgroup of split $G_2$ and study the corresponding parabolic geometry. We relate this to generic $2$-distributions in $5$ dimensions. In Section \ref{confg2} we prove that this $G_2$ geometry is equivalent to conformal geometry with $G_2$ holonomy thus moving the geometry into the realm of conformal geometry. In Section \ref{spinors} we consider the use of spinors in conformal geometry and prove the result that constant spinors for the tractor connection correspond to solutions of the twistor spinor equation. Returning to the $5$-dimensional geometry, the existence of the conformal structure allows one to define, at least locally, spinor bundles. We then prove the existence of a pure spinor for the $2$-distribution satisfying the twistor spinor equation.

Section \ref{examples} considers two examples of this $5$-dimensional geometry, the rolling of two surfaces and a class of differential equations. \\


In Chapter \ref{chap6} we move onto a geometry involving the compact form of $G_2$, namely the coassociative submanifolds of $G_2$ manifolds and coassociative fibrations.

The well-known conjecture of Strominger, Yau and Zaslow \cite{syz} provides a geometric picture of mirror symmetry, at least in the so-called large complex structure limit. The conjecture proposes mirror pairs of Calabi-Yau manifolds which are special Lagrangian torus fibrations over the same base, but with dual fibres. In understanding limiting cases of the conjecture one is motivated to study {\em semi-flat} special Lagrangian fibrations. These are fibrations in which the fibres are flat tori. It is known in this case that the base has natural affine coordinates and a function $\phi$ satisfying the real Monge-Amp\`ere equation ${\rm det}H(\phi)= 1$, where $H(\phi)$ is the Hessian of $\phi$ with respect to the affine coordinates \cite{hit}.

In M-theory $G_2$-manifolds play a role equivalent to Calabi-Yau manifolds in String theory, so it is natural to ask whether there is an analogue of the SYZ conjecture for $G_2$-manifolds. Gukov, Yau and Zaslow argue that the $G_2$ equivalent is a pair of $G_2$-manifolds fibred by coassociative submanifolds over the same base \cite{gyz}. We take this as motivation to study $G_2$-manifolds fibred by flat coassociative tori. More specifically we call a coassociative fibration $X$ {\em semi-flat} if there is a $T^4$-action of isomorphisms of $X$ such that the orbits are coassociative submanifolds. The key result is Theorem \ref{thethm} which states that the base $M$ locally maps into $H^2(T^4,\mathbb{R})$ (equipped with the intersection form) as a minimal $3$-submanifold and conversely such a minimal submanifold gives rise to a semi-flat coassociative fibration.\\

Before studying the specific case of semi-flat $G_2$-manifolds, we investigate the structure of the moduli space of deformations of a compact coassociative submanifold in Section \ref{defo}. Adapting the approach of \cite{hit} which studies the moduli space of special Lagrangians, we find that the moduli space $M$ of deformations of a compact coassociative submanifold $L$ has locally a natural map $u : M \to H^2(L,\mathbb{R})$ defined up to an affine map. The $L^2$ metric on the moduli space is then the induced metric under $u$.\\

Section \ref{seccyl} considers the case of a $G_2$-manifold that is a product $X = Y \times S^1$ of a Calabi-Yau manifold $Y$ and a circle. In this case the coassociative and special Lagrangian moduli spaces can be related. Moreover, in the semi-flat case we show that the minimal submanifold equations reduce to the Monge-Amp\`ere equation, recovering the known result on semi-flat Calabi-Yau manifolds.\\

Section \ref{cfocm} is on compact coassociative fibrations. We show that a compact manifold with holonomy equal to $G_2$ can only have degenerate coassociative fibrations. We then briefly consider the nature of the singularities and provide examples of torus fibrations that might serve as a model for the expected behaviour of coassociative fibrations. In Section \ref{secsemiflat} we consider semi-flat coassociative fibrations. These are coassociative fibrations with a torus action of isometries generating the fibres. In Section \ref{secconstruct} we prove the main result on semi-flat fibrations that they are locally equivalent to positive definite minimal $3$-submanifolds in $H^2(T^4,\mathbb{R}) \simeq \mathbb{R}^{3,3}$.

In \cite{lyz} the authors seek solutions to the Monge-Amp\`ere equation on a $3$-dimensional base that is a cone. This reduces to equations on a surface, in fact the equations for an elliptic affine sphere. This amounts to solving the following equation on a Riemann surface
\begin{equation*}
\psi_{z\overline{z}} + |U|^2e^{-2\psi} + \tfrac{1}{2}e^\psi = 0
\end{equation*}
where $U$ is a holomorphic cubic differential. Up to sign changes this is the Toda equation for the affine Dynkin diagram $A^{(2)}_2$, studied by Tzitz\'eica \cite{tzitz}. In a similar fashion in Section \ref{rtms} we reduce from the minimal submanifold equations on a $3$-manifold to equations on a surface. We consider semi-flat $G_2$-manifolds with a vector field commuting with the $T^4$-action which essentially scales the associative $3$-form. This corresponds to the minimal $3$-manifold being a cone, which in turn is equivalent to a minimal surface in the quadric of unit vectors. We then apply the results of Section \ref{secquadrics} to show that in the case of the unit quadric in $\mathbb{R}^{3,3}$ the equations for a superconformal minimal surface are
\begin{eqnarray*}
2(w_1)_{z\overline{z}} &=& -e^{2w_2-2w_1} - e^{2w_1}, \\
2(w_2)_{z\overline{z}} &=& q\overline{q}e^{-2w_2} + e^{2w_2 - 2w_1}.
\end{eqnarray*}
where $q$ is a holomorphic cubic differential. The case where the $G_2$-manifold is a product of a Calabi-Yau manifold and a circle corresponds to the reduction $e^{2w_2} = q\overline{q}e^{-2w_1}$ in which case the equations reduce to the equation of \cite{lyz}.\\


In Section \ref{extension} we extend the results on semi-flat coassociative fibrations to the case of manifolds with a split $G_2$ structure. In particular if they reduce to equations on a surface we can obtain different real forms of the Toda equations including the case that coincides with the Hitchin component for ${\rm PSp}(4,\mathbb{R})$ which ties in with the earlier work of the Thesis.\\

We conclude with Chapter \ref{further} in which we discuss some further questions that merit investigation.\\



\chapter{Higgs Bundles, harmonic maps and Toda equations}\label{chap2}
In this chapter we examine the links between Higgs bundles, harmonic maps and the affine Toda equations. In Section \ref{higgsandharm} we begin with a review of Higgs bundles followed by harmonic maps. We then explain how Higgs bundles are related to a class of harmonic map. In Section \ref{hitcomp} we introduce a class of Higgs bundles which parameterise a component in a space of representations of the fundamental group of the surface. We introduce a special case of such Higgs bundles which we call cyclic and we relate these to the affine Toda equations. In Section \ref{sectoda} we introduce the affine Toda equations more generally and in Section \ref{msaate} relate them to minimal surfaces in quadrics.

\section{Higgs bundles and harmonic maps}\label{higgsandharm}

Here review introduce the basic theory of Higgs bundles and harmonic maps and show their relationship.
\subsection{Higgs bundles}\label{higgs}

Let $\Sigma$ be a Riemann surface. To each holomorphic vector bundle $E$ over $\Sigma$ we define the {\em slope} $\mu(E)$ by
\begin{equation}
\mu(E) = \frac{{\rm deg}E}{{\rm rank}E}.
\end{equation}
We say that $E$ is {\em semi-stable} if for any proper, non-zero holomorphic sub-bundle $F \subset E$ we have the inequality $\mu(F) \le \mu(E)$. We say $E$ is {\em stable} if the above inequality is strict $\mu(F) < \mu(E)$ and we say $E$ is {\em polystable} if $E$ is a direct sum of stable bundles all with the same slope.\\

The collection of all holomorphic bundles considered up to isomorphism over a compact Riemann surface $\Sigma$ can not in general be given the structure of a Hausdorff topological space in any reasonable way. However Mumford's geometric invariant theory showed that by restricting to stable bundles a reasonable moduli space can be constructed. In fact the moduli space of stable bundles of given rank and degree has a natural structure of an algebraic variety. The stability condition is also important from another point of view. If $E$ is a holomorphic vector bundle of degree zero we can ask whether $E$ can be given a Hermitian metric such that the associated Chern connection is flat. The answer is the well-known result of Narasimhan and Seshadri \cite{nara}. Such a metric exists if and only if $E$ is poly-stable. This provides a link between the moduli space of semi-stable bundles of degree zero and representations of the fundamental group $\pi_1(\Sigma)$ into the unitary groups ${\rm U}(n)$. A similar relation holds for arbitrary degree by considering Hermitian-Einstein connections. In another direction the Narasimhan-Sashadri theorem is extended to non-unitary representations by means of Higgs bundles.\\

The Higgs bundle equations were introduced by Hitchin in \cite{hit2} as a dimensional reduction of the Yang-Mills self-duality equations. Let $\Sigma$ be a Riemann surface with canonical bundle $K$. Let $K$ (not to be confused with the canonical bundle) be the compact real form of a complex semisimple Lie group $G$ with corresponding Lie algebras $\mathfrak{k} \subset \mathfrak{g} = \mathfrak{k} \otimes \mathbb{C}$. Let $\rho$ denote the anti-linear involution corresponding to $\mathfrak{g}$. If $x \otimes y \mapsto k(x,y)$ is the Killing form on $\mathfrak{g}$ then corresponding to $\rho$ is the Hermitian form $x \otimes y \mapsto -k(x , \rho(y))$. The Hermitian adjoint of ${\rm ad}_x$ is $-{\rm ad}_{\rho(x)}$, hence we will denote $-\rho(x)$ by $x^*$.

The Higgs bundle equations for the group $K$ are equations for a pair $(\nabla_A,\Phi)$ consisting of a connection $\nabla_A$ on a principal $K$-bundle $P$ and a section $\Phi$ of the vector bundle ${\rm Ad}P^c \otimes K$ where ${\rm Ad}P^c = P \times_{{\rm Ad}} \mathfrak{g}$ is the complex adjoint bundle. The Higgs bundle equations are
\begin{eqnarray}
F_A + \left[ \Phi , \Phi^* \right] &=& 0 \label{hbe1} \\
\nabla_A^{0,1}\Phi &=& 0 \label{hbe2}
\end{eqnarray}
where $F_A$ is the curvature of $\nabla_A$ and $\nabla_A^{0,1}$ is the $(0,1)$ part of $\nabla_A$. Note also that a solution to these equations defines a flat $G$-connection $\nabla = \nabla_A + \Phi + \Phi^*$ and hence a representation of the fundamental group $\pi_1(\Sigma)$ of $\Sigma$ into $G$. We will have more to say on the link with representations of the fundamental group in Section \ref{higgsharmonic}.\\

Suppose we have a solution $(\nabla_A,\Phi)$ to the ${\rm SU}(n)$ Higgs bundle equations. We have a rank $n$ vector bundle $E$ of degree zero associated to the standard representation of ${\rm SU}(n)$ and $\nabla_A^{0,1}$ defines a holomorphic structure on $E$. We also have that $\Phi$ is a trace free section of ${\rm End}(E) \otimes K$ and from the Higgs bundle equations we have that $\Phi$ is holomorphic. This leads us to the following definition: a {\em Higgs bundle} is a pair $(E,\Phi)$ consisting of a holomorphic vector bundle $E$ and $\Phi$ a holomorphic section of ${\rm End}(E) \otimes K$. Since the Higgs bundle pair consists of holomorphic objects they can be approached from a purely algebro-geometric point of view.\\

So far we have that a solution to the ${\rm SU}(n)$ Higgs bundle equations $(\nabla_A , \Phi)$ defines a Higgs bundle $(E,\Phi)$ where ${\rm deg}E =0$ and $\Phi$ is trace free. Naturally we may ask the converse question. Suppose $(E,\Phi)$ is a Higgs bundle. Given a Hermitian form $h$ on $E$ we then have the Chern connection $\nabla_A$, the unique connection that preserves $h$ and is compatible with the holomorphic structure on $E$. This defines a pair $(\nabla_A, \Phi)$ and we are interested in whether we can find such a pair that solves the Higgs bundle equations (\ref{hbe1}), (\ref{hbe2}). Since $\Phi$ is holomorphic equation (\ref{hbe2}) is automatically satisfied. A necessary condition for a solution to equation (\ref{hbe1}) is that ${\rm deg}E=0$, however this is not sufficient. This leads us to the notion of stability for Higgs bundles.\\

We define the {\em slope} $\mu(E)$ of a Higgs bundle $(E,\Phi)$ to be the slope of $E$, $\mu(E) = {\rm deg}E / {\rm rank}E$. A Higgs bundle $(E,\Phi)$ is said to be {\em semi-stable} if for each proper, non-zero sub-bundle $F \subset E$ which is $\Phi$-invariant, that is $\Phi(F) \subset F \otimes K$ we have $\mu(F) \le \mu(E)$. Similarly we have notions of stable and polystable Higgs bundles.\\

Let us return to the question of the existence of a solution $(\nabla_A,\Phi)$ of the Higgs bundle equations corresponding to a Higgs bundle $(E,\Phi)$ with ${\rm deg}E = 0$. The result of Hitchin \cite{hit2} and Simpson \cite{simp1} is that such a connection $\nabla_A$ exists if and only if $(E,\Phi)$ is polystable and in this case the connection $\nabla_A$ is unique. In the case of a stable bundle the Hermitian metric is unique up to scale.\\

We can express this correspondence on the level of gauge isomorphism classes as follows. Let us fix a smooth vector bundle $E$ of rank $n$ and degree $0$. We restrict to Higgs bundle pairs consisting of a holomorphic structure $\overline{\partial}_E$ on $E$ such that ${\rm det}E = \mathcal{O}$ is the trivial holomorphic line bundle and trace free Higgs field $\Phi$. Two such Higgs bundles are isomorphic if and only if they are ${\rm SL}(n,\mathbb{C})$ gauge equivalent. Fix a Hermitian metric $h$ on $E$ compatible with the volume form given by the ${\rm SL}(n,\mathbb{C})$-structure. Then given a holomorphic structure $\overline{\partial}_E$ on $E$ we have a corresponding Chern connection.

The result of Hitchin and Simpson is that within an ${\rm SL}(n,\mathbb{C})$ gauge isomorphism class of Higgs bundles (here the metric $h$ is fixed and the Higgs bundle $(\overline{\partial}_E , \Phi)$ is gauge transformed) there is a solution to the Higgs bundle equations if and only if the corresponding isomorphism class is polystable. Moreover in the stable case the solution is unique up to ${\rm SU}(n)$ gauge isomorphism.\\

So far we have only considered the correspondence between Higgs bundles and the Higgs bundle equations in the ${\rm SU}(n)$ case. More generally we can set up a correspondence between the Higgs bundle equations for a complex semisimple Lie group $G$ and Higgs bundles of a particular form. Indeed If $(\nabla_A, \Phi)$ is a solution of the Higgs bundle equations with $\nabla_A$ a connection on a principal $K$-bundle $P$ then the complex adjoint bundle $E = {\rm Ad}P^c$ is a holomorphic bundle and since $\Phi$ is a section of $E \otimes K$ it can also be considered as a section of ${\rm End}(E) \otimes K$ thus defining a Higgs bundle pair $(E,\Phi)$ which is clearly polystable.

Conversely we can consider Higgs bundles $(E,\Phi)$ such that $E$ is the adjoint bundle associated to a holomorphic principal $G$-bundle $P^c$ and $\Phi$ is a holomorphic section of ${\rm Ad}P^c \otimes K$. Clearly we have ${\rm deg}E = 0$, so if $(E,\Phi)$ is polystable we know there exists an ${\rm SU}(n)$ solution to the Higgs bundle equations. In fact one can show there is a unique solution $(\nabla_A , \Phi)$ where $\nabla_A$ is a $K$-connection \cite{hit1}. The flat connection $\nabla = \nabla_A + \Phi - \rho(\Phi)$ is then a $G$-connection, hence the monodromy defines a representation of $\pi_1(\Sigma)$ in $G$.\\

\subsection{Harmonic maps}\label{harm}
We review harmonic maps, minimal submanifolds and establish some of their basic properties which we will have need for on several occasions. Standard references for harmonic maps are \cite{eelsam}, \cite{eells}, \cite{eells2}.\\

Let $M,N$ be manifolds with (possibly indefinite) metrics $g,h$ respectively. Given a smooth map $\phi : M \to N$, the differential $\phi_* : TM \to TN$ can be regarded as a section of $T^*M \otimes \phi^{-1}(TN)$. Using the Levi-Civita connections on $M$ and $N$ we have a natural connection on $T^*M \otimes \phi^{-1}(TN)$ which we simply denote $\nabla$. Then we may take the covariant derivative $\nabla \phi_*$ which turns out to be a section of $S^2(T^*M) \otimes \phi^{-1}(TN)$ called the {\em second fundamental form} of $\phi$. The trace of $\nabla \phi_*$ over the $S^2(T^*M)$ factor is a section $\tau(\phi)$ of $\phi^{-1}(TN)$ called the {\em tension field} of $\phi$. We say that $\phi$ is {\em harmonic} if the tension field vanishes.

In local coordinates $\{x^i\}_{i=1}^m$ on $M$ and $\{y^\alpha\}_{\alpha=1}^n$ on $N$ such that $\phi = \{\phi^\alpha(x)\}_{\alpha=1}^n$ the tension field is given by
\begin{equation}
\tau^\gamma(\phi) = g^{ij}\left( \partial^2_{ij} \phi^\gamma - \leftexp{M}{{\Gamma^k}_{ij}} \partial_k \phi^\gamma + \leftexp{N}{{\Gamma^\gamma}_{\alpha \beta}} \partial_i \phi^\alpha \partial_j \phi^\beta \right)
\end{equation}
where $\leftexp{M}{{\Gamma^k}_{ij}}$ and $\leftexp{N}{{\Gamma^\gamma}_{\alpha \beta}}$ are the Christoffel symbols on $M$ and $N$ respectively.\\

We say that $\phi: M \to N$ is a {\em minimal immersion} if $\phi$ is harmonic and $g$ is the pull-back metric $g = \phi^*h$. Equivalently a minimal immersion is an immersion that is a solution of the Euler-Lagrange equations for the volume functional $\phi \mapsto \int_M vol$ where $vol$ is the volume form on $M$ induced by the metric on $N$.\\

We will establish a formula for the Levi-Civita connection on certain homogeneous spaces for later use concerning harmonic maps into such spaces. Similar results can be found in \cite{nom}. Let $G$ be a Lie group, $V \subset G$ a subgroup, $\mathfrak{v} \subset \mathfrak{g}$ the corresponding Lie algebras. Suppose $\mathfrak{v}$ admits an ${\rm Ad}_V$-invariant complement $\mathfrak{h}$ so $\mathfrak{g} = \mathfrak{v} \oplus \mathfrak{h}$ as $V$-modules. Then via the Maurer-Cartan form $\omega : TG \to \mathfrak{g}$ we have the trivialisation
\begin{equation*}
TG = G \times \mathfrak{g}.
\end{equation*}
In a similar fashion we can identify the tangent space of the quotient $G/V$ with
\begin{equation*}
T(G/V) = G \times_V \mathfrak{h}.
\end{equation*}
Write $\omega = \omega^{\mathfrak{v}} + \omega^{\mathfrak{h}}$ with $\omega^{\mathfrak{v}} \in \mathfrak{v}$, $\omega^{\mathfrak{h}} \in \mathfrak{h}$. Then since $\mathfrak{v}$ and $\mathfrak{h}$ are ${\rm Ad}_V$-invariant, this decomposition is itself invariant.\\

Let $p : G \to G/V$ be the natural projection. Suppose now $U \subset G/V$ is an open subset over which we have a local section $\sigma: U \to G$ so $p \, \sigma : U \to G/V$ is simply the inclusion. Then we will denote $\sigma^*\omega^{\mathfrak{v}}$ and $\sigma^* \omega^{\mathfrak{h}}$ simply by $\omega^{\mathfrak{v}}$ and $\omega^{\mathfrak{h}}$. Now $\omega^{\mathfrak{h}}$ provides a canonical identification of $TU$ with $U \times \mathfrak{h}$ by sending $X \in T_uU$ to $(u , \omega^{\mathfrak{h}}(X))$.

We wish to define a torsion free connection over $U$. Given vector fields $X,Y$ it suffices to define $\omega^{\mathfrak{h}}(\nabla_X Y)$. We find that a torsion free connection has the form
\begin{equation}\label{connformula}
\omega^{\mathfrak{h}}(\nabla_X Y) = X(\omega^{\mathfrak{h}} Y) + \frac{1}{2} \left[ \omega X , \omega Y \right]^{\mathfrak{h}} + \frac{1}{2}M(\omega X , \omega Y)
\end{equation}
where $\left[ \; , \; \right]^{\mathfrak{h}}$ denotes the $\mathfrak{h}$-component of the Lie bracket and $M$ is a symmetric map $M : \mathfrak{g} \otimes \mathfrak{g} \to \mathfrak{h}$ (depending on $u \in U$).\\

Suppose now that $\langle \; , \; \rangle$ is a $G$-invariant inner product on $\mathfrak{g}$ such that $\mathfrak{v}$ and $\mathfrak{h}$ are orthogonal. Then $\langle \; , \; \rangle$ transports to a metric on $G$ via $\omega$ and descends to $G/V$ by invariance. So for vector fields $X,Y$ on $U$, the inner product $\langle X , Y \rangle$ is given by
\begin{equation*}
\langle X , Y \rangle = \langle \omega^{\mathfrak{h}} X , \omega^{\mathfrak{h}} Y \rangle.
\end{equation*}
If we require that (\ref{connformula}) preserves the metric this determines $M$. In fact we find
\begin{equation}
M(A,B) = \left[ A^{\mathfrak{v}} , B^{\mathfrak{h}} \right] - \left[ A^{\mathfrak{h}} , B^{\mathfrak{v}} \right].
\end{equation}
Substituting we find that the Levi-Civita connection is given by
\begin{equation}
\omega^{\mathfrak{h}}(\nabla_X Y) = X(\omega^{\mathfrak{h}} Y) + \frac{1}{2} \left[ \omega X , \omega Y \right]^{\mathfrak{h}} + \frac{1}{2} \left[ \omega^{\mathfrak{v}} X , \omega^{\mathfrak{h}} Y \right] - \frac{1}{2} \left[ \omega^{\mathfrak{h}} X , \omega^{\mathfrak{v}} Y \right].
\end{equation}
But since $\mathfrak{v}$ and $\mathfrak{h}$ are ${\rm ad}_{\mathfrak{v}}$-invariant this simplifies a little to
\begin{equation}\label{connformula2}
\omega^{\mathfrak{h}}(\nabla_X Y) = X(\omega^{\mathfrak{h}} Y) + \left[ \omega^{\mathfrak{v}} X , \omega^{\mathfrak{h}} Y \right] + \frac{1}{2} \left[ \omega^{\mathfrak{h}} X , \omega^{\mathfrak{h}} Y \right]^{\mathfrak{h}}.
\end{equation}

Now we shall apply equation (\ref{connformula2}) to the case of a harmonic map from a Riemann surface to a homogeneous space $G/V$ as above. Note that the harmonic equation for a map from a surface $\Sigma$ depends only on the conformal structure on $\Sigma$. Let $\phi : \Sigma \to G/V$ be a map from a Riemann surface to $G/V$.  Locally we compose with the section $\sigma : U \to G$ to get a lift $\tilde{\phi} : U \to G$. We have $\phi^* \sigma^* \omega = \tilde{\phi}^* \omega = \tilde{\phi}^{-1} d\tilde{\phi}$. Suppose $h$ is a metric on $\Sigma$ compatible with the complex structure. Let $z$ denote a local holomorphic coordinate on $\Sigma$. One finds that the map $\phi$ is harmonic if and only if $\nabla_{\overline{z}} \phi_z = 0$. Let us write
\begin{equation}
\tilde{\phi}^* \omega = Adz + Bd\overline{z} = (A^{\mathfrak{v}} + A^{\mathfrak{h}})dz + (B^{\mathfrak{v}} + B^{\mathfrak{h}})d\overline{z}.
\end{equation}
Then the equation for $\phi$ to be harmonic is
\begin{equation}\label{harmsurf1}
\omega^{\mathfrak{h}}(\nabla_{\overline{z}}\phi_z) = A^{\mathfrak{h}}_{\overline{z}} + \left[ B^{\mathfrak{v}} , A^{\mathfrak{h}} \right] + \frac{1}{2} \left[ B^{\mathfrak{h}} , A^{\mathfrak{h}} \right]^{\mathfrak{h}} = 0
\end{equation}
and equivalently we have the conjugate equation
\begin{equation}\label{harmsurf2}
\omega^{\mathfrak{h}}(\nabla_z \phi_{\overline{z}}) = B^{\mathfrak{h}}_z + \left[ A^{\mathfrak{v}} , B^{\mathfrak{h}} \right] + \frac{1}{2} \left[ A^{\mathfrak{h}} , B^{\mathfrak{h}} \right]^{\mathfrak{h}} = 0.
\end{equation}

\subsection{Relation between Higgs bundles and harmonic maps}\label{higgsharmonic}

We consider a certain class of harmonic maps from a Riemann surface into $G/K$ where $G$ is a complex simple Lie group and $K$ is a maximal compact subgroup. Following \cite{cor} we will show how such harmonic maps are related to Higgs bundles.\\

Let $\Sigma$ be a Riemann surface and $\tilde{\Sigma}$ the universal cover. Let $\theta : \pi_1(\Sigma) \to G$ be a representation of the fundamental group of $\Sigma$ in $G$. Consider a $\theta$-equivariant map $\psi : \tilde{\Sigma} \to G/K$, that is if we think of $\tilde{\Sigma}$ as a principal $\pi_1(\Sigma)$-bundle over $\Sigma$ then for each $\gamma \in \pi_1(\Sigma)$ we have $\psi(u\gamma) = \theta(\gamma)^{-1}\psi(u)$.

We can view $\psi$ as a reduction of structure of the principal $G$-bundle $P = \tilde{\Sigma} \times_\theta G$ to the principal $K$-subbundle $P_K = \{ \left[ (x , u ) \right] \in \tilde{\Sigma} \times_\theta G \, | \, u \in \psi(x) \}$ where $ \left[ (x , u ) \right] $ denotes the equivalence class of $(x,u)$ in $\tilde{\Sigma} \times_\theta G$. We can view this reduction of structure as equivalent to defining an anti-involution $\rho$ on the adjoint bundle $P \times_{{\rm Ad}} \mathfrak{g}$ and hence a corresponding Hermitian metric.

Locally we can take a lift $\tilde{\psi} : U \to G$ of $\psi$. Then $\tilde{\psi}$ is a local section of $P_K$. The flat connection over $\Sigma$ corresponding to the monodromy representation $\theta$ descends from the Maurer-Cartan form $\omega$ on $\tilde{\Sigma} \times G$. So in the local trivialisation given by $\tilde{\psi}$ the flat connection is $\tilde{\psi}^*\omega = \tilde{\psi}^{-1} d\tilde{\psi}$.

Let $k$ be the subalgebra of $\mathfrak{g}$ corresponding to $K$ and let $k^\perp$ be the orthogonal complement with respect to the Killing form. Then the compact antilinear involution $\rho$ corresponding to $K$ is given by the identity on $k$ and minus the identity on $k^\perp$. The decomposition $\mathfrak{g} = k \oplus k^\perp$ is $K$-invariant so upon restriction to the principal $K$-subbundle $P_K \subset P$ we have that the flat connection $\omega$ can be invariantly decomposed into $k$ and $k^\perp$-valued parts:
\begin{equation}
\omega = A + \phi.
\end{equation}
Now in the gauge corresponding to $\tilde{\psi}$ we have
\begin{equation}
\omega = \tilde{\psi}^{-1}d\tilde{\psi} = \alpha dz + \beta d\overline{z}.
\end{equation}
Moreover if we let $\alpha = \alpha^k + \alpha^\perp$, $\beta = \beta^k + \beta^\perp$ be the corresponding decompositions then
\begin{eqnarray}
A &=& \alpha^k dz + \beta^k d\overline{z}, \\
\phi &=& \alpha^\perp dz + \beta^\perp d\overline{z}.
\end{eqnarray}
Let us call $\alpha^\perp = \Phi$. Now since $\tilde{\psi}$ is a local lift of $\psi$ we have that $\tilde{\psi}$ corresponds to a local frame such that the anti-involution $\rho$ on the adjoint bundle matches the fixed anti-involution $\rho$ corresponding to $K \subset G$. Therefore $\beta^\perp = \Phi^* = -\rho(\Phi)$. The Maurer-Cartan equations, i.e. the fact that $\omega$ is a flat connection yields
\begin{eqnarray}
F_A + \frac{1}{2} \left[ \phi , \phi \right] &=& 0 \\
d\phi + \left[ A , \phi \right] &=& 0
\end{eqnarray}
where $F_A = dA + \frac{1}{2}\left[ A , A \right]$ is the curvature of the connection defined by $A$.\\

We now consider the condition for $\psi$ to be harmonic with respect to any metric compatible with the conformal structure on $\Sigma$. Recall equation (\ref{connformula2}) giving the Levi-Civita connection on $G/K$. Put $\omega = \omega^k + \omega^\perp$. Then
\begin{equation*}
\omega^\perp (\nabla_X Y) = X(\omega^\perp Y) + \left[ \omega^k X , \omega^\perp Y \right] + \frac{1}{2} \left[ \omega^\perp X , \omega^\perp Y \right]^\perp.
\end{equation*}
In the present case $\left[ k^\perp , k^\perp \right] \subset k$ so the last term vanishes. Put $Y = \partial_z$ and $X = \partial_{\overline{z}}$. So $\omega^k Y = \alpha^k$, $\omega^\perp Y = \Phi$, $\omega^k X = \beta^k$, $\omega^\perp X = \Phi^*$. The harmonic equation $\nabla_{\overline{z}} \phi_z = 0$ is then $0 = \Phi_{\overline{z}} + \left[ \beta^k , \Phi \right]$, that is
\begin{equation*}
(\nabla_A)_{\overline{z}} \Phi = 0
\end{equation*}
where $\nabla_A$ is the connection corresponding to $A$. Therefore we have one of the Higgs bundle equations from the Maurer-Cartan equations $F_A + \left[ \Phi , \Phi^* \right] = 0$ and the second of the Higgs bundle equations is equivalent to the map $\psi$ being harmonic.\\

To summarise we have found that given a representation of the fundamental group $\theta : \pi_1(\Sigma) \to G$, the corresponding flat connection $\nabla$ can be decomposed into $\nabla = \nabla_A + \Phi + \Phi^*$ satisfying the Higgs bundle equations if and only if there exists a $\theta$-equivariant map $\psi: \tilde{\Sigma} \to G/K$ which is harmonic. Conversely a solution of the Higgs bundle equations defines a flat connection with a monodromy representation $\theta$ and we get a $\theta$-equivariant harmonic map.

The natural question is then to which representations does there exist such a harmonic map and hence a corresponding Higgs bundle. The result of Donaldson \cite{don} and Corlette \cite{cor} is that a $\theta$-equivariant harmonic map exists if and only if the representation $\theta$ is {\em reductive}, that is if the induced representation of $\pi_1(\Sigma)$ on $\mathfrak{g}$ is a direct sum of irreducible representations. Moreover when such a map $\psi$ exists it is essentially unique. Any other such solution is of the form $g\psi$ where $g$ is in the centraliser of $\theta(\pi_1(\Sigma))$.

\section{The Hitchin component}\label{hitcomp}

Higgs bundles provide an interpretation for representations of the fundamental group of a surface into a simple Lie group. A particular explicit construction of Higgs bundles by Hitchin \cite{hit1} identifies a component of the space of representations into the split real form of a Lie group. This component is often called the Hitchin component. We consider within the Hitchin component those representations for which only the highest holomorphic differential is non-vanishing and show in this case the Higgs bundle equation reduces to a set of Toda equations.

\subsection{Principal three-dimensional subalgebras}\label{ptd}
Let $\mathfrak{g}$ be a complex simple Lie algebra of rank $l$, let $\mathfrak{h}$ be a Cartan subalgebra with $\Delta$, $\Delta^+$ and $\Pi$ denoting the root system, a system of positive roots and the corresponding simple roots, lastly fix a corresponding basis $\{h_\beta, e_\alpha,e_{-\alpha} | \; \beta \in \Pi, \; \alpha \in \Delta^+\}$. Kostant \cite{kost} defines a subalgebra (unique up to conjugacy) called the {\em principal $3$-dimensional subalgebra} of $\mathfrak{g}$. We construct the principal $3$-dimensional subalgebra as follows. Let
\begin{equation}
x = \dfrac{1}{2}\sum_{\alpha \in \Delta^+}h_\alpha
\end{equation}
then $x = \sum_{\alpha \in \Pi}r_\alpha h_\alpha$ for some positive half-integers $r_\alpha$. We use these to further define
\begin{equation}
e = \sum_{\alpha \in \Pi}\sqrt{r_\alpha}e_\alpha, \; \; \;
\tilde{e} = \sum_{\alpha \in \Pi}\sqrt{r_\alpha}e_{-\alpha}.
\end{equation}
Then we define $\mathfrak{s}$ as the linear span of $\{x,e,\tilde{e}\}$. We must verify that $\mathfrak{s}$ is a subalgebra:
\begin{lemp}\cite{Oni1}
For any $\beta \in \Pi$ we have $\beta(x)=1$.
\begin{proof}
Let $R_\beta$ denote the reflection in the hyperplane in $\mathfrak{h}^*$ orthogonal to $\beta$, namely
\begin{equation*}
R_\beta(\alpha) = \alpha - 2\frac{(\alpha,\beta)}{(\beta,\beta)}\beta.
\end{equation*}
The dual action of $R_\beta$ on $\mathfrak{h}$ is then
\begin{equation}\label{reflection}
R_\beta^t (h) = h-\beta(h)h_\beta.
\end{equation}
From this one verifies the relation $R_\beta^t(h_\alpha) = h_{R_\beta(\alpha)}$. We then have that
\begin{equation*}
R_\beta^t(x) = \dfrac{1}{2}\sum_{\alpha \in \Delta^+} h_{R_\beta(\alpha)}.
\end{equation*}
However we also have that for $\alpha \in \Delta^+$, $R_\beta(\alpha) \in \Delta^+$ except for $\alpha = \beta$ in which case $R_\beta(\beta) = -\beta$. It follows that $R_\beta^t(x) = x - h_\beta$, hence by (\ref{reflection}), $\beta(x)=1$.
\end{proof}
\end{lemp}

Given a root $\lambda \in \Delta$, let $\lambda = \sum_{\alpha \in \Pi}n_\alpha \alpha$. Then we define the {\em $\Pi$-height} of $\lambda$ to be the integer ${\rm height}(\lambda) = \sum_{\alpha \in \Pi}n_{\alpha}$. The above lemma shows that for $y \in \mathfrak{g}_\alpha$, we have $[x,y] = {\rm height}(\alpha)y$, that is $x$ is the grading element corresponding to the gradation of $\mathfrak{g}$ by height. We now deduce the following commutation relations for $\mathfrak{s}$:
\begin{equation}
[x,e] = e, \; \; \; [x,\tilde{e}] = -\tilde{e}, \; \; \; [e,\tilde{e}] = x.
\end{equation}
Thus $\mathfrak{s}$ is a copy of $\mathfrak{sl}(2,\mathbb{C})$.\\

We have that the element $e$ is regular, that is it has an $l$-dimensional centraliser spanned by elements $e_1, \dots , e_l$. Moreover on restriction to $\mathfrak{s}$ the adjoint representation decomposes into irreducible subspaces:
\begin{equation}\label{gdec}
\mathfrak{g} = \bigoplus_{i=1}^l V_i.
\end{equation}
We can take $e_1, \dots , e_l$ as highest weight elements of $V_1, \dots , V_l$ which shows there are indeed $l$ summands. Since $\mathfrak{s}$ itself must appear as one of the $V_i$, we take it to be $V_1$ so we may take $e_1 = e$. Let $m_1, \dots , m_l$ denote the exponents of $\mathfrak{g}$. These can be described as follows \cite{Oni2}: arrange the positive roots of $\mathfrak{g}$ into an array with the $k$-th row consisting of all roots of height $k$, filling in rows from right to left. Then the lengths of the columns from left to right are the exponents. In particular we see that if the highest root $\delta$ has height $M$ then $m_l = M$ is the largest exponent.

Returning to the decomposition of $\mathfrak{g}$ given by (\ref{gdec}), we have that the dimensions of $V_1, \dots , V_l$ are $(2m_1+1), \dots , (2m_l + 1)$, from which we may write
\begin{equation}\label{gdecomp}
\mathfrak{g} = \bigoplus_{i=1}^l {\rm S}^{2m_i}(V)
\end{equation}
where $V$ is the $2$-dimensional fundamental representation for $\mathfrak{s}$. Observe that $[x,e_l] = Me_l$ and thus $e_l$ is a highest weight vector.

We may also decompose $\mathfrak{g}$ according to the action of $x$, i.e. by height:
\begin{equation}
\mathfrak{g} = \bigoplus_{m= -M}^M \mathfrak{g}_m.
\end{equation}

\subsection{Higgs bundle construction}\label{hbc}
We now consider the construction of Higgs bundles in \cite{hit1}. Let $G$ denote the adjoint form of $\mathfrak{g}$. Let $\Sigma$ be a compact Riemann surface of genus $g > 1$ with $K$ the canonical bundle. There are homogeneous generators $p_1 , \dots , p_l$ for the invariant polynomials on $\mathfrak{g}$ such that for all elements $f \in \mathfrak{g}$ of the form
\begin{equation}
f = \tilde{e} + \alpha_1 e_1 + \dots + \alpha_l e_l,
\end{equation}
we have
\begin{equation}\label{poly}
p_i(f) = \alpha_i.
\end{equation}
We also have that $p_i$ has degree $m_i+1$.\\

Recall that if ${\bf M}$ is the moduli space of Higgs bundles $(E,\Phi)$ for the group $G$ then there is a map
\begin{equation}
p : {\bf M} \to \bigoplus_{i = i}^l H^0(\Sigma , K^{m_i+1})
\end{equation}
obtained by applying the invariant polynomials $p_1, \dots , p_l$ to $\Phi$. We will construct a section of this map.\\

Consider the bundle
\begin{equation}\label{E}
E = \bigoplus_{m=-M}^M \mathfrak{g}_m \otimes K^m.
\end{equation}
This is the adjoint bundle of $\mathfrak{g}$ associated to the principal $G$-bundle $P = P_1 \times_i G$ where $P_1$ is the holomorphic principal ${\rm SL}(2,\mathbb{C})$-bundle associated to $K^{1/2}\oplus K^{-1/2}$ (for a choice of $K^{1/2}$) and $i: {\rm SL}(2,\mathbb{C}) \to G$ is the inclusion corresponding to the principal $3$-dimensional subalgebra $\mathfrak{s}$. We see from (\ref{E}) that $E$ is independent of choice of $K^{1/2}$ and defines for us a holomorphic bundle with the structure of $\mathfrak{g}$.

Next we construct a Higgs field $\Phi$. Let $q_1, \dots , q_l$ be holomorphic differentials of degrees $m_1+1, \dots , m_l +1$. Then we may define $\Phi \in H^0(\Sigma , E \otimes K)$ as follows:
\begin{equation}
\Phi = \tilde{e} + q_1 e_1 + \dots + q_l e_l.
\end{equation}
Here, since $\tilde{e} \in \mathfrak{g}_{-1}$, we regard $\tilde{e}$ as a section of $(\mathfrak{g}_{-1} \otimes K^{-1}) \otimes K$ and similarly $q_ie_i$ can be considered as a section of $(\mathfrak{g}_{i} \otimes K^i ) \otimes K$ so that $\Phi$ is a well defined holomorphic section of $E \otimes K$. The Higgs bundles $(E,\Phi)$ constructed here are polystable and in fact correspond to smooth points of the moduli space ${\bf M}$ \cite{hit1}. We have from (\ref{poly}) that $p_i(\Phi) = q_i$ and it follows that $(q_1 , \dots , q_l ) \mapsto (E,\Phi)$ is our desired section.\\

So far we have constructed Higgs bundles with holonomy in $G$. We show that in fact the Higgs bundles constructed have holonomy in the split real form of $G$. We define an involutive automorphism $\sigma$ on $\mathfrak{g}$ as follows:
\begin{equation}
\sigma(e_i) = -e_i, \; \; \; \sigma(\tilde{e}) = -\tilde{e}.
\end{equation}
Since $\mathfrak{g}$ is obtained from $e_1, \dots , e_l$ by repeated application of ${\rm ad}_{\tilde{e}}$, we see that such a $\sigma$ must be unique. We state some properties of $\sigma$ \cite{hit1}; $\sigma$ exists, and defines a Cartan involution corresponding to the split real form of $\mathfrak{g}$. That is if $\rho$ is a compact anti-involution defining a Hermitian metric solving the Higgs bundle equations then in fact $\rho$ and $\sigma$ commute and $\lambda = \rho \sigma$ is an anti-involution for the split real form of $\mathfrak{g}$. Clearly our Higgs field satisfies $\sigma \Phi = -\Phi$. If $(\nabla_A , \Phi)$ satisfies the Higgs bundle equations then so does $(\sigma^*\nabla_A , \sigma^*\Phi)$ and hence so does $(\sigma^*\nabla_A , \Phi)$. Uniqueness of $\nabla_A$ now shows that $\sigma$ is preserved by $\nabla_A$. Note also that $\lambda \Phi = - \rho (\Phi) = \Phi^*$. It now follows that the flat connection $\nabla_A + \Phi + \Phi^*$ has holonomy in the split real form of $G$. Therefore we have constructed representations of the fundamental group of $\Sigma$ into the split real form.\\

Our section $s : \bigoplus_{i = i}^l H^0(\Sigma , K^{m_i+1}) \to {\bf M}$ takes values in the smooth points of ${\bf M}$ and identifies the vector space $V = \bigoplus_{i = i}^l H^0(\Sigma , K^{m_i+1})$ with a submanifold of ${\bf M}$. Moreover the map $s$ actually maps $V$ into the moduli space ${\bf M}^\lambda$ of flat $G^\lambda$-connections where $G^\lambda$ is the split real form of $G$. Around the smooth points ${\bf M}^\lambda$ is a manifold of dimension $(2g-2) \dim(G^\lambda)$. Moreover we find from the Riemann-Roch theorem that $\dim(V) = (2g-2)\dim(G^\lambda)$ also. Since the image $s(V) \subset {\bf M}^\lambda$ is open, closed and connected it defines a smooth component of the space of representations of the fundamental group in $G^\lambda$. We call this the {\em Hitchin component}.

\subsection{Cyclic Higgs bundles}\label{speccase}
We now consider a special case of this construction in which all but the highest differential are set to zero. We will call such a Higgs bundle {\em cyclic}. Thus our Higgs field is of the form

\begin{equation}\label{higgs}
\Phi = \tilde{e} + qe_l.
\end{equation}
Here $q \in {\rm H}^0(\Sigma,K^{M+1})$ is a holomorphic $(M+1)$-differential, where $M = m_l$.

We will also call an element $X \in \mathfrak{g}$ {\em cyclic} if it has the form
\begin{equation}
X = \sum_{\alpha \in \Pi \cup \{-\delta\}} c_\alpha e_{\alpha}
\end{equation}
where $-\delta$ is the lowest root and all $c_\alpha$ are non-zero. Thus our Higgs field is cyclic as an element of $\mathfrak{g}$ at all points where $q \ne 0$.\\

Solutions of the form (\ref{higgs}) have the property that they are fixed points in the moduli space of Higgs bundles of the action of certain roots of unity in $\mathbb{C}^*$. From the properties of the element $x$ we have
\begin{equation*}
[x,\Phi] = -\tilde{e} + Mqe_l.
\end{equation*}
Let us now define $g = {\rm exp}(2\pi i x/(M+1))$. It follows that
\begin{equation}
{\rm Ad}_g \Phi = {\rm exp}(2\pi i M/(M+1))\Phi = \omega\Phi
\end{equation}
where $\omega$ is a $(M+1)$-th root of unity. Now if $(\nabla_A,\Phi)$ is the solution to the Higgs bundle equation corresponding to $\Phi$ then we may gauge transform this solution by $g$ to get the equivalent solution $(g^*\nabla_A, \omega\Phi)$, we can then use the $U(1)$-action to obtain the solution $(g^*\nabla_A,\Phi)$. The uniqueness theorem for solutions to the Higgs bundle equations now implies that $A$ is gauge invariant under $g$, which in turn is equivalent to covariant constancy of $g$ with respect to the connection $\nabla_A$. Thus the action of $g$ and $\nabla_A$ on sections of $E$ commutes. Observe that the Cartan subalgebra $\mathfrak{h}$ is the unity eigenspace of ${\rm Ad}_g$ from which it follows that $\nabla_A$ preserves the subbundle of $E$ corresponding to $\mathfrak{h}$ and thus the connection form $A$ is $\mathfrak{h}$-valued.\\

We now seek an explicit form of the Higgs bundle equations for the special class of solutions (\ref{higgs}). For this we adopt the following point of view: given the holomorphic data $(E,\Phi)$, we seek a Hermitian metric $H$ on $E$ such that the Chern connection $A = H^{-1}\partial H$ satisfies the Higgs bundle equations:
\begin{equation}
F_A + [\Phi,\Phi^*] = 0
\end{equation}
where $F_A = \overline{\partial}(H^{-1}\partial H)$ is the curvature of $\nabla_A$.

More specifically, in our case a Hermitian metric is a reduction of the structure group of $E$ to the maximal compact subgroup $G^\rho$ of $G$. This is equivalent to finding an anti-linear involution $\rho : E \to E$ preserving the Lie algebra structure of the fibres, i.e. it is a choice of compact real form of $\mathfrak{g}$ on each fibre of $E$. The associated Hermitian form is then
\begin{equation}
H(u,v) = h_\rho(u,v) = -k(u,\rho(v))
\end{equation}
where $k$ is the Killing form on $\mathfrak{g}$. Recall that for $u \in \mathfrak{g}$, the adjoint of $u$ thought of as the endomorphism ${\rm ad}_u$ is given by $x^* = -\rho(x)$. Now since $G$ is the adjoint form of $\mathfrak{g}$ we can think of $G$ as a group of linear transformations in $\mathfrak{g}$. Corresponding to the real form $\rho$ on $\mathfrak{g}$ we may define a map $\rho : G \to G$ by $\rho(g) = (g^*)^{-1}$. Here the adjoint is with respect to the corresponding Hermitian form $H$. Note that the differential of $\rho$ at the identity is simply the anti-involution $\rho$ which explains why we denote both by $\rho$.\\

In order to write down an arbitrary such anti-involution $\rho$ we first define the following fixed anti-involution $\hat{\rho}$ of $\mathfrak{g}$:
\begin{equation}
\hat{\rho}(h_\alpha) = -h_\alpha; \; \; \; \hat{\rho}(e_\alpha) = -e_{-\alpha}; \; \; \; \hat{\rho}(e_{-\alpha}) = -e_\alpha.
\end{equation}
The fixed set of $\hat{\rho}$ defines a compact real form of $\mathfrak{g}$ and moreover any other compact real form  is conjugate to $\hat{\rho}$ by an inner automorphism.

We emphasize that $\rho$ is an anti-involution on the adjoint bundle $E$ while $\hat{\rho}$ is a fixed anti-involution on the Lie algebra $\mathfrak{g}$. Consider now a trivialisation of $E$ over an open subset $U$ which identifies the fibres of $E$ with $\mathfrak{g}$ preserving the Lie algebra structure. The anti-involution $\rho : E \to E$ defines on each fibre $E_x$ an anti-involution $\rho_x : E_x \to E_x$ and upon identifying $E_x$ with $\mathfrak{g}$, $\rho_x$ defines an anti-involution on $\mathfrak{g}$. Hence for each $x \in U$ there exists a $p(x) \in G$ such that $\rho_x = \alpha_x \hat{\rho} \alpha_x^{-1}$ where $\alpha_x = {\rm Ad}_{p(x)}$. Allowing $x$ to vary over $U$ we have a map $p : U \to G$ such that in the trivialisation over $U$, $\rho = {\rm Ad}_p \circ \hat{\rho} \circ {\rm Ad}_{p}^{-1}$. Thus
\begin{equation}\label{relate}
\rho = {\rm Ad}_h \circ \hat{\rho}
\end{equation}
where $h = p\hat{\rho}(p)^{-1}$. Equation (\ref{relate}) is a local formula relating $\rho$ to $\hat{\rho}$. Note also that ${\rm Ad}_h$ is a positive symmetric operator with respect to $\hat{\rho}$. The following lemma will show that we then have $h = {\rm exp}(2\Omega)$ for some $\Omega \in \mathfrak{g}$ with the property $\hat{\rho}(\Omega) = -\Omega$.\\

\begin{lemp}
Let $\mathfrak{g}$ be a complex Lie algebra with Hermitian form. If $h \in {\rm Aut}(\mathfrak{g})$ is positive and symmetric then $h = e^H$ where $H$ is symmetric and $H \in {\rm der}(\mathfrak{g})$ is a derivation of $\mathfrak{g}$.
\begin{proof}
We may decompose $\mathfrak{g}$ into the eigenspaces of $h$
\begin{equation*}
\mathfrak{g} = \bigoplus_\lambda \mathfrak{g}_\lambda.
\end{equation*}
Since $h$ is an automorphism we have $\left[ \mathfrak{g}_{\lambda_1} , \mathfrak{g}_{\lambda_2} \right] \subseteq \mathfrak{g}_{\lambda_1 \lambda_2}$. Now we may write $h = e^H$ where $H = {\rm log}(h)$ is symmetric. Further $H$ acts on $\mathfrak{g}_\lambda$ with eigenvalue ${\rm log}(\lambda)$. The fact that $H$ is a derivation follows immediately since ${\rm log}(\lambda_1\lambda_2) = {\rm log}(\lambda_1) + {\rm log}(\lambda_2)$.
\end{proof}
\end{lemp}

Consider how to find the Chern connection associated to the Hermitian metric $h_{\rho}$. In a local trivialisation where $\rho = {\rm Ad}_h \circ \hat{\rho}$ we have:
\begin{eqnarray*}
h_{\rho}(u,v) &=& -k(u,{\rm Ad}_{h}(\hat{\rho}(v))) \\
&=&  h_{\hat{\rho}}(u,{\rm Ad}_{h^{-1}} v).
\end{eqnarray*}
We can regard $h_{\hat{\rho}}$ as a fixed Hermitian form on $\mathfrak{g}$, hence we can think of the metric $h_\rho$ as being locally associated to the matrix $h^{-1} = {\rm exp}(-2{\rm ad}_\Omega)$. The connection form in this trivialisation is then $A = h \partial h^{-1}$.\\

We will need the following
\begin{lemp}
The anti-involution $\rho$ solving the Higgs bundle equations is unique
\begin{proof}
Any other solution is related by a $G$-valued gauge transformation $L: E \to E$ such that $L$ preserves the holomorphic structure and $\Phi$. We will argue that $L$ is the identity. It is a general fact \cite{hit2} that a gauge transformation that is holomorphic and commutes with $\Phi$ is in fact covariantly constant with respect to the unitary connection $\nabla_A$. Now as in (\ref{E}) we may write
\begin{equation}
E = \bigoplus_{m=-M}^M \mathfrak{g}_m \otimes K^m.
\end{equation}
But we have seen that $\nabla_A$ is $\mathfrak{h} = \mathfrak{g}_0$-valued, hence $\nabla_A$ preserves the subbundles $E_m = \mathfrak{g}_m \otimes K^m$. If we decompose $L$ into endomorphisms $L_{ij} : E_i \to E_j$ then each $L_{ij}$ must be covariantly constant. But since $E_i \otimes E_j^* = \mathfrak{g}_i \otimes \mathfrak{g}_j^* \otimes K^{i-j}$ there can only be a non-vanishing covariantly constant section if $i=j$ (since $\Sigma$ is assumed to have genus $g>1$). Therefore $L$ preserves $\mathfrak{g}_i$ for all $i$. Since $L$ preserves the Cartan subalgebra $\mathfrak{h} = \mathfrak{g}_0$ then $L$ is valued in the normaliser $N(T)$ of the maximal torus $T$. But $L$ also preserves $\mathfrak{g}_1 = \bigoplus_{\alpha \in \Pi} \mathfrak{g}_{\alpha}$. Therefore the induced action of $L$ on the roots of $\mathfrak{g}$ preserves the choice of simple roots, hence $L$ is valued in the maximal torus $T$.

We also have that $L$ commutes with $\Phi = \tilde{e} + qe_l$, in particular ${\rm Ad}_L \tilde{e} = \tilde{e}$. It follows that $L$ is the identity.
\end{proof}
\end{lemp}

We can now prove the following
\begin{prop}
The anti-involution $\rho$ locally has the form $\rho = {\rm Ad}_h \circ \hat{\rho}$ where $h = e^{2\Omega}$ and $\Omega$ is valued in $\mathfrak{h}$.
\begin{proof}
We have seen that the connection $A$ is valued in $\mathfrak{h}$ and that $\nabla_A$ preserves the element $g$. Note that $g$ commutes with the action of $x$ so that $g$ is a well defined global gauge transformation. Consider a new globally defined anti-involution given by $\mu = {\rm Ad}_g \circ \rho \circ {\rm Ad}_{g^{-1}} = {\rm Ad}_q \circ \rho$ where $q = g\rho(g)^{-1}$.

Now since both $g$ and $H$ are covariantly constant we see that $\mu$ yields a metric compatible with $A$. Thus the Chern connection corresponding to $\mu$ is also $A$. Now if we take the Higgs bundle equation for $\rho$:
\begin{equation*}
F_A + [\Phi , -\rho(\Phi)] = 0
\end{equation*}
and apply ${\rm Ad}_g$, noting that $g$ commutes with elements of $\mathfrak{h}$ we get
\begin{equation*}
F_A + [\omega\Phi, -{\rm Ad}_g \rho(\Phi)] = 0
\end{equation*}
but since $\omega{\rm Ad}_g \circ \rho (\Phi) = {\rm Ad}_g \circ \rho \circ {\rm Ad}_{g^{-1}} \Phi = \mu(\Phi)$ we get
\begin{equation*}
F_A + [\Phi,-\mu(\Phi)] = 0.
\end{equation*}
By uniqueness of solutions to the Higgs bundle equations we have that $\mu = \rho$, and so ${\rm Ad}_q$ is the identity. Thus $q$ lies in the centre of $G$ which is trivial since $G$ is the adjoint form, hence $\rho(g) = g$.\\

Locally we write $\rho = {\rm Ad}_h \circ \hat{\rho}$ where $h = e^{2\Omega}$. Then $g = \rho(g) = {\rm Ad}_h (\hat{\rho}(g)) = {\rm Ad}_h(g)$. Thus $g$ and $h$ commute. This implies that $h$ preserves the eigenspaces of $g$, in particular $h$ preserves the Cartan subalgebra $\mathfrak{h}$. Thus $h$ is valued in the normaliser $N(T)$. We will further argue that $h$ is valued in the maximal torus $T$.

Since $N(T)/T$ is a finite group (the Weyl group) there is some positive integer $a$ such that $h^a = e^{2a\Omega}$ is valued in $T$, say $h^a = e^{X}$ where $X$ is valued in $\mathfrak{h}$. However $h^a$ is a positive symmetric operator in the sense that $h^a = h^{a/2}\hat{\rho}(h^{-a/2}) = h^{a/2}(h^{a/2})^*$. Therefore $h^{2a} = e^{X - \hat{\rho}(X)}$ is also positive and symmetric. So $h^{2a}$ has a unique positive symmetric $2a$-th root which is $h$. Taking this root yields
\begin{equation*}
h = e^{(X - \hat{\rho}(X))/a}.
\end{equation*}
Replacing $2\Omega$ by $(X - \hat{\rho}(X))/a$ if necessary we have that $h = e^{2\Omega}$ is valued in $T$ and $\Omega$ is valued in $\mathfrak{h}$ with $\hat{\rho}(\Omega) = -\Omega$.
\end{proof}
\end{prop}

We need to understand the transition functions in order to form a global picture. On a trivialising open set $U_\alpha$ the adjoint bundle $E$ becomes $U_\alpha \times \mathfrak{g}$. Let $g_{\alpha \beta}$ be transition functions for the canonical bundle, so that a section $s$ of $K$ consists of local sections $s_\alpha$ such that $s_\alpha = g_{\alpha \beta}s_\beta$. Suppose that the $U_\alpha$ are simply connected so we have $g_{\alpha \beta} = e^{f_{\alpha \beta}}$ where $f_{\alpha \beta}$ is holomorphic. From $(\ref{E})$ we have that the transition functions for $E$ are ${\rm Ad}_{{g_{\alpha \beta}}^x} = {\rm Ad}_{e^{f_{\alpha \beta}x}}$. Note in particular that corresponding to the Cartan subalgebra $\mathfrak{h}$ is a trivial subbundle.\\

On the open set $U_\alpha$ the compact anti-involution $\rho$ is represented by an anti-involution $\rho_\alpha = {\rm Ad}_{h_\alpha} \circ \hat{\rho}$ where $h_\alpha = e^{2\Omega_\alpha}$. The corresponding Hermitian form is
\begin{eqnarray*}
-k(u , \rho_\alpha (v) ) &=& -k(u , {\rm Ad}_{e^{2\Omega_\alpha}} \hat{\rho}(v)) \\
&=& -k(u , \hat{\rho} ( {\rm Ad}_{e^{-2\Omega_\alpha}}v)) \\
&=& h_{\hat{\rho}}(u , {\rm Ad}_{e^{-2\Omega_\alpha}}v).
\end{eqnarray*}
Since the Hermitian forms for $\rho_\alpha$ and $\rho_\beta$ should agree on the overlap we can determine the transition behaviour for $h_\alpha$ and $\Omega_\alpha$:
\begin{eqnarray*}
h_{\hat{\rho}}(u_\alpha , {\rm Ad}_{h_\alpha^{-1}} v_\alpha) &=& h_\rho( {\rm Ad}_{{g_{\alpha \beta}}^x} u_\beta , {\rm Ad}_{h_\alpha^{-1}} ( {\rm Ad}_{{g_{\alpha \beta}}^x} v_\beta ) ) \\
&=& -k({\rm Ad}_{{g_{\alpha \beta}}^x} u_\beta , \hat{\rho} ({\rm Ad}_{h_\alpha^{-1}{g_{\alpha \beta}}^x} v_\beta ) ) \\
&=& -k( u_\beta , {\rm Ad}_{{g_{\alpha \beta}}^{-x}}\hat{\rho} ({\rm Ad}_{h_\alpha^{-1}{g_{\alpha \beta}}^x} v_\beta ) ) \\
&=& -k( u_\beta , \hat{\rho}({\rm Ad}_{{\overline{g}_{\alpha \beta}}^x h_\alpha^{-1}{g_{\alpha \beta}}^x} v_\beta ) ) \\
&=& h_{\hat{\rho}}( u_\beta ,({\rm Ad}_{{\overline{g}_{\alpha \beta}}^x h_\alpha^{-1}{g_{\alpha \beta}}^x} v_\beta ) ) \\
&=& h_{\hat{\rho}}(u_\beta , {\rm Ad}_{h_\beta^{-1}} v_\beta).
\end{eqnarray*}
So $h_\beta^{-1} = e^{\overline{f}_{\alpha \beta} x} h_\alpha^{-1} e^{f_{\alpha \beta}x}$ and hence
\begin{eqnarray}
h_\alpha &=& e^{f_{\alpha \beta}x} h_\beta e^{\overline{f}_{\alpha \beta} x} \\
2\Omega_\alpha &=& (f_{\alpha \beta} + \overline{f}_{\alpha \beta} )x + 2\Omega_\beta.
\end{eqnarray}

In particular the local transformation law for $h_\alpha$ can be rewritten
\begin{equation}
h_\alpha = e^{(f_{\alpha \beta} + \overline{f}_{\alpha \beta})x } h_\beta.
\end{equation}

Having established that $\Omega$ is $\mathfrak{h}$-valued, we have that $\Omega$ and $\partial \Omega$ commute and the connection is simply $A = -2\partial \Omega$. We have in turn that the curvature is $F_A = -2\overline{\partial}\partial \Omega$.

Next we determine $\Phi^* = -\rho(\Phi)$. Recall that $\Phi = \tilde{e} + qe_l$. Thus $\hat{\rho}(\Phi) = -e - \overline{q}\tilde{e_l}$ where $\tilde{e_l}$ is a lowest weight vector for $\mathfrak{g}$. Thus
\begin{eqnarray*}
\Phi^* &=& {\rm Ad}_h (e+ \overline{q}\tilde{e_l}) \\
&=& \sum_{\alpha \in \Pi} \sqrt{r_\alpha}e^{2\alpha(\Omega)}e_\alpha +\overline{q}e^{-2\delta(\Omega)}\tilde{e_l}.
\end{eqnarray*}
We can now determine the commutator $[\Phi,\Phi^*]$:
\begin{eqnarray*}
[\Phi,\Phi^*] &=& \left[ \sum_{\alpha \in \Pi}\sqrt{r_\alpha}e_{-\alpha} + qe_l ,\sum_{\alpha \in \Pi} \sqrt{r_\alpha}e^{2\alpha(\Omega)}e_\alpha +\overline{q}e^{-2\delta(\Omega)}\tilde{e_l} \right] \\
&=& -\sum_{\alpha \in \Pi}r_\alpha e^{2\alpha(\Omega)}h_\alpha -q\overline{q}e^{-2\delta(\Omega)}h_{-\delta}.
\end{eqnarray*}
Although it has been suppressed, this should really be multiplied by a $dz \wedge d\overline{z}$ term. The Higgs bundle equation for $\Omega$ now becomes:
\begin{equation}\label{toda}
-2\Omega_{z\overline{z}} + \sum_{\alpha \in \Pi}r_\alpha e^{2\alpha(\Omega)}h_\alpha +q\overline{q}e^{-2\delta(\Omega)}h_{-\delta} = 0.
\end{equation}

If $\mathfrak{g}$ is a simple Lie algebra of type $L_n$ then these are a version of the {\em affine Toda field equations} \cite{mik} for the Affine Dynkin diagram $L_n^{(1)}$ which we introduce more generally in Section \ref{sectoda}. Note that since $\Omega$ is $\mathfrak{h}$-valued we have $\rho(\Omega) = \hat{\rho}(\Omega) = -\Omega$ and similarly for $h_\alpha, h_{-\delta}$. Thus $\Omega$ is a real linear combination of the $h_\alpha$. However there is an additional constraint on $\Omega$ that we have yet to consider. By examining this condition we will find that in some cases the above equations will reduce to affine Toda field equations for a smaller affine Dynkin diagram.

\subsection{Reality}\label{reality}
The additional constraint on $\Omega$ that we have yet to consider relates to the reality properties of $\Omega$. Recall that we define an involutive automorphism $\sigma$ on $\mathfrak{g}$ by the properties
\begin{equation*}
\sigma(e_i) = -e_i, \; \; \; \sigma(\tilde{e}) = -\tilde{e}.
\end{equation*}
We also note that $\sigma$ commutes with $\hat{\rho}$ and $\hat{\lambda} = \sigma \hat{\rho} = \hat{\rho}\sigma$ is an anti-involution corresponding to the split real form of $\mathfrak{g}$. As before we have $\sigma(\Phi) = -\Phi$ and $\sigma$ is covariantly constant with respect to $\nabla_A$.\\

We shall argue that in fact $\sigma(\Omega) = \Omega $. First note that since $\sigma A = A$ we have $\sigma(\Omega_{z\overline{z}}) = \Omega_{z\overline{z}}$. Thus the remaining terms of the Toda equations (\ref{toda}) must also be fixed by $\sigma$. We also know \cite{hit1} that $\sigma = \phi \circ \hat{\nu}$ where $\phi$ is the inner automorphism corresponding to a rotation by $\pi$ in the principal $3$-dimensional subgroup and $\hat{\nu}$ is a lift of an automorphism $\nu$ of the Dynkin diagram for $\Pi$. In fact $\nu$ is trivial except for the simple Lie algebras of type $A_n,D_{2n+1},E_6$ for which it has order $2$.

In any case, $\phi$ is a conjugation by an element of $\mathfrak{h}$, so that $\sigma$ agrees with $\hat{\nu}$ on $\mathfrak{h}$, that is we have
\begin{equation}
\sigma({h_\alpha}) = h_{\nu(\alpha)}, \; \; \; \sigma(h_{-\delta}) = h_{-\delta}.
\end{equation}
We also have that $\sigma(x) = x$ from which it follows that $r_{\nu(\alpha)} = r_{\alpha}$. By invariance of the right hand side of (\ref{toda}) under $\sigma$ we see that $e^{2\alpha(\Omega)} = e^{2\nu(\alpha)(\Omega)}$, now since $\Omega$ is real with respect to the basis $\{h_\alpha\}$ we may take logarithms to conclude that $\alpha(\Omega) = \nu(\alpha)(\Omega)$ for all $\alpha \in \Pi$, that is $\sigma (\Omega) = \Omega$.\\

From the relation $\sigma(\Omega) = \Omega$ we have that $\rho$ and $\sigma$ commute and thus $\lambda = \sigma \rho = \rho \sigma$ is a real structure on $E$ conjugate to $\hat{\lambda}$, hence $\lambda$ defines a reduction of structure to the split real form $G^\lambda$ of $G$.\\

Now we have that $A$ is invariant under $\rho$ and $\sigma$, hence $\lambda$ also. Therefore $\nabla_A$ has holonomy in $K = G^\rho \cap G^\lambda$, the maximal compact subgroup of $G^\lambda$, in fact $A$ takes values in a maximal torus $T$ of $K$.

Let $\mathfrak{t}$ denote the Lie algebra of $T$, that is $\mathfrak{t} = \mathfrak{h}^\lambda$. The inclusion $i: \mathfrak{t} \to \mathfrak{h}$ induces the restriction map $r: \mathfrak{h}^* \to \mathfrak{t}^*$. Since $\Omega$ is $\mathfrak{t}$-valued we wish to reinterpret the Toda field equations in terms of the restricted roots on $\mathfrak{t}^*$. Intuitively speaking we are dividing out the action of $\nu$.\\

We can realise $\mathfrak{t}^*$ as a subspace of $\mathfrak{h}^*$. In fact the kernel of the restriction map is the $-1$-eigenspace of $\hat{\nu}^t$ so that $\mathfrak{t}^*$ identifies with the $+1$ eigenspace. Under this identification the restriction map becomes the orthogonal projection map
\begin{equation*}
r(\alpha) = \tfrac{1}{2}(\alpha + \hat{\nu}^t(\alpha)).
\end{equation*}
For any $\beta \in \mathfrak{t}^*$ define the dual vector $\tilde{h}_\beta$ by the usual formula
\begin{equation*}
\alpha(\tilde{h}_\beta) = 2\frac{(\alpha,\beta)}{(\beta,\beta)}
\end{equation*}
for all $\alpha \in \mathfrak{t}^*$. Then, for possibly different constants $\tilde{r}_\beta$ the Toda equation (\ref{toda}) reduces to:
\begin{equation}\label{toda2}
-2\Omega_{z\overline{z}} + \sum_{\beta \in r(\Pi)}\tilde{r}_\beta e^{2\beta(\Omega)}\tilde{h}_\beta + q\overline{q}e^{-2\delta(\Omega)}\tilde{h}_{-\delta} = 0.
\end{equation}
The elements of $r(\Pi)$ together with $-\delta = -r(\delta)$ correspond to an affine Dynkin diagram as explained in Section \ref{sectoda}. That is we have a system of vectors in a Euclidean vector space $\mathfrak{t}^*$ whose Dynkin diagram is an affine Dynkin diagram. Equation (\ref{toda2}) is then the affine Toda field equation for the corresponding affine Dynkin diagram.

All that remains is to actually determine which affine root system they correspond to. If $\nu$ is trivial then $\mathfrak{t} = \mathfrak{h}$ and the affine diagram is just the extended Dynkin diagram $L_n^{(1)}$ obtained by adding to the root system $L_n$ the lowest weight $-\delta$. The remaining cases follow from a straightforward calculation and are as follows: for $\mathfrak{g}$ of type $A_{2n},A_{2n-1},D_{2n+1},E_6$, the corresponding affine diagrams are $A_{2n}^{(2)},C_n^{(1)},B_{2n}^{(1)},F_4^{(1)}$ respectively.

\section{The affine Toda equations}\label{sectoda}
We have already encountered a class of equations that we termed affine Toda equations. The Toda equations have been extensively studied from the point of view of integrable systems \cite{mik}, \cite{man}, \cite{oli}, \cite{lez} as well as their relation to minimal surfaces and harmonic maps \cite{bolw} \cite{dol1}, \cite{bol}, \cite{dol2}, \cite{dol3}. This should not be surprising given the link between Higgs bundles and harmonic maps. However our treatment of the affine Toda equations allows for more general real forms of the equations than are usually considered. We also review the links to minimal surfaces and harmonic maps.\\

Let $A$ be an $n \times n$ indecomposable generalised Cartan matrix of affine type. We label the rows and columns by $0,1, \cdots , l$, where $l = n-1$. We arrange so that the $l \times l$ matrix $A^0$ formed by removing row $0$ and column $0$ is the Cartan matrix for the root system of a complex simple Lie algebra. So there is a real $l$-dimensional Euclidean vector space $(\mathfrak{h},\langle \, , \, \rangle)$ with basis $\{h_i\}$, $i = 1, \ldots , l$ such that if we define a corresponding basis $\{\alpha_i \}$, $i = 1, \ldots , l$ of $\mathfrak{h}^*$ by $\alpha_i = 2h_i/\langle h_i , h_i \rangle $ then $A_{ij} = \alpha_j(h_i)$ for $i,j = 1, \ldots , l$.

Moreover \cite{carter} there exists positive integers $a_0,a_1, \ldots , a_l$ with no common factor such that $A (a_0 , \ldots , a_l )^t = 0$. Similarly there exists positive integers $c_0, c_1 , \ldots , c_l$ with no common factor such that $(c_0 , \ldots , c_l)A = 0$. Let us define $h_0 \in \mathfrak{h}$ and $\alpha_0 \in \mathfrak{h}^*$ by
\begin{eqnarray}
h_0 &=& -\sum_{i=1}^l c_i h_i, \\
\alpha_0 &=& -\sum_{i=1}^l a_i \alpha_i.
\end{eqnarray}
Then it follows that
\begin{equation}
\alpha_i (h_j) = A_{ji}, \, i,j = 0,1,\ldots , l.
\end{equation}

Now we may define the affine Toda equations
\begin{defn}
Given an $n \times n$ indecomposable generalised Cartan matrix of affine type let $\mathfrak{h}$, $\{h_i\}$, $\{\alpha_i\}$ be as above. Let $\Sigma$ be a Riemann surface. A map $\Omega : \Sigma \to \mathfrak{h} \otimes \mathbb{C}$ is said to satisfy the {\em affine Toda equations} if it is a solution to
\begin{equation}
2\Omega_{z\overline{z}} = \sum_{i=0}^l k_i e^{2\alpha_i(\Omega)}h_i,
\end{equation}
where $k_i \in \mathbb{C}$ are constants, $k_i \ne 0, \, i = 0, \ldots , l$.
\end{defn}
Typically the affine Toda equations arise with reality conditions on $\Omega$. If however this is not the case then we can replace $\Omega$ by $\Omega + u$ where the constant $u \in \mathfrak{h} \otimes \mathbb{C}$ satisfies $e^{-2\alpha_i(u)} = k_i$ for $i=1, \ldots , l$. This new $\Omega$ then satisfies
\begin{equation}
2\Omega_{z\overline{z}} = \sum_{i=1}^l e^{2\alpha_i(\Omega)}h_i + qe^{2\alpha_0(\Omega)}h_0
\end{equation}
for some constant $q$. However when reality conditions are imposed we might only be able to reduce the $k_i$ to $\pm 1$.\\

Let us now consider a reality condition that will arise in connection with $\tau$-primitive maps. Let $\mu$ be the anti-linear extension to $\mathfrak{h} \otimes \mathbb{C}$ of an involutive isometry on $\mathfrak{h}$. Further assume that there is an involutive permutation $\sigma$ of $\{0,1,\ldots,l\}$ such that $\mu(h_i) = -h_{\sigma(i)}$. We can now impose the following reality condition on solutions $\Omega$ of the affine Toda equations:
\begin{equation}
\mu(\Omega) = -\Omega.
\end{equation}
For consistency this requires the coefficients $k_i$ satisfy $k_{\sigma(i)} = \overline{k_i}$.

If $\sigma(i) = i$ and $i \ne 0$ then $k_i$ is real and by a suitable replacement of $\Omega$ by $\Omega + u$ we may replace $k_i$ by $\pm 1$. On the other hand in the case $\sigma(i) = j$, with $i \ne j$ and $i,j \ne 0$ then a suitable replacement of $\Omega$ allows us to set $k_i = k_j = 1$. Therefore aside from $k_0$ and $k_{\sigma(0)}$ we may set all $k_i$ equal to $\pm 1$.

\subsection{Primitive maps and the affine Toda equations}\label{primandtoda}
We extend the notion of $\tau$-primitive maps defined in \cite{bol} and show they are essentially equivalent to the affine Toda equations.\\

Let $G$ be a connected complex semisimple Lie group with Lie algebra $\mathfrak{g}$. Let $\mathfrak{h}$ be a Cartan subalgebra and $T$ a corresponding maximal torus. As in Section \ref{speccase} let $g = e^{2\pi i x/(M+1)}$ and let $\tau = {\rm Ad}_g$ be the corresponding automorphism of $G$. Then $\tau$ has order $M+1$ and fixes $T$. Let $\mathfrak{g}_m$ denote the eigenspace of $\mathfrak{g}$ on which $\tau$ has eigenvalue $e^{2\pi i m/(M+1)}$. Let us denote by $W$ the orthogonal complement of $\mathfrak{g}_0$ in $\mathfrak{g}$ with respect to the Killing form. Then
\begin{equation}
W = \bigoplus_{m = 1}^M \mathfrak{g}_m.
\end{equation}
and $W$ is preserved by the the adjoint action of $T$. Using the Maurer-Cartan form we have an isomorphism $T(G/T) = G \times_T W$ and $G/T$ inherits a metric from the Killing form on $W$. Hence $\tau$ is also as an isometry on $G/T$. Furthermore the decomposition of $W$ leads to a corresponding decomposition for $T(G/T)$
\begin{equation}
T(G/T) = \bigoplus_{m = 1}^M E_m
\end{equation}
where $E_m = G \times_T \mathfrak{g}_m$.\\

Let $G^\mu$ be the connected subgroup of $G$ generated by the exponents of the fixed points $\mathfrak{g}^\mu$ of an antilinear involution $\mu$ on $\mathfrak{g}$. We say that $\mu$ is {\em compatible} with $\tau$ if $\tau$ preserves $\mathfrak{g}^\mu$ or equivalently if $\tau$ and $\mu$ commute.

Suppose that $\mu$ is a compatible real form. We note that $\mu$ maps $\mathfrak{g}_i$ to $\mathfrak{g}_{-i}$ and in particular $\mu$ preserves $\mathfrak{g}_0$. Let $\mathfrak{g}_0^\mu$ denote the fixed point set of $\mu$ acting on $\mathfrak{g}_0$. The corresponding subgroup $T^\mu$ is a maximal connected abelian subgroup of $G^\mu$. So $\tau$ also acts on $G^\mu/T^\mu$ and the tangent space decomposition now takes the form
\begin{equation}
T(G^\mu/T^\mu) \otimes \mathbb{C} = \bigoplus_{m=1}^M E_m.
\end{equation}

\begin{defn}
Let $\Sigma$ be a connected Riemann surface. A map $\phi : \Sigma \to G^\mu /T^\mu$ is called {\em $\tau$-primitive} if $\phi_*(T^{1,0}\Sigma) \subseteq E_1$ and is cyclic in the sense of Section \ref{speccase} for at least one value of $p \in \Sigma$.
\end{defn}

\begin{prop}
Every $\tau$-primitive map is harmonic.
\begin{proof}
Suppose $\phi : \Sigma \to G^\mu/T^\mu$ is a $\tau$-primitive map. Locally we can find a lift $\tilde{\phi}: U \to G^\mu$ where $U$ is an open subset of $\Sigma$. Let $z$ be a local holomorphic coordinate on $U$. The $\tau$-primitive condition is equivalent to
\begin{equation}
\tilde{\phi}^{-1} \partial_z \tilde{\phi} \in \mathfrak{g}_0 \oplus \mathfrak{g}_1.
\end{equation}
Let us write $\tilde{\phi}^{-1} \partial_z \tilde{\phi} = A_0 + A_1$ with $A_i \in \mathfrak{g}_i$ for $i = 0,1$. Then we also have
\begin{equation}
\tilde{\phi}^{-1} \partial_{\overline{z}} \tilde{\phi} = \mu(A_0) + \mu(A_1) \in \mathfrak{g}_0 \oplus \mathfrak{g}_{-1}.
\end{equation}
Now the Maurer-Cartan equation for $\tilde{\phi}^{-1} d\tilde{\phi}$ gives
\begin{eqnarray}
\mu(A_0)_z - (A_0)_{\overline{z}} + \left[ A_1 , \mu(A_1) \right] &=& 0 \\
-(A_1)_{\overline{z}} + \left[ A_1 , \mu(A_0) \right] &=& 0 \\
\mu(A_1)_z + \left[ A_0 , \mu(A_1) \right] &=& 0.
\end{eqnarray}

Now let us consider the harmonic map equations (\ref{harmsurf1}), (\ref{harmsurf2}) in the case of a map $\phi: \Sigma \to G^\mu/V$ into a homogeneous space such that $V = T^\mu$ is a torus.

In this case $A^{\mathfrak{v}},B^{\mathfrak{v}} \in \mathfrak{g}_0 \otimes \mathbb{C}$, $A^{\mathfrak{h}} \in \mathfrak{g}_1$, $B^{\mathfrak{h}} \in \mathfrak{g}_{-1}$ and we have $\left[ B^{\mathfrak{h}} , A^{\mathfrak{h}} \right]^{\mathfrak{h}} = 0$ since $\left[ \mathfrak{g}_1 , \mathfrak{g}_{-1} \right] \subset \mathfrak{g}_0$. The harmonic map equations are
\begin{eqnarray*}
A_{\overline{z}}^{\mathfrak{h}} + \left[ B^{\mathfrak{v}} , A^{\mathfrak{h}} \right] &=& 0, \\
B_z^{\mathfrak{h}} + \left[ A^{\mathfrak{v}} , B^{\mathfrak{h}} \right] &=& 0.
\end{eqnarray*}
If we put $A^{\mathfrak{v}} = A_0$, $A^{\mathfrak{h}} = A_1$, $B^{\mathfrak{v}} = \mu(A_0)$, $B^{\mathfrak{h}} = \mu(A_1)$ then the above equations are clearly seen to be just be two of the Maurer-Cartan equations for a primitive map. Therefore primitive maps are harmonic.
\end{proof}
\end{prop}

\subsection{Toda framing}\label{todafr}
We follow \cite{bol} in introducing the notion of a Toda framing, which will lead to the affine Toda equations. Let $p_1, \dots , p_l$ be homogeneous generators of the ring of invariant polynomials with $p_l$ having the highest degree $M+1$. If $\phi : \Sigma \to G^\mu/T^\mu$ is $\tau$-primitive then we define a $(M+1,0)$-differential $q$ on $\Sigma$ as
\begin{equation}
q = p_l( \partial \phi , \dots , \partial \phi ).
\end{equation}

\begin{lemp}
The differential $q$ is holomorphic.
\begin{proof}
Let $\nabla$ be the Levi-Civita connection on $G^\mu/T^\mu$. Since $p_l$ is an invariant polynomial it is parallel. Thus
\begin{equation}
\partial_{\overline{z}} p_l( \phi_z , \dots , \phi_z ) = (M+1) p_l ( \nabla_{\overline{z}} \phi_z , \phi_z , \dots , \phi_z ).
\end{equation}
But since $\phi$ is harmonic, this vanishes.
\end{proof}
\end{lemp}

Now if $\phi$ is $\tau$-primitive then by definition $\phi_z$ must be cyclic at at least one point, so the corresponding holomorphic differential $q$ does not identically vanish. So $q$ vanishes only at isolated points and $\phi_z$ is cyclic except at isolated points. As before let $\tilde{\phi}: U \to G^\mu$ be a local lift of $\phi$ and write $\tilde{\phi}^{-1} \partial_z \tilde{\phi} = A_0 + A_1$. Let us define
\begin{equation}
\Phi = e + q e_{-\delta}.
\end{equation}
We take $p_l$ as in (\ref{poly}) so that $p_l(\Phi) = q$. Away from the zeros of $q$, $\Phi$ is cyclic and it follows that locally there exists a map $\Omega : U \to \mathfrak{g}_0$ such that
\begin{equation}\label{omega3}
{\rm Ad}_{e^\Omega}(\Phi) = A_1.
\end{equation}
If we change the lift $\tilde{\phi}$ we may further assume that $\mu(\Omega) = -\Omega$. A lift $\tilde{\phi} : U \to G^\mu$ with this property is called a {\em Toda framing} in \cite{bol}. Now differentiating (\ref{omega3}) gives
\begin{equation}
(A_1)_{\overline{z}} = \left[ \Omega_{\overline{z}} , A_1 \right].
\end{equation}
Combined with the Maurer-Cartan equation we find
\begin{equation}
\left[ \Omega_{\overline{z}} + \mu(A_0), A_1 \right] = 0
\end{equation}
which is only possible if $\mu(A_0) = -\Omega_{\overline{z}}$, that is
\begin{equation}
A_0 = \Omega_z.
\end{equation}
Combining with the first Maurer-Cartan equation now yields
\begin{equation}\label{pretoda}
2\Omega_{z\overline{z}} = \left[ {\rm Ad}_{e^{2\Omega}}(\Phi) , \mu(\Phi) \right].
\end{equation}
Now if we write $e = \sum_{i = 1}^l t_i e_i$ for some positive constants $t_i$ and denote $e_{-\delta}$ by $e_0$. Then
\begin{equation}
\Phi = \sum_{i=1}^l t_i e_i + qe_0.
\end{equation}
Now if we denote $-\delta$ by $\alpha_0$ then
\begin{equation}
{\rm Ad}_{e^{2\Omega}}(\Phi) = \sum_{i=1}^l t_i e^{2\alpha_i(\Omega)} e_i + qe^{2\alpha_0(\Omega)}e_0.
\end{equation}
To work out $\mu(\Phi)$ we need to consider the action of $\mu$ on $\mathfrak{h} = \mathfrak{g}_0$. We define an antilinear action $\mu^*$ of $\mu$ on $\mathfrak{h}^* \otimes \mathbb{C}$ by
\begin{equation}
\mu^*(\alpha)(X) = \alpha( \mu(X)),
\end{equation}
where $\alpha \in \mathfrak{h}^* \otimes \mathbb{C}$, $X \in \mathfrak{h}$. Now if $h \in \mathfrak{h}$ and $\alpha \in \mathfrak{h}^*$ is a root then there exists a non-zero $X \in \mathfrak{g}$ with $\left[ h , X \right] = \alpha(h)X$. If we replace $h$ by $\mu(h)$ and apply $\mu$ we then find $\left[ h , \mu(X) \right] = \mu^*(\alpha) \mu(X)$, since $\alpha(h)$ is real. Thus $\mu^*$ acts as an involution on the roots of $\mathfrak{g}$. In particular $\mu$ preserves $\mathfrak{h}$ and $\mathfrak{h}^*$. It is also clear that $\mu$ preserves the Killing form on $\mathfrak{h}$. Finally since $\mu(\mathfrak{g}_i) = \mathfrak{g}_{-i}$ we have that $\mu^*$ sends $\{\alpha_0 , \alpha_1 , \ldots \alpha_l \}$ to $\{-\alpha_0 , -\alpha_1 , \ldots -\alpha_l \}$. Therefore there is an involutive permutation $\sigma$ of $\{0,1,\ldots , l\}$ such that
\begin{equation}
\mu^*(\alpha_i) = -\alpha_{\sigma(i)}, \; \; \;
\mu(h_i) = -h_{\sigma(i)}.
\end{equation}
Now since $\mu$ sends the root space of $\alpha_i$ to the root space of $\mu^*(\alpha_i)$ we further deduce that
\begin{equation}
\mu(e_i) = s_i f_{\sigma(i)}, \; \; \;
\mu(f_i) = \overline{s_{\sigma(i)}}\,^{-1} e_{\sigma(i)}.
\end{equation}
For some complex constants $s_i \ne 0$.
Now applying $\mu$ to $\left[e_i , f_i \right] = h_i$ we also find $s_{\sigma(i)} = \overline{s_i}$. We can now determine $\mu(\Phi)$:
\begin{equation}
\mu(\Phi) = \sum_{i=1}^l \overline{t_i} s_i f_{\sigma(i)} + \overline{q}s_0 f_0.
\end{equation}
Therefore equation (\ref{pretoda}) becomes
\begin{equation}\label{todaprimitive}
2\Omega_{z\overline{z}} = \sum_{i=0}^l t_i \overline{t_{\sigma(i)}}s_{\sigma(i)} e^{2\alpha_i(\Omega)}h_i
\end{equation}
where $t_0 = q$. Away from the zeros of $q$ we can find a local holomorphic coordinate such that $q$ is constant. Then this equation is an affine Toda equation for the extended Dynkin diagram of $\mathfrak{g}$ with reality condition $\mu(\Omega) = -\Omega$. Conversely a solution of (\ref{todaprimitive}), regardless of whether $q$ vanishes or not defines a solution of the Maurer-Cartan equations and hence a $\tau$-primitive map $\phi : \tilde{\Sigma} \to G^\mu / T^\mu$ where $\tilde{\Sigma}$ is the universal cover of $\Sigma$.

\subsection{Relation to Higgs bundles}\label{rthb}
Here we show that for cyclic Higgs bundles the harmonic primitive map is closely related to the harmonic map corresponding to the Higgs bundle.\\

Suppose we have a cyclic Higgs bundle. Then in the notation of Section \ref{speccase} we may locally define the following $\mathfrak{g}$-valued $1$-forms:
\begin{equation}
\mathcal{B}_\alpha = -\partial\Omega_\alpha + {\rm Ad}_{e^{-\Omega_\alpha}}(\Phi_\alpha) + \overline{\partial}\Omega_\alpha - {\rm Ad}_{e^{\Omega_\alpha}}(\hat{\rho}\Phi_\alpha).
\end{equation}
We have that $\hat{\tau}(\mathcal{B}_\alpha) = \mathcal{B}_\alpha$ and that $d\mathcal{B}_\alpha + \tfrac{1}{2}[\mathcal{B}_\alpha , \mathcal{B}_\alpha ] = 0$. Hence locally there exist maps $\mathcal{D}_\alpha : U_\alpha \to G^{\hat{\tau}}$ into the split real form of $G$ such that $\mathcal{D}_\alpha^{-1} d\mathcal{D}_\alpha = \mathcal{B}_\alpha$. The maps $\mathcal{D}_\alpha$ are defined up to left multiplication by a constant element of $G^{\hat{\tau}}$.

Based on the transformation laws for $\Omega_\alpha$ and $\Phi_\alpha$, we can show that if $\mathcal{D}_\alpha$ is a solution to $\mathcal{D}_\alpha^{-1} d\mathcal{D}_\alpha = \mathcal{B}_\alpha$ then on the intersection $U_\alpha \cap U_\beta$, $\, \mathcal{D}_\alpha e^{\tfrac{1}{2}(\overline{f}_{\alpha \beta} - f_{\alpha \beta})x}$ is a solution to $\mathcal{D}_\beta^{-1} d\mathcal{D}_\beta = \mathcal{B}_\beta$. Therefore if we start with a fixed $U_\alpha$ and fixed $\mathcal{D}_\alpha$, we can choose the $\mathcal{D}_\beta$ on neighboring open sets $U_\beta$ such that on the overlap $\mathcal{D}_\beta = \mathcal{D}_\alpha e^{\tfrac{1}{2}(\overline{f}_{\alpha \beta} - f_{\alpha \beta})x}$. We can continue on in this fashion, though if we continue around a loop we need not return to the original $\mathcal{D}_\alpha$. This means that we get a map
\begin{equation}
\psi : \tilde{\Sigma} \to G^{\hat{\tau}}/(T^{\hat{\tau}} \cap K)
\end{equation}
where $\tilde{\Sigma}$ is the universal cover of $\Sigma$ and $T^{\hat{\tau}}$ is the maximal abelian subgroup of $G^{\hat{\tau}}$ corresponding to $\mathfrak{h}^{\hat{\tau}}$.\\

Note that the connection form $\mathcal{B}_\alpha$ is gauge equivalent to the connection forms
\begin{equation}
\mathcal{A}_\alpha = -2\partial \Omega_\alpha + \Phi_\alpha - {\rm Ad}_{e^{2\Omega_\alpha}}(\hat{\rho}\Phi_\alpha)
\end{equation}
under a gauge change $e^{\Omega_\alpha}$. The $\mathcal{A}_\alpha$ are connection forms for the flat connection corresponding to the Higgs bundle. Note that the Hermitian connection $\nabla_A$ in this gauge is $-2\partial \Omega_\alpha$. If we gauge transform this into the gauge corresponding to the $\mathcal{B}_\alpha$ then it becomes $-\partial\Omega_\alpha + \overline{\partial}\Omega_\alpha$. This Lie algebra valued $1$-form is valued in the fixed compact real form $C$ of $G^{\hat{\tau}}$ defined by $\hat{\rho}$. This shows that the map $\psi : \tilde{\Sigma} \to G^{\hat{\tau}}/(T^{\hat{\tau}} \cap K)$ projects to the map $\pi(\psi) : \tilde{\Sigma} \to G^{\hat{\tau}}/C$ which defines the Hermitian metric satisfying the Higgs bundle equations. That is, the map $\psi$ projects to the harmonic map corresponding to the Higgs bundle. If we instead project from $\psi$ to $p(\psi) : \tilde{\Sigma} \to G^{\hat{\tau}}/T^{\hat{\tau}}$ we see that this map is $\tau$-primitive and hence harmonic. In fact it is straightforward to show that $\psi$ itself is harmonic. Thus $\psi$ is a harmonic mutual lift of the harmonic map corresponding to a Higgs bundle and a $\tau$-primitive map.

\section{Minimal surfaces and affine Toda equations}\label{msaate}
The affine Toda equations are well known to be related to certain cases of minimal surfaces \cite{dol1}, \cite{bol}. The relationship is established through what is known as the harmonic sequence \cite{cal}, \cite{eew}, \cite{hul}. We extend the notion of harmonic sequences to arbitrary signature. Furthermore the harmonic sequence can be related easily to the Higgs bundle picture. This shows how our restricted class of Higgs bundles can be interpreted as a minimal surface into a quadric.

\subsection{Minimal surfaces in quadrics}\label{secquadrics}

We will develop some general theory of minimal surfaces in a quadric. For the most part this is a straightforward generalisation of the case of minimal surfaces in a sphere \cite{cal}. See also \cite{hul} for the case of De Sitter space. This leads us to define superminimal and superconformal maps which are solutions of certain real forms of the Toda field equations \cite{bol}.\\

Let $\mathbb{R}^{p,q}$ denote the $p+q$-dimensional vector space with bilinear form $\langle \; , \; \rangle$ of signature $(p,q)$ and let $Q_{h_0} = \{x \in \mathbb{R}^{p,q} \,| \,\langle x , x \rangle = h_0 \}$ where $h_0 = \pm 1$. Let $\Sigma$ be a connected oriented surface. An immersion $\phi : \Sigma \to Q_{h_0} \subset \mathbb{R}^{p,q}$ that sends each tangent space of $\Sigma$ to a positive definite subspace of $\mathbb{R}^{p,q}$ induces a metric and hence a conformal structure on $\Sigma$. These give $\Sigma$ the structure of a Riemann surface with compatible metric. The map $\phi$ is then a minimal immersion if and only if it is harmonic.

Let us regard $\phi$ as a vector valued function on $\Sigma$ such that $\langle \phi , \phi \rangle = h_0$. Let $D$ denote the trivial connection on $\mathbb{R}^{p,q}$ and $\nabla$ the induced Levi-Civita connection on $Q$. Let $z$ be a local holomorphic coordinate on $\Sigma$. We use $z$ and $\overline{z}$ subscripts to denote partial differentiation with respect to $\tfrac{\partial}{\partial z}$ and $\tfrac{\partial}{\partial \overline{z}}$. The harmonic equation for $\phi$ is
\begin{equation}
\nabla_{\overline{z}}{\phi_z} = 0.
\end{equation}
However we also have that $D = \nabla + \Pi$ where $\Pi$ is the second fundamental form which is valued in the normal bundle of $Q$, since $Q$ is given the induced metric. Therefore the harmonic equation for $\phi$ reduces to
\begin{equation}
\phi_{z\overline{z}} = \lambda \phi
\end{equation}
for some function $\lambda$. In fact we can show $\lambda = -\langle \phi_z , \phi_{\overline{z}} \rangle/h_0$. Note that since $\langle \phi , \phi \rangle = h_0$ we have $\langle \phi , \phi_z \rangle = \langle \phi , \phi_{\overline{z}} \rangle = 0$ and since $\phi$ is a Riemannian immersion $\langle \phi_z , \phi_z \rangle = \langle \phi_{\overline{z}} , \phi_{\overline{z}} \rangle = 0$ and $\langle \phi_z , \phi_{\overline{z}} \rangle = h_1$ defines the induced metric on $\Sigma$.\\

Following the definite signature case \cite{bol} we will inductively construct a sequence $\phi_0,\phi_1,\phi_2, \dots $. We define $\phi_0 = \phi$, $\phi_1 = \phi_z$ and we define $\phi_{i+1}$ from $\phi_i$ under the assumption that $h_i = \langle \phi_i , \overline{\phi_i} \rangle$ is non-vanishing by
\begin{equation}
\phi_{i+1} = (\phi_i)_z - \frac{\langle (\phi_i)_z , \overline{\phi_i} \rangle}{h_i}\phi_i.
\end{equation}
So $\phi_{i+1}$ is the projection of $(\phi_i)_z$ onto the orthogonal complement of $\phi_i$ with respect to the Hermitian form $h(X,Y) = \langle X , \overline{Y} \rangle$. Note that unlike the definite signature case it may be that $h_i$ vanishes even if $\phi_i \ne 0$.\\

Let $r \ge 1$ be an integer. Following \cite{bur}, \cite{hul} we say $\phi$ has {\em isotropy of order $r$} if we can construct the sequence $\phi_0,\phi_1, \dots , \phi_r, \phi_{r+1}$ (so $h_1,\dots, h_r$ are non-vanishing) and $\langle \phi_i , \phi_i \rangle = 0$ for $1\le i \le r$. Since the $h_i$ are non-vanishing we may write $h_i = \epsilon_i H_i$ where $H_i = e^{2w_i}$ and $\epsilon_i = \pm 1$. We then have
\begin{prop}
Let $\phi$ be a minimal surface with isotropy of order $r \ge 1$. We have the following
\begin{itemize}
\item{The sequence $\overline{\phi_r}, \dots , \overline{\phi_1} , \phi_0 , \phi_1, \dots , \phi_r$ is pairwise orthogonal with respect to the Hermitian form,}
\item{$\phi_{r+1}$ and $\overline{\phi_{r+1}}$ are orthogonal to $\overline{\phi_r}, \dots , \overline{\phi_1} , \phi_0 , \phi_1, \dots , \phi_r$,}
\item{$\langle \phi_{r+1} , \phi_{r+1} \rangle$ defines a holomorphic degree $2r+2$-differential on $\Sigma$,}
\item{$(\phi_i)_z = \phi_{i+1} + 2(w_i)_z \phi_i$, for $1 \le i \le r$,}
\item{$(\phi_i)_{\overline{z}} = -(h_{i}/h_{i-1}) \phi_{i-1}$, for $1 \le i \le r+1$,}
\item{$2(w_i)_{z\overline{z}} = h_{i+1}/h_i - h_i/h_{i-1}$, for $1 \le i \le r$.}
\end{itemize}
\begin{proof}
These are all reasonably straightforward verifications nearly identical to the definite signature case.
\end{proof}
\end{prop}


\begin{defn}

Let $\phi:\Sigma \to Q_{h_0} \subset \mathbb{R}^{p,q}$ be a minimal surface. Suppose $\phi$ has isotropy of order $r$ and that $p+q = 2r+2$ or $2r+3$. We say $\phi$ is {\em superminimal} \cite{bur}, \cite{bry1} if $h_{r+1}=0$ and {\em superconformal} \cite{bol}, \cite{bur} if $h_{r+1} \ne 0$.
\end{defn}

We now consider a superconformal or superminimal surface $\phi: \Sigma \to Q_{h_0} \subset \mathbb{R}^{p,q}$ in the case $p+q$ is even. So $\phi$ has isotropy of order $r$ such that $p+q = 2r+2$. Now $\overline{\phi_r}, \cdots , \overline{\phi_1} , \phi_0 , \phi_1, \cdots , \phi_r$ span a real codimension $1$ subspace which is not null, therefore there is a real vector valued function $\tilde{\phi}$ with $\langle \tilde{\phi} , \tilde{\phi} \rangle = \epsilon = \pm 1$ and complex valued function $q$ such that $\phi_{r+1} = q\tilde{\phi}$ and $\overline{\phi_{r+1}} = \overline{q}\tilde{\phi}$. Thus $h_{r+1} = \epsilon q \overline{q}$ and $\langle \phi_{r+1} , \phi_{r+1} \rangle = \epsilon q^2$ is holomorphic. Hence $q(dz)^{r+1}$ is a holomorphic $(r+1)$-differential. We then have the following set of equations:
\begin{equation} \label{signedtoda1}
\begin{aligned}
2(w_i)_{z\overline{z}} &= h_{i+1}/h_i - h_i/h_{i-1}, \; \; \; 1 \le i \le r-1, \\
2(w_r)_{z\overline{z}} &= \epsilon q\overline{q}/h_r - h_r/h_{r-1}.
\end{aligned}
\end{equation}
Notice that the superminimal case where $\phi_{r+1}=0$ corresponds to setting $q=0$. We can also write the equations directly in terms of the $w_i$. Let $\mu_i = \epsilon_i\epsilon_{i+1}$ for $1\le i \le r$ and $\mu_{r+1} = \epsilon \epsilon_{r}$. Then
\begin{equation}
\begin{aligned}
2(w_i)_{z\overline{z}} &= \mu_{i+1}e^{2w_{i+1} - 2w_i} -\mu_i e^{2w_i - 2w_{i-1}}, \; \; \; 1 \le i \le r-1, \\
2(w_r)_{z\overline{z}} &= \mu_{r+1}q\overline{q}e^{-2w_r} -\mu_r e^{2w_r - 2w_{r-1}}.
\end{aligned}
\end{equation}
We find this is a set of affine Toda equations for the affine Dynkin diagram $D^{(2)}_{r+1}$. The different choices of signs $\mu_i = \pm 1$ determine different real forms of the Toda equations.\\

Let us now consider the alternate case of a superconformal or superminimal surface when $p+q$ is odd. So in this case $\phi: \Sigma \to Q_{h_0} \subset \mathbb{R}^{p,q}$ has isotropy of order $r$ and $p+q = 2r+3$. We have that $\overline{\phi_r} , \cdots , \phi_0 , \cdots , \phi_r$ spans $2r+1$ dimensions and $\phi_{r+1}$, $\overline{\phi_{r+1}}$ lie in the $2$-dimensional orthogonal complement. If $\phi_{r+1}$ and $\overline{\phi_{r+1}}$ are linearly dependent then they span a real subspace and we are in the $2r+2$ dimensional situation. Therefore we assume that $\phi_{r+1}$ and $\overline{\phi_{r+1}}$ are linearly independent, at least on some open subset.

There are two cases to consider according to whether the $2$-dimensional space $W$ spanned by $\phi_{r+1}$ and $\overline{\phi_{r+1}}$ is definite or indefinite. We will consider the definite case, the indefinite being similar. Therefore we can locally find a vector valued function $w$ such that $w$ and $\overline{w}$ are a complex basis for $W$ and such that $\langle w , w \rangle = 0$, $\langle w , \overline{w} \rangle = 2\epsilon$ where $\epsilon = \pm 1$ according to whether $W$ is positive or negative definite. We may write
\begin{equation}
\phi_{r+1} = \frac{1}{2}Aw + \frac{1}{2}B\overline{w}
\end{equation}
for two functions $A,B$. Now $\langle \phi_{r+1} , \phi_{r+1} \rangle = \epsilon AB$ is a holomophic differential $q = \epsilon AB$. Aside from the isolated zeros of $q$ we have $A$ and $B$ are non-vanishing. Further we can redefine $w$ so that $A$ is real valued and positive. Set $A = e^{2\eta}$. Now $h_{r+1} = \epsilon /2 (A\overline{A} + B\overline{B}) = \epsilon /2 (e^{2\eta} + q\overline{q} e^{-2\eta})$. Now it remains to find the equation for $(\phi_{r+1})_z$. If we differentiate the orthogonality relations for $\phi_{r+1}$ we find that $(\phi_{r+1})_z$ must be a linear combination of $\overline{\phi_r}$, $\overline{\phi_{r+1}}$ and $\phi_{r+1}$ say
\begin{equation}
(\phi_{r+1})_z = \alpha \overline{\phi_r} + \mu \phi_{r+1} + \lambda \overline{\phi_{r+1}}.
\end{equation}
Furthermore differentiating the orthogonality relation $\langle \phi_{r+1} , \phi_r \rangle = 0$ gives $\alpha = -q/h_r$. Similarly differentiating $\langle \phi_{r+1} , \phi_{r+1} \rangle = q$ and $\langle \phi_{r+1} , \overline{\phi_{r+1}} \rangle = h_{r+1}$ determines $\mu$ and $\lambda$. For simplicity we will use local coordinates $z$ in which the holomorphic differential is $(dz)^{2r+2}$, hence we may set $q=1$. With some care one may put $q$ back into the equations later. We now have
\begin{equation}
\begin{aligned}
\mu + h_{r+1}\lambda &= 0 \\
h_{r+1}\mu + \lambda &= (h_{r+1})_z.
\end{aligned}
\end{equation}
We may also find the equation for $\eta_{z\overline{z}}$. This equation is found by equating the $\phi_{r+1}$ components of $(\phi_{r+1})_{z\overline{z}}$ evaluated in two ways. This gives
\begin{equation}
-h_{r+1}/h_r = \mu_{\overline{z}} + \lambda \overline{\lambda}.
\end{equation}
This simplifies to
\begin{equation}
2\eta_{z\overline{z}} = -\frac{\epsilon \, {\rm sinh}(2\eta)}{h_r}.
\end{equation}

Now we can collect the equations for $h_1, \cdots , h_r , \eta$. At the same time we will reintroduce the holomorphic differential $q$.
\begin{equation}\label{signedtoda3}
\begin{aligned}
2(w_i)_{z\overline{z}} &= h_{i+1}/h_i - h_i/h_{i-1}, \; \; \; 1 \le i \le r-1, \\
2(w_r)_{z\overline{z}} &= \frac{\epsilon}{2h_r} \left( e^{2\eta} + q\overline{q}e^{-2\eta} \right) -h_r/h_{r-1}, \\
2(\eta)_{z\overline{z}} &= -\frac{\epsilon}{2h_r} \left( e^{2\eta} - q\overline{q}e^{-2\eta} \right).
\end{aligned}
\end{equation}
These are a set of affine Toda equations for the affine Dynkin diagram $B^{(1)}_{r+1}$.\\

Similarly if one considers the case when the space $W$ has signature $(1,1)$ one finds the following equations
\begin{equation}\label{signedtoda6}
\begin{aligned}
2(w_i)_{z\overline{z}} &= h_{i+1}/h_i - h_i/h_{i-1}, \; \; \; 1 \le i \le r-1,  \\
2(w_r)_{z\overline{z}} &= \frac{\left( q e^{i\eta} + \overline{q}e^{-i\eta} \right)}{2h_r} -h_r/h_{r-1}, \\
2(\eta)_{z\overline{z}} &= -\frac{\left( q e^{i\eta} - \overline{q}e^{-i\eta} \right)}{2ih_r}.
\end{aligned}
\end{equation}
This is also a set of affine Toda equations for $B^{(1)}_{r+1}$ but with a different reality condition.

\subsection{Flat connection description}\label{fcd}

We will show that conversely a solution of the Toda equations for a minimal surface into a quadric (\ref{signedtoda1}), (\ref{signedtoda3}) or (\ref{signedtoda6}) on a surface $\Sigma$ defines a minimal surface $\phi : \tilde{\Sigma} \to Q_{h_0}$ from the universal cover of $\Sigma$. Indeed we show a solution of the equations defines a flat connection on a $\mathbb{R}^{p,q}$-bundle over $\Sigma$ with a distinguished section $\phi$. Identifying fibres through parallel translation, $\phi$ corresponds to a minimal immersion. That such a flat connection description exists is a well known aspect of integrability \cite{mik}.\\

Let us consider the case with $p+q$ even using the notation of the previous section. The idea is to use the harmonic sequence to build a vector bundle with flat connection. First we define a holomorphic vector bundle $E$
\begin{equation}
E = L_{-r} \oplus \dots \oplus L_{r+1}
\end{equation}
where
\begin{equation}
L_i = K^{-1} \; \; \; i \ne r+1, \; \; \; L_{r+1} = 1.
\end{equation}
With respect to a local holomorphic coordinate $z$ we define a frame $\{e_{-r} , \dots , e_{r+1} \}$ for $E$ as follows
\begin{equation}
e_i = (\partial_z)^i \; \; \; i \ne r+1, \; \; \; e_{r+1} = 1.
\end{equation}
The idea is that $e_0, \dots , e_r$ corresponds to a harmonic sequence $\phi_0, \dots , \phi_r$. So far $E$ has structure group ${\rm SL}(p+q, \mathbb{C})$. We define an antilinear involution $\lambda$ on $E$ reducing the structure group to ${\rm SL}(p+q,\mathbb{R})$ as follows
\begin{equation}
\lambda(e_0)=e_0, \; \; \; \lambda(e_i) = h_i e_{-i} \; \; \; 1 \le i \le r, \; \; \; \lambda(e_{r+1}) = e_{r+1}
\end{equation}
and this determines $\lambda$. We also use $\lambda$ to conjugate endomorphisms by $\overline{T}(x) = \overline{ T(\overline{x})}$. Next we define a non-degenerate bilinear form $\langle \, , \, \rangle$ on $E$
\begin{equation}
\begin{aligned}
\langle e_i , e_j \rangle &= 0 \; \; \; i \ne -j, \\
\langle e_i , e_{-i} \rangle &= 1,
\end{aligned}
\; \; \;
\begin{aligned}
\langle e_0 , e_0 \rangle &= h_0, \\
\langle e_{r+1} , e_{r+1} \rangle &= \epsilon.
\end{aligned}
\end{equation}
This defines a further reduction of structure to ${\rm SO}(p,q)$. We also have a signature $(p,q)$ Hermitian form $h(x,y) = \langle x , \lambda(y) \rangle$. Let $\nabla_A = d + Adz$ denote the associated Chern connection. Then $A$ is given by
\begin{equation}
A(e_i) = 2(w_i)_z e_i
\end{equation}
where $h_i = \epsilon_i e^{2w_i}$, $\epsilon_i = \pm 1$ for $i = 1, \dots , r$ and we put $w_0 = w_{r+1} = 0$.

Now define a $K$-valued endomorphism $\Phi dz$ by
\begin{equation}
\begin{aligned}
\Phi (e_i) &= e_{i+1} \; \; \; 0 \le i \le r-1, \\
\Phi (e_{-i}) &= -e_{-i+1} \; \; \; 2 \le i \le r,
\end{aligned}
\; \; \;
\begin{aligned}
\Phi (e_r) &= qe_{r+1}, & \Phi (e_{r+1}) &= -\epsilon q e_{-r}, \\
\Phi (e_{-1}) &= -h_0 e_0. & &
\end{aligned}
\end{equation}
We see that $\Phi$ is holomorphic. We now see that the Toda equations (\ref{signedtoda1}) are locally equivalent to
\begin{equation}
F_A + \left[ \Phi , \lambda(\Phi) \right] = 0
\end{equation}
where $F_A$ is the curvature of $\nabla_A$. This is also equivalent to flatness of the ${\rm SO}(p,q)$-connection
\begin{equation}
\nabla = \nabla_A + \Phi + \lambda(\Phi).
\end{equation}
If we consider the $h_i$ to transform as $(i,i)$-forms for $i = -r, \dots , r$ then we also have a global correspondence. Now suppose we have a solution of the equations. On the universal cover $\tilde{\Sigma}$ we can trivialise the flat ${\rm SO}(p,q)$-connection. The bundle $E$ has a global section $\phi = e_0$. Note that $\phi$ is real and $\langle \phi , \phi \rangle = h_0$. Thus on the universal cover $\phi$ defines a developing map
\begin{equation}\label{higgsy}
\phi : \tilde{\Sigma} \to \mathbb{R}^{p,q}
\end{equation}
which maps into the quadric $Q_{h_0}$ of vectors $v$ with $\langle v , v \rangle = h_0$. If $\nabla$ has monodromy representation $\rho : \pi_1(\Sigma) \to {\rm SO}(p,q) $ then $\phi$ is $\rho$-equivariant. Moreover we see that $\nabla_z \phi = e_1$ and $\nabla_{\overline{z}}\nabla_z \phi = -h_1/h_0 \phi$. Hence $\phi$ is a minimal immersion. Moreover we can take the harmonic sequence defined by $\phi$ which gives us back the solution to the Toda equations.\\

We now show that under certain circumstances equation (\ref{higgsy}) is a special case of the Higgs bundle equations. Define an involution $\sigma : E \to E$ by
\begin{equation}
\sigma(e_i) = \epsilon_i e_i
\end{equation}
where we have put $\epsilon_{-i} = \epsilon_i$. We let $\sigma$ act on ${\rm End}(E)$ act in the natural way. Then $\sigma$ commutes with $\lambda$ and $\rho = \sigma \lambda$ defines a compact real form of ${\rm SO}(p+q, \mathbb{C})$. Now we can interpret (\ref{higgsy}) as a special solution of the ${\rm SO}(p+q)$ Higgs bundle equations provided that $\sigma(\Phi) = -\Phi$. This in turn is true if and only if
\begin{equation}
\epsilon_{i+1} = -\epsilon_i, \; \; \; -r \le i \le r.
\end{equation}
In this case the flat connection has holonomy ${\rm SO}(r+1,r+1)$. Note that if we wish for the minimal immersion to induce a positive definite (as opposed to a negative definite) metric on $\Sigma$ then we must put $\epsilon_0 = -1$. It is also clear from the form of $(E,\Phi)$ that it is a Higgs bundle in the Hitchin component. We also note that the holomorphic $(r+1)$-differential $q$ corresponds to the Pfaffian of $\Phi$. This is therefore not a cyclic Higgs bundle (as the Pfaffian is not the highest differential), but it does share many similarities with the cyclic case. \\

We have a similar construction in the case $p+q$ is odd. We consider only the case corresponding to equations (\ref{signedtoda3}) the other case being similar. We have that $p+q = 2r+3$. Define a holomorphic bundle
\begin{equation}
E = L_{-r-1} \oplus \dots \oplus L_0 \oplus \dots \oplus L_{r+1}
\end{equation}
where $L_i = K^{-i}$. Define a frame $\{e_{-r-1} , \dots , e_{r+1} \}$ of $E$ where $e_i = (\partial_z)^i$. Define an anti-linear involution $\lambda$ by the conditions
\begin{equation}
\lambda(e_0) = e_0, \; \; \; \lambda(e_i) = e_{-i}h_i \; \; \; 1 \le i \le r, \; \; \; \lambda(e_{r+1}) = \epsilon e^{4\eta}e_{-r-1}
\end{equation}
and an inner product by
\begin{equation}
\langle e_i , e_j \rangle = 0 \; \; \; i \ne -j, \; \; \; \langle e_0 , e_0 \rangle = 0, \; \; \; \langle e_i , e_{-i} \rangle = 1 \; \; \; i \ne 0.
\end{equation}
As before we have a corresponding indefinite Hermitian form $h(x,y) = \langle x , \lambda(y) \rangle$ and we let $\nabla_A$ be the corresponding Chern connection. Define a section $\Phi$ of ${\rm End}(E)\otimes K$ by
\begin{equation}
\begin{aligned}
\Phi(e_i) &= e_{i+1} \; \; \; 0 \le i \le r-1, & \Phi(e_r) &= \frac{1}{\sqrt{2}}( e_{r+1} + q e_{-r-1} ), \\
\Phi(e_{r+1}) &= -\frac{1}{\sqrt{2}}qe_{-r}, & \Phi(e_{-r-1}) &= -\frac{1}{\sqrt{2}}e_{-r}, \\
\Phi(e_{-i}) &= -e_{-i+1} \; \; \; 2 \le i \le r, & \Phi(e_{-1}) &= -h_0 e_0.
\end{aligned}
\end{equation}
As before the Toda equations are equivalent to $\nabla_A + \Phi + \lambda(\Phi)$ being flat. We can again relate this to Higgs bundles in a certain special case. As before define $\sigma : E \to E$ by
\begin{equation}
\sigma(e_i) = \epsilon_i e_i
\end{equation}
where $\epsilon_{r+1} = \epsilon_{-r-1} = \epsilon$. Then $(\nabla_A , \Phi)$ is seen to solve the Higgs bundle equations provided that $\sigma(\Phi) = -\Phi$. This holds if and only if
\begin{equation}
\epsilon_{i+1} = -\epsilon_i, \; \; \; -r-1 \le i \le r.
\end{equation}
In this case it is again clear that the constructed Higgs bundle belongs to the Hitchin component for ${\rm SO}(r+1,r)$ and moreover is a cyclic Higgs bundle. Therefore we can use existence of solutions to the Higgs bundle equation to show the existence of the corresponding minimal immersions.\\

In order to find specific solutions we can consider solutions that are invariant under a translational symmetry. Suppose we have a local coordinate $z = x + iy$ such that $\nabla_A$ and $\Phi$ are $y$-invariant. If we write $\nabla_A = d + Ady$, $\Phi = 1/2(\phi_1 + i\phi_2)dz$ and $\lambda(\Phi) = 1/2(\phi_1 - i\phi_2)d\overline{z}$ such that $A,\phi_1,\phi_2$ are $y$-independent then the equations reduce to
\begin{equation}
\begin{aligned}
\partial_x A &= \left[ \phi_1 , \phi_2 \right] \\
\partial_x \phi_1 &= \left[ A , \phi_2 \right] \\
\partial_x \phi_2 &= \left[ \phi_1 , A \right].
\end{aligned}
\end{equation}
These equations are a different real form of Nahm's equations sometimes called Schmid's equations from their appearance in \cite{sch}. These can be put into Lax form using a spectral parameter $\zeta$. Let $\alpha = (\phi_1 + i \phi_2) + 2\zeta A + \zeta^2(\phi_1 - i\phi_2)$ and $\beta = iA + i\zeta(\phi_1 - i\phi_2)$. Then the Lax equation is
\begin{equation}
\partial_x \alpha = \left[ \alpha , \beta \right].
\end{equation}
It follows that the spectral curve ${\rm det}(\eta - \alpha(\zeta , x)) = 0$ is independent of $x$ and can be thought of as a curve in the total space of $\mathcal{O}(2) = T\mathbb{CP}^1$. By studying the spectral curve one can solve the equations using standard techniques \cite{twist}.



\chapter{Geometric structures from the Hitchin component}\label{chap3}
Given the split real form $G^\tau$ of the adjoint form of a simple Lie group $G$ we have seen that there is a distinguished component $\mathcal{H}(G^\tau)$ of the space of representations of the fundamental group of a compact oriented surface $\Sigma$ of genus $g > 1$ called the Hitchin component. Our aim will be to attempt to interpret the Hitchin component as a moduli space of geometric structures on $\Sigma$. As motivation we note that in the simplest case of $G^\tau = {\rm PSL}(2,\mathbb{R})$ the Hitchin component corresponds to the Fuchsian representations of the fundamental group, hence in this case the Hitchin component identifies naturally with Teichm\"uller space, which is the moduli space of hyperbolic metrics modulo diffeomorphisms isotopic to the identity.

The techniques we use will allow us to understand the Hitchin component when $G^\tau$ has rank $2$. In this case all such representations are described by just two holomorphic differentials, one being quadratic. Based on a result of Labourie \cite{lab} we find that as long as we let the conformal structure on the surface vary it suffices to consider only representations where the quadratic differential vanishes. In the rank $2$ case this leaves just the highest differential non-zero which means we have a  cyclic Higgs bundle so we can use the results we have established for such Higgs bundles.

\section{Quadratic differentials and energy}\label{qdae}
Representations in the Hitchin component are described by a series of holomorphic differentials. The lowest is always a quadratic differential. We introduce the energy functional and explain its relation to the quadratic differential. We recall a conjecture related to the energy functional and quadratic differential and discuss the result of Labourie that provides a partial proof of the conjecture.

\subsection{Holomorphic quadratic differentials}\label{holoqd}
In the Higgs bundle description of the Hitchin component we first choose a conformal structure on $\Sigma$, then the Hitchin component can be identified with the vector space
\begin{equation}\label{hitcomp}
\bigoplus_{i=1}^l H^0(\Sigma , K^{m_i+1})
\end{equation}
where $m_1 \le \dots \le m_l$ are the exponents of $G$. In the case $G^\tau = {\rm PSL}(2,\mathbb{R})$ this is the space of holomorphic quadratic differentials $H^0(\Sigma,K^2)$, which by Serre duality is dual to $H^1(\Sigma , K^*)$ which naturally identifies with the space of deformations of complex structure on $\Sigma$.

More generally for any simple Lie group we have $m_1 = 1$ so that there is always a space of holomorphic quadratic differentials in (\ref{hitcomp}). In fact the holomorphic differential can generally be described as
\begin{equation}
q_2 = k(\Phi , \Phi )
\end{equation}
where $\Phi$ is a Higgs field and $k$ is the Killing form on the adjoint bundle. Suppose we have a representation $\theta : \tilde{\Sigma} \to G^\tau$ in the Hitchin component for $G^\tau$. Choose a conformal structure on $\Sigma$ so that there is a corresponding Higgs bundle $(E,\Phi)$. Let $K$ be a maximal compact subgroup of $G^\tau$ with Lie algebra $k \subset \mathfrak{g}^\tau$. Give $G^\tau/K$ the metric $h$ induced by the Killing form on the orthogonal component $k^\perp$ of $k$. We know there is a $\theta$-equivariant harmonic map $\psi : \tilde{\Sigma} \to G^\tau/K$. Moreover if we take a local lift $\tilde{\psi}$ of $\psi$ then $\Phi$ is the component of $\tilde{\psi}^{-1}\partial\tilde{\psi}$ in $k^\perp$. It now follows that
\begin{equation}
(\psi^* h)^{2,0} = k(\Phi , \Phi ) = q_2.
\end{equation}
Thus the holomorphic quadratic differential $q_2$ vanishes if and only if the map $\psi$ is conformal. This leads us to a particular form of a conjecture of Goldman and Wentworth \cite{gol} and of Labourie \cite{lab}:

\begin{conj}\label{conj1}
Given a representation $\theta$ in the Hitchin component there is a unique conformal structure on $\Sigma$ such that the Higgs bundle $(E,\Phi)$ associated with $\theta$ satisfies $q_2 = k(\Phi , \Phi) = 0$.
\end{conj}

If this conjecture is true then it provides an alternative description of the Hitchin component, namely as a vector bundle over Teichm\"uller space $V \to \mathcal{T}$ with fibres $V_p = \oplus_{i=2}^l H^0(\Sigma , K^{m_i+1})$ consisting of differentials which are holomorphic with respect to the complex structure corresponding to $p \in \mathcal{T}$.

The advantage of such a description is that we need only consider Higgs bundles $(E,\Phi)$ where the holomorphic quadratic differential vanishes simplifying the task of interpreting the Hitchin component. In particular in the rank $2$ case we need only assume that the highest differential is non-zero. Therefore we can use the affine Toda description of Section \ref{speccase} to provide a complete description of the component.

\subsection{The energy functional on Teichm\"uller space}\label{energy}
We will introduce the energy functional \cite{don}, \cite{lab} and explain its relation to the quadratic differential. This will lead to a second form of Conjecture \ref{conj1} in terms of the energy functional. We then examine the extent to which the conjecture is known to hold.\\

Let $\theta$ be a representation of the fundamental group of $\Sigma$ into a semisimple Lie group $G$ and let $K$ be a maximal compact subgroup. Let $\psi : \tilde{\Sigma} \to G/K$ be a $\theta$-equivariant map, not necessarily harmonic. Then $\psi$ defines a reduction of structure of the principal $G$-bundle $P = \tilde{\Sigma} \times_\theta G$ to a principal $K$-bundle $P_K$ and an anti-involution $\rho$ on the adjoint bundle $P \times_\theta \mathfrak{g}$. Let $k$ be the Lie algebra of $K$ and $k^\perp$ the orthogonal component with respect to the Killing form. One finds that the pull-back bundle $\psi^*(T(G/K))$ is naturally identified with $\pi^*(P_K) \times_K k$ where $\pi$ is the projection $\pi : \tilde{\Sigma} \to \Sigma$. Thus $d\psi = \phi$ can be regarded as a $1$-form valued section of the adjoint bundle which is self-adjoint in the sense that $\rho(\phi) = -\phi$.\\

Choose a conformal structure on $\Sigma$. This is equivalent to defining a Hodge star operator $* : T^*\Sigma \to T^*\Sigma$ on the space of $1$-forms. We define the {\em energy} $E_\theta(\psi , *)$ of $\psi$ with respect to $*$ by
\begin{equation}
E_\theta(\psi , *) = \int_\Sigma k( \phi \wedge * \phi)
\end{equation}
where $k$ is the Killing form on the adjoint bundle. If we fix the conformal structure $*$ then the stationary points of $E_\theta$ correspond to the map $\psi$ being harmonic. Let us suppose that the representation $\theta$ is reductive. So from Corlette \cite{cor} we know that such a harmonic map exists and is unique up to an isometry of $G/K$. Therefore if we impose the constraint that $\psi$ be harmonic it follows that $E_\theta(\psi , *)$ depends only on $\theta$ and $*$. Keeping $\theta$ fixed we have that $E_\theta$ depends only on the conformal structure $*$. We write $E_\theta(*)$ for $E_\theta(\psi , *)$ where $\psi$ is any $\theta$-equivariant harmonic map.\\

Suppose now that $f : \Sigma \to \Sigma$ is a diffeomorphism isotopic to the identity. Let $\psi : \tilde{\Sigma} \to G/K$ be a $\theta$-equivariant harmonic map with respect to a conformal structure $*$ on $\Sigma$. Then $\psi \circ f$ is harmonic with respect to the conformal structure $f^{-1}(*)$ and since $f$ is isotopic to the identity we see that $\psi \circ f$ is $\theta_1$-equivariant where $\theta_1$ is a representation conjugate to $\theta$. Applying a left translation by a fixed $g \in G$ we have that $g \psi \circ f$ is a $\theta$-equivariant map harmonic with respect to $f^{-1}(*)$. Therefore it follows that $E_\theta(*)$ depends only on the diffeotopy class of $*$. Therefore $E_\theta$ can be viewed as a map on Teichm\"uller space.\\

It is well known that the stationary points of the energy functional where both the function and conformal structure are allowed to vary correspond to conformal harmonic maps. Therefore the stationary points of $E_\theta$ as a function on Teichm\"uller space are precisely the conformal structures for which the holomorphic quadratic differential of the associated Higgs bundle vanishes. Therefore Conjecture \ref{conj1} is equivalent to the following
\begin{conj}
For any representation $\theta$ in the Hitchin component there exists a unique stationary point of the energy function $E_\theta$ on Teichm\"uller space.
\end{conj}
In fact it is shown in Tromba \cite{tro} that in the case of the Hitchin component for ${\rm PSL}(2,\mathbb{R})$ the energy functions are proper and have a unique minimum. This proves the conjecture in the case of ${\rm PSL}(2,\mathbb{R})$ and suggests a more refined conjecture:
\begin{conj}
For a representation in any Hitchin component the associated energy function is proper and all stationary points are non-degenerate minima.
\end{conj}
Note that this implies the previous conjecture. Indeed since the energy function is positive then if it is also proper it must have a minimum. But if all stationary points are non-degenerate minima then there can be at most one minima.\\

Labourie \cite{lab} shows that in the case of the Hitchin components for ${\rm PSL}(n,\mathbb{R})$, the energy functions are indeed proper. We now argue that this implies the same result for at least some of the other Hitchin components. Suppose $\mathfrak{g}$ is a complex simple subalgebra of $\mathfrak{sl}(n,\mathbb{C})$ with the property that there is a principal $3$-dimensional subalgebra $\iota : \mathfrak{sl}(2,\mathbb{C}) \to \mathfrak{g}$ such that the composition $\mathfrak{sl}(2,\mathbb{C}) \to \mathfrak{sl}(n,\mathbb{C})$ is also a principal $3$-dimensional subalgebra. In this case the Hitchin component for the split real form of $\mathfrak{g}$ identifies as a subspace of the Hitchin component for ${\rm PSL}(n,\mathbb{R})$. This subspace is characterised by the vanishing of a subset of the holomorphic differentials. It is clear that for these representations the energy function is the same as the corresponding energy function for ${\rm PSL}(n,\mathbb{R})$.

Now one can readily show that above situation occurs at the very least for $\mathfrak{so}(2n+1,\mathbb{C}) \subset \mathfrak{sl}(2n+1,\mathbb{C})$, $\mathfrak{sp}(2n,\mathbb{C}) \subset \mathfrak{sl}(2n , \mathbb{C})$ and $\mathfrak{g}_2 \subset \mathfrak{sl}(7,\mathbb{C})$.

Properness of the energy function implies at least in these cases of the existence of a conformal structure such that the corresponding holomorphic quadratic differential vanishes, however it does not establish uniqueness of the conformal structure. To summarise we have
\begin{prop}
For a representation in the Hitchin component associated to a Lie algebra of type $A_n, B_n ,C_n $ or $ G_2$ there exists a conformal structure for which the associated Higgs bundle has vanishing quadratic differential.
\end{prop}

\section{Construction of cyclic representations}\label{cocr}
We are going to consider Higgs bundles for the Hitchin component such that only the highest differential is non-zero which we call cyclic representations. We will outline the general procedure from Section \ref{speccase} to be followed here.\\

Let $\Sigma$ be a Riemann surface with canonical bundle $K$. Let $\mathfrak{g}$ be a complex simple Lie algebra of rank $l$ with exponents $m_1 \le  \dots \le m_l$. Associated to each holomorphic $l+1$-differential $q$ there is a Higgs bundle $(E,\Phi)$ in the Hitchin component for $\mathfrak{g}$ with only the highest differential non-vanishing. Let $i : \mathfrak{sl}(2,\mathbb{C}) \to \mathfrak{g}$ be a principal three dimensional subalgebra and let $x,e,\tilde{e}$ be a basis for $\mathfrak{sl}(2,\mathbb{C})$ with commutation relations
\begin{equation}
\left[ x , e \right] = e, \; \; \; \left[ x , \tilde{e} \right] = -\tilde{e}, \; \; \; \left[ e , \tilde{e} \right] = x.
\end{equation}
Let $e_l$ be a highest weight vector for $\mathfrak{g}$ satisfying $\left[ x , e_l \right] = m_l e_l$.
The holomorphic bundle $E$ and Higgs field $\Phi$ are given by
\begin{eqnarray}
E &=& \bigoplus_{i=1}^l \left( K^{m_i} \oplus K^{m_i - 1} \oplus \dots \oplus K^{-m_i} \right)\\
\Phi &=& \tilde{e} + q e_l
\end{eqnarray}
In fact there is a holomorphic ${\rm SL}(2,\mathbb{C})$-bundle $P$ such that $E$ is the associated bundle $E = P \times_i \mathfrak{g}$. Choose a holomorphic square root $K^{1/2}$ of $K$. Then we may take $P$ to be the ${\rm SL}(2,\mathbb{C})$ frame bundle of $K^{1/2} \oplus K^{-1/2}$.\\

Given a complex representation $\xi : G \to {\rm GL}(V)$ we may take the associated bundle
\begin{equation}
\mathcal{V} = P \times_{\xi \circ i} V.
\end{equation}
We can then regard $\xi(\Phi)$ as a $K$-valued endomorphism of $\mathcal{V}$. Provided that $\xi$ is faithful it is sufficient to consider the Higgs bundle $(\mathcal{V} , \xi(\Phi))$. We will simply write $\Phi$ for $\xi(\Phi)$. The advantage of this is that the bundle $\mathcal{V}$ will provide a simpler description of the Higgs bundle.

For example in the case of $\mathfrak{g} = \mathfrak{sl}(2n+1,\mathbb{C})$ we may take $\xi$ to be the fundamental representation in which case
\begin{equation}
\mathcal{V} = K^n \oplus K^{n-1} \oplus \dots \oplus K^{-n}.
\end{equation}
This is also the case for the fundamental representation of $\mathfrak{so}(2n+1,\mathbb{C}) \subset \mathfrak{sl}(2n+1,\mathbb{C})$ and for $\mathfrak{g}_2 \subset \mathfrak{so}(7,\mathbb{C})$. On the other hand the fundamental representation for $\mathfrak{sl}(2n,\mathbb{C})$ yields
\begin{equation}
\mathcal{V} = K^{n/2} \oplus K^{n/2 - 1} \oplus \dots \oplus K^{-n/2}
\end{equation}
and this likewise applies to the fundamental representation of $\mathfrak{sp}(2n,\mathbb{C}) \subset \mathfrak{sl}(2n,\mathbb{C})$.\\

Now in $\mathfrak{g}$ we define commuting anti-involutions $\hat{\rho},\hat{\lambda}$ and Cartan involution $\sigma = \hat{\rho}\hat{\lambda}$ where $\hat{\rho}$ defines a compact real form and $\hat{\lambda}$ defines a split real form. Moreover we have
\begin{equation}\label{fixedanti}
\begin{aligned}
\hat{\rho}(x) &= -x, & \hat{\rho}(e) &= -\tilde{e}, & \hat{\rho}(\tilde{e}) &= -e, \\
\hat{\lambda}(x) &= -x, & \hat{\lambda}(e) &= \tilde{e}, & \hat{\lambda}(\tilde{e}) &= e.
\end{aligned}
\end{equation}

Now on the adjoint bundle $E$ there will be corresponding anti-involutions $\rho,\lambda$ and Cartan involution $\sigma$. The relation is that over any local trivialisation of $E$ there is a $\mathfrak{h}$-valued function $\Omega$, where $\mathfrak{h}$ is the centraliser of $x$ in $\mathfrak{g}$ such that
\begin{equation}
\begin{aligned}
\rho &= {\rm Ad}_{e^{2\Omega}} \circ \hat{\rho}, & \hat{\rho}(\Omega) &= -\Omega, \\
\lambda &= {\rm Ad}_{e^{2\Omega}} \circ \hat{\lambda}, & \hat{\lambda}(\Omega) &= -\Omega.
\end{aligned}
\end{equation}
and $\sigma$ is unchanged.

Let $k$ be the Killing form for $\mathfrak{g}$ then the Hermitian metric $h_\rho$ on $E$ solving the Higgs bundle equations is given by
\begin{equation}
h_\rho (x,y) = -k(x , \rho(y)).
\end{equation}
We let $\nabla_A$ be the Chern connection associated to $h_\rho$ and let the curvature be $F_A$. Then the Higgs bundle equation is
\begin{equation}
F_A - \left[ \Phi , \rho(\Phi) \right] = 0.
\end{equation}

In the cases we will consider $\rho$ and $\lambda$ have a more direct interpretation in terms of the bundle $\mathcal{V}$. In the case of $\rho$ there will be a Hermitian form $h$ on $\mathcal{V}$ such that if we consider $E$ as a subbundle of ${\rm End}(\mathcal{V})$ then for $A \in E$ we have $\rho(A) = -A^*$ where $A^*$ is the adjoint of $A$ with respect to $h$.

In the case of $\lambda$ there will be an anti-involution on $\mathcal{V}$ which we also denote by $\lambda$ with the property that for $A$ a section of $E$ and $v$ a section of $\mathcal{V}$ we have $\lambda (Av) = \lambda(A) \lambda(v)$. We can think of $\lambda$ as defining a conjugation on $\mathcal{V}$ so $\mathcal{V}$ has a real form $\mathcal{V}^\lambda$ which is preserved by the flat connection $\nabla = \nabla_A + \Phi - \rho(\Phi)$.

\section{Uniformising representations}\label{unif}
We first consider Higgs bundles in the Hitchin component for ${\rm PSL}(2,\mathbb{R})$. Not only is this useful as a warm-up for the following sections but we will need some of the results here for later use in Section \ref{cfps}. A study of these representations is also found in \cite{hit2} where it is shown how this component corresponds to Teichm\"uller space.\\

In the case of ${\rm PSL}(2,\mathbb{R})$ representations we can describe the Hitchin component as follows. The Lie algebra is $\mathfrak{sl}(2,\mathbb{R}) = \langle x , e, \tilde{e} \rangle$ and we suppose that $X,Y$ is a basis for $\mathbb{R}^2$ so that
\begin{equation}
\begin{aligned}
xX &= \frac{1}{2}X, & eX &= 0, & \tilde{e}X &= \frac{1}{\sqrt{2}}Y, \\
xY &= -\frac{1}{2}Y, & eY &= \frac{1}{\sqrt{2}}X, & \tilde{e}Y &= 0.
\end{aligned}
\end{equation}
Choose a conformal structure on $\Sigma$ and a holomorphic quadratic differential $q$. To this data we associate a Higgs bundle $(E,\Phi)$ in the Hitchin component for ${\rm PSL}(2,\mathbb{R})$. In this case $E$ is the adjoint bundle of the following rank $2$ bundle
\begin{equation}
W = K^{1/2} \oplus K^{-1/2}
\end{equation}
where $K^{1/2}$ is a choice of holomorphic square root of the canonical bundle $K$. The Higgs field is
\begin{equation}
\Phi = \tilde{e} + qe.
\end{equation}
Fix an anti-involution $\hat{\rho}$ and an anti-involution $\hat{\lambda} = \hat{\rho} \sigma$ for the split real form as defined by (\ref{fixedanti}) and we have an associated anti-involution $\hat{\lambda}$ on the representation $\mathbb{R}^2$ given by
\begin{equation}
\hat{\lambda}X = Y, \; \; \; \hat{\lambda}Y = X
\end{equation}
with the property that for any endomorphism $A$ and vector $v$, $\hat{\lambda}(Av) = \hat{\lambda}(A)(\hat{\lambda}v)$.\\

Now the Hermitian metric solving the Higgs bundle equations is associated to an anti-involution $\rho$ which locally has the form $\rho = {\rm Ad}_{e^{2\Omega}} \circ \hat{\rho}$ where
\begin{equation}
\Omega = vx
\end{equation}
where if we set $H = e^{2v}$ then $H$ is a $(1/2,1/2)$-form. We find that
\begin{equation}
-\rho(\Phi) = H^2 e + H^{-2}\overline{q} \tilde{e}.
\end{equation}
The connection $\nabla = \nabla_A + \Phi - \rho(\Phi)$ is then
\begin{equation}
\nabla = d -4\partial (v)x + (H^2+q)e  + (1 + H^{-2}\overline{q})\tilde{e}.
\end{equation}
The Higgs bundle equation is
\begin{equation}
4v_{z\overline{z}} = H^2 - q\overline{q}H^{-2}.
\end{equation}
Solutions to this equation correspond to hyperbolic metrics \cite{hit2} and from this one has a correspondence between holomorphic quadratic differentials and Teichm\"uller space.

We will be interested in the case where $q=0$. The equation reduces to $4v_{z\overline{z}} = H^2$ which says that the metric $H^2$ has curvature $-2$. On a compact Riemann surface of genus $g>1$ we can find identify the universal cover $\tilde{\Sigma}$ of $\Sigma$ with the upper half plane $\mathbb{H}^2$ such that
\begin{equation}
H^2 = \frac{dx^2 + dy^2}{2y^2}.
\end{equation}
We seek the real covariant constant sections of $\nabla$ (on the universal cover). Note that the real structure preserved by $\nabla$ on the rank $2$ bundle $W$ is given by
\begin{equation}
\lambda(dz^{1/2}) = H^{-1}dz^{-1/2}, \; \; \; \lambda(dz^{-1/2}) = H dz^{1/2}.
\end{equation}
The real covariant constant sections are $s = s_1dz^{1/2} + s_2 dz^{-1/2}$ where
\begin{equation}
s_2 = (az + b)e^{i\pi/4} dz^{-1/2}, \; \; \; s_1 = \frac{\overline{s_2}}{\sqrt{2}y}.
\end{equation}
This provides the trivialisation for $W$ and $\nabla$ over $\tilde{\Sigma}$.

\section{Convex $\mathbb{RP}^2$-structures}\label{crp2s}
We now turn to the Hitchin component for ${\rm PSL}(3,\mathbb{R})$. In this case the Hitchin component is known to be the moduli space of convex $\mathbb{RP}^2$-structures on the surface \cite{choigold}, \cite{chgo}and these are also known to correspond to hyperbolic affine spheres in $\mathbb{R}^3$ invariant under a representation of the fundamental group. Using cyclic Higgs bundles we can obtain the affine and projective structures from the Higgs bundle in a direct manner. Although these results are already known our method provides an alternative proof for existence of the affine and projective structures associated to such a representation.


\subsection{Flat projective structures}\label{fps}
We introduce the notion of an $\mathbb{RP}^n$-structure and then explain how this is the flat case of the more general notion of a projective structure. Although we are only interested in the flat projective structures the general case which also provides a framework that will be useful even in the flat case. Some references for projective structures are \cite{ovs}, \cite{gol2}.\\

An $\mathbb{RP}^n$-manifold is roughly speaking an $n$-manifold built out of open patches of projective space glued together by projective transformations. More precisely an atlas $\{ U_\alpha , \phi_\alpha \}$ for an $n$-manifold $M$ consisting of an open cover $\{ U_\alpha \}$ and coordinate charts $\phi_\alpha : U_\alpha \to \mathbb{R}^n \subset \mathbb{RP}^n$ is called a projective atlas if the transition maps  $\phi_\beta \circ \phi_\alpha^{-1}$ are restrictions of projective transformations $\mathbb{RP}^n \to \mathbb{RP}^n$. Two projective atlases are called equivalent if their union is also a projective atlas. An {\em $\mathbb{RP}^n$-structure} on $M$ is then an equivalence class of a projective atlas on $M$ and if $M$ has an $\mathbb{RP}^n$-structure then $M$ is called an {\em $\mathbb{RP}^n$-manifold}. The diffeomorphism group of $M$ clearly acts on the set of $\mathbb{RP}^n$-structures on $M$.\\

Given an $\mathbb{RP}^n$-manifold $M$ let $\tilde{M}$ be the universal cover of $M$. Take a point $x \in M$ and let $(U , \phi )$ be a coordinate chart containing $x$ so $\phi$ is a map $\phi : U \to \mathbb{RP}^n$. Now take any curve $\gamma : [0,1] \to M$ in $M$ starting at $x$ and finishing at another point $y \in M$. We can find $0 = t_0 < t_1 < \dots < t_{k-1} < t_k = 1$ partitioning $[0,1]$ such that for $i = 1, \dots , k$, $\gamma([t_i-1 , t_i])$ lies in an open subset $U_i$ where $(U_i , \phi_i)$ is a compatible coordinate chart. By suitably choosing the pairs $(U_i, \phi_i)$ we can arrange so that $\phi_i$ and $\phi_{i+1}$ agree on $U_i \cap U_{i+1}$. To the curve $\gamma$ we then assign the point $\phi_k(y) \in \mathbb{RP}^n$. One shows that the element in $\mathbb{RP}^n$ associated to a curve $\gamma$ depends only on the homotopy class of $\gamma$. Therefore we have a map
\begin{equation}
{\rm dev} : \tilde{M} \to \mathbb{RP}^n
\end{equation}
called a {\em developing map}. The developing map is not unique since it depends on the choice of the initial chart $(U , \phi )$ but all developing maps constructed in this way differ by the action of a projective transformation on $\mathbb{RP}^n$, that is have the form $g \, {\rm dev} $ where $g \in {\rm PGL}(n+1 , \mathbb{R})$. Note that ${\rm dev}$ determines the $\mathbb{RP}^n$-structure on $M$, in fact by suitably restricting the domain of ${\rm dev}$ to coordinate charts we get a compatible projective atlas. Another property of the developing map and most significant for our purposes is that there is an associated monodromy representation. That is there is a representation $\theta : \pi_1(M) \to {\rm PGL}(n+1,\mathbb{R})$ such that ${\rm dev}$ is $\theta$-equivariant, that is
\begin{equation}
{\rm dev} ( m\gamma) = \theta(\gamma)^{-1} {\rm dev}(m)
\end{equation}
where $m \in \tilde{M}$ and $\gamma \in \pi_1(M)$. Note that changing the developing map by an overall projective transformation changes the monodromy representation by conjugation.\\

There is an alternative description of $\mathbb{RP}^n$-structures such that the monodromy representation of the developing map is the monodromy of a flat connection on a rank $n+1$-bundle over $M$. To see this we first introduce the more general notion of a projective structure.\\

Just as the projective geometry of $\mathbb{RP}^n$ is concerned with lines and points, projective differential geometry is concerned with the study of geodesics of an affine connection. Projective structures are an example of the parabolic geometries studied in Chapter \ref{chap4}, see also \cite{kobayashi}, \cite{gol2}, \cite{ovs}. Given an affine connection $\nabla$ on a manifold $M$ the parameterised geodesics are the curves $p(t)$ in $M$ defined by the equation
\begin{equation*}
\nabla_{\dot{p(t)}}\dot{p(t)} = 0
\end{equation*}
or if we allow for arbitrary parameterisations the unparametrised geodesic equation is
\begin{equation*}
\nabla_{\dot{p(t)}}\dot{p(t)} = f(t)\dot{p(t)}
\end{equation*}
for some function $f(t)$. We will consider only torsion free connections since torsion does not enter into the geodesic equation.\\

We say that two connections $\nabla$, $\tilde{\nabla}$ are {\em projectively equivalent} if there is a $1$-form $\lambda$ such that for all vector fields $X,Y$ we have
\begin{equation}
\tilde{\nabla}_X Y = \nabla_X Y + \lambda(X)Y + \lambda(Y)X.
\end{equation}
We have that two torsion free connections have the same (unparametrised) geodesics if and only if they are projectively equivalent. Given this, we define a {\em projective structure} on $M$ to be an equivalence class of projectively equivalent torsion free affine connections on $M$. An affine connection within this equivalence class is called a {\em preferred connection} or {\em Weyl connection}. A projective structure is {\em flat} if each point of $M$ has a neighborhood over which there is a flat preferred connection for the projective structure.

A partition of unity argument shows that for an $\mathbb{RP}^n$-manifold exists a projectively flat affine connection $\nabla$ on $M$ such that locally the geodesics of $\nabla$ correspond to straight lines in $\mathbb{RP}^n$ in any projective coordinate chart. That is $\mathbb{RP}^n$-structures are flat projective structures. Conversely a flat projective structure defines local flat coordinates. One checks that this defines an $\mathbb{RP}^n$-structure so a flat projective structure is precisely an $\mathbb{RP}^n$-structure.\\

Rather than work with a family of equivalent connections we can construct a unique connection on a vector bundle which equivalently describes the geometry. We will assume for simplicity that $M$ is an oriented $n$-manifold with projective structure defined by an affine connection $\nabla$. Since $M$ is oriented we may define the line bundle $L$ such that $L^{-n} = \wedge^n T^*M$. For any bundle $E$ we then define $E \left[ k \right] = E \otimes L^k$ where $k$ may be any real number. For any bundle associated to the tangent bundle by a representation of ${GL}(n,\mathbb{R})$ we let $\nabla$ denote the inherited connection.

Define the {\em tractor bundle} $\mathcal{T}$ to be the ${\rm SL}(n+1,\mathbb{R})$-bundle
\begin{equation}
\mathcal{T} = TM \left[ \mu \right] \oplus L^\mu
\end{equation}
where $\mu = -n/(n+1)$. The tractor bundle also has a natural ${\rm SL}(n+1,\mathbb{R})$-connection $\tilde{\nabla}$ called the {\em tractor connection} for projective geometry. It is defined as follows
\begin{equation}\label{projtractor}
\tilde{\nabla}_X \left[ \begin{matrix} Y \\ s \end{matrix} \right] = \left[ \begin{matrix} \nabla_X Y + Xs \\ \nabla_X s + \mathbf{P}(X,Y) \end{matrix} \right]
\end{equation}
where $\mathbf{P}$ is the tensor
\begin{equation}
\mathbf{P}(X,Y) = -\frac{n}{n^2-1}{\rm Ric}(X,Y) - \frac{1}{n^2-1}{\rm Ric}(Y,X)
\end{equation}
and ${\rm Ric}$ is the Ricci tensor ${\rm tr}(Z \to R(Z,X)Y)$ for $\nabla$, where $R$ is the curvature of $\nabla$.\\

From a representative $\nabla$ for the projective structure on $M$ we construct the pair $(\mathcal{T},\tilde{\nabla})$ consisting of a bundle and connection. If we take a different connection $\nabla^1$ on $M$ which is projectively equivalent to $\nabla$ then we produce an isomorphic pair $(\mathcal{T},\tilde{\nabla}^1)$, i.e. the tractor connection changes by an ${\rm SL}(n+1,\mathbb{R})$ gauge transformation preserving $L^\mu \subset \mathcal{T}$. In this sense the tractor connection is uniquely constructed from the projective structure. Furthermore the line bundle $L^\mu \subset{\mathcal{T}}$ is independent of the choice of affine connection.

A key result for us is that the projective structure on $M$ is flat if and only if the associated tractor connection is flat as a connection. Let us examine the flat case further. On the universal cover $\tilde{M}$ of $M$ we may trivialise the tractor bundle $\mathcal{T} \simeq \tilde{M} \times \mathbb{R}^{n+1}$ such that the tractor connection becomes the trivial connection. The line bundle $L^\mu \subset \mathcal{T}$ then defines a map (defined up to the action of ${\rm SL}(n+1,\mathbb{R})$)
\begin{equation}
{\rm dev} : \tilde{M} \to \mathbb{RP}^n
\end{equation}
which is in fact the developing map. From equation (\ref{projtractor}) we verify that the development map is an immersion. Moreover if $\theta : \pi_1(M) \to {\rm SL}(n+1,\mathbb{R})$ is the monodromy of the tractor connection then $\theta$ is clearly also the monodromy of the developing map.\\

We would like a way of knowing when a flat connection is the tractor connection for a flat projective structure. This leads us to the following notion of a non-degenerate section:
\begin{defn}
Let $E$ be a rank $n+1$ bundle with connection $\tilde{\nabla}$ on an $n$-dimensional manifold. A non-vanishing section $s$ spanning a line subbundle $\mathcal{L}$ of $E$ is called {\em non-degenerate} or {\em generic} if the map $TM \to E/\mathcal{L}$ given by
\begin{equation}
X \to \tilde{\nabla}_X s \; ({\rm mod} \mathcal{L} )
\end{equation}
is an isomorphism. Similarly we say that a line subbundle $\mathcal{L} \subset E$ is generic if any non-vanishing local section of $\mathcal{L}$ is generic.
\end{defn}

Now suppose we have a rank $n+1$ bundle $E$ with flat connection $\tilde{\nabla}$ and generic line subbundle $\mathcal{L} \subset E$. As before we may trivialise $E$ on the universal cover $\tilde{M}$ of $M$ and the generic line bundle $\mathcal{L}$ defines a development map ${\rm dev} : \tilde{M} \to \mathbb{RP}^n$ which is $\theta$-equivariant where $\theta$ is the monodromy of the connection. The fact that $\mathcal{L}$ is generic is then equivalent to ${\rm dev}$ being an immersion so when $\mathcal{L}$ is generic this defines a flat projective structure.

We can further identify $E$ as a tractor bundle over $M$ and find an affine connection on $M$ representing the projective structure as follows. Choose any splitting $E = W \oplus \mathcal{L}$. If $E$ is an ${\rm SL}(n+1,\mathbb{R})$ bundle then we also find $\mathcal{L} = L^\mu = (\wedge^n T^*M)^{1/(n+1)}$. The non-degeneracy of $\mathcal{L}$ gives an isomorphism $e :TM \otimes \mathcal{L} \to W$. Now by projection $\tilde{\nabla}$ defines connections $\nabla^\mathcal{L}$ and $\nabla^W$ on $\mathcal{L}$ and $W$ respectively. Therefore we can use the isomorphism $TM \simeq W \otimes \mathcal{L}^*$ to define an affine connection $\nabla$. Let $s$ be a section of $\mathcal{L}$ and let $p : E \to W$ be the projection onto $W$. Then we define $\nabla$ as follows
\begin{equation}
\tilde{\nabla}_{(\nabla_X Y)} s = \tilde{\nabla}_X (p(\tilde{\nabla}_Y s)) - \tilde{\nabla}_Y \nabla^\mathcal{L}_X s \; ({\rm mod} \mathcal{L} ).
\end{equation}
Note that changing the section $s$ or choosing a different splitting of $E$ will result in projective changes in $\nabla$. It follows that $(E,\tilde{\nabla})$ is the tractor bundle and tractor connection for the projective structure defined by $\nabla$. If we can find a section $s$ of $\mathcal{L}$ that is covariantly constant with respect to $\nabla^\mathcal{L}$ then the formula for the induced affine connection simplifies to
\begin{equation}
\tilde{\nabla}_{(\nabla_X Y)} s = \tilde{\nabla}_X (\tilde{\nabla}_Y s) \; ({\rm mod} \mathcal{L} ).
\end{equation}

\subsection{Convex projective structures}\label{convex}
An $\mathbb{RP}^2$-structure is called {\em convex} if the developing map ${\rm dev} : \tilde{\Sigma} \to \mathbb{RP}^2$ is an embedding of the universal cover $\tilde{\Sigma}$ onto a convex subset $\Omega$ contained in some affine patch. See \cite{gold} for an introduction. In such a case the fundamental group $\pi_1(\Sigma)$ identifies with a discrete subgroup $\Gamma \subset {\rm PSL}(3,\mathbb{R})$ which acts freely and properly on $\Omega$, hence $\Sigma$ identifies with $\Omega/\Gamma$. From the Poincar\'e disc model we see that a uniformising representation is an example of a convex $\mathbb{RP}^2$-structure.\\

We define the moduli space $C(\Sigma)$ of convex $\mathbb{RP}^2$-structures on $\Sigma$ to be the set of all equivalence classes of convex projective $\mathbb{RP}^2$-structures on $\Sigma$ where two structures are considered equivalent if they are related by a diffeomorphism isotopic to the identity. The space has a natural topology and Goldman \cite{gold} shows the monodromy map
\begin{equation}
{\rm hol} : C(\Sigma) \to {\rm Hom}(\pi_1(\Sigma) , {\rm PSL}(3,\mathbb{R})) / {\rm PSL}(3,\mathbb{R})
\end{equation}
is an embedding onto an open, Hausdorff subspace homeomorphic to a Euclidean space of dimension $16g-16$. Since $C(\Sigma)$ is connected it is contained within a component of the space of representations. But since the uniformising representations yield convex $\mathbb{RP}^2$-structures this component must be the Hitchin component. Therefore the monodromy map identifies $C(\Sigma)$ with an open subspace of the Hitchin component for ${\rm PSL}(3,\mathbb{R})$. Naturally one might expect $C(\Sigma)$ to be the whole component and this was indeed proved by Choi and Goldman \cite{choigold}.\\

Our contribution will be to show how one can obtain the convex $\mathbb{RP}^2$-structures directly from the Higgs bundle construction of the Hitchin component. In doing so the connection between convex projective structures and affine spheres will naturally emerge.\\

Recall the procedure of Section \ref{hitcomp} for constructing Higgs bundles in the Hitchin component. We are interested in the $\mathfrak{sl}(3,\mathbb{C})$ case. A representation in this component is described by a quadratic and cubic differential. We use the result of Labourie that we can choose a conformal structure on $\Sigma$ so that the associated quadratic differential vanishes. This leaves a holomorphic cubic differential.\\

Let $V$ be a complex $3$-dimensional vector space. Let $\hat{\rho}$ be the anti-involution $\hat{\rho}(A) = -\overline{A}^t$ for a compact real form and $\hat{\lambda}$ the anti-involution $\hat{\lambda}(A) = H\overline{A}H$ for a split real form where
\begin{equation}
H = \left[ \begin{matrix} 0 & 0 & 1 \\ 0 & 1 & 0 \\ 1 & 0 & 0 \end{matrix} \right].
\end{equation}
Moreover we have an anti-involution also denoted $\hat{\lambda}$ on $V$ given by $\hat{\lambda} v = H\overline{v}$. The split real form for $\hat{\lambda}$ preserves the fixed point subspace of $\hat{\lambda}$ on $V$.\\

We have a principal $3$-dimensional subalgebra $\mathfrak{s} = \langle x , e , \tilde{e} \rangle$ where
\begin{equation}
x = \left[ \begin{matrix} 1 & 0 & 0 \\ 0 & 0 & 0 \\ 0 & 0 & -1 \end{matrix} \right], \; e = \left[ \begin{matrix} 0 & 1 & 0 \\ 0 & 0 & 1 \\ 0 & 0 & 0 \end{matrix} \right], \; \tilde{e} = \left[ \begin{matrix} 0 & 0 & 0 \\ 1 & 0 & 0 \\ 0 & 1 & 0 \end{matrix} \right].
\end{equation}
Let $\mathcal{W}$ be the associated bundle on $\Sigma$, so from the form of $x$ we have $\mathcal{W} = K \oplus 1 \oplus K^{-1}$. Let $q$ be a holomorphic cubic differential. The corresponding Higgs field $\Phi$ is then
\begin{equation}
\Phi = \left[ \begin{matrix} 0 & 0 & q \\ 1 & 0 & 0 \\ 0 & 1 & 0 \end{matrix} \right].
\end{equation}
Let $h = e^{2\Omega}$ be the corresponding Hermitian metric on $\mathcal{W}$. Then
\begin{equation}
h = \left[ \begin{matrix} h_1 & 0 & 0 \\ 0 & h_1^{-1}h_2^{-1} & 0 \\ 0 & 0 & h_2 \end{matrix} \right], \; \Omega = \left[ \begin{matrix} w_1 & 0 & 0 \\ 0 & -w_1 - w_2 & 0 \\ 0 & 0 & w_2 \end{matrix} \right]
\end{equation}
but since $\hat{\rho}\Omega = \hat{\lambda}\Omega = -\Omega$ we find that $w_2 = -w_1$ so way may write
\begin{equation}
h = \left[ \begin{matrix} h & 0 & 0 \\ 0 & 1 & 0 \\ 0 & 0 & h^{-1} \end{matrix} \right], \; \Omega = \left[ \begin{matrix} w & 0 & 0 \\ 0 & 0 & 0 \\ 0 & 0 & -w \end{matrix} \right].
\end{equation}
We know that $h$ transforms as a $(1,1)$-form, hence it can be regarded as a metric on $\Sigma$.\\

The compact and split anti-involutions on the adjoint bundle are $\rho = {\rm Ad}_h \circ \hat{\rho}$, $\lambda = {\rm Ad}_h \circ \hat{\lambda}$. Moreover, there corresponds to $\lambda$ a real structure on $\mathcal{W}$ given by $\lambda v = h H \overline{v}$. For $v = [A,B,C]^t$ we have $\lambda v = [h\overline{C} , \overline{B} , h^{-1} \overline{A} ]^t$. The associated flat connection preserves the real subbundle $\mathcal{W}^\lambda$ of fixed points of $\lambda$.

We may now determine $\Phi^* = -\rho \Phi$ and the connection form $\mathcal{A} = -2\partial \Omega + \Phi + \Phi^*$:
\begin{equation}
\Phi^* = \left[ \begin{matrix} 0 & h & 0 \\ 0 & 0 & h \\ h^{-2}\overline{q} & 0 & 0 \end{matrix} \right], \; \mathcal{A} = \left[ \begin{matrix} -2\partial w & h & q \\ 1 & 0 & h \\ h^{-2}\overline{q} & 1 & 2\partial w \end{matrix} \right].
\end{equation}
\\

To obtain a projective structure on $\Sigma$ from this data we also need to choose a generic real line subbundle. An obvious candidate is the real subspace of the trivial factor of $\mathcal{W} = K \oplus 1 \oplus K^{-1}$. Let $s$ denote the section $s(x) = [0,1,0]^t$, $L = 1^\lambda$ the real line bundle spanned by $s$. We have
\begin{equation}
\nabla s = \left[ \begin{matrix} h \\ 0 \\ 1 \end{matrix} \right] = \left[ \begin{matrix} h'dz \otimes d\overline{z} \\ 0 \\ dz^{-1} \otimes dz \end{matrix} \right]
\end{equation}
where $h = h' dz \otimes d\overline{z}$. This proves $s$ is generic and hence we have a projective structure on $\Sigma$.\\

So far we have not shown the projective structure is convex. For this we will need to examine the projective structure in more detail.\\

Our flat connection on a rank $3$ bundle together with a generic global section $s$ not only defines a development map $\tilde{\Sigma} \to \mathbb{RP}^2$ but also a lift $\tilde{\Sigma} \to \mathbb{R}^3$. This gives $\tilde{\Sigma}$ the structure of an immersed surface in $\mathbb{R}^3$. There is an induced affine connection on $\tilde{\Sigma}$ representing the projective structure. In fact the isomorphism $T\Sigma \to \mathcal{W}^\lambda / L$ given by $X \to \nabla_X s \, ( {\rm mod} \, L)$ induces an affine connection $\nabla^a$ from $\nabla$. It is defined by the relation
\begin{equation}
\nabla_{(\nabla^a_X Y)} s = \nabla_X \nabla_Y s \, ( {\rm mod} \, L).
\end{equation}
Let $X = u\tfrac{\partial}{\partial z} + \overline{u}\tfrac{\partial}{\partial \overline{z}}$ be a vector field on $\Sigma$. We have
\begin{equation}
\nabla \nabla_X s = \left[ \begin{matrix} (d(h'u) - 2(\partial w)h'\overline{u} + q'u dz) \otimes dz \\ 0 \\ (du + \overline{uq'}/h' d\overline{z} + 2(\partial w) u ) \otimes dz^{-1} \end{matrix} \right] \, ( {\rm mod} \, L)
\end{equation}
where $q = q'(dz)^3$. It follows, writing tangent vectors as column vectors $X = [u , \overline{u} ]^t$ that
\begin{equation}
\nabla^a \left[ \begin{matrix} u \\ \overline{u} \end{matrix} \right] = \left[ \begin{matrix} du + 2(\partial w)u + \overline{uq'}/h' d\overline{z} \\ d\overline{u} + 2\overline{u}dw - 2(\partial w)\overline{u} + q'u/h' dz \end{matrix} \right].
\end{equation}
Therefore the connection form $\mathcal{A}^a$ for $\nabla^a$ is
\begin{equation}
\mathcal{A}^a = \left[ \begin{matrix} 2\partial w & \overline{q}/h \\ q/h & 2\overline{\partial}w \end{matrix} \right].
\end{equation}

The Higgs bundle equation $-2\overline{\partial}\partial \Omega + [\Phi , \Phi^*] = 0$ becomes the following equation for $h$:
\begin{equation}
2\partial \overline{\partial} w + q \overline{q} e^{-4w} - e^{2w} = 0.
\end{equation}
This real form of the $\mathfrak{sl}(3,\mathbb{C})$ affine Toda equation is known as the Tzitzeica equation \cite{lof}. We will see in Section \ref{affsph} that this structure is precisely the requirement for ${\rm dev} : \tilde{\Sigma} \to \mathbb{R}^3$ to be an immersed hyperbolic affine sphere. Moreover $2h$ is the affine metric and $\nabla^a$ is the Blaschke connection \cite{lof}. Since the metric $2h$ on $\tilde{\Sigma}$ is complete it will follow from the results in Section \ref{affsph} that $\tilde{\Sigma}$ is in fact a properly embedded submanifold of $\mathbb{R}^3$ and under the projection $\mathbb{R}^3 \to \mathbb{RP}^2$ maps bijectively with a convex subset $\Omega \subset \mathbb{R}^2 \subset \mathbb{RP}^2$ contained in an affine subspace. Hence the projective structure constructed is indeed convex.

\subsection{Affine spheres}\label{affsph}
We provide a brief review on affine differential geometry and affine spheres leading to the connection between hyperbolic affine spheres in $\mathbb{R}^3$ and convex $\mathbb{RP}^2$-structures. Some references for affine differential geometry and affine spheres are \cite{nomsa}, \cite{lof}.\\

Affine differential geometry concerns the properties of hypersurfaces in $\mathbb{R}^{n+1}$ invariant under the action of the special affine group ${\rm SL}(n+1,\mathbb{R}) \ltimes \mathbb{R}^{n+1}$. Let $D$ denote the trivial affine connection on $\mathbb{R}^{n+1}$. Given an immersed hypersurface $f : M \to \mathbb{R}^{n+1}$ and a transverse vector field $\xi$ on $M$, the decomposition $f^*(T\mathbb{R}^{n+1}) = TM \oplus \mathbb{R}\xi$ induces a torsion free affine connection $\nabla$ on $M$ with the properties
\begin{eqnarray}
D_X Y &=& \nabla_X Y + h(X,Y)\xi \label{affconn1} \\
D_X \xi &=& -S(X) + \tau(X)\xi \label{affconn2}
\end{eqnarray}
for any two vector fields $X,Y$ on $M$. Here $h$ is a symmetric bilinear form, $S$ an endomorphism of $TM$ and $\tau$ a $1$-form on $M$. We will restrict attention to convex hypersurfaces, which amounts to assuming that $h$ is positive definite and that $\xi$ points to the convex side. In this case $h$ defines a metric on $M$.\\

For a convex hypersurface the {\em affine normal} is the unique transverse vector field $\xi$ with the properties that $\xi$ points to the convex side, $\tau  = 0$ and ${\rm dvol}_h = \iota_\xi \omega$ where $\omega$ is the standard volume form on $\mathbb{R}^{n+1}$. From now on we always take the affine normal so that (\ref{affconn1})-(\ref{affconn2}) reduce to
\begin{eqnarray}
D_X Y &=& \nabla_X Y + h(X,Y)\xi \\
D_X \xi &=& -S(X).
\end{eqnarray}
In this case $h$ is called the {\em affine metric}, $\nabla$ the {\em Blaschke connection} and $S$ the {\em shape operator}. Since $h$ is non-degenerate we also have the Levi-Civita connection $\nabla^h$ for $h$ on $M$. The difference $C = \nabla^h - \nabla$ is called the {\em Pick form}.\\

An {\em affine sphere} is a convex hypersurface $f : M \to \mathbb{R}^{n+1}$ such that the shape operator has the form $S = \lambda I$ where $I$ is the identity operator and $\lambda$ is a constant. We say that $M$ is an {\em elliptic affine sphere} if $\lambda > 0$, {\em parabolic} if $\lambda = 0$ and {\em hyperbolic} if $\lambda < 0$. We see that $\lambda f + \xi$ is constant along $M$, so in the parabolic case $\xi$ is constant along $M$ while in the other cases we may translate $M$ so that $\xi = -\lambda f$. Moreover we can rescale so that $\lambda = 1$ in the elliptic case and $\lambda = -1$ in the hyperbolic case.\\

The key result that ties affine spheres to convex projective structures is the following \cite{lof}
\begin{prop}
Given a convex bounded domain $\Omega \subset \mathbb{R}^n$ where $\mathbb{R}^n \subset \mathbb{R}^{n+1}$ is the affine plane $x_{n+1} = 1$, there is a unique properly embedded hyperbolic affine sphere with $\lambda = -1$ and centre $0$ asymptotic to the boundary of the cone $\{ t\Omega \, |  \, t>0 \}$. Any immersed hyperbolic affine sphere is properly embedded if and only of the affine metric is complete. Any such affine sphere $M$ is asymptotic to the boundary of the cone given by the convex hull of $M$ and $0$.
\end{prop}

We saw in Section \ref{convex} how for a compact Riemann surface $\Sigma$ of genus $g>1$, a choice of complex structure and cubic holomorphic differential gives rise to an immersed affine sphere $\tilde{\Sigma} \to \mathbb{R}^3$. Moreover the affine metric on $\tilde{\Sigma}$ is complete so that $\tilde{\Sigma}$ is in fact properly embedded and thus gives rise to a convex projective structure.

Conversely a convex $\mathbb{RP}^2$-structure arises from a convex subset $\Omega \subset \mathbb{RP}^2$ which is acted upon by a subgroup of $\Gamma \subset {\rm SL}(3,\mathbb{R})$. Since there is a unique hyperbolic affine sphere $M \to \mathbb{R}^3$ asymptotic to $\Omega$ it follows that $\Gamma$ acts on $M$ and the quotient $M / \Gamma$ identifies with $\Sigma$. Now the affine metric on $M$ descends to $\Sigma$ defining a conformal structure. The affine connection likewise descends, in particular the Pick form is defined on $\Sigma$ and identifies with the holomorphic cubic differential of Section \ref{convex}.\\

The correspondences just described are mutually inverse so we have constructed a bijection between pairs $(*,q)$ consisting of a complex structure on $\Sigma$ and a holomorphic cubic differential and the convex $\mathbb{RP}^2$-structures on $\Sigma$. Recall the result of Choi and Goldman that the moduli space of convex $\mathbb{RP}^2$-structures is the Hitchin component for ${\rm PSL}(3,\mathbb{R})$. Therefore we have shown in the case of ${\rm PSL}(3,\mathbb{R})$ that for each representation in the Hitchin component there is a {\em unique} conformal structure for which the associated holomorphic quadratic differential vanishes.

\section{Hitchin component for ${\rm PSp}(4,\mathbb{R})$}\label{hcfpsp}
We now turn our attention to the Hitchin component for ${\rm PSp}(4,\mathbb{R})$. In this case the Hitchin component is seen to identify with a subspace of the Hitchin component for ${\rm PSL}(4,\mathbb{R})$ which was studied by Guichard and Wienhard \cite{gui} in which it was shown that representations in this component correspond to {\em convex-foliated projective structures}, a class of projective structure on the unit tangent bundle of the surface. Given a compact Riemann surface $\Sigma$ the unit tangent bundle $M$ is the bundle of unit tangent vectors with respect to any Riemannian metric on $\Sigma$. Within this space the projective structures corresponding to ${\rm PSp}(4,\mathbb{R})$ are precisely those with a contact structure which in local projective coordinates matches the canonical contact structure on $\mathbb{RP}^3$.\\

In Section \ref{proj} we use Higgs bundle techniques to construct projective structures on the unit tangent bundle $M$ with contact structures. We will then argue in Section \ref{cfps} that our construction fills out the space of projective structures (with contact structure) of Guichard and Wienhard. Finally in Section \ref{rtlcams} we relate our projective structures to line congruences and minimal surfaces.\\

\subsection{Projective structures for ${\rm PSp}(4,\mathbb{R})$}\label{proj}
We begin by describing the complex Lie algebra $\mathfrak{sp}(4,\mathbb{C})$. Let $W$ be a complex $4$-dimensional vector space with symplectic form $\omega(u,v) = u^t \omega v$ where $\omega$ is the matrix
\begin{equation}
\omega = \left[ \begin{matrix} 0 & 0 & 0 & 1 \\ 0 & 0 & -1 & 0 \\ 0 & 1 & 0 & 0 \\ -1 & 0 & 0 & 0 \end{matrix} \right]
\end{equation}
then $\mathfrak{sp}(4,\mathbb{C})$ is the subalgebra of $\mathfrak{gl}(W)$ of endomorphisms $A$ such that $A^t\omega + \omega A =0$. We fix compact and split real structures $\hat{\rho}$ and $\hat{\tau}$ as follows: $\hat{\rho}(A) = -\overline{A}^t$, $\hat{\tau}(A) = H\overline{A}H$, where
\begin{equation}
H = \left[ \begin{matrix} 0 & 0 & 0 & 1 \\ 0 & 0 & 1 & 0 \\ 0 & 1 & 0 & 0 \\ 1 & 0 & 0 & 0 \end{matrix} \right].
\end{equation}
Moreover corresponding to $\hat{\tau}$ is an anti-involution on $W$ also denoted $\hat{\tau}$. It is given by $\hat{\tau}v = H\overline{v}$ and the fixed point subspace defines the real $4$-dimensional subspace which the split real form preserves.

We can choose for the principal $3$-dimensional subgroup $\mathfrak{s} = \langle x, e, \tilde{e} \rangle$, where
\begin{equation}
\begin{aligned}
x &= \left[ \begin{matrix} 3/2 & 0 & 0 & 0 \\ 0 & 1/2 & 0 & 0 \\ 0 & 0 & -1/2 & 0 \\ 0 & 0 & 0 & -3/2 \end{matrix} \right], &
e &= \left[ \begin{matrix} 0 & \sqrt{3/2} & 0 & 0 \\ 0 & 0 & \sqrt{2} & 0 \\ 0 & 0 & 0 & \sqrt{3/2} \\ 0 & 0 & 0 & 0 \end{matrix} \right], \\
\tilde{e} &= \left[ \begin{matrix} 0 & 0 & 0 & 0 \\ \sqrt{3/2} & 0 & 0 & 0 \\ 0 & \sqrt{2} & 0 & 0 \\ 0 & 0 & \sqrt{3/2} & 0 \end{matrix} \right]. & &
\end{aligned}
\end{equation}

Now recall the construction of the Higgs bundle $(E,\Phi)$ for the Hitchin component. Choose a spin structure $K^{1/2}$. The bundle $E$ is the ${\rm PSp}(4,\mathbb{C})$-bundle associated to the ${\rm SL}(2,\mathbb{C})$-bundle $K^{1/2}\oplus K^{-1/2}$ under the inclusion of the principal $3$-dimensional subalgebra. However we can lift the structure group to ${\rm Sp}(4,\mathbb{C})$ so that $E = \mathfrak{sp}(\mathcal{W})$ where $\mathcal{W}$ is the rank $4$-symplectic vector bundle $\mathcal{W} = K^{3/2} \oplus K^{1/2} \oplus K^{-1/2} \oplus K^{-3/2}$. Let $q \in {\rm H}^0(\Sigma, K^4)$ be a holomorphic quartic differential. The associated Higgs field is
\begin{equation}
\Phi = \left[ \begin{matrix} 0 & 0 & 0 & q \\ \sqrt{3/2} & 0 & 0 & 0 \\ 0 & \sqrt{2} & 0 & 0 \\ 0 & 0 & \sqrt{3/2} & 0 \end{matrix} \right].
\end{equation}
Let $h = e^{2\Omega}$ be the associated Hermitian structure:
\begin{equation}
h = \left[ \begin{matrix} h & 0 & 0 & 0 \\ 0 & k^{-1} & 0 & 0 \\ 0 & 0 & k & 0 \\ 0 & 0 & 0 & h^{-1} \end{matrix} \right], \;
\Omega = \left[ \begin{matrix} a & 0 & 0 & 0 \\ 0 & -b & 0 & 0 \\ 0 & 0 & b & 0 \\ 0 & 0 & 0 & -a \end{matrix} \right].
\end{equation}
Then the compact and split anti-involutions for the adjoint bundle are $\rho = {\rm Ad}_h \circ \hat{\rho}$, $\tau = {\rm Ad}_h \circ \hat{\tau}$. The real real rank $4$ subbundle $\mathcal{W}^\tau$ preserved by $\nabla$ is the fixed point set of the anti-involution $\tau v = h \hat{\tau} v$. Now we may calculate $\Phi^* = -\rho \Phi$ and the connection form $\mathcal{A} = -2\partial \Omega + \Phi + \Phi^*$:
\begin{eqnarray}
\Phi^* &=& \left[ \begin{matrix} 0 & \sqrt{3/2}hk & 0 & 0 \\ 0 & 0 & \sqrt{2}/k^2 & 0 \\ 0 & 0 & 0 & \sqrt{3/2}kh \\ \overline{q}/h^2 & 0 & 0 & 0 \end{matrix} \right], \\
\mathcal{A} &=& \left[ \begin{matrix} -2\partial a & \sqrt{3/2}hk & 0 & q \\ \sqrt{3/2} & 2\partial b & \sqrt{2}/k^2 & 0 \\ 0 & \sqrt{2} & -2\partial b & \sqrt{3/2}kh \\ \overline{q}/h^2 & 0 & \sqrt{3/2} & 2\partial a \end{matrix} \right].
\end{eqnarray}

Let $\pi : U\Sigma \to \Sigma$ be the unit tangent bundle over $\Sigma$. We may pull back $\mathcal{W}$ and $\nabla$ to $U\Sigma$. In particular, $\pi^*\mathcal{W}^\tau$ is a real rank $4$ bundle over the $3$-dimensional space $U\Sigma$ with flat connection $\pi^*\nabla$. We describe a canonical real non-trivial line subbundle $L \subset \pi^*\mathcal{W}^\tau$ and show that any non-vanishing local section of $L$ is generic with respect to $\pi^*\nabla$. This shows there is a canonical flat projective structure associated to the Higgs bundle.\\

Let $z$ be a local holomorphic coordinate for an open subset $U$ of $\Sigma$. We have local coordinates $z,\theta$ for the open set $\pi^{-1}(U)$ of $U\Sigma$ where for a given tangent vector $X \in T_zU$, $\theta$ represents the angle between $\tfrac{\partial}{\partial x}$ and $X$. The canonical line bundle $L$ is defined as follows: the fibre $L_{(z,\theta)}$ is the real $1$-dimensional subspace spanned by $\pm e^{i\theta/2}dz^{-1/2} + \tau(\pm e^{i\theta/2}dz^{-1/2})$. Note that the points $\pm e^{i\theta/2}dz^{-1/2}$ in $\mathbb{P}_{\mathbb{R}}(K^{-1/2})$ correspond to the ray of tangent vectors in the direction $e^{i\theta}\tfrac{\partial}{\partial z}$ where $K^{-1}$ is identified with the tangent bundle. This shows that $L$ is canonically defined. \\

For a vector $v = [A,B,C,D]^t$ we have that $\tau v = [h\overline{D}, k^{-1}\overline{C} , k\overline{B} , h^{-1}\overline{A} ]^t$. Therefore we may write
\begin{equation}
\tau(e^{i/2\theta}dz^{-1/2}) = k^{-1}e^{-i\theta/2}{d\overline{z} \,}^{-1/2} = A = A'dz^{1/2} \in K^{1/2}.
\end{equation}
If we write $k = k' dz^{-1/2} {d\overline{z} \,}^{-1/2}$, $h = h'dz^{3/2}{d\overline{z } \, }^{-3/2}$ where $k',h'$ are non-vanishing real functions then
\begin{equation}
A' = (k')^{-1}e^{-i\theta/2}
\end{equation}
and $L_{(z,\theta)}$ is the real span of
\begin{equation}
s(z,\theta) = [0, A , \overline{A}k , 0 ]^t.
\end{equation}
Now $A$ is non-vanishing so $s$ defines a local section of $L$. However $s$ does not extend to a global section for if we go once around the fibre we end up with $-s$. In fact $L$ restricted to any fibre is clearly non-orientable. We could on the other hand pass to the double cover of $U\Sigma$, the unit circle bundle of $K^{-1/2}$ on which $L$ is clearly trivial. \\

In order to prove the genericity of $L$ it suffices to consider the local sections of form $s$ as defined. We have
\begin{equation}
\nabla s = \left[ \begin{matrix} \sqrt{3/2}h'k'A'dz^{3/2} \otimes d\overline{z} \\ dA +2(\partial b )A + \sqrt{2}k^{-1}\overline{A} \\ d(\overline{A}k) + \sqrt{2} A - 2(\partial b)\overline{A}k \\ \sqrt{3/2}\overline{A'}k' dz^{-3/2} \otimes dz \end{matrix} \right].
\end{equation}
The bottom and top entries are non-vanishing multiples of $d\overline{z}$ and $dz$ respectively, so to prove non-degeneracy it suffices to show that $\iota_{\tfrac{\partial}{\partial \theta}} \nabla s$ does not lie in $L$ for any point $(z,\theta)$. We have
\begin{equation}
\iota_{ \tfrac{\partial}{\partial \theta}} \nabla s  = \left[ \begin{matrix} 0 \\ \tfrac{\partial A'}{\partial \theta} dz^{1/2} \\ \tfrac{\partial (\overline{A'}k')}{\partial \theta} dz^{-1/2} \\ 0 \end{matrix} \right] = \left[ \begin{matrix} 0 \\ -\tfrac{i}{2} A \\ \tfrac{i}{2} \overline{A}k \\ 0 \end{matrix} \right]
\end{equation}
which proves non-degeneracy since $-\tfrac{i}{2}A$ is a non-vanishing imaginary multiple of $A$. Thus we have found a flat projective structure on $U\Sigma$. The monodromy of this structure is a representation in the Hitchin component for ${\rm PSp}(4,\mathbb{R})$. We make two observations from this: first the monodromy of the projective structure around the fibres is trivial, second since ${\rm PSp}(4,\mathbb{R})$ preserves a contact distribution on $\mathbb{RP}^3$, there is an associated contact distribution on $U\Sigma$. In fact a contact form is obtained by contracting the symplectic form on $\pi^*\mathcal{W}^\tau$ in the direction corresponding to the line bundle $L$, so we see that the contact distribution is transverse to the fibres.

One final observation is that the fibres of $U\Sigma$ are geodesics in the projective structure. This amounts to showing $\nabla_{\tfrac{\partial}{\partial \theta}} ( \nabla_{\tfrac{\partial}{\partial \theta}} s ) = 0 \; {\rm mod}(s)$ which follows from a direct computation.\\

\subsection{Convex-foliated projective structures}\label{cfps}
We will show that the projective structures constructed by the Higgs bundle approach coincide with the convex-foliated projective structures for uniformising representations. \\

Let us first explain the convex-foliated projective structure for a uniformising representation \cite{gui}. Let $\mathbb{H}^2$ represent $2$-dimensional hyperbolic space with the hyperbolic metric. As usual let $\Sigma$ be a compact Riemann surface of genus $g>1$. If $\Sigma$ is given a constant curvature metric then $\tilde{\Sigma}$ is identified with $\mathbb{H}^2$ up to isometry, such that $\pi_1(\Sigma)$ acts by isometries, hence defining a uniformising representation $\rho : \pi_1(\Sigma) \to {\rm PSL}(2,\mathbb{R})$ up to conjugation.

Let $U\mathbb{H}^2$ denote the unit tangent bundle of $\mathbb{H}^2$. We may define a projective structure on $U\Sigma$ with trivial holonomy around the fibres of $U\Sigma \to \Sigma$ by describing a developing map, that is a local diffeomorphism $\psi : U\mathbb{H}^2 \to \mathbb{RP}^3$ such that $\psi(u \gamma^l) = \hat{\rho}(\gamma)^{-1}\psi(u)$ where $\gamma^l \in \pi_1(U\Sigma)$ is any lift of $\gamma \in \pi_1(\Sigma)$ and $\hat{\rho}$ is the composition of $\rho$ with the map ${\rm PSL}(2,\mathbb{R}) \to {\rm PSL}(4,\mathbb{R})$ given by the principal $3$-dimensional subalgebra.\\

Let us note that $U\mathbb{H}^2$ can be identified with ${\rm PSL}(2,\mathbb{R})$. In fact ${\rm PSL}(2,\mathbb{R})$ acts freely and transitively on $U\mathbb{H}^2$. Moreover if $\mathbb{H}^2$ is identified with ${\rm PSL}(2,\mathbb{R})/{\rm U}(1)$, then $U\mathbb{H}^2$ is identified with the principal ${\rm U}(1)$-bundle ${\rm PSL}(2,\mathbb{R}) \to \mathbb{H}^2$ over $\mathbb{H}^2$.

Under the inclusion ${\rm SL}(2,\mathbb{R}) \to {\rm SL}(4,\mathbb{R})$ of the principal $3$-dimensional subalgebra, the fundamental representation of ${\rm SL}(4,\mathbb{R})$ becomes the third symmetric power $S^2(\mathbb{R}^2)$ of the fundamental representation of ${\rm SL}(2,\mathbb{R})$. This is clear since the adjoint representation of ${\rm SL}(4,\mathbb{R})$ under the principal $3$-dimensional subgroup contains an $S^6(\mathbb{R}^2)$ factor. We may consider $\mathbb{RP}^3$ as $\mathbb{P}(S^3(\mathbb{R}^2))$, the projectivisation of the third symmetric power of the fundamental representation for ${\rm SL}(2,\mathbb{R})$. Then ${\rm PSL}(2,\mathbb{R})$ acts on $\mathbb{RP}^3$ and has two open orbits. Elements of $\mathbb{RP}^3$ can be thought of a cubic polynomials up to scale, hence they are given by specifying their zeros. The two open orbits are those cubics with positive or negative discriminant. If the discriminant of a real cubic is positive then it has $3$ distinct real roots, if it is negative then it has $1$ real root and two distinct complex conjugate roots.\\

We may utilise either of the two open orbits to define a developing map. Taking a cubic $C$ with non-zero discriminant we have a map ${\rm PSL}(2,\mathbb{R}) \to \mathbb{RP}^3$ given by $g \mapsto g C$. Upon identifying $U\mathbb{H}^2$ with ${\rm PSL}(2,\mathbb{R})$ we have an equivariant map and since the orbit is open and $3$-dimensional the map is a local diffeomorphism, hence serves as a developing map. Note that if the discriminant of $C$ is positive then the stabiliser of $C$ is $\mathbb{Z}_3$ while if the discriminant is negative then the stabiliser is trivial. The choice of open orbit with negative discriminant yields the convex foliated projective structure for $\rho$ \cite{gui}. We examine this structure more carefully.\\

Given an element $u \in U\mathbb{H}^2$ we can associate to $u$ an element $g \in {\rm PSL}(2,\mathbb{R})$ and hence a cubic $g C$. However there is a more direct way of associating a cubic to $u$ with negative discriminant. To describe a cubic up to scale it suffices to specify the roots, one real and two conjugate. Let the unit tangent vector $u$ lie over the base point $\alpha \in \mathbb{H}^2$. This is a complex number in the upper half plane, hence $\alpha$ and $\overline{\alpha}$ will serve to define the conjugate roots. Consider the geodesic starting at $\alpha$ and heading in the direction $u$. We may follow this geodesic to the boundary and associate to it the boundary point $a \in \mathbb{RP}^1$. This will serve to define the real root. Since ${\rm PSL}(2,\mathbb{R})$ acts on $\mathbb{H}^2$ by isometries it is clear that we have an equivariant sequence of maps
\begin{equation*}
{\rm PSL}(2,\mathbb{R}) \to U\mathbb{H}^2 \to \mathbb{H}^2 \times \mathbb{RP}^1.
\end{equation*}
An equivariant map $U\mathbb{H}^2 \to \mathbb{P}(S^3(\mathbb{R}^2))$ is then obtained by taking the symmetric product of equivariant maps $\psi_2 :\mathbb{H}^2 \to \mathbb{P}(S^2(\mathbb{R}^2))$ and $\psi_1 :\partial \mathbb{H}^2 = \mathbb{RP}^1 \to \mathbb{P}(\mathbb{R}^2)$. These maps are as follows:
\begin{eqnarray*}
\psi_2(\alpha) &=& (\alpha Z + 1)(\overline{\alpha}Z + 1), \\
\psi_1([a_0,a_1]) &=& (a_0Z + a_1)
\end{eqnarray*}
where $Z$ is an indeterminate and $a = [a_0,a_1] \in \mathbb{RP}^1$. Note that the roots of these polynomials are not the points $\alpha, \overline{\alpha} , a$ but rather the corresponding points under the isomorphism of representations $\mathbb{R}^2 \simeq {\mathbb{R}^2}^*$. This is so that the polynomials transform under the representation $S^k(\mathbb{R}^2)$ rather than $S^k({\mathbb{R}^2}^*)$, that is a polynomial $P(Z)$ of degree $k$ transforms as
\begin{equation*}
\left[ \begin{matrix} a & b \\ c & d \end{matrix} \right] \cdot P(Z) = P \left( \frac{aZ+b}{cZ+d} \right) \cdot (cZ+d)^k.
\end{equation*}
By equivariance this is simply an alternative way of describing the developing map $\psi : U\mathbb{H} \to \mathbb{RP}^3$. It has the advantage that the map is naturally a product of two simpler maps $\psi_1, \psi_2$. By finding Higgs bundle interpretations for these two maps we will be able to identify this projective structure with the projective structure we constructed by Higgs bundle methods.\\

Now we show that the projective structure constructed in Section \ref{proj} matches this. To do this we must trivialise the bundle $E = K^{3/2} \oplus K^{1/2} \oplus K^{-1/2} \oplus K^{-3/2}$ with respect to the flat connection corresponding to the uniformising representation. The non-degenerate line bundle then defines a developing map and we check that up to a projective transformation this is the same map we get from the construction for the convex foliated projective structure.\\

Since the uniformising representation has holonomy in the principal three dimensional subgroup ${\rm SL}(2,\mathbb{R}) \subset {\rm SL}(4,\mathbb{R})$ we may first trivialise the associated two dimensional representation $K^{1/2} \oplus K^{-1/2}$ with respect to the flat connection and then identify the rank $4$ bundle $E$ with $S^3(K^{1/2} \oplus K^{-1/2})$.\\

Since we work with the uniformising representation we identify the universal cover of the surface with the upper half-plane $\mathbb{H}^2$ with the constant curvature $-2$ metric $\frac{dz d\overline{z}}{2y^2}$. The real covariant constant sections of $K^{1/2} \oplus K^{-1/2}$ are
\begin{equation}
s = \left[ \begin{matrix} s_1 \\ s_2 \end{matrix} \right]
\end{equation}
such that
\begin{equation}
\begin{aligned}
s_1 &= \tfrac{\overline{s_2}}{\sqrt{2}y} \\
s_2 &= (az+b)e^{i\pi/4}dz^{-1/2}
\end{aligned}
\end{equation}
where $a,b$ are real constants. Setting $(a,b) = (1,0),(0,1)$ we get two sections
\begin{eqnarray*}
e_1 &=& \tfrac{\overline{z}}{\sqrt{2}y}\lambda^{-1}X + z\lambda Y \\
e_2 &=& \tfrac{\lambda^{-1}}{\sqrt{2}y}X + \lambda Y
\end{eqnarray*}
where $\lambda = e^{i\pi/4}$, $X = \left[ \begin{matrix} 1 \\ 0 \end{matrix} \right]$, $Y = \left[ \begin{matrix} 0 \\ 1 \end{matrix} \right]$.\\

The non-degenerate line bundle $L \subset K^{3/2} \oplus K^{1/2} \oplus K^{-1/2} \oplus K^{-3/2}$ defining the projective structure constructed by Higgs bundle means can be thought of as the product of a line bundle $L_1 \subset K^{1/2} \oplus K^{-1/2}$ with the evident trivial line bundle $L_2 = 1 \subset K \oplus 1 \oplus K^{-1} = S^2(K^{1/2} \oplus K^{-1/2})$. When identifying $K \oplus 1 \oplus K^{-1}$ with $S^2(K^{1/2} \oplus K^{-1/2})$ and similarly for the third symmetric power we must be careful to ensure that the identification recovers the flat connection on the rank $4$ bundle $E$ since this is the connection we wish to trivialise. We have a standard matrix form of the principal three dimensional subgroup as subgroups of ${\rm SL}(3,\mathbb{R})$ and ${\rm SL}(4,\mathbb{R})$. With respect to these one finds that $dz , 1 , dz^{-1/2}$ should identify with $X^2,\sqrt{2}XY,Y^2$ and $dz^{3/2}, dz^{1/2} , dz^{-1/2} , dz^{-3/2}$ should identify with $X^3 , \sqrt{3}X^2Y , \sqrt{3}XY^2 , Y^3$.\\

The line bundle $L_2$ corresponds to $\sqrt{2}XY$, so we need to express this in terms of the covariantly constant basis $e_1e_1, e_1e_2, e_2e_2$. One finds that
\begin{equation*}
\sqrt{2}XY = \tfrac{1}{2y}(e_1 - ze_2)(e_1 - \overline{z}e_2).
\end{equation*}
Now the non-degenerate line bundle $L$ is spanned locally by a section $[0,A,\overline{A}k,0]^t$ where $A = \pm k^{-1}e^{i\theta/2}$. In terms of $X$ and $Y$ this is $\sqrt{3}AX^2Y + \sqrt{3}\overline{A}kXY^2 = \sqrt{3/2}(AX + \overline{A}kY)(\sqrt{2}XY)$. The line bundle $L_1$ is therefore spanned by $\sqrt{3/2}(AX + \overline{A}kY)$. If we write this as $a_1e_1 + a_2e_2$ then we solve for $a_1,a_2$. Since the line bundle is real we need only equate $Y$ coefficients:
\begin{equation*}
\pm \sqrt{3/2}e^{i\theta/2} = (a_1z + a_2)e^{i\pi/4}.
\end{equation*}
Thus
\begin{equation}
\begin{aligned}
\pm a_1 &= \sqrt{3/2}\left( \dfrac{1}{y} \right) {\rm sin}(\theta/2 - \pi/4) \\
\pm a_2 &= \sqrt{3/2}{\rm cos}(\theta/2 - \pi/4) - \sqrt{3/2} \left( \dfrac{x}{y} \right) {\rm sin}(\theta/2 -\pi/4).
\end{aligned}
\end{equation}
Now we can express the non degenerate line bundle $L$ as the span of
\begin{equation}
\phi = \sqrt{3/2}\frac{1}{2y}(e_1 - z e_2)(e_1 - \overline{z}e_2)(a_1e_1 + a_2e_2).
\end{equation}
Now let us compare this with the convex foliated projective structure. This has developing map $\psi : U\mathbb{H}^2 \to \mathbb{RP}^3 = \mathbb{P}(S^3(\mathbb{R}^2))$ given by
\begin{equation*}
\psi(z , u) = (zZ + 1)(\overline{z}Z+1)(aZ+1)
\end{equation*}
where $u$ is a unit tangent to $z \in \mathbb{H}^2$ and $a \in \mathbb{RP}^1 = \partial \mathbb{H}^2$ is the point of the boundary which the geodesic through $z$ in the direction $u$ approaches.\\

As before let $\theta$ denote the angle of clockwise rotation from $\partial_x$ to the tangent direction $u$. Then we may express $a$ in terms of $z$ and $\theta$ as follows. We may suppose $u = {\rm cos}(\theta)\partial_x + {\rm sin}(\theta)\partial_y$. Let the geodesic $\gamma$ through $z$ in the direction $u$ be given by the equation
\begin{equation*}
(\tilde{x}-c)^2 + \tilde{y}^2 = r^2
\end{equation*}
for constants $c,r$. Then by the definition of $\gamma$ and $a$ we have
\begin{eqnarray*}
(x-c)^2 + y^2 &=& r^2 \\
(a-c)^2 &=& r^2 \\
(x-c){\rm cos}(\theta) + y \, {\rm sin}(\theta) &=& 0.
\end{eqnarray*}
Solving we find that
\begin{equation*}
a = x + y \left( \frac{1 + {\rm sin}(\theta)}{{\rm cos}(\theta)} \right).
\end{equation*}
Now let us compare developing maps $\phi$ and $\psi$. We may relate the basis $e_1,e_2$ to $Z,1$ by the projective transformation $e_1/e_2 = -1/Z$. Then in terms of $Z$ $\phi$ becomes
\begin{equation*}
\phi = \sqrt{3/2}\frac{1}{2y} (zZ+1)(\overline{z}Z + 1)(a_2Z - a_1).
\end{equation*}
To check that $\phi$ and $\psi$ are the same map it remains only to show that $(a_2Z - a_1)$ is a multiple of $(aZ+1)$ (continuity takes care of the case $\theta = \pi/2$ where $a = \infty \in \mathbb{RP}^1$). A simple calculation shows $aa_1 + a_2 = 0$ and this proves the two linear factors are multiples of one another, hence the developing maps are equal.\\

Now our construction takes a conformal structure on $\Sigma$ and a holomorphic quartic differential and produces a projective structure on $U\Sigma$. In the case of a uniformising representation we see that the projective structure agrees with the corresponding convex-foliated projective structure of Guichard and Wienhard. But the convex-foliated projective structures forms a component of the space of projective structures on $U\Sigma$ \cite{gui}. By continuity our construction maps into the same component so we have constructed the same projective structures up to homeomorphism. Moreover the monodromy map on this component is a homeomorphism with the Hitchin component. It follows that our construction fills out the entire space of convex-foliated projective structures with contact structure. If Conjecture \ref{conj1} holds then in fact our construction is a bijection.

\subsection{Relation to line congruences and minimal surfaces}\label{rtlcams}
So far we have provided an alternative description for the convex-foliated projective structures but we have not provided any characterisation of our projective structures. We will relate the projective structures to line congruences and minimal surfaces.\\

We have constructed a projective structure on the unit tangent bundle $U\Sigma$ with the properties that the fibres are lines and the holonomy around the fibres is trivial. If we let $\tilde{\Sigma}$ be the universal cover of $\Sigma$ and $U\tilde{\Sigma}$ the unit tangent bundle of $\tilde{\Sigma}$ then the developing map descends to a map ${\rm dev} : U\tilde{\Sigma} \to \mathbb{RP}^3$ with the property that the fibres of $U\tilde{\Sigma} \to \tilde{\Sigma}$ are mapped bijectively to lines in $\mathbb{RP}^3$. By means of the Pl\"ucker embedding we have a map
\begin{equation}
\phi : \tilde{\Sigma} \to \mathbb{P}(\wedge^2 \mathbb{R}^4).
\end{equation}
Moreover if the projective structure has holonomy $\theta$ then $\phi$ is $\theta$-equivariant. We call such a map a $\theta$-equivariant {\em line congruence}. It is a $2$-parameter family of lines in $\mathbb{RP}^3$.\\

Let us recall the Klein quadric. The space of lines in $\mathbb{RP}^3$ is a projective variety embedded in $\mathbb{RP}^5 = \mathbb{P}(\wedge^2 \mathbb{R}^4)$ by the Pl\"ucker embedding. This sends the projective line spanned by two vectors $v,w$ to the point in projective space corresponding to $v \wedge w$.

A $2$-vector $\mu \in \wedge^2 \mathbb{R}^4$ is the wedge product of two vectors if and only if $\mu \wedge \mu = 0$. Therefore the Grassmannian of lines in $\mathbb{RP}^3$ is the variety $Q = \{ \mu \in \mathbb{P}(\wedge^2 \mathbb{R}^4) \; | \; \mu \wedge \mu = 0 \}$. Moreover the wedge product $\wedge^2 \mathbb{R}^4 \otimes \wedge^2 \mathbb{R}^4 \to \wedge^4 \mathbb{R}^4$ defines a signature $(3,3)$-inner product (understood as representations of ${\rm SL}(4,\mathbb{R})$) for which $Q$ is the projectivised null quadric, usually called the {\em Klein quadric}. From this point of view it follows that $Q$ has a natural signature $(2,2)$-conformal structure.

We now define what is meant by a line congruence \cite{sasa}:
\begin{defn}
Let $\Sigma$ be a surface. A {\em line congruence} on $\Sigma$ is a map $\phi : \Sigma \to Q$ from the surface into the Klein quadric. Suppose further that $\tilde{\Sigma}$ is the universal cover of $\Sigma$. Let $\rho : \pi_1(\Sigma) \to {\rm SL}(4,\mathbb{R})$ be a representation of the fundamental group of $\Sigma$ into ${\rm SL}(4,\mathbb{R})$. A {\em $\rho$-equivariant line congruence} on $\Sigma$ is a map $\phi : \tilde{\Sigma} \to Q$ such that $\phi \circ \gamma = \rho(\gamma)^{-1}\phi$ for all $\gamma \in \pi_1(\Sigma)$.
\end{defn}

Consider now a flat projective structure on a circle bundle $B$ over $\Sigma$. Let us assume that the fibres of $B$ are lines in the sense that development once around a fibre maps bijectively onto a line in $\mathbb{RP}^3$. Notice that the holonomy around the loop must preserve every point of the corresponding line. Since this is true for all fibres and since the development map is a local diffeomorphism, it follows that the holonomy around a loop must in fact fix each point in an open subset of $\mathbb{RP}^3$ and hence is trivial.

The holonomy of such a projective structure must then descend to a representation $\rho : \pi_1(\Sigma) \to {\rm SL}(4,\mathbb{R})$. If $\tilde{B}$ is the universal cover of $B$ then the developing map $\psi : \tilde{B} \to \mathbb{RP}^3$ descends to a map $\psi : \overline{B} \to \mathbb{RP}^3$ where $\overline{B}$ is a circle bundle over $\tilde{\Sigma}$ (in fact the pull-back of $B$ to a bundle over $\tilde{\Sigma}$). Since each fibre of $\overline{B} \to \tilde{\Sigma}$ maps bijectively to a line in $\mathbb{RP}^3$ we obtain a $\rho$-equivariant line congruence on $\Sigma$ where $\rho$ is the holonomy of the projective structure.\\

Given a $\rho$-equivariant line congruence on $\Sigma$ we will produce a circle bundle $B$ over $\Sigma$ and a $\rho$-equivariant map $\psi : \overline{B} \to \mathbb{RP}^3$. If the map $\psi$ is a local diffeomorphism then it defines a developing map for a projective structure on $B$ where the fibres develop into lines.

The Klein quadric $Q$ has a natural $\mathbb{RP}^1$-bundle over it $E \to Q$. In terms of the geometry of $Q$ this is one of the two bundles of maximal isotropic subspaces (the $\alpha$-planes). Viewing $Q$ as the Grassmannian of lines in $\mathbb{RP}^3$ the bundle $E$ is the tautological bundle. Evidently there is a natural map $e : E \to \mathbb{RP}^3$. Define $\overline{B} = \phi^*(E)$, the pull-back of $E$ to a circle bundle over $\overline{\Sigma}$. There is a natural map $\psi : \overline{B} \to \mathbb{RP}^3$ obtained by composing the natural maps $\phi^*(E) \to E$ and $e : E \to \mathbb{RP}^3$.\\

Now we find a natural action of $\pi_1(\Sigma)$ on $\overline{B}$ lifting the action on $\tilde{\Sigma}$. Let $a \in \phi^*(E)_u$ where $u \in \tilde{\Sigma}$, we have that $\psi(a) \in \phi(u)$ thought of as a line in $\mathbb{RP}^3$. Given $\gamma \in \pi_1(\Sigma)$ we have $\rho(\gamma)^{-1} \psi(a) \in \rho(\gamma)^{-1} \phi(u) = \phi(u \gamma)$. Now the line $\phi(u\gamma)$ is the image under $\psi$ of the fibre of $\overline{B}$ over $u\gamma$ so there is a unique point $b \in \overline{B}_{u\gamma}$ such that $\psi(b) = \rho(\gamma)^{-1}\psi(a)$. Hence define $a\gamma = b$. This gives the desired action which is clearly smooth. Now $\overline{B}$ is the pull-back to $\tilde{\Sigma}$ of the bundle $B = \overline{B}/\pi_1(\Sigma)$ over $B$. Now if $\psi$ is a local diffeomorphism it follows that $\psi$ is the developing map of a projective structure on $B$ such that development maps a fibre bijectively to a line.\\

We now determine the condition under which $\psi : \overline{B} \to \mathbb{RP}^3$ is a local diffeomorphism. Note that since the Klein quadric $Q$ has a $(2,2)$-conformal structure we can use $\phi : \tilde{\Sigma}$ to pull this back to a (possibly degenerate) conformal structure on $\tilde{\Sigma}$, moreover ${\rm SL}(4,\mathbb{R})$ acts on $Q$ by conformal transformations so the pull-back structure descends to $\Sigma$.

\begin{lemp}
The map $\psi$ is a local diffeomorphism if and only if the pull-back conformal structure is definite (positive or negative definite).
\begin{proof}
Locally we may write $\phi = v \wedge w$ where $v,w$ are vector valued functions on a coordinate chart with complex coordinate function $z$. Then locally $E$ has coordinates $(z,u)$ where $u$ is real and $\psi$ is given by $\psi(z,u) = v(z) + uw(z)$. Now $\psi$ is a local diffeomorphism if and only if $\psi \wedge \psi_u \wedge \psi_x \wedge \psi_y$ is non-vanishing. Clearly $\psi \wedge \psi_u = v \wedge w = \phi$. Suppose at some point $\psi \wedge \psi_u \wedge \psi_x \wedge \psi_y = 0$. Then
\begin{equation*}
A(v + uw)_x + B(v + uw)_y = 0 \; (\; {\rm mod} \; v,w ).
\end{equation*}
That is
\begin{equation*}
v_X + uw_X = 0 \; ( \; {\rm mod} \; v,w )
\end{equation*}
where $X = A\partial_x + B\partial_y$. Now
\begin{equation*}
\phi_X = v_X \wedge w + v \wedge w_X
\end{equation*}
so
\begin{equation*}
\phi_X \wedge \phi_X = 2v_X \wedge w \wedge v \wedge w_X = 0.
\end{equation*}
Thus the tangent vector $X$ is null in the induced conformal structure. So if $\psi$ is a local diffeomorphism then the induced conformal structure is definite. Reversing the argument proves the converse.
\end{proof}
\end{lemp}

If we start with a projective structure on a circle bundle $B$ over $\Sigma$ where development of each fibre is a bijection to a line in $\mathbb{RP}^3$, then it is clear that locally there exist coordinates $(z,u) \in \mathbb{C} \times \mathbb{RP}^1$ on $B$ such that the development map has the form $\psi(z,u) = v(z) + uw(z)$. Therefore the induced line congruence induces a definite conformal structure on $\Sigma$. We say that two projective structures on a circle bundle $B \to \Sigma$ are gauge equivalent if they are related by a diffeomorphism $\phi : B \to B$ lifting the identity $\Sigma \to \Sigma$. We have thus shown the following:
\begin{prop}
Let $\theta : \pi_1(\Sigma) \to {\rm SL}(4,\mathbb{R})$ be a representation. There is a bijection between $\theta$-equivariant line congruences inducing a definite conformal structure on $\Sigma$ and gauge equivalence classes of flat projective structures on circle bundles $B \to \Sigma$ with monodromy $\theta$ such that each fibre develops bijectively into a line of $\mathbb{RP}^3$.
\end{prop}

Now let us restrict to line congruences with symplectic monodromy. Then there is a preserved symplectic form $\omega$ on on $\mathbb{R}^4$ and a corresponding symplectic form on $(\mathbb{R}^4)^*$, that is an element $\omega^* \in \wedge^2 \mathbb{R}^4$. Suppose the line congruence $\phi$ induces a definite conformal structure so that there is a projective structure on the circle bundle $B = \phi^*(E)$. The symplectic form $\omega$ induces a contact structure on $B$ such that the contact distribution is transverse to the fibres if and only of $\phi \wedge \omega^*$ is non-vanishing at all points on $\tilde{\Sigma}$. In particular this is the case for the projective structures we have constructed.

Let us assume this transversal condition. We can then uniquely lift $\phi$ to a map $\phi : \tilde{\Sigma} \to \wedge^2 \mathbb{R}^4$ such that $\phi \wedge \omega^* = -\omega^* \wedge \omega^* = -{\rm dvol}_4$. Then consider the map $s : \tilde{\Sigma} \to \wedge^2 \mathbb{R}^4$ defined by
\begin{equation}
s = \phi + \omega^*.
\end{equation}
This map has the properties that $s \wedge s = -{\rm dvol}_4$ and $s \wedge \omega^* = 0$. Therefore $s$ maps into the $5$-dimensional orthogonal complement of $\omega^*$ which is an irreducible representation of ${\rm Sp}(4,\mathbb{R})$ corresponding to the $5$-dimensional fundamental representation of ${\rm SO}(2,3)$. Moreover $s$ maps into the quadric of vectors of length $-1$. We have the following
\begin{prop}
Let $\phi$ be a line congruence constructed from a representation $\theta$ in the Hitchin component for ${\rm PSp}(4,\mathbb{R})$. Then the associated map $s$ into the quadric of length $-1$ vectors in $(\omega^*)^\perp$ is a $\theta$-equivariant minimal immersion. Conversely a $\theta$-equivariant minimal immersion into this quadric such that the image does not lie in any proper subspace defines precisely the line congruence and projective structure we have constructed.
\begin{proof}
The rank $5$ bundle associated to the $5$-dimensional representation is $V = K^2 \oplus K \oplus 1 \oplus K^{-1} \oplus K^{-2}$ and we find that $s$ corresponds to the line bundle $1 \subset V$. We have already seen that this corresponds to a minimal immersion in the quadric.
\end{proof}
\end{prop}

\section{$G_2$ Representations}\label{g2reps}
We will show that representations into the $G_2$ Hitchin component yield almost complex curves into a $6$-dimensional quadric. For details on the split real form of $G_2$ and the split octonions see Section \ref{octonions}. Our approach is is based on similar results for almost complex curves in $S^6$, which relates to the compact form of $G_2$ \cite{bol}.\\

Consider the $7$-dimensional representation $V$ of the split real form of $G_2$. This space identifies with the imaginary split octonions and has a metric of signature $(3,4)$. Consider the quadric $Q_+ = \{ x \in V \, | \, \langle x , x \rangle = 1 \}$. Then $x \in V$ is an element of $Q$ if and only if $x^2 = -1$. The map $J_x : V \to V$ given by
\begin{equation}
J_x(y) = x \times y
\end{equation}
where $\times$ denotes the split octonion cross product defines a complex structure on $T_xQ$. Letting $x$ vary over all elements of $Q$ this defines an almost complex structure $J$ on $Q$. Consider now a Riemann surface $\Sigma$ and a map $\phi : \Sigma \to Q$. The map $\phi$ is called an {\em almost complex curve} if it intertwines the almost complex structures of $\Sigma$ and $Q$, that is if
\begin{equation}
\phi \times \partial \phi = i \partial \phi.
\end{equation}
We see that for such a map $\langle \partial \phi , \partial \phi \rangle = 0$. Assuming $\phi$ is an immersion then we have that $\phi$ is conformal in the sense that if $g$ is the metric on $Q$ then either $\phi^*(g)$ or $-\phi^*(g)$ is a metric representing the conformal class on $\Sigma$. In the split $G_2$ case both possibilities can occur. In either case the cone $\{ t\phi(x) \, | t > 0 \}$ locally defines an associative submanifold of $V$ which can have signature $(3,0)$ or $(1,2)$. In fact a cone in $V$ with non-degenerate induced metric is associative if and only if the associated surface in the quadric is an almost complex curve. It also follows that almost complex curves are minimal surfaces of $Q$.\\

We now construct almost complex curves using Higgs bundles. First we describe the necessary algebra. Let $V = \mathbb{C}^7$ be a complex $7$-dimensional vector space. We represent elements of $V$ as column vectors. Define an inner product
\begin{equation}
\langle X , Y \rangle = X^t E Y
\end{equation}
where
\begin{equation}
E = \left[ \begin{matrix} 0 & 0 & 0 & 0 & 0 & 0 & -1 \\ 0 & 0 & 0 & 0 & 0 & 1 & 0 \\ 0 & 0 & 0 & 0 & -1 & 0 & 0 \\
0 & 0 & 0 & 1 & 0 & 0 & 0 \\ 0 & 0 & -1 & 0 & 0 & 0 & 0 \\ 0 & 1 & 0 & 0 & 0 & 0 & 0\\ -1 & 0 & 0 & 0 & 0 & 0 & 0 \end{matrix} \right].
\end{equation}
Then $\langle \; , \; \rangle$ has signature $(3,4)$. Let $\hat{\rho}(X) = -\overline{X}^t$ be an anti-involution for the compact real form ${\rm SO}(7,\mathbb{R})$ indeed one may check that the endomorphisms $X$ of $V$ such that $X^tE + EX = 0$ and $\overline{X}^t + X = 0$ form a Lie algebra isomorphic to $\mathfrak{so}(7,\mathbb{R})$. Let us define a Cartan involution $\sigma$ on $\mathfrak{so}(7,\mathbb{C})$
\begin{equation}
\sigma(X) = -HX^tH
\end{equation}
where
\begin{equation}
H = \left[ \begin{matrix} 0 & 0 & 0 & 0 & 0 & 0 & 1 \\ 0 & 0 & 0 & 0 & 0 & 1 & 0 \\ 0 & 0 & 0 & 0 & 1 & 0 & 0 \\
0 & 0 & 0 & 1 & 0 & 0 & 0 \\ 0 & 0 & 1 & 0 & 0 & 0 & 0 \\ 0 & 1 & 0 & 0 & 0 & 0 & 0\\ 1 & 0 & 0 & 0 & 0 & 0 & 0 \end{matrix} \right].
\end{equation}
We then have an anti-involution $\hat{\lambda} = \hat{\rho} \sigma$ which defines the split form $\mathfrak{so}(3,4)$. Indeed there is a corresponding anti-involution also denoted $\hat{\lambda}$ which acts on $V$ such that $\hat{\lambda}(Ax) = \hat{\lambda}(A)\hat{\lambda}(x)$, namely
\begin{equation}
\hat{\lambda}(x) = H\overline{x}.
\end{equation}
On the real subspace $V^{\hat{\lambda}}$ the inner product restricts to a real inner product of signature $(3,4)$.\\

Now we introduce the $G_2$ structure of $V$. Let $e_3 , e_2 , \dots , e_{-3}$ be the standard basis on $V$, that is for any vector we have $v = \left[ v^3 , v^2 , \dots , v^{-3} \right]^t = v^3e_3 + \dots + v^{-3}e_{-3}$. We define a $G_2$ structure on $V$ by the following identification:
\begin{equation}
\begin{aligned}
e_3 &= \frac{1}{\sqrt{2}}(jl + \sqrt{-1}kl), &
e_2 &= \frac{1}{\sqrt{2}}(j + \sqrt{-1}k), \\
e_1 &= \frac{1}{\sqrt{2}}(l + \sqrt{-1}il), &
e_0 &= i.
\end{aligned}
\end{equation}
and $e_{-i}$ is obtained from $e_i$ by replacing $\sqrt{-1}$ with $-\sqrt{-1}$. Here $i,j,k,l,il,jl,kl$ is a basis for the split octonions. With this assignment the anti-involutions $\hat{\rho}$, $\hat{\lambda}$ preserve the Lie algebra $\mathfrak{g}_2 \subset \mathfrak{so}(7,\mathbb{C})$.

We may define a principal $3$-dimensional subalgebra $x,e,\tilde{e}$ as follows:
\begin{equation}
\begin{aligned}
x &= {\rm diag}(3,2,1,0,-1,-2,-3) \\
e &= \left[ \begin{matrix} 0 & \sqrt{3} & 0 & 0 & 0 & 0 & 0 \\ 0 & 0 & \sqrt{5} & 0 & 0 & 0 & 0 \\ 0 & 0 & 0 & \sqrt{-6} & 0 & 0 & 0 \\
0 & 0 & 0 & 0 & \sqrt{-6} & 0 & 0 \\ 0 & 0 & 0 & 0 & 0 & \sqrt{5} & 0 \\ 0 & 0 & 0 & 0 & 0 & 0 & \sqrt{3} \\ 0 & 0 & 0 & 0 & 0 & 0 & 0 \end{matrix} \right] \\
\tilde{e} &= -\hat{\rho}(e).
\end{aligned}
\end{equation}

Now we may carry out the Higgs bundle construction for the Hitchin component for $G_2$. Such representations are described by holomorphic differentials of degrees $2$ and $6$. As usual for any such representation we may choose a complex structure on the Riemann surface so that the quadratic differential vanishes leaving only a degree $6$ differential $q$. The holomorphic adjoint bundle can be described as the $G_2$ endomorphisms of the following rank $7$ holomorphic bundle
\begin{equation}
\mathcal{W} = K^3 \oplus K^2 \oplus \dots \oplus K^{-3}.
\end{equation}
The Higgs field $\Phi$ has the form
\begin{equation}
\Phi = \left[ \begin{matrix} 0 & 0 & 0 & 0 & 0 & q & 0 \\ \sqrt{3} & 0 & 0 & 0 & 0 & 0 & q \\ 0 & \sqrt{5} & 0 & 0 & 0 & 0 & 0 \\ 0 & 0 & -\sqrt{-6} & 0 & 0 & 0 & 0 \\
0 & 0 & 0 & -\sqrt{-6} & 0 & 0 & 0 \\ 0 & 0 & 0 & 0 & \sqrt{5} & 0 & 0 \\ 0 & 0 & 0 & 0 & 0 & \sqrt{3} & 0 \end{matrix} \right].
\end{equation}
The metric is $h = e^{2\Omega}$ where $\Omega$ has the form
\begin{equation}
\Omega = {\rm diag}(X+Y , X , Y , 0 , -Y , -X , -(X+Y)).
\end{equation}
We let $\rho = {\rm Ad}_{e^{2\Omega}} \circ \hat{\rho}$ and $\lambda = \rho \sigma$ be the compact and split anti-involutions preserved by the associated connection $\nabla_A$. Let us also define $s = e^{2X}$ and $t = e^{2Y}$ so $s$ is a $(2,2)$-form and $t$ is a $(1,1)$-form. Then
\begin{equation}
\Phi^* = -\rho(\Phi) = \left[ \begin{matrix} 0 & \sqrt{3}t & 0 & 0 & 0 & 0 & 0 \\ 0 & 0 & \sqrt{5}s/t & 0 & 0 & 0 & 0 \\ 0 & 0 & 0 & \sqrt{-6}t & 0 & 0 & 0 \\
0 & 0 & 0 & 0 & \sqrt{-6}t & 0 & 0 \\ \overline{q}/(s^2t) & 0 & 0 & 0 & 0 & 0 & \sqrt{3}t \\ 0 & \overline{q}/(s^2t) & 0 & 0 & 0 & 0 & 0 \end{matrix} \right]
\end{equation}
and the connection form for $\nabla_A$ is $-2\partial \Omega$.\\

Consider the section $s = \left[ 0 , 0 , 0 , 1 , 0 , 0 , 0 \right]^t$. We have that $s$ is real with respect to the real structure on $\mathcal{W}$ associated to $\lambda$. In terms of the $G_2$ structure we can identify $s$ with $e_0 = i$, so $s^2 = -1$. Therefore $s$ develops into a map
\begin{equation}
\phi : \tilde{\Sigma} \to Q \subset V^{\hat{\lambda}}.
\end{equation}
This map is an almost complex curve. This amounts to showing $s \times \nabla_z s = i \nabla_z s$ where $\nabla = \nabla_A + \Phi - \rho(\Phi)$ is the flat connection. This is a straightforward computation. It is also clear that $\phi$ is an immersion and that the corresponding associative cone is of signature $(1,2)$. Therefore we have constructed almost complex curves to each representation in the Hitchin component for $G_2$.



\chapter{Parabolic Cartan geometries}\label{chap4}
Cartan geometries are a curved generalisation of homogeneous spaces that emerged from the works of E. Cartan \cite{cartan2}, \cite{car3} and also from T. Y. Thomas \cite{tho1}, \cite{tho2} from a different perspective. To each homogeneous space $G/H$ the associated Cartan geometries are spaces locally modeled on the homogeneous space. Parabolic geometries are a special case where the model space $G/P$ involves a parabolic subgroup. We have already encountered one type of parabolic geometry namely projective geometry for which the homogeneous model is projective space. Another familiar parabolic geometry is conformal geometry and Chapter \ref{chap5} is concerned with a parabolic geometry associated to the split form of $G_2$. In this chapter we will provide a rapid introduction to the theory of parabolic geometries as preparation for Chapter \ref{chap5}.\\

Section \ref{cartgeom} introduces Cartan geometries in general before specialising to parabolic geometries which require some algebraic background on parabolic subalgebras. Section \ref{tract} introduces the tractor connection and tractor bundles. These allow Cartan geometries to be understood in terms of more familiar objects namely vector bundles and connections. We finish with Section \ref{confgeom} which examines the particular case of conformal geometry in more detail.\\

\section{Cartan and parabolic geometries}\label{cartgeom}
We give a general introduction into the theory of parabolic geometries. References for parabolic geometries include \cite{beg}, \cite{cap3}, \cite{cap2}, \cite{cap}. See also \cite{sharpe} for an introduction to Cartan geometries.

\subsection{Cartan geometries}\label{cartgeom2}

\begin{defn}
Let $\mathfrak{g}$ be a Lie algebra and $\mathfrak{p}$ a subalgebra. Let $P$ be a Lie group with Lie algebra $\mathfrak{p}$. A {\em Cartan geometry} over a manifold $M$ is a principal $P$-bundle $\mathcal{G}$ together with a $\mathfrak{g}$-valued differential form $\omega \in \Omega^1(\mathcal{G},\mathfrak{g})$, called a {\em Cartan connection}, such that the following properties hold:
\begin{enumerate}
\item{At each point $u \in \mathcal{G}$, $\omega_u : T_u \mathcal{G} \to \mathfrak{g}$ is an isomorphism.}
\item{If $\tilde{X}$ is the vector field on $\mathcal{G}$ associated to $X \in \mathfrak{p}$ then $\omega (\tilde{X}) = X$.}
\item{If $R_p$ denotes the right translation on $\mathcal{G}$ associated to $p \in B$ then $R_p^* \omega = {\rm Ad}_{p^{-1}} (\omega)$.}
\end{enumerate}
\end{defn}

When referring to Cartan geometries or Cartan connections we sometimes need to make the group $P$ and Lie algebra $\mathfrak{g}$ clear in which case we will speak of a $(P, \mathfrak{g})$ Cartan geometry or $(P , \mathfrak{g})$ Cartan connection. Often a group $G \supset P$ for which $\mathfrak{g}$ is the Lie algebra will be present but it is not necessary for the definition. The idea behind Cartan geometries is that they are a curved generalisation of homogeneous spaces, specifically the principal $P$-bundle $\mathcal{G} \to M$ generalises the principal $P$-bundle $G \to G/P$ over the homogeneous space. The Cartan connection identifies the tangent space at each point of $\mathcal{G}$ with the tangent space $\mathfrak{g}$ of the homogeneous bundle.\\

A homogeneous space $G/P$ is equipped with a canonical Cartan connection, the Maurer-Cartan form for $G$. Recall that the Maurer-Cartan form $\omega$ of a Lie group $G$ is the $\mathfrak{g}$-valued $1$-form defined by sending $X \in T_gG$ to $L_{g^{-1}*}(X) \in T_eG = \mathfrak{g}$. Therefore Cartan connections can be thought of as a generalisation of the Maurer-Cartan forms. The key difference is that the translational property of a Cartan connection is only assumed to hold for the subgroup $P$ of $G$.\\

One of the key concepts in Cartan geometries is that of curvature. For a Cartan geometry $\mathcal{G}\to M$ with connection $\omega$ the {\em curvature} of the Cartan geometry is the $\mathfrak{g}$-valued $2$-form $\kappa \in \Omega^2(\mathcal{G},\mathfrak{g})$ given by
\begin{equation}\label{curvature}
\kappa = d\omega + \tfrac{1}{2}[\omega , \omega ]
\end{equation}
where the term $[\omega,\omega]$ is the combination of exterior product and the Lie bracket for $\mathfrak{g}$, namely $\tfrac{1}{2}[\omega,\omega](X,Y) = [\omega(X),\omega(Y)]$. Clearly $\kappa$ measures the failure of $\omega$ to satisfy the Maurer-Cartan equation and thus how deformed the Cartan geometry is from the homogeneous space. Via $\omega$ we can also view the curvature as a $\bigwedge^2\mathfrak{g}^* \otimes \mathfrak{g}$-valued function on $\mathcal{G}$ which we denote by $K$, so $K_u(X,Y) = \kappa_u(\omega_u^{-1}(X),\omega_u^{-1}(Y))$. Then (\ref{curvature}) becomes
\begin{equation}
K(X,Y) = [X,Y] - \omega( [\omega^{-1}(X),\omega^{-1}(Y)]),
\end{equation}
that is the curvature measures the difference between the Lie bracket on $\mathfrak{g}$ and the Lie bracket of the associated vector fields on $\mathcal{G}$.

The transformation property $R^*_p \omega = {\rm Ad}_{p^{-1}}\omega$ of a Cartan connection can be differentiated to give $\mathcal{L}_{\tilde{Y}}\omega = -[Y,\omega]$ where $Y \in \mathfrak{p}$ and $\tilde{Y} = \omega^{-1}(Y)$ is the associated vertical vector field. This simplifies to $d\omega(\tilde{Y}, \; ) = -[\omega(\tilde{Y}),\omega ]$. It follows that the curvature vanishes if either argument is vertical. Thus $K$ can be thought of as a $\bigwedge^2(\mathfrak{g}/\mathfrak{p})^* \otimes \mathfrak{g}$-valued function on $\mathcal{G}$. Lastly the transformation property of $\omega$ implies the same transformation property for the curvature $R^*_p \kappa = {\rm Ad}_{p^{-1}}\kappa $. These properties allow us to view curvature as a $2$-form valued section of the adjoint bundle $\mathcal{G} \times_P \mathfrak{g}$ on $M$, similar to the usual notion of the curvature of a connection.

\subsection{Parabolic subalgebras}\label{para}
We shall consider a special case of Cartan geometries known as parabolic geometries. To understand these geometries we first have to undertake some algebraic preliminaries on parabolic subalgebras.\\

Let $\mathfrak{g}$ be a real or complex semisimple Lie algebra and $\mathfrak{p}$ a subalgebra. Recall that $\mathfrak{p}$ is called a {\em parabolic subalgebra} \cite{cap2} if $\mathfrak{p}$ contains a maximal solvable subalgebra, i.e. a Borel subalgebra. An alternative definition is given in \cite{cald2}.\\

By means of Dynkin diagrams (or Satake diagrams for real algebras), all parabolic subalgebras can be dealt with in a uniform manner. Any parabolic subalgebra $\mathfrak{p} \subset \mathfrak{g}$ of a (real or complex) semisimple Lie algebra can be described as follows \cite{cap2}: there exists a Cartan subalgebra $\mathfrak{h} \subset \mathfrak{g}$, a system of positive roots $\Delta^+$ for $(\mathfrak{h},\mathfrak{g})$ and a subset $\Sigma \subset \Delta_0$ of the corresponding simple roots. We define the $\Sigma$-height of a root $\alpha = \Sigma_{\alpha_j \in \Delta_0} n_j \alpha_j$ as the integer $\Sigma_{\alpha_j \in \Sigma} n_j$. Then $\mathfrak{p}$ is the subalgebra of $\mathfrak{g}$ spanned by all root spaces of non-negative $\Sigma$-height (this contains $\mathfrak{h}$).

Conversely given any subset $\Sigma \subset \Delta_0$ of a system of simple roots the above construction defines a parabolic subalgebra, provided that $\mathfrak{g}$ is over the complex numbers. In the real case there are restrictions on which subsets $\Sigma$ define a subalgebra and this is described by Satake diagrams. In either case we see that parabolic subalgebras can be described by taking the Dynkin diagram for $\mathfrak{g}$ and crossing out the simple roots corresponding to $\Sigma$.

This method of describing parabolic subalgebras also gives $\mathfrak{g}$ the structure of a $|k|$-graded Lie algebra, namely we have that $\mathfrak{g} = \mathfrak{g}_{-k} \oplus \cdots \oplus \mathfrak{g}_0 \oplus \cdots \oplus \mathfrak{g}_k$ where $\mathfrak{g}_j$ is the subspace of $\mathfrak{g}$ with $\Sigma$-height equal to $j$. It should be noted that this grading is not canonical as it depends on the Cartan subalgebra and choice of positive roots, though it is unique up to conjugacy.\\

We define the following subalgebras of $\mathfrak{g}$: first $\mathfrak{g}_- = \mathfrak{g}_{-k} \oplus \cdots \oplus \mathfrak{g}_{-1}$ which as a representation of $\mathfrak{g}_0$ is isomorphic to $\mathfrak{g}/\mathfrak{p}$ and second $\mathfrak{p}_+ = \mathfrak{g}_1 \oplus \cdots \oplus \mathfrak{g}_k$.

\subsection{Lie algebra cohomology}\label{cohomo}
The Lie algebra cohomology \cite{chei} of $\mathfrak{g}_-$ with coefficients in $\mathfrak{g}$ appears in the theory of parabolic geometries in two different ways. The first cohomology appears as obstructions to the prolongation procedure and the second cohomology is related to possible values the curvature can take. As such we provide a definition and state a key result. Powerful results for calculating Lie algebra cohomology are obtained by harmonic Hodge theory \cite{kos}, \cite{och}.\\

We define the space of $n$-cochains as $C^n(\mathfrak{g}_-,\mathfrak{g}) = \bigwedge^n \mathfrak{g}_-^* \otimes \mathfrak{g}$ viewed as alternating multilinear maps. The differential $\partial : C^n(\mathfrak{g}_-,\mathfrak{g}) \to C^{n+1}(\mathfrak{g}_-,\mathfrak{g})$ is given by
\begin{eqnarray}\label{del}
&&(\partial \phi)(X_0,X_1, \dots , X_n) = \sum_{i=0}^n (-1)^i[X_i,\phi(X_0 , \dots , \hat{X_i} , \dots , X_n)] \nonumber \\
& & { } + \sum_{i<j} (-1)^{i+j}\phi([X_i,X_j],X_0 , \dots , \hat{X_i} , \dots , \hat{X_j} , \dots , X_n).
\end{eqnarray}
The differential satisfies $\partial^2 = 0$ and the cohomology groups are denoted $H^n(\mathfrak{g}_-,\mathfrak{g})$. The spaces $C^n(\mathfrak{g}_-,\mathfrak{g})$ can be decomposed into homogeneous components where $\phi \in C^n(\mathfrak{g}_-,\mathfrak{g})$ has homogeneity $l$ if for $X_j \in \mathfrak{g}_{i_j}$, $\phi(X_1, \dots , X_n ) \in \mathfrak{g}_{i_1 + \dots + i_n + l}$. We denote by $C^n_l(\mathfrak{g}_-,\mathfrak{g})$ the homogeneity $l$ subspace of $C^n(\mathfrak{g}_-,\mathfrak{g})$. The differential $\partial$ preserves homogeneity so it follows that the cohomology groups likewise have a decomposition $H^n(\mathfrak{g}_-,\mathfrak{g}) = \oplus_l H^n_l(\mathfrak{g}_-,\mathfrak{g})$ into homogeneous subspaces. The Lie algebra $\mathfrak{g}_0$ has a natural action on the cochains and the differential commutes with this action. It follows that the cohomology groups $H^n_l(\mathfrak{g}_-,\mathfrak{g})$ have naturally defined $\mathfrak{g}_0$-module structures.\\

Via the Killing form, the subalgebras $\mathfrak{g}_-$ and $\mathfrak{p}_+$ are dual. It follows that the negative of the dual of the differential $\partial : C^n(\mathfrak{p}_+,V) \to C^{n+1}(\mathfrak{p}_+,V)$, (defined similarly to (\ref{del})) where $V$ is any $\mathfrak{g}$-module, is a linear map $\partial^* : C^{n+1}(\mathfrak{g}_-,V^*) \to C^n(\mathfrak{g}_-,V^*)$. In particular we have a map $\partial^* : C^{n+1}(\mathfrak{g}_-,\mathfrak{g}) \to C^n(\mathfrak{g}_-,\mathfrak{g})$ satisfying $(\partial^*)^2 = 0$. The map $\partial^*$ is called the {\em codifferential}. In this context it is actually more natural to work with the codifferential and the resulting Lie algebra homology since there is a natural action of $\mathfrak{p}$, hence when $\mathfrak{g}$ is replaced with an adjoint bundle with structure group $P$, the codifferential is naturally defined as a bundle map.\\

Let $B$ denote the Killing form on $\mathfrak{g}$. There exists an involution $\rho : \mathfrak{g} \to \mathfrak{g}$, linear in the real case or antilinear in the complex case for which $B^*(X,Y) = -B(X,\rho(Y))$ defines a positive definite inner product in the real case or a positive definite Hermitian inner product in the complex case \cite{cap2}. Moreover $\rho$ can be chosen so that $\rho(\mathfrak{g}_i) = \mathfrak{g}_{-i}$. The inner product $B^*$ on $\mathfrak{g}$ induces inner products on the $C^n(\mathfrak{g}_-,\mathfrak{g})$, which will also be denoted $B^*$. The key properties of the codifferential are given by the following:
\begin{prop}\cite{cap2}
The differential $\partial$ and codifferential $\partial^*$ are adjoint with respect to $B^*$, i.e., $B^*(\partial \phi , \psi) = B^*(\phi , \partial^* \psi)$. Each space $C^n_l(\mathfrak{g}_-,\mathfrak{g})$ splits as a direct sum of the image of $\partial $ and the kernel of $\partial^*$. Each cohomology class contains a unique $\partial^*$-closed representative (i.e., $\partial$-closed and $\partial^*$-closed).
\end{prop}

An element of $C^n(\mathfrak{g}_-,\mathfrak{g})$ that is $\partial$-closed and $\partial^*$-closed is called {\em harmonic}. Thus the harmonic elements define unique representatives for each cohomology class.

\subsection{Parabolic geometries}\label{parageom}
In this section following \cite{cap2} we take the algebraic structure of parabolic subalgebras and translate them into a geometric picture via Cartan geometries.\\

Let $\mathfrak{g}$ be a real or complex semisimple Lie algebra and $\mathfrak{p}$ a parabolic subalgebra. As in Section \ref{para}, give $\mathfrak{g}$ the structure of a $|k|$-graded algebra.

For the purpose of giving a uniform treatment it will be useful to assume we have a given Lie group $G$ with Lie algebra $\mathfrak{g}$. Let $P \subset G$ be the subgroup of elements $g$ such that ${\rm Ad}_g$ preserves the filtration $\mathfrak{g}_{(-k)} \supset \dots \supset \mathfrak{g}_{(k)}$ where $\mathfrak{g}_{(i)} = \mathfrak{g}_i \oplus \mathfrak{g}_{i+1} \oplus \cdots \oplus \mathfrak{g}_k$, $i= -k, \dots , k$. Further define $G_0 \subset G$ as the subgroup of elements $g$ such that ${\rm Ad}_g$ preserves the gradation of $\mathfrak{g}$. Then $P$ has Lie algebra $\mathfrak{p}$ and $G_0$ has Lie algebra $\mathfrak{g}_0$ \cite{cap2}. The structure of $P$ is given by the following:

\begin{prop}\cite{cap2}\label{groupform}
Given $g \in P$ there exists unique elements $g_0 \in G_0$ and $X_i \in \mathfrak{g}_i$ for $i=1,2, \dots k$ such that $g = g_0 {\rm exp}(X_1) \cdots {\rm exp}(X_k)$.
\end{prop}

Now define the subgroup $P_+ \subset P$ as the image of $\mathfrak{p}_+$ under the exponential map. One finds that $P$ is the semidirect product of $G_0$ and $P_+$.\\

Having established the various groups involved we now proceed to the geometry itself.

\begin{defn}
Let $(P,\mathfrak{g})$ be as above. A $(P,\mathfrak{g})$ Cartan geometry is called a {\em parabolic geometry}.
Let $\mathcal{G} \to M$ be a parabolic geometry with Cartan connection $\omega$ and curvature $K \in \Omega^0(\mathcal{G},\bigwedge^2 \mathfrak{g}^*_- \otimes \mathfrak{g})$. The Cartan connection $\omega$ is called {\em normal} if $\partial^* K = 0$ where $\partial^*$ is the codifferential $\partial^* : C^2(\mathfrak{g}_-,\mathfrak{g}) \to C^1(\mathfrak{g}_-,\mathfrak{g})$.
\end{defn}

Suppose we have a parabolic geometry $\pi :\mathcal{G}\to M$ with Cartan connection $\omega$. Given any $x \in M$ and $u \in \pi^{-1}(x)$ the exact sequence
\begin{equation}\label{exact}
0 \longrightarrow \mathfrak{p} \longrightarrow \mathfrak{g} \buildrel\pi_*\omega_u^{-1}\over\longrightarrow T_xM \longrightarrow 0
\end{equation}
induces an isomorphism $\phi_u :\mathfrak{g}_- = \mathfrak{g}/\mathfrak{p} \to T_xM$ of vector spaces and the transformation property of $\omega$ implies $\phi_{up} \circ {\rm Ad}_{p^{-1}} = \phi_u$. Therefore we have an isomorphism
\begin{equation*}
TM \simeq \mathcal{G} \times_{\rm Ad} \mathfrak{g} / \mathfrak{p}.
\end{equation*}
Now when $P$ acts on $\mathfrak{g}$ it preserves the filtration $\mathfrak{g}_{(-k)} \supset \dots \supset \mathfrak{g}_{(k)}$ and thus it also preserves the induced filtration $\mathfrak{g}_{(-k)}/\mathfrak{p} \supset \cdots \supset \mathfrak{g}_{(-1)}/\mathfrak{p} \supset \mathfrak{g}_{(0)}/\mathfrak{p} = \{0\}$ on $\mathfrak{g}_-$. Hence there is a corresponding filtration
\begin{equation*}
TM = T^{-k}M \supset \cdots \supset T^{-1}M \supset T^0M = 0
\end{equation*}
of the tangent bundle. Now it is clear that we also have isomorphisms
\begin{equation*}
T^jM / T^{j+1}M \simeq \mathcal{G} \times_{\rm Ad} \mathfrak{g}_{(j)}/\mathfrak{g}_{(j+1)}
\end{equation*}
for $j = -k , \dots , -1$. Therefore we can give the associated graded space $T_x^{-k}M/T_x^{-k+1}M \oplus \cdots \oplus T_x^{-1}M/\{0\}$ the structure of the Lie algebra $\mathfrak{g}_-$. Now the action of $P$ on $\mathfrak{g}_{(-k)}/\mathfrak{g}_{(-k+1)} \oplus \cdots \oplus \mathfrak{g}_{(-1)}/\mathfrak{g}_{(0)}$ factors through to $G_0$ which acts by automorphisms of the Lie algebra structure. Thus the bundle
\begin{equation}
{\rm Gr}TM = T^{-k}M/T^{-k+1}M \oplus \cdots \oplus T^{-1}M
\end{equation}
has on each fibre the structure of the Lie algebra $\mathfrak{g}_-$. We call ${\rm Gr}TM$ the {\em associated graded bundle}. The associated graded bundle has a reduction of structure group to $G_0$. Note that while $G_0$ preserves the algebraic structure on $\mathfrak{g}_-$ it can happen that $G_0$ is a proper subgroup of ${\rm Aut}(\mathfrak{g}_0)$. In this sense the associated graded bundle may have more structure than just the algebraic bracket operation.\\

Given a manifold $M$ with a filtration $TM = T^{-k}M \supset \cdots \supset T^{-1}M$ of the tangent bundle we can likewise form the associated graded bundle ${\rm Gr}TM$. Suppose the filtration has the property that if $X$ and $Y$ are respectively sections of $T^rM$ and $T^sM$, then their commutator $[X,Y]$ is a section of $T^{r+s}M$. We denote this property by $[T^rM,T^sM] \subset T^{r+s}M$. The commutator bracket then induces a bracket $ [\; , \; ]: T^{r}M/T^{r+1}M \otimes T^{s}M/T^{s+1}M \to T^{r+s}M/T^{r+s+1}M$. Note that only zero order terms survive the quotient so that this is an algebraic Lie bracket operation defined on the fibres of ${\rm Gr}TM$. We say that $M$ satisfies the {\em structure equations} for the pair $(P,\mathfrak{g})$ if
\begin{enumerate}
\item{The associated graded bundle ${\rm Gr}TM$ has a reduction of structure to $G_0$, hence in particular each fibre has the algebraic structure of $\mathfrak{g}_-$.}
\item{The Lie algebra structure on each fibre induced from vector field commutators matches the Lie algebra structure obtained from the reduction of structure to $G_0$.}
\end{enumerate}

\begin{rem}
Note that in the $|1|$-graded case the requirement that $M$ satisfies the structure equations reduces to the requirement that $TM$ has a reduction of structure to $G_0$.
\end{rem}

Suppose we have a $(P,\mathfrak{g})$ parabolic geometry $\mathcal{G} \to M$ inducing a filtration on $TM$ and an associated graded bundle ${\rm Gr}TM$. We will determine the conditions under which the structure equations hold.

Let $\kappa = \sum_{i = -k + 2}^{3k} \kappa^i$ be the decomposition of the curvature into homogeneous components so that if $X \in \mathfrak{g}_r$ and $Y \in \mathfrak{g}_s$ then $\kappa(X,Y) \in \mathfrak{g}_{r+s+i}$. Recall that the curvature measures the difference between Lie brackets in $\mathfrak{g}$ and their associated vector fields. It follows that a parabolic geometry satisfies the structure equations if and only if $\kappa^i = 0$ for $i \le 0$. We call a Cartan connection {\em regular} if it satisfies this homogeneity condition on the curvature.\\

A regular parabolic geometry yields a filtration on the tangent bundle solving the structure equations. Naturally one is interested in the converse question of whether a solution of the structure equations corresponds to a regular parabolic geometry. This is closer to the original notion of Cartan geometries and Cartan's equivalence method where structures on the base manifold define a canonical Cartan connection. The following proposition will confirm this point of view. Note that similar results but expressed differently can be found in earlier works \cite{tan}, \cite{tan2}, \cite{mori}:

\begin{prop}\label{cartanconn}\cite[Corollary 3.23]{cap2}
Let $M$ be a manifold that satisfies the structure equations for $(P,\mathfrak{g})$. Assume that $(\mathfrak{p},\mathfrak{g})$ is such that all the cohomology groups ${\rm H}^1_l(\mathfrak{g}_-,\mathfrak{g})$ with $l>0$ are trivial. Then there is a bijective correspondence between reductions of the associated graded bundle to $G_0$ and isomorphism classes of regular, normal $(P,\mathfrak{g})$ Cartan connections which recover the structure on ${\rm Gr}TM$.
\end{prop}
Under the hypotheses of this proposition there will generally exist many Cartan connections that yield the given structure on ${\rm Gr}TM$ but there exists a unique such Cartan connection that is normal.\\

\begin{rem}
The parabolic subalgebras $\mathfrak{p} \subset \mathfrak{g}$ such that not all the cohomology groups ${\rm H}^1_l(\mathfrak{g}_-,\mathfrak{g})$ with $l>0$ are trivial are known \cite{cap2}. These cases can be interpreted to mean that aside from the structure equations, additional structure has to be given to define the parabolic geometry. For the parabolic geometries we shall be henceforth considering this issue will not come into play, although it does matter for projective geometries. In the case of projective geometries the structure equations are vacuous and so additional structure is required to describe such geometries (namely the equivalence class of affine connections).
\end{rem}

\section{Tractor description of parabolic geometries}\label{tract}
Now that we have introduced parabolic geometries, we introduce the tractor bundle framework for describing such geometries in the more familiar terms of vector bundles and connections (in the usual sense). The formalism dates back to the work of T. Y. Thomas and was revived in modern form in \cite{beg}, see also \cite{cap3}.\\


Let $\mathcal{G}\to M$ be a $(P,\mathfrak{g})$ Cartan geometry with Cartan connection $\omega$. Let $G \supset P$ be a Lie group with Lie algebra $\mathfrak{g}$. Then the bundle $\tilde{\mathcal{G}} = \mathcal{G} \times_P G$ has an ordinary $\mathfrak{g}$-valued connection $\tilde{\omega}$ such that $i^*\tilde{\omega} = \omega$ where $i$ is the inclusion $\mathcal{G} \to \tilde{\mathcal{G}}$. To see this note that $\tilde{\omega}$ is defined on $i_*(T\mathcal{G})$ to agree with $\omega$ and then $\tilde{\omega}$ is automatically defined on all other tangent vectors by means of the defining properties of a connection. This connection is called the {\em tractor connection}.\\

Let $V$ be any representation of $G$. Then there is an associated bundle
\begin{equation*}
\mathcal{V} = \mathcal{G} \times_P V = \tilde{\mathcal{G}} \times_G V.
\end{equation*}
Such bundles are useful because the tractor connection can then be considered as a connection $\overrightarrow{\nabla}$ on $\mathcal{V}$. We call $\mathcal{V}$ a {\em tractor bundle}. Typically a parabolic subgroup $P \subset G$ is the subgroup of $G$ preserving some partial flag in $V$ and this induces a flag structure on $\mathcal{V}$ which defines a reduction of structure from $\tilde{\mathcal{G}}$ to $\mathcal{G}$. More generally if a tractor bundle $\mathcal{V}$ can be given structure so as to define a reduction of structure from $\tilde{\mathcal{G}}$ to $\mathcal{G}$ then we can recover the Cartan geometry from $\mathcal{V}$ and $\overrightarrow{\nabla}$.\\

In the case of parabolic geometries it is possible to relate the tractor connection to certain affine connections which we now introduce. A {\em Weyl structure} is a reduction of structure $\sigma : \mathcal{G}_0 \to \mathcal{G}$ to a principal $G_0$-bundle. Such sections always exist globally \cite{slovak}, since $P/G_0 \simeq P_+$ is contractible.

A Weyl structure provides a splitting of the filtration on the tangent bundle and hence an isomorphism $TM \simeq {\rm Gr}TM$. Furthermore the $G_0$-invariant grading of $\mathfrak{g}$ allows us to decompose $\sigma^*\omega = \omega_{-k} + \dots + \omega_k$ where $\omega_i$ is $\mathfrak{g}_i$-valued. One sees that the $\mathfrak{g}_0$ component $\sigma^*(\omega_0)$ is an ordinary $\mathfrak{g}_0$-valued connection $\nabla^\sigma$ on $\mathcal{G}_0$. This connection can then be viewed as a connection on any vector bundle associated to a representation of $G_0$. In particular $\nabla^\sigma$ can be regarded as a connection on $TM$. We call $\nabla^\sigma$ a {\em Weyl connection} or {\em preferred connection}.\\

Given a Weyl structure $\sigma : \mathcal{G}_0 \to \mathcal{G}$ let us write $\omega_- = \omega_{-k} + \dots + \omega_{-1}$ and $\omega_+ = \omega_1 + \dots + \omega_k$ so that $\sigma^*\omega = \omega_- + \omega_0 + \omega_+$. Now since $\omega_-$ and $\omega_+$ vanish on the vertical distribution for $\mathcal{G}_0$ they can be viewed as $1$-form valued sections of the adjoint bundle $\mathcal{A} = \mathcal{G} \times_{\rm Ad} \mathfrak{g} \simeq \mathcal{A}_- \oplus \mathcal{A}_0 \oplus \mathcal{A}_+$. The relation between the tractor connection $\overrightarrow{\nabla}$ and the Weyl connection $\nabla^\sigma$ can then be written as
\begin{equation*}
\overrightarrow{\nabla} = \omega_- + \nabla^\sigma + \omega_+.
\end{equation*}
We note further that $\omega_- : TM \to \mathcal{A}_-$ provides an isomorphism $TM \simeq \mathcal{A}_-$. Now since $\mathcal{A}_0$ acts faithfully on $\mathcal{A}_-$ we can identify $\mathcal{A}_0$ with a subbundle of ${\rm End}(TM)$. Via the Killing form $\mathcal{A}_+$ and $\mathcal{A}_-$ are dual so $\mathcal{A}_+ \simeq T^*M$. The adjoint bundle now has the form
\begin{equation*}
\mathcal{A} = TM \oplus \mathcal{A}_0 \oplus T^*M.
\end{equation*}
So for any tangent vector $X$ we can identify $\omega_-(X)$ with $X$ and $\omega_+(X)$ can be identified with a $1$-form which we denote $\mathbf{P}^\sigma(X)$. Under these identifications we have
\begin{equation}\label{weyltra}
\overrightarrow{\nabla}_X = X + \nabla_X^\sigma + \mathbf{P}^\sigma(X).
\end{equation}
This is the fundamental relation between the tractor connection and the Weyl connections. The map $\mathbf{P}^\sigma : TM \to T^*M$ is called the {\em rho tensor} and depends on the choice of $\sigma$.

\section{Conformal geometry}\label{confgeom}
We now consider in more detail a specific instance of a parabolic geometry, namely conformal geometry. Although conformal geometry can be approached as a topic within Riemannian geometry, it is beneficial to take a more invariant approach, the history of which traces back to the work of Cartan in \cite{cartan2}. We now understand this work as an example of a parabolic Cartan geometry, see for example \cite{cap3}, \cite{kobayashi}.

\subsection{M\"obius space}\label{mobius}
Our first step is to understand the homogeneous model space for conformal geometry. Let $V$ be a vector space of dimension $n$ and $[g]$ a conformal class of $(p,q)$ signature inner products (henceforth referred to as metrics) on $V$. Throughout this section we will assume $n \ge 3$. We define the {\em conformal group} ${\rm CO}(V) = {\rm O}(V) \times \mathbb{R}^*$ and index $2$ subgroup ${\rm CO}_+(V) = \{ g \in {\rm CO}(V) \, | \, \det{g} > 0 \}$.\\

Our homogeneous model for conformal geometry consists of Lie algebras $\mathfrak{p} \subset \mathfrak{g}$ where $\mathfrak{g} = \mathfrak{g}_{-1} + \mathfrak{g}_0 + \mathfrak{g}_1 \simeq V + \mathfrak{co}(V) + V^*$ and $\mathfrak{p} = \mathfrak{g}_0 + \mathfrak{g}_1$. Thus $\mathfrak{g}$ is a $|1|$-graded Lie algebra given by the relations
\begin{equation}
\begin{aligned}
\lbrack v, \hat{v}\rbrack &= \lbrack \lambda,\hat{\lambda} \rbrack  = 0\\
\lbrack A,\hat{A} \rbrack &= A\hat{A}- \hat{A}A \\
\lbrack A, v \rbrack &= Av \\
\lbrack \lambda , A \rbrack &= A^*\lambda \\
\lbrack v , \lambda \rbrack &= \lambda \otimes v - (gv) \otimes (g^{-1}\lambda) + \lambda (v)I
\end{aligned}
\end{equation}
where $v,\hat{v} \in V, \lambda , \hat{\lambda} \in V^* , A, \hat{A} \in \mathfrak{co}(V)$. Notice that the last of these commutators depends only on the conformal class of the metric $g$.

There is an isomorphism $ V + \mathfrak{co}(V) + V^* \simeq \mathfrak{so}(E)$ where $E$ is an $(n+2)$-dimensional vector space with signature $(p+1,q+1)$ metric $\tilde{g}$. As vector spaces we write $E = \mathbb{R} \oplus V \oplus \mathbb{R}$, then the metric is given by
\begin{equation}\label{matrixA}
\tilde{g} = \left[ \begin{matrix} 0 & 0 & -1 \\ 0 & g & 0 \\ -1 & 0 & 0 \end{matrix} \right]
\end{equation}
and the isomorphism of algebras is given by sending $(v,A,\lambda) \in V + \mathfrak{co}(V) + V^*$ to
\begin{equation}\label{matrixB}
\left[ \begin{matrix} -\alpha & \lambda & 0 \\ v & \hat{A} & g^{-1}\lambda \\ 0 & gv & \alpha \end{matrix} \right]
\end{equation}
where $A = \hat{A} + \alpha I$ and $\hat{A} \in \mathfrak{so}(V)$ (so $\alpha = {\rm tr}(A)/n$). It follows that as a representation of $\mathfrak{co}(V)$, we have $E \simeq L[-1] \oplus V[-1] \oplus L[1]$, where $L[w]$ denotes the $1$-dimensional representation given by $\mathfrak{co}(V) \ni \phi \mapsto (w/n){\rm tr}(\phi)$ and we use the notation $W[w] = W \otimes L[w]$ for a representation $W$.

Thus far the construction of the metric $\tilde{g}$ on $E$ depends on the choice of representative $g$ of the conformal structure on $V$. However we may instead consider the conformal metric ${\bf g} = {\rm det}(g)^{-1/n}g$. A short calculation shows that ${\bf g} \in S^2( V^*) [2]$ and is independent of the representative metric $g$. Now ${\bf g}$ is a well-defined metric on $V[-1]$ and so we can replace $g$ by ${\bf g}$ in the formulas (\ref{matrixA}) and (\ref{matrixB}) to obtain a uniquely defined metric on $E$. To summarise we have that $E = L[-1] \oplus V[-1] \oplus L[1]$ with metric
\begin{equation}
\tilde{g} = \left[ \begin{matrix} 0 & 0 & -1 \\ 0 & {\bf g} & 0 \\ -1 & 0 & 0 \end{matrix} \right]
\end{equation}
and an isomorphism of Lie algebras $V + \mathfrak{co}(V) + V^* \simeq \mathfrak{so}(\mathcal{V})$ where $(v,A,\lambda)$ maps to
\begin{equation}
\left[ \begin{matrix} -\alpha & \lambda & 0 \\ v & \hat{A} & {\bf g}^{-1}\lambda \\ 0 & {\bf g}v & \alpha \end{matrix} \right].
\end{equation}
We shall call $E$ the {\em tractor space}.\\

Our homogeneous model for conformal geometry is now ${\rm SO}(E)/P$ where $P$ is the stabiliser in ${\rm SO}(E)$ of the null line $L[-1] \subset L[-1] \oplus V[-1] \oplus L[1]$. Note that the Lie algebra of $P$ is $\mathfrak{p}$. The group ${\rm SO}(E)$ acts transitively on the projectivised null quadric $Q = \{ u \in \mathbb{P}(\mathcal{V}) | \; \tilde{g}(u,u)=0 \}$ and $P$ is the stabiliser of a point in $Q$ so the homogeneous space ${\rm SO}(E)/P$ is diffeomorphic to the quadric $Q$. Note that in place of ${\rm SO}(E)$ we could just as easily use ${\rm O}(E)$ and replace $P$ by the appropriate stabiliser. The only difference is that ${\rm SO}(E)$ preserves an orientation on $Q$ so we are considering oriented conformal geometry.\\

Given a scale, that is a non-zero element $\xi \in L[1]$ we can embed $V$ in $E$ via $v \mapsto (\xi^{-1},\xi^{-1}v,1/2 \, \xi g(v,v))^t$ where $g = \xi^{-2}{\bf g}$ is the metric corresponding to the scale. Although this map depends on a choice of scale, the induced map $V \to \mathbb{P}(E)$ is uniquely defined as $v \mapsto \left[1,v,1/2 \, {\bf g}(v,v) \right]^t$. Clearly this map gives a diffeomorphism between $V$ and a neighborhood of the origin $o = [1,0,0]^t$ such that $0 \in V$ maps to the origin. \\

The conformal structure on $V$ can be extended to $Q$ in such a way that ${\rm SO}(E)$ acts as the group of conformal transformations of $Q$ (or more strictly ${\rm PSO}(E)$). In this context ${\rm SO}(E)$ may be called the {\em M\"{o}bius group} and $Q$ {\em M\"{o}bius space}. We can think of the M\"{o}bius space $Q$ as the conformal compactification of $V$, on which all the infinitesimal automorphisms of the flat conformal structure can be integrated to automorphisms, not just locally defined as is the case for $V$.

\subsection{Normal conformal Cartan connections}\label{nccc}
We now consider an $n$-dimensional oriented manifold $M$ with signature $(p,q)$ conformal structure $[g]$. The conformal structure can be equivalently viewed as a ${\rm CO}_+(p,q)$-structure on $M$, let $\mathcal{G}_0$ denote this ${\rm CO}_+(p,q)$-principal bundle. Naturally one wishes to classify these and to that end one may apply the equivalence method. This was achieved by Cartan in \cite{cartan2}, the result is that to each conformal structure is associated a canonical normal Cartan connection, which is called the {\em normal conformal Cartan connection}.\\

The principal bundle for this Cartan geometry is the first prolongation of $\mathcal{G}_0$ \cite{kobayashi} and one finds that the Cartan geometry is modeled on $G/P = Q$, the M\"{o}bius space of Section \ref{mobius} (strictly speaking $G$ should be ${\rm PSO}(p+1,q+1)$ but in odd dimensions this is isomorphic to ${\rm SO}(p+1,q+1)$). In fact one can deduce the existence and uniqueness of the normal conformal connection without going through the prolongation procedure. It follows from Proposition \ref{cartanconn} that normal conformal connections are in bijection with reductions of the structure group of the tangent bundle to $G_0 = {\rm CO}_+(p,q)$, which implies the existence and uniqueness.\\

The normal conformal connection is distinguished by its curvature $\kappa$ having the following form: viewed as a $\mathfrak{g}$-valued $2$-form we have that the $\mathfrak{g}_{-1} = V$-valued component vanishes, the $\mathfrak{g}_0 = \mathfrak{co}(p,q)$-valued component is the conformal Weyl curvature. Moreover it then follows that the $\mathfrak{g}_1 = V^*$-valued component is the Cotton-York tensor \cite{cap3}.\\

Consider the tractor bundle $\mathcal{E}$ associated to the standard representation of $G = {\rm SO}(p+1,q+1)$. In our previous terminology we called this representation the tractor space $E$, hence the tractor bundle is $\mathcal{E} = \tilde{\mathcal{G}} \times_G E$. As a representation of $G_0$, we have found that $E = L[-1] \oplus V[-1] \oplus L[+1]$, where $V$ is the fundamental representation of $G_0$. Thus on restriction to structure group $G_0$ via a Weyl structure we have that
\begin{equation}
\mathcal{E} = \mathcal{L}[-1] \oplus TM[-1] \oplus \mathcal{L}[+1]
\end{equation}
where we use the notation $\mathcal{L}[w]$ to denote the line bundle associated to the representation $L[w]$.

It can be shown that the preferred connections corresponding to the normal conformal Cartan connection are precisely those connections that are torsion free and preserve the conformal structure \cite{slovak}. Thus amongst the preferred connections are the Levi-Civita connections of representative metrics of the conformal class; however not every preferred connection preserves a metric. In fact a preferred connection preserves a metric if and only if it preserves a scale, i.e., a non-vanishing section of $\mathcal{E}[1]$.\\

Let $\overrightarrow{\nabla}$ denote the tractor connection and let $\nabla^\sigma$ denote the preferred connection corresponding to a Weyl structure $\sigma$. Then as in equation (\ref{weyltra}) we have
\begin{equation}
\overrightarrow{\nabla}_X = \nabla^\sigma_X + X + \mathbf{P}^\sigma(X).
\end{equation}
In the case of conformal geometry the rho tensor $\mathbf{P}^\sigma$ can be expressed in terms of the Ricci curvature of the corresponding Weyl connection \cite{cap3}. We let $R$ be the Riemann curvature of $\nabla^\sigma$, ${\rm Ric}(X,Y) = {\rm tr}(Z \to R(Z,X)Y)$ the Ricci curvature and $s = {\rm tr}({\rm Ric})$ the scalar curvature. Then
\begin{equation}
\mathbf{P}^\sigma(X,Y) = -\frac{1}{n-2}\left( \frac{1}{n}{\rm Ric}(X,Y) + \frac{n-1}{n}{\rm Ric}(Y,X) - \frac{1}{2n-2}sg(X,Y) \right).
\end{equation}

Given a Cartan geometry we have an intrinsically defined connection, the tractor connection. Thus we can consider the holonomy of this connection, which we call tractor holonomy. In the case of conformal geometry we call this the {\em conformal holonomy} of the conformal structure. In Chapter \ref{chap5} we will consider $G_2$ conformal holonomy.



\chapter{Conformal $G_2$ holonomy}\label{chap5}
In this chapter we will study signature $(2,3)$ conformal geometry with holonomy in the split form of $G_2$. As discussed in the introdution this has a history that traces back to the work of Cartan \cite{cartan1}. This geometry has examined for several different perspectives including the works of Bryant and Hsu \cite{bryant}, Agrachev \cite{agrachev}, Nurowski \cite{nurowski1}, \cite{nurowski2}, Zelenko \cite{zel1}, \cite{zel2}, \v{C}ap and Sagerschnig \cite{casa} and Hammerl \cite{ham}. Our contribution is mostly to clarify the geometry at first from the point of view of parabolic geometries then later purely in terms of conformal geometry. We start off in Section \ref{octonions} by studying split $G_2$ and the corresponding homogeneous space. The corresponding parabolic geometry is then considered and shown to be equivalent to a generic $2$-plane distribution in $5$ dimensions. Then in Section \ref{confg2} we relate this $G_2$ geometry to conformal $G_2$ holonomy. In Section \ref{spinors} we consider the geometry from the point of view of spinors and show that $G_2$ holonomy corresponds to a spinor satisfying the twistor spinor equation and further that the generic $2$-plane distribution corresponds to the annihilator of this spinor. Finally in Section \ref{examples} we give two examples of such generic $2$-distributions.


\section{Split octonions and split $G_2$}\label{octonions}
We will study a parabolic geometry for the split real form of $G_2$. In order to understand the geometry we must first examine the algebraic structure of split $\mathfrak{g}_2$ and the parabolic subalgebra.\\

The split form of $G_2$ can be conveniently described in terms of the split octonions, a split signature version of the more familiar octonions. For more details see \cite{harvey}.

We define the {\em split octonions} $\widetilde{\mathbb{O}}$ as the $8$-dimensional algebra over $\mathbb{R}$ spanned by elements $1,i,j,k,l,li,lj,lk$ satisfying the multiplication given in Figure \ref{mult}.
\begin{figure}[h]
\begin{center}
\leavevmode
\begin{tabular}[c]{|c|c|c|c|c|c|c|c|}
\hline
$1$ & $i$ & $j$ & $k$ & $l$ & $li$ & $lj$ & $lk$ \\
\hline
$i$ & $-1$ & $k$ & $-j$ & $-li$ & $l$ & $-lk$ & $lj$ \\
\hline
$j$ & $-k$ & $-1$ & $i$ & $-lj$ & $lk$ & $l$ & $-li$ \\
\hline
$k$ & $j$ & $-i$ & $-1$ & $-lk$ & $-lj$ & $li$ & $l$ \\
\hline
$l$ & $li$ & $lj$ & $lk$ & $1$ & $i$ & $j$ & $k$ \\
\hline
$li$ & $-l$ & $-lk$ & $lj$ & $-i$ & $1$ & $k$ & $-j$ \\
\hline
$lj$ & $lk$ & $-l$ & $-li$ & $-j$ & $-k$ & $1$ & $i$ \\
\hline
$lk$ & $-lj$ & $li$ & $-l$ & $-k$ & $j$ & $-i$ & $1$ \\
\hline
\end{tabular}
\end{center}
\caption{Multiplication table for the split octonions}\label{mult}
\end{figure}
\\

Let $x = x_0 1 + x_1 i + \cdots + x_7 lk$ denote an element of $\widetilde{\mathbb{O}}$. We define the conjugate as $\bar{x} = x_0 - x_1 i - \cdots - x_7 lk$, and the norm of $x$ as ${\rm N}(x) = x\bar{x} = x_0^2 + x_1^2 + x_2^2 + x_3^2 -x_4^2 -x_5^2 -x_6^2 - x_7^2$. In fact the norm is the quadratic form associated to the split signature metric $g(x,y) = {\rm Re}(x\bar{y})$.\\

Consider the group ${\rm Aut}(\widetilde{\mathbb{O}})$ of linear isomorphisms of $\widetilde{\mathbb{O}}$ preserving the algebra. This is a Lie group and we can show that the Lie algebra is $\tilde{\mathfrak{g}}_2$, the split real form of $G_2$. We note that any automorphism must preserve the norm so that ${\rm Aut}(\widetilde{\mathbb{O}})$ is a subgroup of ${\rm O}(4,4)$. But an automorphism must also preserve the identity $1 \in \widetilde{\mathbb{O}}$ and hence acts trivially on the $1$-dimensional subspace ${\rm Re}(\widetilde{\mathbb{O}})$. Thus the automorphisms are determined by their action on the orthogonal complement ${\rm Re}(\widetilde{\mathbb{O}})^\perp = {\rm Im}(\widetilde{\mathbb{O}})$ which has a signature $(3,4)$ metric. Thus the automorphism group is a subgroup of ${\rm O}(3,4)$, in fact a subgroup of ${\rm SO}(3,4)$.\\

Let $x,y \in {\rm Im}(\widetilde{\mathbb{O}})$. We define the cross product $ \times : {\rm Im}(\widetilde{\mathbb{O}}) \times {\rm Im}(\widetilde{\mathbb{O}}) \to {\rm Im}(\widetilde{\mathbb{O}})$ by
\begin{equation}
x \times y = \frac{1}{2}(xy - yx)
\end{equation}
and we also note since $x$ and $y$ are imaginary we have
\begin{equation}
g(x,y) = -\frac{1}{2}(xy + yx).
\end{equation}

The algebraic structure of the split octonions is encoded in the trilinear form $\phi$ on ${\rm Im}(\widetilde{\mathbb{O}})$ given by
\begin{equation}
\phi(x,y,z) = g(x \times y , z).
\end{equation}
As the split octonions are an alternative algebra we find that $\phi$ is skew-symmetric. In fact
\begin{equation}
\phi = dx^{123} + dx^{145} + dx^{167} + dx^{246} + dx^{275} + dx^{347} + dx^{356}
\end{equation}
where $dx^{ijk}$ denotes $dx^i \wedge dx^j \wedge dx^k$. Since the $3$-form $\phi$ encodes the algebraic structure of the split octonions, the automorphism group is simply the stabiliser of $\phi$ in ${\rm GL}({\rm Im}(\widetilde{\mathbb{O}}))$. We denote the automorphism group by $\tilde{G}_2$ and the corresponding Lie algebra by $\tilde{\mathfrak{g}}_2$. We shall see that this is indeed the split real form of $G_2$.\\





We use a basis that most clearly shows the parabolic structure of $\tilde{\mathfrak{g}}_2$. With respect to the ordered basis
\begin{equation*}
\left\{\tfrac{1}{\sqrt{2}}(i +li),\tfrac{1}{\sqrt{2}}(lj -j),\tfrac{1}{\sqrt{2}}(k -lk),l,\tfrac{1}{\sqrt{2}}(k +lk),\tfrac{1}{\sqrt{2}}(j+lj),\tfrac{1}{\sqrt{2}}(i -li) \right\}
\end{equation*}
the algebra $\tilde{\mathfrak{g}}_2$ consists of matrices of the form:
\begin{equation}\label{g2}
\left[ \begin{matrix} u_1+u_4 & \lambda^1 & \lambda^2 & \lambda^3 & \lambda^4 & \lambda^5 & 0 \\ v^1 & u_1 & u_2 & \sqrt{2}\lambda^2 & \tfrac{1}{\sqrt{2}}\lambda^3 & 0 & \lambda^5 \\ v^2 & u_3 & u_4 & -\sqrt{2}\lambda^1 & 0 & \tfrac{1}{\sqrt{2}}\lambda^3 & -\lambda^4 \\ v^3 & \sqrt{2}v^2 & -\sqrt{2}v^1  & 0 &  -\sqrt{2}\lambda^1 & -\sqrt{2}\lambda^2 & \lambda^3 \\ v^4 & \tfrac{1}{\sqrt{2}}v^3 & 0 & -\sqrt{2}v^1 & -u_4 & u_2 & -\lambda^2 \\ v^5 & 0 & \tfrac{1}{\sqrt{2}}v^3 & -\sqrt{2}v^2 & u_3 & -u_1 & \lambda^1 \\ 0 & v^5 & -v^4 & v^3 & -v^2 & v^1 & -u_1-u_4 \end{matrix} \right].
\end{equation}

The parabolic structure now becomes clear. The $14$ variables $u_1, \dots , \lambda^5$ can be thought of as linear coordinates on $\tilde{\mathfrak{g}}_2$ and thus as a basis for the dual $\tilde{\mathfrak{g}}_2^*$. The entries $u_1,u_4$ correspond to a Cartan subalgebra, with the element $u_1 = u_4 = 1$ serving as a grading element. The root space decomposition is as shown in Figure \ref{root2}. Note we have included a dividing line to establish a system of positive roots.

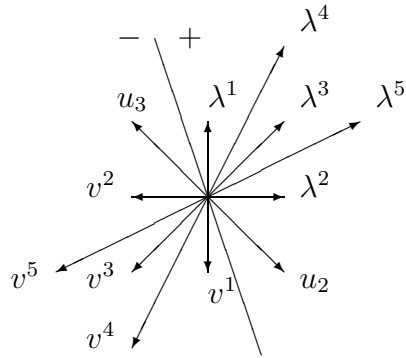
\begin{figure}[h]
\setlength{\unitlength}{1mm}
\begin{picture}(90,60)
\put(60,30){\vector(2,1){20}} \put(82,42){$\lambda^5$}
\put(60,30){\vector(1,2){10}} \put(72,52){$\lambda^4$}
\put(60,30){\vector(1,1){10}} \put(72,42){$\lambda^3$}
\put(60,30){\vector(1,0){10}} \put(72,30){$\lambda^2$}
\put(60,30){\vector(0,1){10}} \put(60,42){$\lambda^1$}
\put(60,30){\vector(1,-1){10}} \put(72,18){$u_2$}
\put(60,30){\vector(-1,1){10}} \put(48,42){$u_3$}
\put(60,30){\vector(-1,0){10}} \put(44,30){$v^2$}
\put(60,30){\vector(0,-1){10}} \put(60,16){$v^1$}
\put(60,30){\vector(-1,-1){10}} \put(44,18){$v^3$}
\put(60,30){\vector(-2,-1){20}} \put(34,18){$v^5$}
\put(60,30){\vector(-1,-2){10}} \put(44,10){$v^4$}
\linethickness{0.05mm}
\put(60,30){\line(-1,3){7}}
\put(60,30){\line(1,-3){7}}
\put(48,50){$-$}
\put(56,50){$+$}
\end{picture}
\caption{Root space decomposition}\label{root2}
\end{figure}
We see that for the given system of positive roots the simple roots correspond to $u_2$ and $\lambda^1$. A simple calculation shows $u_2$ is the larger of the two. Let $\tilde{\mathfrak{p}}$ be the parabolic subalgebra corresponding to the diagram $\times \Lleftarrow \circ$. We see that $\tilde{\mathfrak{p}}$ consists of those elements with $v^1 = \dots = v^5 = 0$. Moreover this is precisely the subalgebra that preserves the line $o$ spanned by $i+li$. The graded structure associated to the parabolic subalgebra $\tilde{\mathfrak{p}}$ is as follows:
\begin{center}
\leavevmode
\begin{tabular}{|c|c|c|c|c|c|c|}
\hline
$\mathfrak{g}_{-3}^*$ & $\mathfrak{g}_{-2}^*$ & $\mathfrak{g}_{-1}^*$ & $\mathfrak{g}_0^*$ & $\mathfrak{g}_1^*$ & $\mathfrak{g}_2^*$ & $\mathfrak{g}_3^*$ \\
\hline
$v^5,v^4$ & $v^3$ & $v^2 , v^1$ & $u_1,u_2,u_3,u_4$ & $\lambda^2,\lambda^1$ & $\lambda^3$ & $\lambda^5 , \lambda^4$ \\
\hline
\end{tabular}
\end{center}
Now let us define $\tilde{P}$ and $\tilde{G}_0$ to be the subgroups of $\tilde{G}_2$ preserving respectively the filtration and grading on $\tilde{\mathfrak{g}}_2$. Then the corresponding Lie algebras are $\tilde{\mathfrak{p}}$ and $\mathfrak{g}_0$.
\begin{lemp}\label{lemg0}
The natural map
\begin{equation}
\phi: \tilde{G}_0 \to {\rm GL}(\mathfrak{g}_{-1})
\end{equation}
is an isomorphism.
\begin{proof}
First we show $\phi$ is injective. Suppose $g \in \tilde{G}_0$ and $\phi(g)=0$, that is $g$ acts trivially on $\mathfrak{g}_-$. Taking Lie brackets of elements of $\mathfrak{g}_{-1}$ generates $\mathfrak{g}_-$ so $g$ acts trivially on $\mathfrak{g}_-$. Moreover $\mathfrak{g}_0$ can be identified with $\mathfrak{gl}(\mathfrak{g}_{-1})$ so $g$ also acts trivially on $\mathfrak{g}_0$. Now suppose $Y \in \mathfrak{g}_i$ for $i >0$ then for all $X \in \mathfrak{g}_{-i}$ we have
\begin{equation*}
\left[ {\rm Ad}_g Y - Y , X \right] = 0.
\end{equation*}
This is only possible if ${\rm Ad}_g Y = Y$. Thus we have shown $g$ acts trivially on $\tilde{\mathfrak{g}}_2$. We can think of $g$ as an endomorphism of the imaginary split octonions that commutes with $\tilde{\mathfrak{g}}_2$. But the only such endomorphisms are multiples of the identity and then in fact then we must have $g$ is the identity since no other multiple of the identity is contained in $\tilde{G}_2$.\\

Next we show $\phi$ is surjective. Since $\mathfrak{g}_0$ is isomorphic to $\mathfrak{gl}(2)$ we need only show that in the image of $\phi$ is an orientation reversing element. For this define an involution $\tau$ which is the identity on $k,l,lk$ and acts as multiplication by $-1$ on $i,j,li,lj$. Then it is clear that $\tau \in \tilde{G}_2$. In fact $\tau \in \tilde{G}_0$ and the action of $\tau$ on $\mathfrak{g}_{-1}$ is orientation reversing.
\end{proof}
\end{lemp}

\begin{prop}
Let $P_o$ be the subgroup of $\tilde{G}_2$ preserving the line spanned by $o = i + li$. Then $P_o = \tilde{P}$.
\begin{proof}
From Lemma \ref{lemg0} we find that $\tilde{G}_0$ consists of matrices of the form
\begin{equation*}
\left[ \begin{matrix} a & 0 & 0 & 0 & 0 \\ 0 & A & 0 & 0 & 0 \\ 0 & 0 & 1 & 0 & 0 \\ 0 & 0 & 0 & a^{-1}A & 0 \\ 0 & 0 & 0 & 0 & a^{-1} \end{matrix} \right]
\end{equation*}
where $A \in {\rm GL}(2)$ and $a = {\rm det}(A)$. Furthermore we have that $\tilde{P} = \tilde{G}_0 \ltimes \tilde{P}_+$ where $\tilde{P}_+$ is the group generated by exponentials of elements of $\tilde{\mathfrak{p}}_+ = \mathfrak{g}_1 \oplus \mathfrak{g}_2 \oplus \mathfrak{g}_3$. From this we see immediately that $\tilde{P}$ preserves the null line spanned by $i + li$.\\

Conversely suppose $g \in \tilde{G}_2$ preserves the line spanned by $i +li$. Since $g$ preserves the cross product structure we can show that $g$ has the form
\begin{equation*}
g = \left[ \begin{matrix} a & * & * & * & * \\ 0 & A & * & * & * \\ 0 & 0 & 1 & * & * \\ 0 & 0 & 0 & a^{-1}A & * \\ 0 & 0 & 0 & 0 & a^{-1} \end{matrix} \right]
\end{equation*}
where again $A \in {\rm GL}(2)$ and $a = {\rm det}(A)$. From this it is straightforward to show that $g$ preserves the filtration on $\tilde{\mathfrak{g}}_2$ hence $g \in \tilde{P}$.
\end{proof}
\end{prop}

Let $Q$ be the space of all null lines in ${\rm Im}(\tilde{\mathbb{O}})$. This space can be identified with the null quadric in $\mathbb{P}({\rm Im}(\tilde{\mathbb{O}}))$. The group $\tilde{G}_2$ acts transitively on this space and the stabiliser of the line spanned by $o = i+li$ is the parabolic subgroup $P_o = \tilde{P}$. Therefore $Q$ identifies with the homogeneous space $\tilde{G}_2 / \tilde{P}$. This is the homogeneous space on which the parabolic geometry for $(\tilde{P} , \tilde{\mathfrak{g}}_2)$ is modeled.\\

Suppose we have a regular parabolic geometry $\tilde{\mathcal{P}} \to M$ for $(\tilde{P},\tilde{\mathfrak{g}}_2)$. Then we have a filtration of $TM$
\begin{equation}
V^{-1} \subset V^{-2} \subset V^{-3} = TM
\end{equation}
where $V^{-1}$ is $2$-dimensional and $V^{-2}$ is $3$-dimensional. We also have the associated graded bundle
\begin{equation}
{\rm Gr}TM = V^{-1} \oplus \left( V^{-2}/V^{-1} \right) \oplus \left( V^{-3}/V^{-2} \right).
\end{equation}
Since the parabolic geometry is regular we have that $M$ satisfies the structure equations. Let us examine the structure equations in more detail. The Lie algebra $\mathfrak{g}_- = \mathfrak{g}_{-1} \oplus \mathfrak{g}_{-2} \oplus \mathfrak{g}_{-3}$ is spanned by $a_1, \dots , a_5$ where $a_1, a_2$ span $\mathfrak{g}_{-1}$, $a_3$ spans $\mathfrak{g}_{-2}$ and $a_4, a_5$ span $\mathfrak{g}_{-3}$. The algebra is determined by the relations
\begin{equation}
\left[ a_1 , a_2 \right] = a_3, \; \; \; \left[ a_1 , a_3 \right] = a_4, \; \; \; \left[ a_2 , a_3 \right] = a_5
\end{equation}
with all other commutators given by skew symmetry or vanish. It follows that $M$ satisfies the structure equations if and only if the $2$-distribution $V = V^{-1} \subset TM$ is maximally non-integrable or {\em generic}. By this we mean that $\partial V = V + \left[ V , V \right]$ is three dimensional and $\partial (\partial V) = \partial V + \left[ \partial V , \partial V \right]$ is $5$-dimensional. These are the largest possible dimensions that $\partial V$ and $\partial \partial V$ can have.

Now the Lie algebra $\mathfrak{g}_0$ is isomorphic to $\mathfrak{gl}(2)$ and as representations of $\mathfrak{g}_0$ we have
\begin{equation}
\begin{aligned}
\mathfrak{g}_{-1} &= V^* \\
\mathfrak{g}_{-2} &= \wedge^2 V^* \\
\mathfrak{g}_{-3} &= V^* \otimes \wedge^2 V^*
\end{aligned}
\end{equation}
where $V$ is the standard $2$-dimensional representation. But this is precisely the reduction of structure provided by the Lie bracket of vector fields, therefore there is a unique reduction of structure of ${\rm Gr}TM$ to $\mathfrak{g}_0$ which is compatible with the structure equations. In light of this we now have

\begin{prop}
Let $M$ be a $5$-manifold with generic $2$-plane distribution $V$. There exists a unique normal, regular parabolic geometry $\mathcal{P} \to M$ modeled on the homogeneous space $Q = \tilde{G}_2 / \tilde{P}$ such that the filtration on $TM$ induced by the parabolic geometry agrees with the filtration induced by $V$
\end{prop}


\section{Conformal $G_2$ geometry}\label{confg2}
As shown in the previous section the Cartan geometry $(\tilde{\mathcal{P}},\omega)$ obtained from a generic $2$ distribution on a $5$-dimensional manifold $M$ is modeled on the homogeneous space $\tilde{G_2}/\tilde{P}$ which is signature $(2,3)$ conformal M\"{o}bius space. We now show that the manifold $M$ has a naturally defined signature $(2,3)$ conformal structure $[g]$. This conformal structure was discovered by Nurowski \cite{nurowski1}. The connection between the $G_2$-geometry and the conformal geometry can be understood as a special case of a more general relation between parabolic geometries of different types \cite{dosl}.\\

Let us note that given a generic $2$-distribution $V$ on a $5$-dimensional manifold $M$, the associated graded bundle ${\rm Gr}TM \simeq V \oplus \wedge^2 V \oplus (\wedge^2 V \otimes V)$ has a natural $(\wedge^2 V )^2$-valued pairing such that $V$ is isotropic and $V^\perp = V \oplus \wedge^2 V$, so it is perhaps not so surprising that $M$ should have a natural $(2,3)$-conformal structure. However a conformal structure on the associated graded bundle does not uniquely determine a conformal structure on the tangent bundle. The additional information is encoded into the parabolic geometry associated to the distribution.\\

Recall that $Q = \tilde{G}_2 / \tilde{P}$ is the space of null lines in the $7$-dimensional representation $V = {\rm Im}(\tilde{\mathbb{O}})$. But $\tilde{G}_2 \subset {\rm SO}(3,4)$ and $Q$ can also be identified with the homogeneous space $Q = {\rm SO}(3,4)/P$ where $P$ is the stabiliser in ${\rm SO}(3,4)$ of a null line. This suggests a relation between Parabolic geometries for $(\tilde{P} , \tilde{\mathfrak{g}}_2)$ and for $(P , \mathfrak{so}(3,4))$. We will show that a $(\tilde{P} , \tilde{\mathfrak{g}}_2)$ geometry defines a conformal parabolic geometry and that conversely a conformal geometry with $\tilde{G}_2$ holonomy reduces to a $(\tilde{P} , \tilde{\mathfrak{g}}_2)$ geometry.\\

First consider going from a parabolic geometry for $(\tilde{P} , \tilde{\mathfrak{g}}_2)$ to a conformal geometry. For this we have the following:
\begin{prop}
Let $G$ be a Lie group with Lie algebra $\mathfrak{g}$ and let $G_1,G_2$ be Lie subgroups with corresponding Lie algebras $\mathfrak{g}_1,\mathfrak{g}_2$ such that $\mathfrak{g} = \mathfrak{g}_1 + \mathfrak{g}_2$. Let $\mathcal{G} \to M$ be a $(\mathfrak{g}_2,G_1 \cap G_2)$ Cartan geometry with Cartan connection $\omega$. Then there exists a unique $(\mathfrak{g},G_1)$ Cartan connection $\omega_1$ on $\mathcal{G}_1 = \mathcal{G} \times_{G_1 \cap G_2} G_1$ such that $i^*\omega_1 = \omega$ where $i^*$ is the inclusion $i : \mathcal{G} \to \mathcal{G}_1$.
\end{prop}
\begin{proof}
This is a straightforward exercise in using the Cartan connection axioms.
\end{proof}
In our case we have that $G = {\rm SO}(3,4)$, $G_2 = \tilde{G}_2$, $G_1 = P$ and $G_1 \cap G_2 = \tilde{P}$. We have $\mathfrak{g} = \mathfrak{g}_1 + \mathfrak{g}_2$. It follows that $\omega$ uniquely extends to a conformal Cartan connection. When $\omega$ is normal it is not immediately clear that the corresponding conformal connection is normal. It was recently proven that in fact the associated conformal connection is indeed normal \cite{ham}.\\

It follows that the tangent bundle of $M$ has a corresponding reduction of structure to ${\rm CO}(2,3)$, that is $M$ has an associated conformal structure. Furthermore the conformal structure has conformal holonomy contained in $\tilde{G}_2$.\\

We now consider the converse situation. Let $\mathcal{P}\to M$ be a conformal parabolic geometry with Cartan connection $\omega$. Let $\mathcal{G} = \mathcal{P} \times_P {\rm SO}(3,4)$ be the associated tractor bundle and extend $\omega$ to $\mathcal{G}$ defining the associated tractor connection. We are interested in the case where the tractor connection has holonomy in $\tilde{G}_2$. This means that there is a reduction of structure to a principal $\tilde{G}_2$-subbundle $\tilde{\mathcal{G}} \subset \mathcal{G}$ preserved by the tractor connection. Therefore $\omega$ restricts to a $\tilde{\mathfrak{g}}_2$-valued connection form on $\tilde{\mathcal{G}}$.

Now $\mathcal{P}$ is the bundle of frames in $\mathcal{G}$ that fix a certain line bundle $\mathcal{L} \subset \mathcal{T}$ in the tractor bundle. Therefore we have a principal $\tilde{P}$-subbundle $\tilde{\mathcal{P}}$ of $\tilde{\mathcal{G}}$ defined as the bundle of frames of $\tilde{\mathcal{G}}$ preserving the line bundle $\mathcal{L}$. The tractor connection restricts to a $\mathfrak{g}_2$-valued $1$-form $\omega$ on $\tilde{\mathcal{P}}$ which we can show is a Cartan connection.
\begin{prop}
The $\tilde{\mathfrak{g}}_2$-valued $1$-form defined on $\tilde{\mathcal{P}}$ as above is a Cartan connection. Moreover if the conformal connection is torsion free then the $(\tilde{P}, \tilde{\mathfrak{g}}_2)$ Cartan connection is regular.
\begin{proof}
The only thing that is not immediate is that for each $u \in \tilde{\mathcal{P}}$, $\omega : T_u \tilde{\mathcal{P}} \to \tilde{\mathfrak{g}}_2$ is an isomorphism. This is easily seen from the following commutative diagram
\begin{equation*}\xymatrix{
T_u \tilde{\mathcal{G}} \ar[dr]^\omega \ar[rrr] & & & T_u\mathcal{G} \ar[dl]_\omega \\
& \tilde{\mathfrak{g}}_2 \ar[r] & \mathfrak{so}(3,4)  & \\
T_u\tilde{\mathcal{P}} \ar[ur]^\omega \ar[uu] \ar[rrr] & & & T_u\mathcal{P} \ar[ul]_\omega \ar[uu]}
\end{equation*}
and the fact that $\omega : T_u\mathcal{P} \to \mathfrak{so}(3,4)$ is an isomorphism.
\end{proof}
\end{prop}

We note again that when the conformal connection is normal the associated $(\tilde{P},\tilde{\mathfrak{g}}_2)$ Cartan connection is normal \cite{ham}. Thus we have two constructions that are mutually inverse which allow us to regard the geometry as either a conformal geometry with $\tilde{G}_2$ holonomy or as a $(\tilde{P},\tilde{\mathfrak{g}}_2)$ geometry.\\

Let $\tilde{\mathcal{P}} \to M$ be the parabolic geometry associated to a generic $2$-distribution on $M$. Associated to the $7$-dimensional representation of $\tilde{G}_2$ is a tractor bundle $\mathcal{T}$ on $M$ which is equipped with the tractor connection. Now since $\tilde{\mathcal{P}}$ has structure group $\tilde{P}$ which is the subgroup of $\tilde{G}_2$ fixing a null line, there is a corresponding line subbundle $\mathcal{L} \subset \mathcal{T}$ which is null with respect to the signature $(3,4)$ metric on $\mathcal{T}$. We will say that a local frame $\{ e_1, \dots , e_7 \}$ for $\mathcal{T}$ is a {\em $\tilde{P}$-frame} if the identification of $\{ e_1 , \dots , e_7 \}$ with our standard basis $\{\tfrac{1}{\sqrt{2}}(i +li),\tfrac{1}{\sqrt{2}}(lj -j),\tfrac{1}{\sqrt{2}}(k -lk),l,\tfrac{1}{\sqrt{2}}(k +lk),\tfrac{1}{\sqrt{2}}(j+lj),\tfrac{1}{\sqrt{2}}(i -li) \}$ respects the $\tilde{G}_2$ structure and if the line bundle $\mathcal{L}$ is spanned by $e_1$. Clearly we can always locally find a $\tilde{P}$-frame.

In a $\tilde{P}$-frame the tractor connection will have a connection matrix of the form (\ref{g2}) where the coefficients $v_1,v_2, \dots $ are now to be considered as $1$-forms. The conformal structure on $M$ can be described as follows. Let $\langle \, , \, \rangle$ be the metric on $\mathcal{T}$. Choose a local non-vanishing section $s$ of the line bundle $\mathcal{L}$. Then a local metric $g$ on $M$ representing the conformal structure is given by
\begin{equation}
g(X,Y) = \langle \overrightarrow{\nabla}_X s , \overrightarrow{\nabla}_Y s \rangle.
\end{equation}
Taking $s = e_1$ as a local section we find the conformal structure is represented by
\begin{equation}
g = -2v^1 v^5 + 2v^2v^4 - (v^3)^2.
\end{equation}
The $1$-forms $v^1, \dots , v^5$ define a coframe on $M$. Moreover the graded structure of $\tilde{\mathfrak{g}}_2$ shows that the annihilator of $v^3,v^4,v^5$ is the generic $2$-distribution $V^{-1}$ and the annihilator of $v^4,v^5$ is the $3$-distribution $V^{-2} = V + \left[ V , V \right]$. From this we note that $V^{-1}$ is a bundle of maximal isotropics of the conformal structure and $V^{-2}$ is the annihilator of $V^{-1}$.\\


\section{Twistor spinor formulation}\label{spinors}
We show for arbitrary dimension and signature that a constant spinor for the tractor bundle corresponds to a solution to the twistor spinor equations. Now the group $G_2$ occurs as the stabiliser of a spinor in 7 dimensions. This implies in the case of $(2,3)$ conformal structures arising from $2$-distributions that there exists a constant spinor for the tractor connection. Since in this case we are considering indefinite conformal structures we can also consider pure spinors. It so happens that in signature $(2,3)$, all non-vanishing spinors are pure so that in particular the twistor spinor defines a $2$-dimensional maximal isotropic distribution. We show that this is in fact the original distribution. Thus the twistor spinor has a direct interpretation in terms of the $5$-dimensional geometry.


\subsection{The twistor spinor equation}
We shall refer to a manifold with a ${\rm Spin}(p,q) \times \mathbb{R}^*$ structure as a {\em conformal spin manifold}. Any representation of ${\rm Spin}(p,q)$ is also a representation of ${\rm Spin}(p,q) \times \mathbb{R}^*$ for which $\mathbb{R}^*$ acts trivially. Hence one can define spin bundles for a conformal spin manifold. Moreover one can take the tensor product with a $1$-dimensional representation of $\mathbb{R}^*$ to obtain weighted spinor bundles.\\

On a $(p,q)$ signature conformal manifold we also have the tractor bundle $\mathcal{E}$ which is an ${\rm SO}(p+1,q+1)$-bundle, so we can also consider lifting the structure of $\mathcal{E}$ to ${\rm Spin}(p+1,q+1)$. The existence of splittings $\mathcal{E} \simeq \mathcal{L}[-1] \oplus T[-1] \oplus \mathcal{L}[1]$ shows that on a conformal spin manifold this is possible. Given a spinor representation for ${\rm Spin}(p+1,q+1)$ we can then form the associated spinor bundle. In any case we can always form this bundle locally. In this section we study such spinors and show how to relate them to spinors of the underlying manifold. In particular covariant constancy under the tractor connection will lead to the twistor spinor equation \cite{peri}, \cite{hab}.\\

Let $V$ be a vector space with $(p,q)$ signature conformal structure $[g]$ and let ${\bf g}$ be the conformal metric on $V[-1]$. Let $E = L[-1] \oplus V[-1] \oplus L[1]$ be the associated tractor space with signature $(p+1,q+1)$ metric
\begin{equation}
\tilde{g} = 2e^0e^\infty + {\bf g},
\end{equation}
where $e^0 \in L^*[-1]$, $e^\infty \in L^*[1]$. Note for convenience both $\tilde{g}$ and ${\bf g}$ have changed sign from the conventions of Section \ref{mobius}.\\

We use the convention that the Clifford algebra $\mathit{Cl}(W)$, associated to a vector space $W$ with metric $g$ is defined by the relation $x^2 = -g(x,x)$ \cite{lawson}. Let $S(E)$ denote an irreducible $\mathit{Cl}(E)$ module, so that by restriction $S(E)$ is a ${\rm Spin}(E)$ module. Our first objective is to understand the structure of $S(E)$ with respect to the isomorphism $\mathfrak{spin}(E) \simeq V \oplus \mathfrak{co}(V) \oplus V^*$. We consider each of the various parts in turn.\\

Consider $\mathfrak{spin}(V[-1]) \subset \mathfrak{spin}(E)$. By restriction $S(E)$ is a $\mathit{Cl}(V[-1])$-module and the action of $\mathfrak{spin}(V[-1])$ on $S(E)$ is simply this Clifford multiplication. We now make use of the isomorphism $\mathit{Cl}_{p+1,q+1} \simeq \mathit{Cl}_{1,1} \otimes \mathit{Cl}_{p,q}$. We can think of $L[-1] \oplus L[1]$ as a $2$-dimensional space with metric
\begin{equation}
g_{1,1} = 2e^0e^\infty = \hat{e}^2 - \hat{\epsilon}^2
\end{equation}
where $\{\hat{e},\hat{\epsilon}\}$ is a dual basis to the basis $\{e,\epsilon \}$ given by
\begin{equation}
\left[ \begin{matrix} e \\ \epsilon \end{matrix} \right] = \frac{1}{\sqrt{2}}\left[ \begin{matrix} 1 & 1 \\ 1 & -1 \end{matrix} \right] \left[ \begin{matrix} e_0 \\ e_\infty \end{matrix} \right].
\end{equation}
Now $\mathit{Cl}_{1,1} \simeq {\rm End}(\mathbb{R}^2)$ with an explicit representation given by
\begin{equation}\label{matrixrep1}
e = \left[ \begin{matrix} 0 & 1 \\ -1 & 0 \end{matrix} \right] \qquad \epsilon = \left[ \begin{matrix} 0 & 1 \\ 1 & 0 \end{matrix} \right].
\end{equation}
Consider the linear map $\phi : E \to \mathit{Cl}_{1,1} \otimes \mathit{Cl}(V[-1])$ defined by
\begin{eqnarray}
\phi(u) &=& e\epsilon \otimes u \quad u \in V, \nonumber \\
\phi(e) &=& e \otimes 1, \\
\phi(\epsilon) &=& \epsilon \otimes 1. \nonumber
\end{eqnarray}
Since $\phi$ has the property $\phi(x)^2 = -\tilde{g}(x,x)$, $\phi$ extends uniquely to an isomorphism $\phi : \mathit{Cl}(E) \to \mathit{Cl}_{1,1} \otimes \mathit{Cl}(V[-1])$. Now as a $\mathit{Cl}(V[-1])$-module, $S(E)$ splits into the $\pm 1$ eigenspaces $S_\pm$ of $e\epsilon$:
\begin{equation}
S(E) = S_+ \oplus S_-.
\end{equation}
Moreover, the action of $e$ and $\epsilon$ intertwine $S_\pm$ so we see that $S_\pm$ are isomorphic as vector spaces with Clifford multiplication by $V[-1]$ differing only by a factor of $-1$. It follows that $S_\pm$ are irreducible $\mathit{Cl}(V[-1])$-modules. Under this decomposition we see that $\mathfrak{spin}(V[-1])$ just acts diagonally by Clifford multiplication.\\

To understand how the remaining parts of $\mathfrak{spin}(E)$ act we must understand the isomorphism $\mathfrak{so}(E) \to \mathfrak{spin}(E)$. Given $x,y \in E$, let $x \wedge y$ denote the element of $\mathfrak{so}(E)$ given by
\begin{equation}
(x \wedge y)z = \tilde{g}(x,z)y - \tilde{g}(y,z)x.
\end{equation}
Then the isomorphism $\mathfrak{so}(E) \to \mathfrak{spin}(E)$ is given by \cite{lawson}
\begin{equation*}
x \wedge y \mapsto \tfrac{1}{4}[x,y] = \tfrac{1}{4}(xy-yx),
\end{equation*}
the product being Clifford multiplication.\\

Now consider the homothety part $kI \in \mathfrak{co}(V)$ as contained in $\mathfrak{so}(E)$ and hence in $\mathfrak{spin}(E)$. One verifies that the corresponding element is $\tfrac{1}{2}k\epsilon e$. This acts on $S_\pm$ as multiplication by $\mp\tfrac{1}{2}k$ and it follows that $S_+$ has conformal weight $-\tfrac{1}{2}$ and $S_-$ has conformal weight $\tfrac{1}{2}$ so a better way to express the decomposition of $S(E)$ is
\begin{equation}
S(E) = S_+[-\tfrac{1}{2}] \oplus S_-[\tfrac{1}{2}].
\end{equation}

Now consider the action of $V \subset \mathfrak{spin}(E)$. One verifies the corresponding Clifford action is $-\tfrac{1}{2}Xe_\infty$ or in matrix form this is
\begin{equation}
\rho(X) = \left[ \begin{matrix} 0 & 0 \\ -\tfrac{1}{\sqrt{2}}X \cdot & 0 \end{matrix} \right].
\end{equation}
Note that this is consistent with the fact that Clifford multiplication by an element of $V$ will raise the conformal weight by $1$.\\

Lastly consider $\lambda \in V^*$. The corresponding element of $\mathfrak{spin}(E)$ is $\tfrac{1}{2}({\bf g}^{-1}\lambda )e_0$, or in matrix form
\begin{equation}
\rho(\lambda) = \left[ \begin{matrix} 0 & \tfrac{1}{\sqrt{2}}({\bf g}^{-1}\lambda ) \cdot \\ 0 & 0 \end{matrix} \right].
\end{equation}
Once again note that this is consistent with the conformal weights.\\

The tractor connection $\overrightarrow{\nabla}$ on a conformal spin manifold $M$ induces a connection on the spinor bundle associated to $S(E)$. We denote this bundle by $S(\mathcal{E})$ and the associated connection will be denoted the same as the tractor connection $\overrightarrow{\nabla}$. After reducing the structure group of $\mathcal{E}$ to ${\rm Spin}(p,q)$ by a choice of Weyl structure we have that $S(\mathcal{E})$ decomposes as
\begin{equation}
S(\mathcal{E}) = S_+(M)[-\tfrac{1}{2}] \oplus S_-(M)[\tfrac{1}{2}]
\end{equation}
where $S_+(M),S_-(M)$ are the spin bundles associated to $S_+,S_-$. A choice of Weyl structure also allows us to write the tractor connection as $\overrightarrow{\nabla}_X = \nabla_X + \rho(X) + \rho(\mathbf{P}(X))$ from which we deduce the tractor connection on $S(\mathcal{E})$ has the following form:
\begin{equation}
\overrightarrow{\nabla}_X = \left[ \begin{matrix} \nabla_X & \tfrac{1}{\sqrt{2}}{\bf g}^{-1}\mathbf{P}(X)\cdot \\ -\tfrac{1}{\sqrt{2}}X\cdot & \nabla_X \end{matrix} \right].
\end{equation}

Now suppose $\Psi$ is a covariantly constant spinor with respect to the tractor connection. Let
\begin{equation}
\Psi = \left[ \begin{matrix} \psi_1 \\ \psi_2 \end{matrix} \right]
\end{equation}
be the decomposition of $\Psi$ into $\pm 1$ eigenspaces of $e\epsilon$. The the equation $\overrightarrow{\nabla}\Psi = 0$ becomes
\begin{equation}\label{covconst}
\left[ \begin{matrix} \nabla_X \psi_1 +\tfrac{1}{\sqrt{2}}{\bf g}^{-1}\mathbf{P}(X)\cdot \psi_2 \\ \nabla_X \psi_2 - \tfrac{1}{\sqrt{2}}X\cdot \psi_1 \end{matrix} \right] = \left[ \begin{matrix} 0 \\ 0 \end{matrix} \right]
\end{equation}
for all vector fields $X$. Contracting the second of the two equations over a dual basis via Clifford multiplication yields
\begin{equation}\label{twispi2}
D\psi_2 + \tfrac{1}{\sqrt{2}}n \psi_1 = 0
\end{equation}
where $D$ is the Dirac operator and $n={\rm dim}(M)$. Substituting back into the second equation of (\ref{covconst}), we arrive at the {\em twistor spinor equation}:
\begin{equation}\label{twispi}
\nabla_X \psi_2 + \tfrac{1}{n}X \cdot D\psi_2 = 0.
\end{equation}

Now $\psi_2$ is the conformal weight $\tfrac{1}{2}$ component of $\Psi$. Let us now consider the parabolic subalgebra $\mathfrak{p} = \mathfrak{co}(V) \oplus V^*$. This preserves the subspace of conformal weight $-\tfrac{1}{2}$ spinors and gives rise to an exact sequence:
\begin{equation}
0 \to S_+[-\tfrac{1}{2}] \to S(\mathcal{E}) \to S_-[\tfrac{1}{2}] \to 0.
\end{equation}
Thus a spinor with values in $S(\mathcal{E})$ projects onto a weight $\tfrac{1}{2}$ spinor on $M$ and since the $V^*$ part of $\mathfrak{p}$ acts trivially on the projection, this spinor is independent of choice of Weyl structure.\\

Conversely suppose we have a spinor $\psi$ on $M$ of conformal weight $\tfrac{1}{2}$ that satisfies the twistor spinor equation (\ref{twispi}). We can show this lifts to a unique spinor $\Psi$ for the bundle $S(\mathcal{E})$ that is preserved by the tractor connection. We define $\Psi = (\psi_1 , \psi_2)^t$ by setting $\psi_2 = \psi$ and we define $\psi_1$ so as to satisfy equation (\ref{twispi2}) by $\psi_1 = -\sqrt{2}/n D\psi$. We then find that the second equation of (\ref{covconst}) holds. In fact the first equation holds as well although the computation is somewhat lengthy. Thus we have deduced:
\begin{prop}
A covariantly constant spinor for the tractor connection on $S(\mathcal{E})$ projects to a solution of the twistor spinor equation. Conversely a spinor satisfying the twistor spinor equation uniquely lifts to a covariantly constant section of $S(\mathcal{E})$.
\end{prop}


\subsection{Spinors in signatures (2,3) and (3,4)}
We now specialise to the case of $(2,3)$ conformal structures. Consider a basis $\alpha,\beta,c,u,v$ with dual basis $\hat{u},\hat{v},\hat{c},\hat{\alpha},\hat{\beta}$ for $V[-1]$ in which the metric becomes
\begin{equation}
{\bf g} = 4\hat{u}\hat{\beta} - 4\hat{v}\hat{\alpha} - \hat{c}^2.
\end{equation}
The utility of this basis is that it has a convenient matrix representation. We have that $\mathit{Cl}_{2,3} \simeq {\rm End}(\mathbb{R}^4) \oplus {\rm End}(\mathbb{R}^4)$ and the two $(2,3)$ spin representations are the corresponding copies of $\mathbb{R}^4$. Letting $S_{2,3}$ denote one of these copies we can find an explicit representation of $\mathit{Cl}_{2,3}$ on $S_{2,3}$ given by
\begin{equation}\label{matrixrep2}
\begin{aligned}
u &= \left[ \begin{matrix} 0 & 1 & 0 & -1 \\ 1 & 0 & -1 & 0 \\ 0 & 1 & 0 & -1 \\ 1 & 0 & -1 & 0 \end{matrix} \right], & v &= \left[ \begin{matrix} 0 & -1 & 0 & 1 \\ 1 & 0 & 1 & 0 \\ 0 & 1 & 0 & -1 \\ 1 & 0 & 1 & 0 \end{matrix} \right], & c &= \left[ \begin{matrix} 1 & 0 & 0 & 0 \\ 0 & -1 & 0 & 0 \\ 0 & 0 & 1 & 0 \\ 0 & 0 & 0 & -1 \end{matrix} \right], \\
\alpha &= \left[ \begin{matrix} 0 & 1 & 0 & 1 \\ -1 & 0 & 1 & 0 \\ 0 & 1 & 0 & 1 \\ 1 & 0 & -1 & 0 \end{matrix} \right], & \beta &= \left[ \begin{matrix} 0 & -1 & 0 & -1 \\ -1 & 0 & -1 & 0 \\ 0 & 1 & 0 & 1 \\ 1 & 0 & 1 & 0 \end{matrix} \right]. & &
\end{aligned}
\end{equation}

Viewing $E$ as the imaginary split octonions we can realise the metric $\tilde{g} = 4\hat{u}\hat{\beta} - 4\hat{v}\hat{\alpha} - \hat{c}^2 + e^2 - \epsilon^2$ as the standard metric through the following identification of bases:
\begin{equation}
\begin{aligned}
u &= k-lk, & v &= lj-j, & c &= l, & \alpha &= j + lj, \\
\beta &= k+lk, & e &= i, & \epsilon &= li. & &
\end{aligned}
\end{equation}
In particular we have that the generic distribution corresponds to the isotropic $2$-plane $V$ spanned by $u,v$ and the corresponding $3$-plane distribution $W$ is spanned by $u,v,c$. Up to scale the pure spinor for this distribution is
\begin{equation}
\psi = \left[ \begin{matrix} 0 \\ 1 \\ 0 \\ 1 \end{matrix} \right].
\end{equation}

Now $\tilde{G}_2$ can be realised as the subgroup of ${\rm SO}(3,4)$ fixing a spinor. Using the matrix form of $\tilde{\mathfrak{g}}_2$ given in (\ref{g2}) we can determine the spinor. For this to work we first note that the matrix of (\ref{g2}) is with respect to the basis $\{\tfrac{1}{\sqrt{2}}(e+\epsilon),\tfrac{1}{\sqrt{2}}v,\tfrac{1}{\sqrt{2}}u,c,\tfrac{1}{\sqrt{2}}\beta,\tfrac{1}{\sqrt{2}}\alpha,\tfrac{1}{\sqrt{2}}(e-\epsilon)\}$. From this one finds that up to scale, the unique spinor $\Psi$ preserved by $\tilde{G}_2$ is
\begin{equation}
\Psi = (\tfrac{1}{4}\alpha \beta \psi ,\psi)
\end{equation}
thus $\psi$ satisfies the twistor spinor equation. We summarise this result as follows:
\begin{prop}
Let $M$ be a $5$-manifold with a generic $2$-distribution $V$. On $M$ there is a naturally associated signature $(2,3)$ conformal structure with $\tilde{G}_2$ conformal holonomy, such that $V$ becomes a distribution of maximal isotropics. Locally there exists a pure spinor for $V$ which satisfies the twistor spinor equation.
\end{prop}

With this result we have come full circle in the following sense: starting with a generic $2$-distribution on a $5$-manifold we found a uniquely defined $(2,3)$ conformal structure with $\tilde{G}_2$ conformal holonomy. Since $\tilde{G}_2$ preserves a spinor, locally there is a constant spinor for the tractor connection. This spinor projects to a solution to the twistor spinor equation. Taking the annihilator of the twistor spinor yields a distribution of maximal isotropics which happens to be the distribution we began with.

\begin{rem}
We can also describe the $3$-distribution $W$ directly in terms of the twistor spinor $\psi$. It is the space of vectors satisfying $X \cdot \psi = k\psi$ for some real number $\psi$. Moreover one can easily verify that this space is the orthogonal complement of the $2$-distribution from this description.
\end{rem}


\section{Examples}\label{examples}
A well-known example of generic $2$-plane distributions in $5$ dimensions occurs in the description of two surfaces with metrics in contact. The distribution corresponds to motions of the surfaces without slipping or twisting. A second example is given by a class of differential equations. In both instances we can actually determine explicitly the conformal structure, though without further assumptions the conformal structure is too complicated to be particularly useful.


\subsection{Surfaces in contact}
We follow \cite{agrachev} in describing the motions of two rigid bodies in contact. We consider the two rigid bodies as Riemannian surfaces $M,\hat{M}$ with metrics $g,\hat{g}$ although the extension to split signature metrics is straightforward. For simplicity we assume $M$ and $\hat{M}$ are oriented. Then the configuration space $Q$ describing all possible ways in which $M$ and $\hat{M}$ can touch shall be defined as
\begin{equation}
Q = \{ \varphi : T_xM \to T_{\hat{x}}\hat{M} \, | \, x \in M , \hat{x} \in \hat{M} , \phi \textrm{ is an orientation preserving isometry} \}.
\end{equation}
Note that we are ignoring any constraints corresponding to admissibility of contact when embedded in $\mathbb{R}^3$. This is a $5$-dimensional principal $\rm{SO}(2)$-bundle fibred over the product $M \times \hat{M}$.

We consider the local description of $Q$. Let $\{e_1, e_2\}$ and $\{\hat{e}_1, \hat{e}_2\}$ denote oriented orthonormal frames for $M$ and $\hat{M}$ and let $\{e^1,e^2\}$ and $\{\hat{e}^1,\hat{e}^2\}$ be the dual coframes. Then $Q$ can be given local coordinates $(x,\hat{x},\theta)$ where $x \in M, \hat{x} \in \hat{M}$ and the point $\varphi(x,\hat{x},\theta) \in Q$ is the isometry sending $\{e_1,e_2\}$ to $\{\hat{e}_1,\hat{e}_2 \}$ rotated clockwise by an angle $\theta$, i.e.,
\begin{equation}
\varphi \left[ \begin{matrix} e_1 \\ e_2 \end{matrix} \right] = \left[ \begin{matrix} \cos{\theta} & \sin{\theta} \\ -\sin{\theta} & \cos{\theta} \end{matrix} \right] \left[ \begin{matrix} \hat{e}_1 \\ \hat{e}_2 \end{matrix} \right].
\end{equation}

We shall now show that on $Q$ there is a canonically defined connection. Let $\mathcal{F}M$ and $\mathcal{F}\hat{M}$ denote the $\rm{SO}(2)$-frame bundles for $M$ and $\hat{M}$. Let $V$ denote a $2$-dimensional vector space with the standard Riemannian metric. Then we can view $\mathcal{F}M$ as the set of isometries $u : V \to T_xM$ and similarly $\mathcal{F}\hat{M}$ is the set of isometries $v : V \to T_{\hat{x}}\hat{M}$. Therefore given $(u,v) \in \mathcal{F}M \times \mathcal{F}\hat{M}$ we can define the isometry $\varphi = v \circ u^{-1} : T_xM \to T_{\hat{x}}\hat{M}$. It follows that $Q$ can be expressed as $Q = \mathcal{F}\hat{M} \times_{{\rm SO}(2)} \mathcal{F}M$ where the left action on $\mathcal{F}M$ is just the inverse of the usual right action. Note that only because ${\rm SO}(2)$ is abelian do we get that the left and right actions on $\mathcal{F}M$ commute, and only for this reason $Q$ is a principal bundle.\\

On $\mathcal{F}M$ and $\mathcal{F}\hat{M}$ we have the Levi-Civita connection forms and one finds that the difference of these two forms is a connection form $\omega$ defined on $Q$. In fact since ${\rm SO}(2)$ is $1$-dimensional we can view $\omega$ as a differential form on $Q$. Let $B = \alpha e^1 + \beta e^2$ and $\hat{B} = \hat{\alpha}\hat{e}^1 + \hat{\beta}\hat{e}^2$ denote the Levi-Civita connection forms on $M$ and $\hat{M}$. Then
\begin{equation}
\omega = d\theta + (\alpha e^1 + \beta e^2) - (\hat{\alpha}\hat{e}^1 + \hat{\beta}\hat{e}^2).
\end{equation}
Note that the structure functions $\alpha,\beta$ are given by
\begin{equation}
[e_1,e_2] = \alpha e_1 + \beta e_2
\end{equation}
and similarly for $\hat{\alpha},\hat{\beta}$. We can show that $\omega$ is uniquely defined independent of the choice of coframes on $M$ and $\hat{M}$.\\

The motions of the two surfaces are described by curves $\gamma(t) : I \to Q$ from an interval $I$ into $Q$. The corresponding paths $x(t),\hat{x}(t)$ on $M$ and $\hat{M}$ are given by projecting $\gamma$ to $M$ and $\hat{M}$. We say that the motion $\gamma$ corresponds to rolling without slipping if
\begin{equation}
\varphi(\gamma(t))(\partial_t x(t) ) = \partial_t \hat{x}(t).
\end{equation}

We further say that $\gamma$ is a motion without twisting if a parallel vector field along $x(t)$ gets mapped under $\phi(\gamma)$ to a parallel vector field along $\hat{x}(t)$. This condition can be rephrased as saying that the canonical section of $\gamma^*{\rm Hom}(TM,T\hat{M})$ defined by $\gamma$ is covariantly constant with respect to the corresponding pull-back connection induced from the Levi-Civita connections. One readily checks that $\gamma$ defines a motion without twisting if and only if $\gamma^*(\omega)=0$, that is if and only if $\dot{\gamma}$ is horizontal.\\

Thus a motion $\gamma$ corresponds to rolling without slipping and without twisting if and only if $\dot{\gamma}$ lies in a $2$-dimensional distribution $V$ on $Q$ defined as
\begin{equation}
V = \{ X \in T_\varphi Q \, | \, \omega(X)=0, \; {\rm and} \; \varphi (\pi_*(X)) = \hat{\pi}_*(X) \}
\end{equation}
where $\pi$ and $\hat{\pi}$ denote the projections from $Q$ to $M$ and $\hat{M}$. We can further write down an explicit basis for $V$. Let $E_1,E_2,\hat{E}_1,\hat{E}_2$ denote the horizontal lifts of the vector fields $e_1,e_2,\hat{e}_1,\hat{e}_2$. Thus
\begin{equation}
\begin{aligned}
E_1 &= e_1 - \alpha \partial_{\theta}, & E_2 &= e_2 - \beta \partial_{\theta}, \\
\hat{E}_1 &= \hat{e}_1 + \hat{\alpha}\partial_{\theta}, & \hat{E}_2 &= \hat{e}_2 + \hat{\beta}\partial_{\theta}.
\end{aligned}
\end{equation}
Let us further define the following vector fields
\begin{equation}
\left[ \begin{matrix} \bar{E}_1 \\ \bar{E}_2 \end{matrix} \right] = \left[ \begin{matrix} \cos{\theta} & \sin{\theta} \\ -\sin{\theta} & \cos{\theta} \end{matrix} \right] \left[ \begin{matrix} \hat{E}_1 \\ \hat{E}_2 \end{matrix} \right].
\end{equation}
Then a basis for $V$ is given by $\{X_4,X_5\}$ where
\begin{equation}
\begin{aligned}
X_4 &= E_1 + \bar{E}_1 \\
X_5 &= E_2 + \bar{E}_2.
\end{aligned}
\end{equation}
Now one can verify the following commutation relation
\begin{equation}
[X_4,X_5] = -(k-\hat{k})\partial_{\theta} + \alpha X_4 + \beta X_5
\end{equation}
where $k$ and $\hat{k}$ are the Gaussian curvatures of $M$ and $\hat{M}$, that is $dB = ke^1 \wedge e^2$, $d\hat{B} = \hat{k}\hat{e}^1 \wedge \hat{e}^2$. We further have the following commutation relations:
\begin{equation}
\begin{aligned}
\left[ X_4,\partial_{\theta} \right] &= -\bar{E}_2 \\
\left[ X_5,\partial_{\theta} \right] &= \bar{E}_1.
\end{aligned}
\end{equation}
We conclude that the distribution $V$ is generic if and only if the difference in Gaussian curvatures $\eta = k-\hat{k}$ is non-vanishing. \\

From this point one can proceed to explicitly determine the conformal structure but we omit this long and tedious calculation. Note that motions of rolling without slipping or twisting correspond to curves that are tangent to the generic distribution. On the other hand we see that a null geodesic in the conformal structure that is initially tangent to the $2$-distribution will be tangent to the distribution at all other points. Therefore motion of the two surfaces rolling without slipping or twisting can be described by a class of null geodesic.


\subsection{Differential equation example}
Here we give an example of a generic distribution arising from a differential equation studied in \cite{nurowski1} and \cite{nurowski2}. Consider the differential equation
\begin{equation}\label{diffequ}
z' = F(x,y,y',y'',z)
\end{equation}
for two functions $y = y(x)$ and $z = z(x)$ of one variable $x$. The standard technique for studying such equations is to view (\ref{diffequ}) as defining a submanifold inside a jet bundle. Therefore let $J$ be the $5$-dimensional jet bundle consisting of variables $x,y,z,p=y',q=y''$. Then equation (\ref{diffequ}) is encoded in the following system of $1$-forms:
\begin{eqnarray}\label{forms}
\omega^1 &=& dz - F(x,y,z,p,q)dx \nonumber \\
\omega^2 &=& dy - pdx \\
\omega^3 &=& dp - qdx. \nonumber
\end{eqnarray}
The forms $\omega^2$, $\omega^3$ are the contact forms which determine when a section $\gamma:\mathbb{R} \to J$ is the prolongation of a section $\mathbb{R} \ni x \mapsto (x,y(x),z(x)) \in \mathbb{R}^3$, while the form $\omega^1$ represents the original differential equation. Thus a solution to (\ref{diffequ}) is equivalent to a section $\gamma: (a,b) \to J$ such that $\gamma^*(\omega^i)=0$ for $i=1,2,3$. \\

Define the vector field $D = \partial_x + p\partial_y + q\partial_p + F\partial_z$. Then the annihilator $V$ of $\omega^1,\omega^2,\omega^3$ is the $2$-dimensional distribution spanned by $D,\partial_q$. Thus we have a $5$-dimensional manifold $J$ with a $2$ distribution $V$ such that solutions of (\ref{diffequ}) are precisely the curves $\gamma \subset J$ tangent to $V$.\\

We consider now the question of genericity of $V$. First one has that $[\partial_q , D] = \partial_p + F_q \partial_z \ne 0 \;{\rm mod}(V)$. Thus $V$ is never integrable and taking commutators yields a $3$ distribution $W = V^{(1)}$. Taking commutators a second time, we find that $V^{(2)} = TJ$ if and only $F_{qq}$ is non-vanishing at each point of $J$. Thus the condition for genericity of $V$ is the non-vanishing of $F_{qq}$. \\

As in the previous example one can proceed to determine the corresponding conformal structure by a lengthy calculation. The conformal structure can be found in \cite{nurowski2}. We also note that solutions to the differential equation now correspond to null geodesics tangent to the distribution at one point (hence at all points).



\chapter{Coassociative submanifolds and semi-flat coassociative fibrations}\label{chap6}

As mentioned in the Introduction we expect that for $G_2$ manifolds there is an analogue of the SYZ conjecture involving coassociative fibrations. Therefore we study the deformation theory of coassociative submanifolds and the structure of coassociative fibrations, in particular a special class we call semi-flat where there is a torus action by isometries generating the fibres. The main result is that semi-flat coassociative fibrations are locally equivalent to positive definite minimal $3$-submanifolds in $\mathbb{R}^{3,3}$. We then reduce to equations on a surface which relate to minimal surfaces in a quadric in $\mathbb{R}^{3,3}$. This ties in with our earlier work on the affine Toda equations.\\

In Section \ref{secspeclag} we review the structure of the moduli space of deformations of a compact special Lagrangian submanifold in a Calabi-Yau manifold. Section \ref{seccoasub} adapts this approach to study the moduli space of deformations of a compact coassociative submanifold in a $G_2$-manifold. In Section \ref{seccyl} we show the relation between two moduli spaces for a $G_2$-manifold $X$ that is cylindrical, i.e. $X = Y \times S^1$ where $Y$ is a Calabi-Yau manifold. Section \ref{cfocm} considers coassociative fibrations of compact manifolds. We show that if the holonomy is equal to $G_2$ then such a fibration must have singularities. We give examples of torus fibrations which provide a model for the possible singularities we might expect. In Section \ref{secsemiflat} we study semi-flat coassociative fibrations. These are analogous to semi-flat special Lagrangian fibrations. In Section \ref{secconstruct} we prove Theorem \ref{thethm} giving a local equivalence between semi-flat $G_2$-manifolds and minimal $3$-manifolds in $\mathbb{R}^{3,3}$. In Section \ref{seccylsem} we show that in the cylindrical case the minimal submanifold equations reduce to the Monge-Amp\`ere equation. In Section \ref{secredtosurf} we show how an additional scaling symmetry of the semi-flat $G_2$-manifold corresponds to a minimal surface in a quadric. In Section \ref{secquadrics2} we consider the equations for a superconformal minimal surface in the unit quadric in $\mathbb{R}^{3,3}$. We finish with Section \ref{extension} in which we extend the results on semi-flat coassociative fibrations to the case of split $G_2$. This involves other real forms of the Toda equations which we can relate to our earlier work.


\section{Deformations of coassociative submanifolds}\label{defo}


\subsection{Deformations of special Lagrangian submanifolds}\label{secspeclag}
We review the case of special Lagrangian submanifolds, following Hitchin \cite{hit}. Let $X$ be a Calabi-Yau manifold of complex dimension $n$, $\omega$ the K\"ahler form and $\Omega_1$ the real part of the holomorphic volume form $\Omega$. A submanifold $L \subset X$ will be called {\em special Lagrangian} if it is a Lagrangian submanifold for $\omega$ such that $\Omega_1$ vanishes when restricted to $L$. Suppose $L \subset X$ is a compact special Lagrangian submanifold and let $M$ denote the (local) moduli space of special Lagrangian submanifolds isotopic to $L$. By McLean \cite{mclean}, we have
\begin{prop}\label{mclean1}
A normal vector field $\nu$ to a compact special Lagrangian submanifold $L$ is the vector field normal to a deformation through special Lagrangian submanifolds if and only if the corresponding $1$-form $\iota_\nu \omega|_L$ is harmonic. There are no obstructions to extending a first order deformation to a family of actual deformations.
\end{prop}

In particular this implies that the moduli space $M$ is a smooth manifold of dimension $b^1(L)$. For each $t \in M$ let $L_t$ denote the corresponding special Lagrangian submanifold. Fix a basepoint $t_0 \in M$ and let $L = L_{t_0}$. The proposition says we have a natural isomorphism $\lambda_t : T_tM \to \mathcal{H}^1(L_t,\mathbb{R})$, where $\mathcal{H}^1(L_t,\mathbb{R})$ is the space of harmonic $1$-forms on $L_t$. Since all the $L_t$ are isotopic we have (over a simply connected open subset of $M$) canonical isomorphisms $[ \; ] : \mathcal{H}^1(L_t,\mathbb{R}) \to H^1(L,\mathbb{R})$ by taking cohomology classes. Working locally, we will not distinguish between $M$ and a simply connected open subset in $M$. Therefore we have a natural $H^1(L,\mathbb{R})$-valued $1$-form $[\lambda]$ on $M$.

\begin{prop}\label{speclagprop}
The $1$-form $[\lambda]$ on $M$ is closed. Thus locally we may write $[\lambda] = du$ where $u : M \to H^1(L,\mathbb{R})$ is a local diffeomorphism, unique up to a translation.
\begin{proof}
Let $\mathcal{M} = M \times L$ and let $f: \mathcal{M} \to X$ represent the full local family of deformations. If $\pi: \mathcal{M} \to M$ is projection onto the first factor and we let $f_t$ denote $f|_{\pi^{-1}(t)}$ then $f_t$ is a diffeomorphism $f_t : \pi^{-1}(t) \to L_t \subset X$ and we can take $f_{t_0} = {\rm id}$. Since $L$ is a calibrated submanifold it is oriented and we have also assumed $L$ is compact, hence given a homology class $A \in H_1(L,\mathbb{R})$ let $\eta_A \in \Omega^{n-1}(L,\mathbb{R})$ represent the Poincar\'e dual, that is for any $\sigma \in H^1(L,\mathbb{R})$,
\begin{equation*}
\langle \sigma , A \rangle = \int_L \sigma \wedge \eta_A.
\end{equation*}
Extend $\eta_A$ to a form on $\mathcal{M} = M \times L$ in the natural way. We claim that the $1$-form $\langle \lambda , A \rangle$ is given by $p_* (f^*\omega \wedge \eta_A)$ where $p_*$ represents integration over the fibres of $\mathcal{M}$. Since $\omega$ is closed and integration over fibres takes closed forms to closed forms, this will prove the proposition.\\

Let $t^1,\dots ,t^m$ be local coordinates on $M$. Then by the definition of $f$ we have for each $i$ and each $t \in M$ a vector field representing the corresponding deformation $\nu_i : L_t \to TX|_{L_t}$ given by $\nu_i(a) = f_*(t,a)\partial_i$. Note that $\nu_i$ need not be normal to $L_t$. Define the $1$-forms $\theta_i$ on $L_t$ by $\theta_i = (\iota_{\nu_i} \omega)|_{L_t}$. Since $\omega|_{L_t}=0$ only the normal component of $\nu_i$ contributes so the $\theta_i$ are harmonic by \ref{mclean1}. Moreover since $M$ represents all local deformations the $\theta_i$ span $\mathcal{H}^1(L_t,\mathbb{R})$. We also have $[\lambda] = [\theta_i] dt^i$.

Now consider $f^*\omega$. Since each submanifold $L_t$ is Lagrangian we have $(f^*\omega)|_{\pi^{-1}(t)}=0$. It follows that we can write $f^*\omega = dt^i \wedge \hat{\theta}_i$ for some $1$-forms $\hat{\theta}_i$. Moreover we see that $\hat{\theta}_i|_{\pi^{-1}(t)} = (f_t)^*\theta_i(t)$. Now when we perform integration over the fibres of $\mathcal{M}$ we find
\begin{eqnarray*}
p_* (f^*\omega \wedge \eta_A) &=& p_* ( dt^i \wedge \hat{\theta}_i \wedge \eta_A) \\
&=& dt^i \int_{L_t} \hat{\theta_i}|_{\pi^{-1}(t)} \wedge \eta_A \\
&=& dt^i \int_A \hat{\theta_i}|_{\pi^{-1}(t)} \\
&=& dt^i\int_{(f_t)_*A} \theta_i \\
&=& \langle \lambda , A \rangle,
\end{eqnarray*}
where we have used the fact that $f_{t_0} = {\rm id}$.
\end{proof}
\end{prop}

So far we have not used the form $\Omega_1$. By a similar argument we have that a tangent vector in $T_tM$ corresponds to a harmonic $(n-1)$-form, in fact it is just the Hodge dual of the above $1$-forms. Therefore we have natural isomorphisms $\mu_t = *_t \lambda_t : T_tM \to \mathcal{H}^{n-1}(L_t,\mathbb{R})$, where $*_t$ is the Hodge star using the induced metric on $L_t$. Taking cohomology classes we get a $H^{n-1}(L,\mathbb{R})$-valued $1$-form on $M$, $[\mu]$. The same proof as before shows that $[\mu] = dv$ where $v: M \to H^{n-1}(L,\mathbb{R})$ is a local diffeomorphism defined up to a translation.\\

The moduli space $M$ has a natural metric $G$ induced by the $L^2$-metric on harmonic forms. Indeed if $U \in T_tM$ then $\lambda(U) \in \mathcal{H}^1(L_t,\mathbb{R})$ is the corresponding harmonic $1$-form and we take
\begin{equation}
G(U,U) = \int_L \lambda(U) \wedge * \lambda(U).
\end{equation}

But we know that $*\lambda(U) = \mu(U)$, so then $G(U,U)$ is just the Poincar\'e duality pairing. Thus $G$ is the pull-back under $F = (u,v) : M \to H^1(L,\mathbb{R}) \oplus H^{n-1}(L,\mathbb{R})$ of the natural duality induced metric. The space $H^1(L,\mathbb{R}) \oplus H^{n-1}(L,\mathbb{R})$ also has a natural symplectic form $\omega$, again from the duality pairing. Explicitly for $(u_1,u_2),(v_1,v_2) \in H^1(L,\mathbb{R}) \oplus H^{n-1}(L,\mathbb{R})$,
\begin{equation}\label{sympl}
\omega( (u_1,u_2),(v_1,v_2) ) =  \langle u_1 \smallsmile v_2 - v_1 \smallsmile u_2 , [L] \rangle.
\end{equation}
Using the algebraic property $\omega \wedge \Omega = 0$, contracting twice and restricting to $L$, we readily find that for $U,V \in T_tM$,
\begin{equation}
[\lambda(U)] \smallsmile [\mu(V)] -[\lambda(V)]\smallsmile [\mu(U)] = 0.
\end{equation}
Thus $F$ embeds $M$ as a Lagrangian submanifold of $H^1(L,\mathbb{R}) \oplus H^{n-1}(L,\mathbb{R})$.

\subsection{Deformations of coassociative submanifolds}\label{seccoasub}
We consider now the moduli space $M$ of deformations of a coassociative submanifold $L$ of a $G_2$-manifold. 

Let us briefly review some $G_2$-geometry. A general reference for $G_2$-manifolds is \cite{joyce}. By a $G_2$-manifold we mean a Riemannian manifold with holonomy contained in $G_2$. Recall that $G_2$ is the stabiliser of a generic $3$-form $\phi$ on $\mathbb{R}^7$ called the {\em associative $3$-form}. Since $G_2 \subset {\rm SO}(7)$, we also have that $G_2$ preserves a $4$-form $\psi$ called the {\em coassociative $4$-form}. We use the convention in \cite{joyce} for the standard $3$ and $4$ forms on $\mathbb{R}^7$. Hence if ${e^1, \dots, e^7}$ are a standard basis of $1$-forms for $\mathbb{R}^7$, the standard $3$ and $4$-forms $\phi_0,\psi_0$ are taken as
\begin{eqnarray}
\phi_0 &=& e^{123} + e^1 \! \! \wedge \! (e^{45} + e^{67}) + e^2 \! \! \wedge \! (e^{46} - e^{57}) + e^3 \! \! \wedge \! (-e^{47} - e^{56}) \\
\psi_0 &=& e^{4567} + e^{23} \! \! \wedge \! (e^{45} + e^{67}) + e^{31} \! \! \wedge \! (e^{46} - e^{57}) + e^{12} \! \! \wedge \! (-e^{47} - e^{56})
\end{eqnarray}
where we use the notation $e^{ij\dots k}=e^i \wedge e^j \wedge \dots \wedge e^k$.\\

A $G_2$-structure on a $7$-manifold thus defines a $3$-form $\phi$, a $4$-form $\psi$ and a metric $g$. A well known result \cite{fergr} is that for a $G_2$-structure the Levi-Civita connection for $g$ has $G_2$-holonomy if and only if $\phi$ and $\psi$ are closed.\\

In \cite{harvlaw} the notion of calibrations and calibrated submanifolds were defined. Moreover the $3$ and $4$-forms $\phi$ and $\psi$ are shown to be calibrations. The corresponding calibrated submanifolds are called {\em associative} and {\em coassociative} submanifolds respectively. It is also shown in \cite{harvlaw} that a $4$-dimensional submanifold $L$ is {\em coassociative} if and only if $\phi|_L=0$ and in such a case $L$ inherits an orientation $\psi|_L$. The deformation theory of coassociative submanifolds was considered in \cite{mclean} where it was shown that
\begin{prop}
A normal vector field $\nu$ to a compact coassociative submanifold $L$ is the vector field normal to a deformation through coassociative submanifolds if and only if the corresponding $2$-form $\iota_\nu \phi|_L$ is a closed self-dual $2$-form, hence harmonic. There are no obstructions to extending a first order deformation to a family of actual deformations.
\end{prop}

Actually, McLean uses a convention for the associative $3$-form leading to anti-self dual harmonic forms. Once again we have a moduli space $M$ of coassociative submanifolds isotopic to $L$. The proposition shows that $M$ is a smooth manifold of dimension $b^2_+(L)$.

We now further investigate the structure of the moduli space. From the proposition we have for each $t\in M$ a natural isomorphism $\lambda_t : T_tM \to \mathcal{H}^2_+(L_t,\mathbb{R})$, where $\mathcal{H}^2_+(L_t,\mathbb{R})$ is the space of harmonic self-dual $2$-forms on $L_t$. Since the self-dual cohomology on $L_t$ depends on the metric which varies with $t$ we do not have a natural identification of all the $\mathcal{H}^2_+(L_t,\mathbb{R})$. However we can still take the cohomology class of $\lambda_t$ giving an element of $H^2(L,\mathbb{R})$. Therefore we have a natural $H^2(L,\mathbb{R})$-valued $1$-form $[\lambda]$ on $M$. Repeating the proof of \ref{speclagprop} we find
\begin{prop}
The $1$-form $[\lambda]$ on $M$ is closed. Thus locally we may write $[\lambda] = du$ where $u : M \to H^2(L,\mathbb{R})$ is an immersion unique up to a translation.
\end{prop}

We will see that the map $u$ is in a sense the correct analogy of the map $F$ obtained in the special Lagrangian case. From the previous proposition we immediately have
\begin{prop}
The natural $L^2$-metric $G$ on $M$ given by
\begin{equation}
G(U,U) = \int_L \lambda(U) \wedge \lambda(U)
\end{equation}
is the pull-back under $u:M \to H^2(L,\mathbb{R})$ of the non-degenerate inner product on $H^2(L,\mathbb{R})$ given by the intersection form. Since $u$ maps each tangent space to a positive definite subspace, $G$ is positive definite.
\end{prop}

Compared with the special Lagrangian case, we have lost the Lagrangian aspect of $M$ as there is no canonical symplectic structure on $H^2(L,\mathbb{R})$.\\

Next we turn to global issues. Suppose we have two points $s,t \in M$ and two curves joining them. Along the curves we get isotopies between $L_s$ and $L_t$ and hence corresponding isomorphisms of their cohomology. If the two curves are not homotopic these two isomorphisms need not agree, though in either case a basis for $H^2(L_s,\mathbb{Z})$ (modulo torsion) will be sent to a basis for $H^2(L_t,\mathbb{Z})$ preserving the intersection form. Therefore the ambiguity is an element of ${\rm \bf SO}(Q,\mathbb{Z}) = {\rm \bf SO}(Q) \cap {\rm \bf SL}(b^2,\mathbb{Z})$ where $Q$ denotes the intersection form and $b^2$ is the second Betti number of $L$. The unit determinant condition follows since for any closed curve we have an isotopy through coassociative submanifolds preserving the orientation induced by the calibrating form $\psi$. In the case where $L = T^4$, cohomology is generated by the $1$-cocycles so the ambiguity can be thought of as an element of ${\rm \bf SL}(4,\mathbb{Z})$ acting via the representation ${\rm \bf SL}(4,\mathbb{Z}) \to {\rm \bf SO}(3,3,\mathbb{Z})$ on $2$-forms.\\

The above ambiguity prevents us from defining $[\lambda]$ as a global $1$-form on $M$. However we can define $[\lambda]$ on the universal cover $\hat{M}$ of $M$. Then since $[\lambda]$ is closed and $\hat{M}$ simply connected we may write $[\lambda] = d\hat{u}$ giving a developing map $\hat{u} : \hat{M} \to H^2(L,\mathbb{R})$ of the universal cover $\hat{M}$ of $M$. Note that while $[\lambda]$ is defined on $M$ up to ambiguity in ${\rm \bf SO}(Q,\mathbb{Z})$, there is a second ambiguity in $u$, namely translations. Therefore the monodromy representation has the form $\rho : \pi_1(M) \to {\rm \bf SO}(Q,\mathbb{Z}) \ltimes \mathbb{R}^{b^2(L)}$.

\subsection{Cylindrical $G_2$-manifolds}\label{seccyl}
Following \cite{joy2} we call a $G_2$-manifold {\em cylindrical} if it has the form $\hat{X} = X \times S^1$ where $X$ is a Calabi-Yau manifold. In this case we expect a relation between the moduli space of special Lagrangian submanifolds $L \subset X$ and the corresponding coassociative submanifold $\hat{L} = L \times S^1 \subset \hat{X}$. This will be shown to be the case and justifies our claim that we have found the analogous structure on the moduli space of coassociative submanifolds.\\

\begin{defn}
We say that a $G_2$-manifold $(\hat{X},\phi,\hat{g})$ is {\em cylindrical} if there is a Calabi-Yau manifold $(X,g,\omega,\Omega)$ such that $\hat{X}$ is diffeomorphic to $X \times T$, where $T$ is the flat torus $(\mathbb{R}/\mathbb{Z},dt^2)$, and such that the metric $\hat{g}$ and $3$-form $\phi$ are given by
\begin{eqnarray}
\hat{g} &=& g + dt^2 \\
\phi &=& {\rm Re}(\Omega) + \omega \wedge dt.
\end{eqnarray}
\end{defn}

To be more explicit give $\mathbb{C}^3$ a dual basis $(z^1,z^2,z^3)=(e^1+ie^4,e^2+ie^5,e^3+ie^6)$. Let $\omega$ be the K\"ahler form, $\Omega = \Omega_1 + i\Omega_2$ the holomorphic $3$-form, namely
\begin{eqnarray}
\omega &=& e^1 \wedge e^4 + e^2 \wedge e^5 + e^3 \wedge e^6 \\
\Omega &=& (e^1+ie^4)\wedge (e^2+ie^5) \wedge (e^3 + ie^6).
\end{eqnarray}
Therefore
\begin{equation}
\phi = e^{123} - e^{345} + e^{246} - e^{156} + (e^{14} + e^{25} + e^{36})\wedge dt.
\end{equation}
The volume form is $-e^{123456}\wedge dt$. Therefore we have
\begin{eqnarray}
\psi &=& -(e^{456}-e^{126}-e^{153}-e^{423})\wedge dt - e^{2536} - e^{1436} - e^{1425}\\
&=& {\rm Im}(\Omega) \wedge dt - \tfrac{1}{2}\omega^2.
\end{eqnarray}
It is clear that the last identity $\psi = {\rm Im}(\Omega) \wedge dt - \tfrac{1}{2}\omega^2$ will be true for any cylindrical $G_2$-manifold.

Let $L \subset X$ be a compact special Lagrangian submanifold, then $\hat{L} = L \times T \subset \hat{X}$ is a coassociative submanifold. We wish to compare the deformation theory of these two submanifolds. In fact they are easily seen to be the same as follows from the following straightforward proposition
\begin{prop}
Let $(L,g)$ be a compact oriented Riemannian $3$-manifold and let $(\hat{L},\hat{g})$ be the Riemannian $4$-manifold with $\hat{L} = L \times S^1$, $\hat{g} =g + dt^2$ and orientation $dV_{\hat{L}} = dV_L \wedge dt$. Let $\mathcal{H}^1(L)$ denote harmonic $1$-forms on $L$ and $\mathcal{H}^2_+(\hat{L})$ the harmonic self-dual $2$-forms on $\hat{L}$. Then the map $e : \mathcal{H}^1(L) \to \mathcal{H}^2_+(\hat{L})$ defined by
\begin{equation}\label{e}
e(\alpha) = \tfrac{1}{\sqrt{2}}(*_g \alpha + \alpha \wedge dt)
\end{equation}
is an isomorphism and an isometry of $L^2$ metrics.
\end{prop}

Now suppose we have a deformation of $L$ in $X$ through special Lagrangian submanifolds. Let $\nu$ be the corresponding normal vector field on $L$. We can also consider $\nu$ to be a normal vector field to $\hat{L}$. Then
\begin{eqnarray*}
\iota_\nu \phi &=& \iota_\nu \Omega_1 + \iota_\nu \omega \wedge dt \\
&=& *_g \iota_\nu \omega + \iota_\nu \omega \wedge dt \\
&=& \sqrt{2}e(\iota_\nu \omega).
\end{eqnarray*}

We see now that not only do tangent vectors in the moduli space coincide, but the corresponding metrics coincide, up to a constant factor. Therefore, at least locally, the moduli space of deformations of $L$ in $X$ and $\hat{L}$ in $\hat{X}$ are isometric. More explicitly let $M$ be the (local) moduli space. Let $F = (u,v) : M \to H^1(L,\mathbb{R}) \oplus H^2(L,\mathbb{R})$ be the natural immersion obtained in the special Lagrangian case. Let $U : M \to H^2(\hat{L},\mathbb{R})$ be the natural immersion in the coassociative case. The map $e$ of equation (\ref{e}) can be replaced by a more general isometry $e : H^1(L,\mathbb{R}) \oplus H^2(L,\mathbb{R}) \to H^2(\hat{L},\mathbb{R})$ given by
\begin{equation}
e(\alpha,\beta) = \tfrac{1}{\sqrt{2}}(\beta + \alpha \smallsmile [dt]).
\end{equation}
Then we may take $U = \sqrt{2} e \circ F$ (recall $F$ and $U$ are only defined up to a translation).

\section{Coassociative fibrations}\label{coassf}
We now turn our attention to the case of coassociative fibrations. In particular we will consider a relatively simple class of fibrations we call {\em semi-flat} (perhaps more correctly $4/7$ flat) which extend the similar notion of a semi-flat special Lagrangian fibration. To start though, we look at some global aspects of coassociative fibrations of compact manifolds in Section \ref{cfocm}.

\subsection{Coassociative fibrations of compact manifolds}\label{cfocm}
We will be considering fibrations $\pi : X \to B$ of $G_2$-manifolds such that the fibres are coassociative submanifolds. Such fibrations may have important application to string theory and M-theory \cite{gyz}. 

Given the general difficulty of constructing $G_2$-manifolds it should be no surprise that finding coassociative fibrations with $G_2$-holonomy is a difficult task. It follows from the work of Bryant \cite{bry2} that many examples of coassociative fibrations exist, but these need not be compact or complete. An example of a complete coassociative fibration $S^3 \times \mathbb{R}^4 \to S^3$ is found in \cite{brsa}.\\

We can also consider fibrations which degenerate, that is $\pi$ need not be a locally trivial fibre bundle. In fact we will show that if $X$ is compact and has finite fundamental group then any fibration coassociative or not must degenerate. Even allowing degeneracy it is still difficult to find coassociative fibrations. The author is aware of only one class of examples of compact coassociative fibrations of $G_2$-holonomy \cite{kov}.

\begin{prop}
Let $X$ be a compact $G_2$-manifold with finite fundamental group. Then $X$ admits no locally trivial fibre bundles $X \to B$ onto a $3$-dimensional base.
\begin{proof}
We may assume $X$ is connected. Further it suffices to replace $X$ by its universal cover which is also compact, so we assume $X$ is simply connected. Let $\pi : X \to B$ be a locally trivial fibre bundle where $B$ is a $3$-manifold and let $F$ be one of the fibres. If $F$ is not connected we may replace $B$ by its universal cover $\tilde{B}$ to get another fibration $\pi_1 : X \to \tilde{B}$ with fibre equal to a connected component of $F$. So we may as well assume $B$ is simply connected and $F$ is connected.

Now since $B$ is a compact simply connected $3$-manifold it is oriented and $H_1(B,\mathbb{Z}) = H_2(B,\mathbb{Z}) = 0$. Then by the Hurewicz theorem $\pi_2(B) = 0$ also (of course the Poincar\'e conjecture implies the base is diffeomorphic to the $3$-sphere but we don't need this fact). The long exact sequence of homotopy groups implies that $\pi_1(F) = 0$. Further $F$ must be oriented since $X$ and $B$ are.\\

We will make use of the Leray-Serre spectral sequence (with coefficients in $\mathbb{R}$) in order to gain information on the cohomology of $X$. Since $B$ is simply connected we have $E^{p,q}_2 = H^p(B,\mathbb{R}) \otimes H^q(F,\mathbb{R})$. In order to calculate $H^2(X,\mathbb{R})$ and $H^3(X,\mathbb{R})$ there is only one relevant non-trivial differential to consider $d_3 : E^{0,2}_3 \to E^{3,0}_3$, with $E^{0,2}_3 = H^2(F,\mathbb{R})$ and $E^{3,0}_3 = H^3(B,\mathbb{R})$. We then have $H^2(X,\mathbb{R}) = {\rm ker}(d_3)$ and $H^3(X,\mathbb{R}) = {\rm coker}(d_3)$. However the differential $d_3$ must vanish for otherwise we have $H^3(X,\mathbb{R}) = 0$ which is impossible on a compact $G_2$-manifold. Thus $d_3 = 0$ and it follows that the following maps are isomorphisms
\begin{equation}
\begin{aligned}
i^* &: H^2(X , \mathbb{R}) \to H^2(F,\mathbb{R}) \\
\pi^* &: H^3(B,\mathbb{R}) \to H^3(X,\mathbb{R})
\end{aligned}
\end{equation}
where $i^*$ is induced from the inclusion of some fibre $i : F \to X$ and $\pi^*$ is induced by the projection $\pi : X \to B$.\\

Let $\phi$ be the $G_2$ $3$-form on $X$. Then the cohomology class of $\phi$ has the form $\left[ \phi \right] = c \pi^* \left[ dvol_B \right]$ where $dvol_B$ is a volume form on $B$ such that $\int_B dvol_B = 1$ and $c$ is some non-vanishing constant. We claim that $\pi^*(dvol_B)$ is Poincar\'e dual to the fibre $F$. This is a straightforward consequence of fibre integration. Therefore if $\mu$ is a closed $4$-form on $X$ we have
\begin{equation}\label{poinc}
\int_F i^*\mu = \int_X \mu \wedge \pi^*(dvol_B) = c^{-1} \int_X \mu \wedge \phi.
\end{equation}
Let us recall two cohomological properties of compact $G_2$ manifolds \cite{joyce} where we continue to assume that $H^1(X,\mathbb{R}) = 0$. First there is a symmetric bilinear form $\langle \, , \, \rangle$ on $H^2(X,\mathbb{R})$ given by
\begin{equation}\label{pairing}
\langle \eta , \xi \rangle = \int_X \eta \wedge \xi \wedge \phi.
\end{equation}
This form is negative definite. Secondly if $p_1(X) \in H^4(X,\mathbb{R})$ is the first Pontryagin class of $X$ then
\begin{equation}\label{lessthanzero}
\int_X p_1(X) \wedge \phi \, < 0.
\end{equation}
Combining (\ref{pairing}) with (\ref{poinc}) and the fact that $i^* : H^2(X,\mathbb{R}) \to H^2(F,\mathbb{R})$ is an isomorphism we find that the intersection form on $F$ is negative definite. Hence Donaldson's theorem implies the intersection form of $F$ is diagonalisable, i.e. of the form ${\rm diag}(-1,-1, \dots , -1)$. Now $F$ is a spin manifold because $X$ is spin and the normal bundle $N$ of $F$ in $X$ is trivial so that $w_2(TF) = w_2(TF \oplus N) = i^* w_2(TX) = 0$ where $w_2$ denotes the second Stiefel-Whitney class. But now from Wu's formula the intersection form of $F$ must be even. Therefore the intersection form must be trivial and $H^2(F,\mathbb{R}) = 0$.

Now we also have that $p_1(F) = p_1(TF \oplus N) = i^*p_1(X)$. Therefore
\begin{equation*}
\int_X p_1(X) \wedge \phi = c \int_F p_1(F) = 0
\end{equation*}
where the last equality follows from the Hirzebruch signature theorem. But this contradicts (\ref{lessthanzero}), hence such fibre bundles $\pi : X \to B$ can not exist.
\end{proof}
\end{prop}

\begin{rem}
The above proposition assumed that $X$ is a $G_2$-manifold with finite fundamental group. For a compact $G_2$-manifold $X$ this is equivalent to the holonomy of $X$ being equal to the whole of $G_2$ \cite{joyce}.
\end{rem}

Next we consider the question of what smooth fibres a compact coassociative fibration can have. Clearly if $F$ is such a fibre we must have $b^2_+(F) \ge 3$ and moreover $\wedge^2_+ T^*F$ has a trivialisation by harmonic forms (in particular $F$ has an ${\rm SU}(2)$-structure). From \cite{bry2} any such $4$-manifold is a fibre of a coassociative fibration, though it need not be compact or complete. We also have:
\begin{prop}\label{riemfibr}
Let $\pi : X \to B$ be a coassociative fibration with compact fibres. Then the base $B$ can be given a metric such that $\pi$ is a Riemannian submersion (around non-singular fibres) if and only if the fibres are Hyperk\"ahler. Moreover in this case the base identifies with the moduli space of deformations of a fibre and the base metric $g_B$ is related to the moduli space $L^2$ metric $g_{L^2}$ by
\begin{equation}\label{basemetric}
g_B = \frac{1}{2 {\rm vol(F)}}g_{L^2}.
\end{equation}
\begin{proof}
Let $b \in B$ and $F = \pi^{-1}(b)$. Choose a basis $v_i$ for $T_bB$ and let $\tilde{v}_i$ be the horizontal lifts. Define corresponding harmonic $2$-forms $\omega_i$ by $\omega_i = \iota_{\tilde{v}_i} \phi |_F$. Then the $\omega_i$ are a frame for $\wedge^2_+T^*F$. Now if $g$ is the metric on $X$ then
\begin{equation}
g(A , B) dvol_X = \frac{1}{6}\phi \wedge \iota_A \phi \wedge \iota_B \phi.
\end{equation}
From this it follows that
\begin{equation}\label{wedge}
\omega_i \wedge \omega_j = 2g(\tilde{v}_i , \tilde{v}_j ) dvol_F.
\end{equation}
Now suppose $B$ has a metric such that $\pi$ is a Riemannian submersion. Choose the $v_i$ to be an orthonormal basis. Then the $\omega_i$ are Hyperk\"ahler forms on $F$. Conversely suppose $F$ is a compact Hyperk\"ahler $4$-manifold. Then since $b^2_+(F) = 3$ we have that the space of self-dual harmonic $2$-forms is $3$-dimensional, spanned by the Hyperk\"ahler forms. Thus the $\omega_i$ are constant linear combinations of the Hyperk\"ahler forms, hence $g(\tilde{v}_i , \tilde{v}_j)$ is constant along the fibres. So $B$ can be given a metric $g_B$ making $\pi$ a Riemannian submersion. In this case since $b^2_+(F) = 3$ we see that the base exhausts all deformations of a fibre through coassociative submanifolds so that the base identifies with the moduli space of deformations. Moreover if we integrate (\ref{wedge}) over $F$ we get (\ref{basemetric}).
\end{proof}
\end{prop}
Note that a smooth Hyperk\"ahler $4$-manifold is either a torus or a K3 surface. Therefore these are likely candidates for the fibres of a compact coassociative fibration.\\

We now move on to the question of what sort of singularities can occur for a compact coassociative fibration $f:X \to B$ and what does the discriminant locus $\Delta = \{ b \in B \, | \, \exists \, x \in f^{-1}(b), \, {\rm rank}(df_x) < 3 \} \subset B$ look like?

In the case of special Lagrangian fibrations $Y \to B$ that are sufficiently well behaved the discriminant locus has codimension $2$ \cite{bai}. Under assumptions that the singularities are well behaved Baier \cite{bai} shows that if $\Delta$ is smooth then $Y$ has vanishing Euler characteristic, so generally we expect $\Delta$ not to be smooth.

Returning to the case of coassociative fibrations $X \to B$ we might likewise expect under reasonable assumptions on the singularities that the discriminant locus $\Delta \subset B$ has codimension $2$. Interestingly the constraint on smoothness in the special Lagrangian case no longer seems to be an issue since $X$ is odd dimensional, hence always has vanishing Euler characteristic. Moreover Kovalev has constructed examples of compact coassociative K3 fibrations with discriminant locus a smooth link \cite{kov}.\\

We will not investigate the issue of singularities in any depth and instead we will simply provide a model for producing examples of compact coassociative fibrations on manifolds with $G_2$-structures with torsion, that is the $3$-form will not be closed.

Suppose we have a compact Hyperk\"ahler $8$-manifold $M$ with holomorphic Lagrangian fibration $\pi : M \to \mathbb{CP}^2$ (with singularities). That is if $\omega_I, \omega_J, \omega_K$ are the Hyperk\"ahler forms then the non-singular fibres of $\pi$ are complex submanifolds with respect to $\omega_I$ and are Lagrangian with respect to the holomorphic symplectic form $\Omega = \omega_J + i\omega_K$. Note that the non-singular fibres are necessarily tori. We can give $M$ the structure of a ${\rm Spin}(7)$-manifold where the $4$-form is
\begin{equation}
\Phi = \frac{1}{2}\omega_I^2 + \frac{1}{2}\omega_J^2 - \frac{1}{2}\omega_K^2.
\end{equation}
This makes the fibres of $\pi$ into Cayley $4$-folds, indeed since we can write $\Phi = \omega_I^2/2 + \Omega \wedge \overline{\Omega}/2$ it follows that $\Phi|_{{\rm ker}\pi_*} = \omega_I^2/2$ which is the volume form on the fibres since the fibres are complex submanifolds. We have therefore produced an example of a compact Cayley $4$-fold fibration (albeit with holonomy in ${\rm Sp}(2)$). Note that the discriminant locus $\Delta \subset \mathbb{CP}^2$ is an algebraic curve $\mathbb{CP}^2$.

To get a coassociative fibration let us take a smooth embedded $3$-sphere $S^3 \subset \mathbb{CP}^2$ that meets $\Delta$ transversally and avoids any singular points of $\Delta$. Therefore if $\Delta$ has real dimension $2$ then $\Delta \cap S^3$ will be a smooth link in $S^3$. Let $X = \pi^{-1}(S)$. We claim that $X$ has an almost $G_2$ structure such that $\pi: X \to S^3$ is a coassociative fibration with discriminant locus the link $\Delta \cap S^3$. To see this let $v$ be the unit normal to $TX$ and let $v^* = g(v, \, )$ where $g$ is the metric on $M$. Then on $TM|_X$ we may write
\begin{equation}
\Phi = v^* \wedge \phi + \psi
\end{equation}
where $\iota_v \phi =  \iota_v \psi = 0$. Then $\phi$ defines an associative $3$-form on $X$ which is generally not closed and $\psi$ is the corresponding $4$-form. Note on the other hand that $\psi$ is closed since $\Phi|_X = \psi$. This also shows that the smooth fibres of $\pi$ are coassociative submanifolds.
To make this example more explicit let $f:S \to \mathbb{CP}^1$ be an elliptic fibration of a K3 surface over $\mathbb{CP}^1$. Then we will take our Hyperk\"ahler $8$-manifold $M$ to be the Hilbert scheme ${\rm Hilb}^2S$ of pairs of points on $S$. This can be concretely described as follows: the product $S \times S$ has a natural $\mathbb{Z}_2$-action given by exchanging points. The Hilbert scheme ${\rm Hilb}^2S$ is obtained from the quotient $(S \times S) / \mathbb{Z}_2$ by blowing up the diagonal. Now if we combine the map $M = {\rm Hilb}^2S \to (S \times S) / \mathbb{Z}_2$ with the map $f : S \to \mathbb{CP}^1$ we see that we can map each point of $M$ to a pair of unordered points in $\mathbb{CP}^1$ which we can think of as the zeros of a quadratic polynomial, hence we get an induced map $\pi : M \to \mathbb{CP}^2$. This is in fact an example of holomorphic symplectic fibration.\\

To complete our example we should discuss what sort of singularities can be obtained and what the discriminant locus looks like. Matsushita \cite{mat} determines the types of singularities that can occur (except over a finite set of points) in a holomorphic symplectic fibration of the type we are considering. In fact most of the singularities listed in \cite{mat} are realised by taking the Hilbert scheme of two points on a K3 elliptic fibration in the previously described way. From this one should be able to the determine the monodromy representation on the fibre homology. It is also known the type of links that can occur as the result of placing a $3$-sphere around a singularity of a curve in $\mathbb{CP}^2$ \cite{boi}. For example a singularity of the form $z_1^p + z_2^q = 0$ (in affine coordinates) where $p$ and $q$ are coprime leads to a $(p,q)$-torus knot.

\subsection{Semi-flat coassociative fibrations}\label{secsemiflat}

\begin{defn}
Let $X$ be a $G_2$-manifold. We say that $X$ is {\em semi-flat} if there is a $T^4$-action on $X$ preserving the $3$-form and such that the orbits are coassociative tori.
\end{defn}

In this section we will denote the metric associated to a $G_2$-manifold $X$ by $g$. The cross product $ \times : \wedge^2 TX \to TX$ is then defined by the relation
\begin{equation}
g( A \times B , C) = \phi(A,B,C)
\end{equation}
for all tangent vectors $A,B,C$ in a given tangent space.\\

Let $X$ be a semi-flat $G_2$-manifold. We have a distribution $V$ on $X$ tangent to the $T^4$-action which we call the vertical distribution. We also have the corresponding distribution of normals $V^\perp$ which we call the horizontal distribution. We will show that the horizontal distribution is integrable.

\begin{lemp}
Let $K$ be a vector field preserving the $3$-form on a $G_2$-manifold. Then $K$ preserves the metric and $4$-form.
\begin{proof}
We note that for any two tangent vectors $A,B$
\begin{equation}\label{gdvol}
g(A , B) dvol = \frac{1}{6}\phi \wedge \iota_A \phi \wedge \iota_B \phi.
\end{equation}
This is an isomorphism $TX \to T^*X \otimes \wedge^7 T^*X$. Taking the determinant gives an isomorphism $\wedge^7 TX \to (\wedge^7 T^*X)^8$ or a section of $(\wedge^7 T^*X)^9$. This is essentially $dvol^9$. Since $K$ preserves $\phi$ it now follows that $K$ preserves $dvol$ and hence by (\ref{gdvol}) the metric. Finally $\psi = *\phi$ must then be preserved as well.
\end{proof}
\end{lemp}

\begin{lemp}\label{formofphi}
Let $X$ be a $G_2$-manifold. Suppose $V$ is a distribution of coassociative subspaces and $V^\perp$ the corresponding orthogonal distribution of associative subspaces. Then locally there exists orthonormal frames $\{e_j\}$, $1\le j \le 4$ for $V$ and $\{a_i\}$, $1\le i \le 3$ for $V^\perp$ such that the $3$-form $\phi$ and $4$-form $\psi$ have the following forms with respect to the corresponding coframe:
\begin{eqnarray}
\phi \! \! &=& \! \! a^{123} + a^1 \! \! \wedge \! (e^{12}+e^{34}) + a^2 \! \! \wedge \! (e^{13} - e^{24}) + a^3 \! \! \wedge \! (-e^{14} - e^{23}) \label{standardphi} \\
\psi \! \! &=& \! \! e^{1234} + a^{23} \! \! \wedge \! (e^{12}+e^{34}) + a^{31} \! \! \wedge \! (e^{13} - e^{24}) + a^{12} \! \! \wedge \! (-e^{14} - e^{23}) \label{standardpsi}.
\end{eqnarray}
\begin{proof}
Follows since $G_2$ is transitive on the set of associative (or coassociative) subspaces \cite{harvlaw}.
\end{proof}
\end{lemp}

\begin{prop}
Let $X$ be a semi-flat $G_2$-manifold. The horizontal distribution is integrable.
\begin{proof}
We denote the vertical distribution spanned by the $T^4$-action by $V$ and the corresponding horizontal distribution by $V^\perp$. Let $e_j$, $j=1,2,3,4$ denote a basis for the infinitesimal action. The $e_j$ define a commuting frame for $V$. Moreover since the $e_j$ span a coassociative distribution, the products $e_j \times e_k$ are horizontal vector fields that span $V^\perp$. For any vector fields $A,B$ we have
\begin{equation}\label{leibnitz}
\mathcal{L}_{e_j}(A \times B) = \mathcal{L}_{e_j}(A) \times B + A \times \mathcal{L}_{e_j}(B).
\end{equation}
Indeed for any vector field $C$ we have the relation $g( A \times B , C ) = \phi(A,B,C)$. Taking the Lie derivative with respect to $e_j$ we obtain (\ref{leibnitz}). If we take $A = e_k, B = e_l$ we now obtain $[e_j , e_k \times e_l] = 0$. We can therefore take any three linearly independent products $e_k \times e_l$ to obtain a frame for $V^\perp$. Denote such a frame $a_1,a_2,a_3$. Therefore we have $[e_j, a_i] = [e_j , e_k] = 0$. Let $e^j$ and $a^i$ denote the corresponding coframe. We then have $\mathcal{L}_{e_j}(e^k) = \mathcal{L}_{e_j}(a^i) = 0$.\\

We now introduce a bi-grading on differential forms as follows. We identify $V^*$ and $(V^\perp)^*$ as subbundles of $T^*X$. Then a differential form of type $(p,q)$ is a section of $\wedge^p (V^\perp)^* \otimes \wedge^q V^*$. It follows directly that Lie differentiating with respect to $e_j$ preserves the bi-grading. Let $\alpha$ be a $(p,q)$-form invariant under the $T^4$-action. We claim that $d\alpha$ is a sum of homogeneous components, each of type $(a,b)$ with $b \le q$. Indeed we have $de^j(e_k,e_l) = -e^j([e_k,e_l]) = 0$ and $de^j(e_k,a_i) = -e^j([e_k,a_i]) = 0$ so $de^j$ is of type $(2,0)$. Similarly $da^i$ is of type $(2,0)$. The general result now follows.\\

It follows from (\ref{formofphi}) that $\phi$ is a sum of $(3,0)$ and $(1,2)$ terms and $\psi$ is a sum of $(2,2)$ and $(0,4)$ terms. Now since $\phi$ and $\psi$ are preserved by the $e_j$, their grading homogeneous components must also be preserved. In particular if we decompose $\psi$ as $\psi = \psi_{(0,4)} + \psi_{(2,2)}$, we find that $d(\psi_{(0,4)})$ is of type $(1,4)+(2,3)$ while $d(\psi_{(2,2)})$ is of type $(3,2)$. Therefore since $d\psi = 0$ the $(2,3)$ component of $d(\psi_{(0,4)})$ must vanish. Now since the vertical $V$ consists of coassociative subspaces $\psi_{(0,4)} = fe^{1234}$ for a nowhere vanishing function $f$. We find
\begin{equation}
[d(\psi_{(0,4)})]_{(2,3)} = f de^1 \wedge e^{234} - f e^1 \wedge de^2 \wedge e^{34} + f e^{12} \wedge de^3 \wedge e^4 - fe^{123} \wedge de^4.
\end{equation}
Setting this to zero implies that $de^1 = \dots = de^4 = 0$. It follows that the horizontal distribution is integrable.
\end{proof}
\end{prop}

From the proof of this proposition we see that on a semi-flat $G_2$-manifold we can find invariant local frames $\{e_j\}$ for the vertical $V$ and $\{a_i\}$ for the horizontal $V^\perp$. Invariance means we have commutation relations $[e_j,e_k] = [e_j,a_i] = 0$. Since $\phi$ is also invariant we can make invariant frame changes such that $\phi$ and $\psi$ have the forms (\ref{standardphi}) and (\ref{standardpsi}). This frame is necessarily orthonormal.\\

Let $w \in U \subset \mathbb{R}^3$ be local coordinates for a leaf of the horizontal distribution. We can take $U$ sufficiently small that we have an embedded submanifold $i:U \to X$. Define a map $f:U \times T^4 \to X$ by $f(w,x) = x \cdot i(w)$ where $x$ acts on $i(w)$ by the $T^4$-action. It is immediate that $f$ is an immersion sending $T(T^4)$ to the vertical distribution and $TU$ to the horizontal. Moreover by sufficiently restricting $U$ we may assume $f$ is injective. Let $\pi : X \to U$ denote the locally defined projection. The fibres of $\pi$ are coassociative submanifolds. Therefore we can identify $U$ as the local moduli space of coassociative deformations of fibres as in Proposition \ref{riemfibr}.\\


The metric $g$ has the form $g = g_V + g_{V^\perp}$ where $g_V$ and $g_{V^\perp}$ are metrics on the vertical and horizontal respectively. Now since $g$ is $T^4$-invariant we have that $g_V$ is a flat metric on each orbit and $g_{V^\perp}$ is the pull-back under $\pi$ of a metric $g_U$ on the base $U$ which is related to the $L^2$ moduli space metric by (\ref{basemetric}).


We see that a semi-flat $G_2$ manifold is a principal $T^4$-bundle $\pi: X \to M$ over a base $M$ that can be identified with the moduli space of deformations of the fibres.

\subsection{Construction of semi-flat $G_2$-manifolds}\label{secconstruct}
We find an equivalent local characterisation of semi-flat $G_2$-manifolds in terms of minimal submanifolds.\\

Suppose for the moment that we have a semi-flat $G_2$-manifold. As usual let us take an invariant local frame $a_1,a_2,a_3,e_1, \dots , e_4$ such that $\phi$ has the form (\ref{standardphi}). Let $w = (w^1,w^2,w^3)$ denote local horizontal coordinates and let $x^i \in \mathbb{R}/\mathbb{Z}$, $1 \le i \le 4$ be standard coordinates for $T^4$. So the $dw^i$ are $(1,0)$-forms and the $dx^a$ are $(0,1)$-forms. Hence we see that we can uniquely write $\phi$ as
\begin{equation}
\phi = dvol_M + dw^i \wedge \theta_i
\end{equation}
where the $\theta_i$ are $(0,2)$-forms. It is also clear that $\theta_i|_{T_w}$ is a harmonic self-dual $2$-form on $T_w$ representing the deformation of $T_w$ in the $\partial / \partial w^i$ direction.\\

Thinking of $M$ as the moduli space of coassociative deformations we recall that there is a locally defined function $u : M \to H^2(T,\mathbb{R})$ such that $du$ is the $H^2(T,\mathbb{R})$-valued $1$-form $w \mapsto [ \theta_i |_{T_w} ]dw^i$. We can give $H^2(T,\mathbb{R})$ coordinates $a_{ij}$ such that $a_{ij}$ corresponds to the cohomology class $[a_{ij}dx^{ij}]$. Therefore we have $6$ functions $u_{ab}$ on $M$ such that $u(w) = [u_{ab}(w)dx^{ab}]$. We then have
\begin{eqnarray*}
du &=& d[u_{ab}dx^{ab}]\\
&=& \left[ \frac{\partial u_{ab}}{\partial w^i} dx^{ab} \right] dw^i \\
&=& [\theta_i |_{T_w}] dw^i.
\end{eqnarray*}
Hence $[\theta_i |_{T_w}] = [\tfrac{\partial u_{ab}}{\partial w^i}(w) dx^{ab} ]$, in fact since the $\theta_i$ have no $dw^j$ terms and are constant with respect to fibre coordinates (being harmonic) we have
\begin{equation}\label{theta}
\theta_i = \frac{\partial u_{ab}}{\partial w^i} dx^{ab}.
\end{equation}

By (\ref{basemetric}) the base metric $g_M$ and $L^2$ moduli space metric on $M$ are related by $g_{L^2} = 2{\rm vol}(T^4)g_M$. We can explicitly integrate the $L^2$ metric. If $g_M = g_{ij}dw^i dw^j$ then
\begin{eqnarray*}
g_{ij} &=& \dfrac{1}{2{\rm vol(T^4)}}\int_{T^4} u_* \left( \frac{\partial}{\partial w^i} \right) \wedge u_* \left( \frac{\partial}{\partial w^j} \right) \nonumber \\
&=& \dfrac{1}{2{\rm vol(T^4)}}\int_{T^4} \frac{\partial u_{ab}}{\partial w^i}dx^{ab} \wedge \frac{\partial u_{cd}}{\partial w^j}dx^{cd} \nonumber \\
&=& \dfrac{1}{2{\rm vol(T^4)}}\epsilon^{abcd} \frac{\partial u_{ab}}{\partial w^i} \frac{\partial u_{cd}}{\partial w^j}
\end{eqnarray*}
where $\epsilon^{abcd}$ is antisymmetric and $\epsilon^{1234}=1$. So we have
\begin{equation}\label{metric}
2g_M(A,B){\rm vol(T^4)}dx^{1234} = u_*(A) \wedge u_*(B),
\end{equation}
so $g_M$ is essentially the pull-back of the wedge product.\\

We are now in a position to reverse the construction. Suppose $M$ is an oriented $3$-manifold with a function $u : M \to H^2(T,\mathbb{R})$, $u = [u_{ab}dx^{ab}]$. We assume that $u_*$ sends the tangent spaces of $M$ into maximal positive definite subspaces of $H^2(T,\mathbb{R})$. Choose a positive constant $\tau$. The significance of $\tau$ is that it represents the volume of the coassociative fibres. We may pull back the intersection form to define a positive definite metric $h$ on $M$ given by
\begin{equation}\label{metrich}
2h(A,B)\tau dx^{1234} = u_*(A) \wedge u_*(B).
\end{equation}
Let $dvol_h$ denote the volume form for this metric. We may then define self-dual $2$-forms $\theta_i$ by equation (\ref{theta}), hence we can define a $3$-form $\phi$ on $M \times T^4$ by
\begin{equation}\label{constructedphi}
\phi = dvol_h + dw^i \wedge \theta_i.
\end{equation}
We can easily verify that $\phi$ is closed by noting that from (\ref{theta}) we have $(\partial/\partial w^j)\theta_i = (\partial/\partial w^i)\theta_j$.\\

It is clear that for any given $u$ and $\tau$, $\phi$ has the correct algebraic form for an associative $3$-form. By (\ref{gdvol}), $\phi$ determines a metric $g$ on $X$ and a corresponding volume form $dvol_X = dvol_M \wedge dvol_T$. It follows from (\ref{constructedphi}) that the induced metric $g$ agrees with $h$ on the horizontal distribution so that $dvol_M= \, dvol_h$. Moreover one can further show $dvol_T = \tau dx^{1234}$.\\

Now consider $\psi = *\phi$. Since the $\theta_i$ are self-dual $2$-forms on each fibre we find
\begin{equation}
\psi = dvol_T + *_3 dw^i \wedge \theta_i
\end{equation}
where $*_3$ denotes the Hodge star with respect to $g$ restricted to the base. We can see that a necessary condition for $\psi$ to be closed is that $c$ is constant. Indeed if $\psi$ is closed then since it is a calibrating form and the fibres of $X$ are isotopic calibrated submanifolds they must have equal volume. Therefore assume $c$ is constant. We calculate
\begin{eqnarray*}
d\psi &=& 0 + d(*_3 dw^i) \wedge \theta_i + *_3 dw^i \wedge d\theta_i \\
&=& \Delta w^i dvol_M \wedge \theta_i + g^{ij} \iota_j dvol_M \wedge dw^k \wedge \frac{\partial^2 u_{ab}}{\partial w^k \partial w^i} dx^{ab} \\
&=& \Delta w^i dvol_M \wedge \frac{\partial u_{ab}}{\partial w^i} dx^{ab} + g^{ij} dvol_M \wedge \frac{\partial^2 u_{ab}}{\partial w^j \partial w^i} dx^{ab} \\
&=& \left( \Delta w^i \frac{\partial u_{ab}}{\partial w^i} + g^{ij} \frac{\partial^2 u_{ab}}{\partial w^j \partial w^i} \right) dvol_M \wedge dx^{ab}.
\end{eqnarray*}
Hence $\psi$ is closed if and only if for each $a,b$ we have
\begin{equation}\label{uequ}
g^{ij} \frac{\partial^2 u_{ab}}{\partial w^j \partial w^i} + \Delta w^i \frac{\partial u_{ab}}{\partial w^i} = 0.
\end{equation}
Note that the Laplacian on functions on the base $\Delta = *^{-1}d*d = -\delta d$ is given by
\begin{equation}
\Delta f = g^{ij} \left(\frac{\partial^2 f}{\partial w^i \partial w^j} - {\Gamma^k}_{ij} \frac{\partial f}{\partial w^k} \right).
\end{equation}
Where ${\Gamma^k}_{ij}$ are the Christoffel symbols
\begin{equation}
{\Gamma^k}_{ij} = \frac{1}{2}g^{km}( \partial_i g_{jm} + \partial_j g_{im} - \partial_m g_{ij} ).
\end{equation}
In particular, applied to a coordinate function we have $\Delta w^k = -g^{ij}{\Gamma^k}_{ij}$. Substituting this into equation (\ref{uequ}) we get
\begin{equation}
\Delta u_{ab} = 0.
\end{equation}
This says that the map $u : M \to H^2(T,\mathbb{R})$ is harmonic where $M$ is given the metric $g$. However the pull-back metric on $M$ induced by $u$ differs from $g$ only by a constant so equivalently $u$ is harmonic with respect to the induced metric. Another way of saying this is that the map $u$ is a minimal immersion or that $M$ is locally embedded as a minimal $3$-submanifold \cite{eells}.
\begin{thm}\label{thethm}
Let $M$ be an oriented $3$-manifold and $u : M \to \wedge^2 \mathbb{R}^4$ a map with the property that $u$ maps the tangent spaces of $M$ into maximal positive definite subspaces of $\wedge^2 \mathbb{R}^4$ and let $\tau$ be a positive constant. Let $h$ be the pull-back metric defined in equation (\ref{metrich}) with volume $dvol_h$. Let $X = M \times \left( \mathbb{R}/\mathbb{Z} \right)^4$ and define $\phi \in \Omega^3(X,\mathbb{R})$ by
\begin{equation}
\phi = \, dvol_h + du,
\end{equation}
where $u$ is considered as a $2$-form on $X$. Then $(X,\phi)$ is a semi-flat $G_2$-manifold if and only if $u$ is a minimal immersion. Moreover every semi-flat $G_2$-manifold locally has this form.
\end{thm}

We can improve on Theorem \ref{thethm} by determining the global properties of semi-flat $G_2$-manifolds. For completeness we may define a {\em locally semi-flat} $G_2$-manifold as a locally trivial $T^4$-fibration $\pi : X \to M$ with semi-flat local trivialisations. Now $M$ identifies as the moduli space of deformations of the fibres, so we have a developing map $\tilde{u}: \tilde{M} \to \wedge^2 \mathbb{R}^4$ on the universal cover $\tilde{M}$ of $M$ with monodromy representation $\rho:\pi_1(M) \to {\rm \bf SL}(4,\mathbb{Z}) \ltimes \wedge^2 \mathbb{R}^4$. Moreover $\tilde{u}$ is a minimal immersion. Therefore we have a semi-flat $G_2$-manifold $\tilde{X} = \tilde{M} \times \left( \mathbb{R}/\mathbb{Z} \right)^4$. The $3$-form $\tilde{\phi}$ is
\begin{equation*}
\tilde{\phi} = \, dvol_{h} + d\tilde{u}.
\end{equation*}
Clearly $\tilde{\phi}$ can be identified with the lift of the $3$-form $\phi$ on $X$.\\

Considering transition functions for $X$ as a bundle over $M$ we have $X = \tilde{M} \times_\lambda \left( \mathbb{R}/\mathbb{Z} \right)^4$ where $\lambda$ is a representation $\lambda : \pi_1(M) \to {\rm \bf SL}(4,\mathbb{Z}) \ltimes \left( \mathbb{R}/\mathbb{Z} \right)^4$. Now for $\tilde{\phi}$ to descend to $\phi$ on $X$ we must have that the ${\rm \bf SL}(4,\mathbb{Z})$ parts of $\rho$ and $\lambda$ agree.\\

Conversely given an oriented $3$-manifold $M$ we see that locally semi-flat $G_2$-manifolds with base $M$ correspond to data $(\tau,\tilde{u},\rho,\lambda)$ where
\begin{itemize}
\item{$\tau$ is a positive constant,}
\item{$\rho$ is a representation $\rho: \pi_1(M) \to {\rm \bf SL}(4,\mathbb{Z}) \ltimes \wedge^2 \mathbb{R}^4$,}
\item{$\lambda$ is a representation $\lambda: \pi_1(M) \to {\rm \bf SL}(4,\mathbb{Z}) \ltimes \left( \mathbb{R}/\mathbb{Z} \right)^4$,}
\item{$\tilde{u}$ is a map $\tilde{u} : \tilde{M} \to \wedge^2 \mathbb{R}^4$ sending the tangent spaces of $\tilde{M}$ to positive definite subspaces,}
\end{itemize}
with the properties that $\tilde{u}$ is a minimal immersion and the ${\rm \bf SL}(4,\mathbb{Z})$ parts of $\rho$ and $\lambda$ agree. The semi-flat case occurs when the ${\rm \bf SL}(4,\mathbb{Z})$ parts of $\rho$ and $\lambda$ are trivial.

\subsection{Cylindrical semi-flat $G_2$-manifolds}\label{seccylsem}
We will show that the semi-flat $G_2$ equations reduce to the Monge-Amp\`ere equation if we assume that the resulting $G_2$-manifold is cylindrical. This coincides with the result of Hitchin \cite{hit} in which special Lagrangian fibrations with flat fibres are produced from the Monge-Amp\`ere equation.\\

Let us write the $4$-torus $T^4$ as a product $T^4 = T^3 \times T^1$. We have an isometry $e : H^1(T^3,\mathbb{R}) \oplus H^2(T^3,\mathbb{R}) \to H^2(T^4,\mathbb{R})$ as follows:
\begin{equation}
e(\alpha, \beta) = (\alpha \smallsmile [dt] + \beta)/\sqrt{2}
\end{equation}
where $t \in \mathbb{R}/\mathbb{Z}$ is the standard coordinate for $T$. We can express a map $\tilde{u} : M \to H^2(T^4,\mathbb{R})$ as the composition $\tilde{u} = \sqrt{2}e( u,v)$ where $u : M \to H^1(T^3,\mathbb{R})$, $v: M \to H^2(T^3,\mathbb{R})$.\\

From the map $\tilde{u} : M \to H^2(T^4,\mathbb{R})$ we can construct a semi-flat $G_2$-manifold $\pi : X \to M$ as before. Now we assume there is a $T^1$ subgroup of the $T^4$-action on $X$ such that the orbits have constant length. Therefore we can write $X = Y \times T$ and the $G_2$-equations then imply that $X$ is a cylindrical $G_2$-manifold constructed from a Calabi-Yau manifold $Y$. Note also that $Y$ is a semi-flat special Lagrangian fibration over the same base.\\

As in Section \ref{secspeclag} we have a symplectic structure (\ref{sympl}) on $H^1(T^3,\mathbb{R}) \oplus H^2(T^3,\mathbb{R})$ arising from duality such that $(u,v)$ locally embeds $M$ as a Lagrangian submanifold. The maps $u,v$ are local diffeomorphisms hence writing $u = [u_i dx^i]$ and $v = [\tfrac{1}{2!} \epsilon_{ijk}v^idx^{jk}]$ we may take either the $u_i$ or $v^j$ as coordinates on $M$. Moreover since $M$ is Lagrangian it is locally the graph of a gradient, that is there exist functions $\phi,\psi$ on $M$ such that $v^i = \frac{\partial \phi}{\partial u_i}$, $u_i = \frac{\partial \psi}{\partial v^i}$. A little arithmetic shows that after applying an element of ${\rm SL}(4,\mathbb{Z})$ we may assume $t = x^4$. Now we write $\tilde{u} = [\tilde{u}_{ij}dx^{ij}]$ then we find
\begin{eqnarray}
\tilde{u}_{i4} &=& u_i, \\
\tilde{u}_{ij} &=& \tfrac{1}{2}\epsilon_{ijk}v^k, \; \; i,j,k \le 3.
\end{eqnarray}
If we check the formula for the metric induced by $\tilde{u}$ we find that up to a multiple it is given by
\begin{equation}
g = \frac{\partial^2 \phi}{\partial u_i \partial u_j} du_i du_j = \frac{\partial^2 \psi}{\partial v^i \partial v^j} dv^i dv^j.
\end{equation}
Now starting with the relation $v^j = \frac{\partial \phi}{\partial u_j}$ we find $dv^j = \frac{\partial^2 \phi}{\partial u_j \partial u_k}du_k$. Substituting this into the expression for $g$ we find that $h^{ij} = \frac{\partial^2 \psi}{\partial v^i \partial v^j}$ is the inverse of $g_{ij} = \frac{\partial^2 \phi}{\partial u_i \partial u_j}$, that is $g_{ij}h^{jk} = \delta^k_i$. Let us introduce some notation: $\partial_i = \tfrac{\partial}{\partial u_i}$, $\partial^j = \tfrac{\partial}{\partial v^j}$, $\phi_{ij \dots k} = \partial_i \partial_j \dots \partial_k \phi$, $\psi^{ij \dots k} = \partial^i \partial^j \dots \partial^k \psi$. We also note that $\partial^j = \psi^{jk}\partial_k$. Now we can calculate the Christoffel symbols in the $u_i$ coordinates:
\begin{equation}
{\Gamma^k}_{ij} = \frac{1}{2}\psi^{km}\phi_{mij}.
\end{equation}
Now we calculate $\Delta \tilde{u}_{ab}$. First suppose $b=4$ so that $\tilde{u}_{ab} = u_a$. Then we find
\begin{eqnarray}
\Delta u_a &=& \psi^{ij}\left( 0 - \frac{1}{2}\psi^{km}\phi_{mij}\delta^a_k \right) \nonumber \\
&=& -\frac{1}{2}\psi^{ij}\psi^{am}\phi_{mij}. \label{laplacian}
\end{eqnarray}
If we take the relation $\phi_{ij}\psi^{jk} = \delta^k_i$ and differentiate we find
\begin{equation}
\phi_{mij}\psi^{jk} = -\phi_{ij}\phi_{mr} \psi^{rjk}.
\end{equation}
Substituting into equation (\ref{laplacian}) we find that
\begin{equation}
\Delta u_a = \frac{1}{2}\phi_{ij}\psi^{aij}.
\end{equation}
Similarly we find
\begin{equation}
\Delta v^a = \frac{1}{2}\psi^{ij}\phi_{aij}.
\end{equation}
Therefore the $G_2$-equations in this case reduce to
\begin{eqnarray*}
\phi_{ij}\psi^{aij} &=& 0, \\
\psi^{ij}\phi_{aij} &=& 0.
\end{eqnarray*}

Now let us recall Jacobi's formula in the case where the matrix valued function $\phi$ is invertible:
\begin{equation}
d {\rm det}(\phi) = {\rm det}(\phi) {\rm tr}(\phi^{-1}d\phi).
\end{equation}
Therefore the Monge-Amp\`ere equation $d {\rm det}(\phi) = 0$ is equivalent to ${\rm tr}(\phi^{-1}d\phi) = 0$. But
\begin{eqnarray*}
{\rm tr}(\phi^{-1}d\phi) &=& {\rm tr}( \psi d\phi) \\
&=& \psi^{ij}\phi_{aij} du_a.
\end{eqnarray*}
Similarly we can interchange the roles of $\phi$ and $\psi$. Hence the $G_2$ equations in this case are equivalent to the Monge-Amp\`ere equation.
\begin{rem}
We note that in \cite{hit} the Monge-Amp\`ere equation is also shown to be equivalent to $M$ being calibrated with respect a calibrating form that is a linear combination of the volume forms of $H^1(T^3,\mathbb{R})$ and $H^2(T^3,\mathbb{R})$. This agrees with the fact that $M$ is minimally immersed.
\end{rem}

\section{Relation to minimal surfaces}\label{rtms}
We impose an additional degree of symmetry on a semi-flat $G_2$-manifold. The additional symmetry is shown to correspond to the base locally having the structure of a minimal cone in $\mathbb{R}^{3,3}$ which in turn is equivalent to a minimal surface in the quadric of unit vectors.\\

\subsection{Reduction to surface equations}\label{secredtosurf}
Let $\pi:X \to M$ be a semi-flat $G_2$-manifold constructed from a minimal immersion $u: M \to \mathbb{R}^{3,3}$. We suppose there is a vector field $U$ on $X$ such that $U$ commutes with the $T^4$-action. It is not possible for such a vector field to preserve the $3$-form $\phi$ up to scale, that is $\mathcal{L}_U \phi = \lambda \phi$ for some non-vanishing function $\lambda$, for in this case the fibres of the $T^4$-fibration would not have constant volume. Therefore we consider a slightly different symmetry. We suppose that
\begin{equation}\label{scale}
\mathcal{L}_U\phi = \lambda \phi + 2 \lambda dvol_M
\end{equation}
as such a symmetry will preserve the volume of the fibres.\\

Let $U =V+W = V^i  \tfrac{\partial}{\partial w^i} + W^a \tfrac{\partial}{\partial x^a} $. For $U$ to commute with the $T^4$-action we must have the $W^a$ and $V^i$ are independent of $x$. We will show that $W$ is a vector field generated by the $T^4$-action, hence we need only consider $V$.

Recall that locally a semi-flat $G_2$-manifold $X = M \times T^4$ with coordinates $(w,x)$ has the $3$-form
\begin{equation}
\phi = dvol_M + du_{ab} \wedge dx^{ab} = dvol_M + dw^i \wedge \theta_i
\end{equation}
where $u : M \to H^2(T^4,\mathbb{R})$ is a minimal immersion and $\theta_i = \tfrac{\partial u_{ab}}{\partial w^i}dx^{ab}$. For simplicity we will take ${\rm vol}(T^4)=1$.

Since $\phi$ is closed, the condition on $U$ is that $\mathcal{L}_U\phi = d(i_U\phi)= \lambda \phi + 2\lambda dvol_M$. We find that
\begin{equation*}
i_U\phi = i_V dvol_M + V^i \theta_i + dw^i \wedge i_W \theta_i
\end{equation*}
and that
\begin{equation*}
d(i_U \phi) = {\rm div}(V)dvol_M + d(V^i\theta_i) - dw^i \wedge d(W^a \frac{\partial u_{ab}}{\partial w^i}dx^b ).
\end{equation*}
Equating this to (\ref{scale}) we find
\begin{eqnarray}
{\rm div}(V) &=& 3\lambda, \\
d(V^i \theta_i) &=& \lambda dw^i \wedge \theta_i, \\
\frac{\partial W^a}{\partial w^i}\frac{\partial u_{ab}}{\partial w^j} &=& \frac{\partial W^a}{\partial w^j}\frac{\partial u_{ab}}{\partial w^i}.
\end{eqnarray}
We can show that $W$ must be independent of the base variables. This follows from the algebraic fact that if $\theta_1,\theta_2,\theta_3$ are a basis of self-dual $2$-forms and $A_1,A_2,A_3$ are vectors such that $i_{A_i}\theta_j = i_{A_j}\theta_i$ then $A_i = 0$. Hence $W$ is a vector field coming from the $ T^4$-action. Therefore we ignore $W$.\\

Now the equations for $V$ are $3\lambda = {\rm div}(V)$ and $d(V^i \theta_i) = \lambda dw^i \wedge \theta_i = \lambda du$. Taking exterior derivatives we find $0 = d\lambda \wedge du$ hence $\tfrac{\partial \lambda}{\partial w^i}\tfrac{\partial u}{\partial w^j} = \tfrac{\partial \lambda}{\partial w^j}\tfrac{\partial u}{\partial w^i}$, that is $\tfrac{\partial \lambda}{\partial w^i}\theta_j = \tfrac{\partial \lambda}{\partial w^j}\theta_i$. But $\{\theta_i\}$ are linearly independent so we have $d\lambda = 0$. Therefore the second equation for $V$ becomes $d(V^i\theta_i - \lambda u)=0$ or $V^i \tfrac{\partial u}{\partial w^i} = \lambda u + c$ where $c$ is constant. There are now two cases to consider; when $\lambda =0$ and $\lambda \ne 0$.

The $\lambda = 0$ case can readily be shown to correspond to minimal surfaces in $\mathbb{R}^{2,3}$. Such minimal surfaces correspond locally to holomorphic maps $\tau : \Sigma \to Q$ from a Riemann surface into an open subset of a quadric given by $Q = \{ v \in \mathbb{C} \otimes \mathbb{R}^{2,3} \; | \; \langle v , v \rangle = 0, \; \langle v , \overline{v} \rangle > 0 \}$. There is a local Weierstrass representation for the corresponding minimal immersion $\phi$ \cite{li}:
\begin{equation*}
\phi(z) = \phi(0) + {\rm Re}\int_0^z \tau(\zeta) d\zeta.
\end{equation*}
Now assume $\lambda \ne 0$. Then we can redefine $u$ to absorb the constant $c$ so we have $u_*(V) = V^i \tfrac{\partial u}{\partial w^i} = \lambda u$. Now we can rescale $V$ such that $u_*(V) = u$. If the vector field $V$ vanishes at a point $w \in M$ then $u(w) = u_*(V_w)=0$. Now the map $u:M \to H^2(T^4,\mathbb{R})$ is an immersion so $V$ vanishes at isolated points. Away from the zeros of $V$ we may find local coordinates $(x,y,t) \in \Sigma \times I$ such that $V = \tfrac{\partial}{\partial t}$. Hence we have $\tfrac{\partial u}{\partial t}(x,y,t) = u(x,y,t)$. The solution is of the form $u(x,y,t) = u(x,y)e^t$. The induced metric $g$ on $M$ has the property that $g(x,y,t) = e^{2t}g(x,y)$, hence $\tfrac{\partial}{\partial t}$ satisfies ${\rm div}(\tfrac{\partial}{\partial t})=3$ as required.\\

We will attempt to find local coordinates that diagonalise the metric on $M$. By changing the local slice $\Sigma \to M$ along the $t$ direction we have freedom $u(x,y) \mapsto u(x,y)e^{\rho(x,y)}$ where $\rho$ is an arbitrary smooth function on $\Sigma$. We calculate (on $t=0$)
\begin{equation*}
2 g \left( \frac{\partial}{\partial t} , \frac{\partial}{\partial x} \right) dx^{1234} = e^{2\rho}( u \wedge \frac{\partial u}{\partial x} + u \wedge u \frac{\partial \rho}{\partial x} )
\end{equation*}
and similarly for $g \left( \frac{\partial}{\partial t} , \frac{\partial}{\partial y} \right)$. Let us define functions $r,s$ by
\begin{eqnarray}
r u \wedge u &=& -u \wedge \frac{\partial u}{\partial x} \\
s u \wedge u &=& -u \wedge \frac{\partial u}{\partial y}.
\end{eqnarray}
Note that this is possible because $2g(\tfrac{\partial}{\partial t} , \tfrac{\partial}{\partial t} )dx^{1234} = u \wedge u \neq 0$. Then we can locally find a function $\rho(x,y)$ such that $g(\tfrac{\partial}{\partial t} , \tfrac{\partial}{\partial x} ) = g(\tfrac{\partial}{\partial t} , \tfrac{\partial}{\partial y}) = 0$ if and only if $\tfrac{\partial r}{\partial y} = \tfrac{\partial s}{\partial x}$. This follows easily from the definitions of $r$ and $s$. Therefore our metric now has the form
\begin{equation*}
g(x,y,t) = e^{2t}(c(x,y)dt^2 + g_{\Sigma}(x,y) ).
\end{equation*}
Setting $r = e^t$ we may write this as
\begin{equation*}
g(x,y,r) = c(x,y)dr^2 + r^2g_{\Sigma}(x,y).
\end{equation*}
Let us also note that $2c dx^{1234} = u \wedge u$ so that $2\tfrac{\partial c}{\partial x} dx^{1234} =  2 u \wedge \tfrac{\partial u}{\partial x} = 0$ and similarly for $y$. Thus $c$ is constant and by scaling we can assume $c=1$. Therefore the metric is
\begin{equation}
g = dr^2 + r^2 g_{\Sigma}
\end{equation}
and our minimal $3$-fold is a cone.\\

It is well known that given a minimal $M$ submanifold in $S^n$, the cone $M \times (0,\infty)$ over $M$ is a minimal submanifold of $\mathbb{R}^{n+1}$ \cite{sim}. We give a generalization of this result. For a manifold $M$ with (possibly indefinite) metric $g$ let $(CM,\hat{g})$ denote the cone where $CM = M \times (0,\infty)$ , $\hat{g} = dr^2 + r^2g$. Given a map $\phi : (M,g) \to (N,h)$ we define the {\em radial extension} $\hat{\phi} : CM \to CN$ by $\hat{\phi}(x,r) = (\phi(x),r)$.

\begin{prop}
The radial extension $\hat{\phi} : CM \to CN$ is minimal if and only if $\phi : M \to N$ is minimal.
\begin{proof}
First we note that $\phi$ is a Riemannian immersion if and only if $\hat{\phi}$ is a Riemannian immersion. Let us use coordinates $x^1, \dots , x^m$ on $M$ and let $r = x^0$. We use the convention that indices $i,j,k,\dots $ do not take the value $0$. Likewise give $CN$ coordinates $r=y^0,y^1, \dots , y^n$. We have
\begin{equation}
\begin{aligned}
\hat{g}_{00} &= 1, \; \; \; & \hat{g}_{0i} &= 0, \; \; \; & \hat{g}_{ij} &= r^2 g_{ij} \\
\hat{g}^{00} &= 1, \; \; \; & \hat{g}^{0i} &= 0, \; \; \; & \hat{g}^{ij} &= \frac{1}{r^2} g^{ij}.
\end{aligned}
\end{equation}
We readily verify the following relation between the Christoffel symbols on $M$ and $CM$:
\begin{equation}
\begin{aligned}
\leftexp{CM}{{\Gamma^k}_{ij}} &= \leftexp{M}{{\Gamma^k}_{ij}}, \; \; \; &
\leftexp{CM}{{\Gamma^0}_{ij}} &= -rg_{ij}, \; \; \; &
\leftexp{CM}{{\Gamma^k}_{0j}} &= \frac{1}{r}\delta^k_j, \\
\leftexp{CM}{{\Gamma^k}_{00}} &= 0, \; \; \; &
\leftexp{CM}{{\Gamma^0}_{0j}} &= 0, \; \; \; &
\leftexp{CM}{{\Gamma^0}_{00}} &= 0.
\end{aligned}
\end{equation}
There are similar relations for the Christoffel symbols on $CN$. The map $\hat{\phi}$ has the properties
\begin{eqnarray*}
\phi^0 &=& x^0 \\
\partial_0 \phi^\gamma &=& 0.
\end{eqnarray*}
Recall that the tension field $\tau(\phi)$ for a map $\phi :M \to N$ is the section of $\phi^*(TN)$ obtained by taking the trace of the second fundamental form:
\begin{equation}
\tau^\gamma(\phi) = g^{ij}\left( \partial^2_{ij} \phi^\gamma - \leftexp{M}{{\Gamma^k}_{ij}} \partial_k \phi^\gamma + \leftexp{N}{{\Gamma^\gamma}_{\alpha \beta}} \partial_i \phi^\alpha \partial_j \phi^\beta \right).
\end{equation}
The map $\phi$ is harmonic if and only if $\tau(\phi)=0$. Similarly we have a torsion field $\tau(\hat{\phi})$. We calculate
\begin{eqnarray}
\tau^\gamma(\hat{\phi}) &=& \frac{1}{r^2} \tau^\gamma(\phi) \\
\tau^0(\hat{\phi}) &=& \frac{1}{r} g^{ij} \left( g_{ij} - h_{\alpha \beta}\partial_i \phi^\alpha \partial_j \phi^\beta \right).
\end{eqnarray}
The result follows.
\end{proof}
\end{prop}

Let $\mathbb{R}^{p,q}$ denote $\mathbb{R}^n$ with a signature $(p,q)$ inner product. We say that a submanifold $X$ of $\mathbb{R}^{p,q}$ is a cone if $X$ is diffeomorphic to $\Sigma \times (0,\infty)$ such that $i(x,r) = ri(x,1)$ where $i$ is the inclusion $i : X \to \mathbb{R}^{p,q}$ and the induced metric on $X$ is of the form $dr^2 + r^2g_\Sigma$ where $g_\Sigma$ is independent of $r$.

Restricting to $r=1$ we have an inclusion $ i : \Sigma \to Q \subset \mathbb{R}^{p,q}$ where $Q = \{ v \in \mathbb{R}^{p,q} | \; \langle v , v \rangle = 1 \} = {\rm O}(p,q)/{\rm O}(p-1,q)$. Conversely such a map defines a cone in $\mathbb{R}^{p,q}$. We thus have
\begin{corp}
There is a bijection between minimal cones with definite induced metric in $\mathbb{R}^{p,q}$ and minimal submanifolds of ${\rm O}(p,q)/{\rm O}(p-1,q)$ with definite induced metric.
\end{corp}

In our situation we have a minimal surface $u : \Sigma \to {\rm O}(3,3)/{\rm O}(2,3)$.

\subsection{Minimal surfaces of signature $(3,3)$}\label{secquadrics2}

Let us apply the results of Section \ref{secquadrics} to the case of a minimal surface $\phi$ into $Q = \{ x \in \mathbb{R}^{3,3} \, | \, \langle x , x \rangle = 1 \}$. In this case the maximum possible isotropy order is $2$. In this case we have $\mu_1 = 1$, $\mu_2 = -1$, $\mu_3 = 1$, $q$ is a holomorphic cubic differential and the equations become
\begin{eqnarray}
2(w_1)_{z\overline{z}} &=& -e^{2w_2-2w_1} - e^{2w_1}, \label{d23toda1} \\
2(w_2)_{z\overline{z}} &=& q\overline{q}e^{-2w_2} + e^{2w_2 - 2w_1} \label{d23toda2}.
\end{eqnarray}
Let $a_1 = H_1 = e^{2w_1}$ and $a_2 = H_2/H_1 = e^{2w_2-2w_1}$. Then $a_1$ and $a_2$ are positive $(1,1)$-forms and can be thought of as metrics on $\Sigma$. But from the above equations we see that $a_1$ has strictly positive curvature while $a_2$ has strictly negative curvature. Therefore there are no compact solutions without singularities. The equation for an elliptic affine sphere in $\mathbb{R}^3$ \cite{lyz} appears as a special case of equations (\ref{d23toda1}),(\ref{d23toda2}). Indeed the elliptic affine sphere equation is
\begin{equation}\label{tzitz1}
2(w_1)_{z\overline{z}} = -q\overline{q}e^{-4w_1} - e^{2w_1}.
\end{equation}
If we set $H_2 = q\overline{q}/H_1$ then equation (\ref{tzitz1}) yields a solution to (\ref{d23toda1}) and (\ref{d23toda2}) (strictly speaking it is a solution away from the zeros of $q$, however one can show the corresponding minimal immersion extends over the zeros of $q$) . In fact we can explain this reduction more directly in terms of $G_2$ geometry. Equation (\ref{tzitz1}) corresponds to the equation for a cylindrical semi-flat $G_2$-manifold with scaling symmetry as in Section \ref{secredtosurf}.

There are other cases of minimal surfaces in $Q \subset \mathbb{R}^{3,3}$ that we can find similar equations for, namely real forms of the affine Toda equations for the affine Dynkin diagrams $A^{(1)}_1$ and $B^{(1)}_2$. These are minimal surfaces for which the image lies in a proper subspace of $\mathbb{R}^{3,3}$. The equations follow from the results of Section \ref{secquadrics}.

\section{Extension to split $G_2$-manifolds}\label{extension}
With a little work we can extend the semi-flat construction to produce split $G_2$-manifolds. Consider the split octonions $\tilde{\mathbb{O}} = \mathbb{R} \oplus {\rm Im}\, \tilde{\mathbb{O}}$. We say that a $3$-dimensional subspace $V$ of ${\rm Im}\, \tilde{\mathbb{O}}$ is {\em associative} if the split octonion metric is non-degenerate on $V$ and $\mathbb{R} \oplus V$ is closed under multiplication. The algebra $\mathbb{R} \oplus V$ is isomorphic to either the quaternions $\mathbb{H}$ when $V$ is positive definite, or the split quaternions $\tilde{\mathbb{H}}$ when $V$ has signature $(1,2)$. We call such associative subspaces {\em definite associative} and {\em split associative} respectively. This partitions the associative subspaces into two orbits of split $G_2$.

Similarly we call a $4$-dimensional subspace of ${\rm Im}\, \tilde{\mathbb{O}}$ {\em coassociative} if it is the orthogonal complement in ${\rm Im}\, \tilde{\mathbb{O}}$ of an associative subspace. If $V$ is a definite associative subspace then the corresponding coassociative subspace $V^\perp$ has signature $(0,4)$ and will be called {\em definite coassociative}. If $V$ is split associative then $V^\perp$ has signature $(2,2)$ and will be called {\em split coassociative}. Coassociative subspaces are characterised by the fact that they are non-degenerate $4$-dimensional subspaces $W$ such that the split octonion cross product of any two vectors in $W$ is orthogonal to $W$.\\

In what follows we restrict attention to ${\rm Im}\, \tilde{\mathbb{O}}$ equipped with the split octonion cross product. The $3$-form $\phi$ is given by
\begin{equation}
\phi(a,b,c) = \langle a \times b , c \rangle.
\end{equation}
The signature $(3,4)$ metric by
\begin{equation}
g(A,B)dvol = \frac{1}{6}\iota_A \phi \wedge \iota_B \phi \wedge \phi
\end{equation}
provided that we use the orientation opposite to the standard orientation on ${\rm Im}\, \tilde{\mathbb{O}}$. Finally the $4$-form $\psi$ is defined by $\psi = * \phi$. Given a decomposition ${\rm Im}\, \tilde{\mathbb{O}} = A \oplus C$ into an associative and corresponding coassociative subspace, we can describe $\phi$ as follows. Let $a_1,a_2,a_3$ be an orthogonal basis for $V$ with $\langle a_i,a_i \rangle = \epsilon_i$ where $(\epsilon_1,\epsilon_2,\epsilon_3) = (1,1,1)$ in the definite associative case or $(1,-1,-1)$ in the split associative case. Let $a^i$ be the corresponding coframe. There is a corresponding basis of self-dual $2$-forms $\omega_1,\omega_2,\omega_3 \in \wedge^2 C^*$ such that $\omega_i \wedge \omega_j = 2\delta_{ij}\epsilon_i vol_C$ where $vol_C$ is the volume form for $C$. Then the $3$-form is
\begin{equation}
\phi = a^{123} - a^i \wedge \omega_i.
\end{equation}
The volume form is $a^{123} \wedge vol_C$.\\

Consider a $7$-manifold $X$ with split $G_2$-structure. This structure is completely determined by a $3$-form $\phi \in \Omega^3(X)$. We also have a corresponding $4$-form $\psi = *\phi$. We say $X$ is a {\em split $G_2$-manifold} if $\phi$ and $\psi$ are closed.

We may define associative and coassociative submanifolds but in this case there are two types of each. Since we wish to construct coassociative fibrations we have a choice as to which type of coassociative submanifolds to consider. Let $M$ be an oriented $3$-manifold and consider a map $u : M \to \mathbb{R}^{3,3} \simeq \wedge^2 \mathbb{R}^4$. We also fix a constant volume form $dvol_C$ on $\mathbb{R}^4$. We require that $u$ sends each tangent space of $M$ to a non-degenerate subspace of $\mathbb{R}^{3,3}$ with respect to the signature $(3,3)$ metric, hence this induces a metric on $M$. In the definite coassociative case this metric has signature $(3,0)$ and in the split coassociative case the signature is $(1,2)$. The metric $h$ is given by
\begin{equation}
2h(A,B)dvol_C = u_*(A) \wedge u_*(B).
\end{equation}
Now consider $u$ as a $2$-form on $X = M \times T^4$ where $T^4 = (\mathbb{R}/ \mathbb{Z})^4$. Then on $X$ we define a $3$-form $\phi$ by
\begin{equation}
\phi = dvol_M - du
\end{equation}
where $dvol_M$ is the volume form induced by the orientation on $M$ and the metric $h$. It follows that $\phi$ is closed and gives $X$ a split $G_2$-structure. Therefore $\phi$ defines a signature $(3,4)$ metric $g$ on $X$ and we find that $g$ restricted to $TM$ agrees with $h$. Now we may determine when the $4$-form $\psi = *\phi$ is closed. The derivation is essentially the same as in the compact case. We find that $\psi$ is closed if and only if $u$ is a minimal immersion. We therefore have that locally semi-flat coassociative fibred split $G_2$ manifolds correspond to minimal immersions in $\mathbb{R}^{3,3}$ of signature $(3,0)$ in the definite coassociative case and signature $(1,2)$ in the split coassociative case. \\

If we consider the case where $M$ is a cone over a surface such that the radial direction is negative definite we find this corresponds to minimal surfaces in the quadric $Q_{-1} = \{x \in \mathbb{R}^{3,3} | \langle x , x \rangle = -1 \}$. This is the other quadric than in the compact $G_2$ case. We have seen that such minimal surfaces can be constructed from Higgs bundles in the Hitchin component for ${\rm SO}(3,3)$ or ${\rm SO}(2,3) = PSp(4,\mathbb{R})$. This gives a further interpretation of the Hitchin component for $PSp(4,\mathbb{R})$.


\chapter{Further questions}\label{further}
We discuss some unresolved questions raised from our work.\\

In relation to Higgs bundles, the Toda equations and the Hitchin component:
\begin{enumerate}
\item{Although we were able to link the affine Toda equations to cyclic Higgs bundles this only holds under specific reality conditions. It would be useful to have some existence results which cover all reality conditions, in particular for the case arising in relation to coassociative fibrations.}
\item{In an attempt to prove the energy functionals of Section \ref{energy} have non-degenerate minima their Hessians were investigated. The difficulty here is that the Hessian involves terms related to gauge invariance. In order to show positive definiteness one would need to make sharp estimates of these terms.}
\item{Our method for investigating the Hitchin components is only adequate for the rank $2$ case. In higher rank cases there are more differentials and the entire component can not be exhausted by cyclic Higgs bundles alone. Therefore it seems that a different approach would be required. However we can still gain partial information by considering the special case of cyclic Higgs bundles.}
\item{In the case of ${\rm PSp}(4,\mathbb{R})$ we showed that the convex-foliated projective structures of Guichard and Wienhard were equivalent to a class of projective structures on the unit tangent bundle where the fibres were lines. However the convex-foliated projective structures correspond to the Hitchin component for ${\rm PSL}(4,\mathbb{R})$. Therefore a natural question to ask is whether a similar statement holds for all convex-foliated structures. That is given any convex-foliated projective structure is it homeomorphic to a projective structure on the unit tangent bundle where the fibres are lines and if so does this projective structure have any other distinguishing features?}

\end{enumerate}

In relation to parabolic geometries and $G_2$ conformal holonomy:
\begin{enumerate}
\item{Although we have given examples of conformal structure with $G_2$ holonomy, the complexity of the formulas has prohibited further progress. In particular it is not at all clear under what conditions the holonomy is the whole of $G_2$, nor have we investigated what subgroups of $G_2$ can occur as conformal holonomy groups. In the case of definite signature conformal Einstein structures the possible local holonomy groups have been classified \cite{armstrong} but the case of indefinite and non-Einstein conformal structures remains incomplete. For some results in the indefinite case see \cite{lei}.}

\item{Aside from the two examples given there are other situations in which a generic $2$-plane distribution occurs in $5$ dimensions. In particular consider a $4$-manifold $M$ with $(2,2)$-conformal structure. The bundle of maximal isotropics is a $5$-manifold. In the case where the $M$ is orientable the maximal isotropics can put into two classes, the so-called $\alpha$-planes and $\beta$-planes. Thus there are two $5$-dimensional $S^1$-bundles over $M$, each with a naturally defined plane distribution \cite{akivis}. This geometry is a real version of the twistor space in Riemannian signature.}



\item{Consider the parabolic geometry corresponding to the other maximal parabolic of split real $G_2$. This is in fact a $5$-dimensional contact geometry. However, what is interesting is that there is a duality between the two homogeneous spaces. The homogeneous space corresponding to the other maximal parabolic can be described as the collection of lines tangent to the $2$-plane distribution in the homogeneous space corresponding to the first. The point line duality also works in reverse, that is a certain class of lines in one homogeneous space correspond to points in the other. It is interesting to note that this point line duality has an analogue in the corresponding parabolic geometries. If $M$ is a $5$-manifold with generic $2$-distribution, one can consider a special class of curves in $M$ known as abnormal extremals. In fact one can show that these curves are precisely the null geodesics everywhere tangent to the $2$-distribution. These form a $5$-dimensional contact manifold \cite{bryant}. The problem to investigate is whether this space can be given the structure of a parabolic geometry corresponding to the other maximal parabolic of $G_2$, and if so investigate the consequences of this duality.}
\end{enumerate}

In relation to coassociative fibrations:
\begin{enumerate}
\item{Since we have shown that a compact coassociative fibration (with $G_2$ holonomy) must degenerate the obvious question is what kind of singularities can occur? More specifically are there reasonable conditions under which the discriminant locus can be shown to have codimension $2$ or to be smooth? Also there is the question of what type of smooth fibres can occur. The two obvious candidates are a $4$-torus and a $K3$ surface but are there any others?}
\item{Seperate from this thesis we considered $G_2$ manifolds fibred by coassociative ALE manifolds. We found that under the further assumption that the fibration is a Riemannian submersion $X \to B$, the map $u : B \to H^2(F,\mathbb{R})$ is still harmonic where $F$ is the fibre. Beyond this it seems very difficult to say much about ALE fibrations in full generality. With additional simplifying assumptions one may be able to reduce to a more tractable problem.}
\item{Since we found that semi-flat coassociative fibrations correspond to minimal $3$-submanifolds in $\mathbb{R}^{3,3}$ it would be good to have a better understanding of such submanifolds and in particular some non-trivial explicit examples. It would also be good to examine the possible singularities that can occur in semi-flat fibrations.}
\end{enumerate}




\addcontentsline{toc}{chapter}{Bibliography}

\end{document}